\documentclass[a4paper,11pt]{article}
\usepackage[dvipdfmx]{graphicx}
\usepackage{amsmath,amssymb}
\usepackage{bm,dsfont,mathtools,stmaryrd}
\usepackage{xparse,xstring,xspace}
\usepackage{amsthm,thmtools}
\usepackage[top=80pt, bottom=90pt, left=70pt, right=70pt]{geometry}
\usepackage{upref}
\usepackage{comment}
\usepackage{color}
\usepackage{enumitem}
\usepackage[%
    bookmarksnumbered,
    hypertexnames=false,
    pdfdisplaydoctitle,
    pdfusetitle,
    unicode,
    colorlinks=true,
    linkcolor=blue,
    citecolor=blue
]{hyperref}
\hypersetup{
  pdfauthor={Yuni Iwamasa}
}
\usepackage[capitalise,noabbrev]{cleveref}
\usepackage{bookmark}

\crefname{line}{Line}{Lines}
\Crefname{line}{Line}{Lines}

\usepackage{autonum}
\makeatletter
\autonum@generatePatchedReferenceCSL{Cref}
\makeatother

\usepackage{mleftright}
\mleftright

\usepackage{tikz}
\usetikzlibrary{cd}
\usetikzlibrary{positioning}
\usetikzlibrary{intersections}
\usetikzlibrary{calc}
\usetikzlibrary{arrows.meta}
\usetikzlibrary{shapes.geometric}
\usetikzlibrary{shapes.misc}
\usetikzlibrary{decorations.pathmorphing}
\usetikzlibrary{decorations.markings}
\usetikzlibrary{tikzmark}

\def\numnode{1pt}
\def\sp{5pt}
\makeatletter
\def\pgf@lib@dec@arrowhead#1#2{
  \pgfarrowtotallength{#2}
  \pgftransformxshift{\pgf@x/2}
  \pgftransformxscale{#1}
  \pgfarrowdraw{#2}
}
\makeatother

\usepackage{array}
\usepackage{booktabs}
\usepackage{tabulary}
\usepackage{tabularx}
\newcolumntype{Y}{>{\raggedright\arraybackslash}X}
\newcolumntype{P}[1]{>{\raggedright\arraybackslash}p{#1}}
\newcommand{\coxeterchambercell}[2]{%
    \begin{minipage}[c]{\linewidth}
        \centering
        Spherical:
        #1
        \par
        \vspace{0.45em}
        Euclidean:
        #2
    \end{minipage}%
}

\newcommand{\typesep}{%
    \specialrule{\lightrulewidth}{0.6em}{0.6em}%
}

\tikzset{
    dynkin vertex/.style={
        circle,
        draw,
        line width=0.4pt,
        minimum size=4.8mm,
        inner sep=0pt,
        font=\scriptsize
    },
    dynkin edge/.style={
        line width=0.45pt
    },
    dynkin double edge/.style={
        line width=0.35pt,
        double,
        double distance=1.1pt,
        -{Stealth[length=2.4pt,width=2.8pt]}
    },
    dynkin label/.style={
        font=\scriptsize
    }
}
\newcommand{\dynkindiagram}[1]{%
    \ensuremath{%
        \vcenter{\hbox{#1}}%
    }%
}

\newcommand{\DynkinDiagramA}{%
    \dynkindiagram{%
        \begin{tikzpicture}
            \node[
                circle,
                draw,
                fill=white,
                inner sep=\numnode
            ] (1) at (1,0) {$1$};

            \node[
                circle,
                draw,
                fill=white,
                inner sep=\numnode
            ] (2) at (2,0) {$1$};

            \node at (3,0) {$\cdots$};

            \node[
                circle,
                draw,
                fill=white,
                inner sep=\numnode
            ] (n-1) at (4,0) {$1$};

            \node[
                circle,
                draw,
                fill=white,
                inner sep=\numnode
            ] (n) at (5,0) {$1$};

            \node[below=\sp] at (1) {$v_1$};
            \node[below=\sp] at (2) {$v_2$};
            \node[below=\sp] at (n-1) {$v_{n-1}$};
            \node[below=\sp] at (n) {$v_n$};

            \draw (1) -- (2);
            \draw (2) -- (2.5,0);
            \draw (3.5,0) -- (n-1);
            \draw (n-1) -- (n);
        \end{tikzpicture}%
    }%
}

\newcommand{\DynkinDiagramB}{%
    \dynkindiagram{%
        \begin{tikzpicture}
            \node[
                circle,
                draw,
                fill=white,
                inner sep=\numnode
            ] (1) at (1,0) {$1$};

            \node[
                circle,
                draw,
                fill=white,
                inner sep=\numnode
            ] (2) at (2,0) {$2$};

            \node at (3,0) {$\cdots$};

            \node[
                circle,
                draw,
                fill=white,
                inner sep=\numnode
            ] (n-1) at (4,0) {$2$};

            \node[
                circle,
                draw,
                fill=white,
                inner sep=\numnode
            ] (n) at (5,0) {$2$};

            \node[below=\sp] at (1) {$v_1$};
            \node[below=\sp] at (2) {$v_2$};
            \node[below=\sp] at (n-1) {$v_{n-1}$};
            \node[below=\sp] at (n) {$v_n$};

            \draw (1) -- (2);
            \draw (2) -- (2.5,0);
            \draw (3.5,0) -- (n-1);

            \draw[
                decoration={
                    markings,
                    mark=at position 0.5 with {\arrow{>}}
                },
                postaction=decorate,
                double distance=1.5pt
            ] (n-1) -- (n);
        \end{tikzpicture}%
    }%
}

\newcommand{\DynkinDiagramC}{%
    \dynkindiagram{%
        \begin{tikzpicture}
            \node[
                circle,
                draw,
                fill=white,
                inner sep=\numnode
            ] (1) at (1,0) {$2$};

            \node[
                circle,
                draw,
                fill=white,
                inner sep=\numnode
            ] (2) at (2,0) {$2$};

            \node at (3,0) {$\cdots$};

            \node[
                circle,
                draw,
                fill=white,
                inner sep=\numnode
            ] (n-1) at (4,0) {$2$};

            \node[
                circle,
                draw,
                fill=white,
                inner sep=\numnode
            ] (n) at (5,0) {$1$};

            \node[below=\sp] at (1) {$v_1$};
            \node[below=\sp] at (2) {$v_2$};
            \node[below=\sp] at (n-1) {$v_{n-1}$};
            \node[below=\sp] at (n) {$v_n$};

            \draw (1) -- (2);
            \draw (2) -- (2.5,0);
            \draw (3.5,0) -- (n-1);

            \draw[
                decoration={
                    markings,
                    mark=at position 0.5 with {\arrow{>}}
                },
                postaction=decorate,
                double distance=1.5pt
            ] (n) -- (n-1);
        \end{tikzpicture}%
    }%
}

\newcommand{\DynkinDiagramDThree}{%
    \dynkindiagram{%
        \begin{tikzpicture}
            \node[
                circle,
                draw,
                fill=white,
                inner sep=\numnode
            ] (1) at (1,0) {$1$};

            \node[
                circle,
                draw,
                fill=white,
                inner sep=\numnode
            ] (2) at (2,0.55) {$1$};

            \node[
                circle,
                draw,
                fill=white,
                inner sep=\numnode
            ] (3) at (2,-0.55) {$1$};

            \node[below left=\sp] at (1) {$v_1$};
            \node[above=\sp] at (2) {$v_2$};
            \node[below=\sp] at (3) {$v_3$};

            \draw (1) -- (2);
            \draw (1) -- (3);
        \end{tikzpicture}%
    }%
}

\newcommand{\DynkinDiagramD}{%
    \dynkindiagram{%
        \begin{tikzpicture}
            \node[
                circle,
                draw,
                fill=white,
                inner sep=\numnode
            ] (1) at (1,0) {$1$};

            \node[
                circle,
                draw,
                fill=white,
                inner sep=\numnode
            ] (2) at (2,0) {$2$};

            \node at (3,0) {$\cdots$};

            \node[
                circle,
                draw,
                fill=white,
                inner sep=\numnode
            ] (n-2) at (4,0) {$2$};

            \node[
                circle,
                draw,
                fill=white,
                inner sep=\numnode
            ] (n-1) at (5,0.55) {$1$};

            \node[
                circle,
                draw,
                fill=white,
                inner sep=\numnode
            ] (n) at (5,-0.55) {$1$};

            \node[below=\sp] at (1) {$v_1$};
            \node[below=\sp] at (2) {$v_2$};
            \node[below=\sp] at (n-2) {$v_{n-2}$};
            \node[above=\sp] at (n-1) {$v_{n-1}$};
            \node[below=\sp] at (n) {$v_n$};

            \draw (1) -- (2);
            \draw (2) -- (2.5,0);
            \draw (3.5,0) -- (n-2);
            \draw (n-2) -- (n-1);
            \draw (n-2) -- (n);
        \end{tikzpicture}%
    }%
}

\newcommand{\ExtendedDynkinDiagramAOne}{%
    \dynkindiagram{%
        \begin{tikzpicture}
            \node[
                circle,
                draw,
                fill=white,
                inner sep=\numnode
            ] (0) at (1,0) {$1$};

            \node[
                circle,
                draw,
                fill=white,
                inner sep=\numnode
            ] (1) at (2.2,0) {$1$};

            \node[below=\sp] at (0) {$v_0$};
            \node[below=\sp] at (1) {$v_1$};

            \draw (0) -- (1);
            \node at (1.6,0.28) {$\infty$};
        \end{tikzpicture}%
    }%
}

\newcommand{\ExtendedDynkinDiagramA}{%
    \dynkindiagram{%
        \begin{tikzpicture}
            \node[
                circle,
                draw,
                fill=white,
                inner sep=\numnode
            ] (1) at (1,0) {$1$};

            \node[
                circle,
                draw,
                fill=white,
                inner sep=\numnode
            ] (2) at (2,0) {$1$};

            \node at (3,0) {$\cdots$};

            \node[
                circle,
                draw,
                fill=white,
                inner sep=\numnode
            ] (n-1) at (4,0) {$1$};

            \node[
                circle,
                draw,
                fill=white,
                inner sep=\numnode
            ] (n) at (5,0) {$1$};

            \node[
                circle,
                draw,
                fill=white,
                inner sep=\numnode
            ] (0) at (3,1.05) {$1$};

            \node[below=\sp] at (1) {$v_1$};
            \node[below=\sp] at (2) {$v_2$};
            \node[below=\sp] at (n-1) {$v_{n-1}$};
            \node[below=\sp] at (n) {$v_n$};
            \node[above=\sp] at (0) {$v_0$};

            \draw (1) -- (2);
            \draw (2) -- (2.5,0);
            \draw (3.5,0) -- (n-1);
            \draw (n-1) -- (n);
            \draw (0) -- (1);
            \draw (0) -- (n);
        \end{tikzpicture}%
    }%
}

\newcommand{\ExtendedDynkinDiagramBTwo}{%
    \dynkindiagram{%
        \begin{tikzpicture}
            \node[
                circle,
                draw,
                fill=white,
                inner sep=\numnode
            ] (1) at (1,0) {$1$};

            \node[
                circle,
                draw,
                fill=white,
                inner sep=\numnode
            ] (2) at (2,0) {$2$};

            \node[
                circle,
                draw,
                fill=white,
                inner sep=\numnode
            ] (0) at (3,0) {$1$};

            \node[below=\sp] at (1) {$v_1$};
            \node[below=\sp] at (2) {$v_2$};
            \node[below=\sp] at (0) {$v_0$};

            \draw[
                decoration={
                    markings,
                    mark=at position 0.5 with {\arrow{>}}
                },
                postaction=decorate,
                double distance=1.5pt
            ] (1) -- (2);

            \draw[
                decoration={
                    markings,
                    mark=at position 0.5 with {\arrow{>}}
                },
                postaction=decorate,
                double distance=1.5pt
            ] (0) -- (2);
        \end{tikzpicture}%
    }%
}

\newcommand{\ExtendedDynkinDiagramB}{%
    \dynkindiagram{%
        \begin{tikzpicture}
            \node[
                circle,
                draw,
                fill=white,
                inner sep=\numnode
            ] (0) at (1,0.55) {$1$};

            \node[
                circle,
                draw,
                fill=white,
                inner sep=\numnode
            ] (1) at (1,-0.55) {$1$};

            \node[
                circle,
                draw,
                fill=white,
                inner sep=\numnode
            ] (2) at (2,0) {$2$};

            \node at (3,0) {$\cdots$};

            \node[
                circle,
                draw,
                fill=white,
                inner sep=\numnode
            ] (n-1) at (4,0) {$2$};

            \node[
                circle,
                draw,
                fill=white,
                inner sep=\numnode
            ] (n) at (5,0) {$2$};

            \node[above=\sp] at (0) {$v_0$};
            \node[below=\sp] at (1) {$v_1$};
            \node[below=\sp] at (2) {$v_2$};
            \node[below=\sp] at (n-1) {$v_{n-1}$};
            \node[below=\sp] at (n) {$v_n$};

            \draw (0) -- (2);
            \draw (1) -- (2);
            \draw (2) -- (2.5,0);
            \draw (3.5,0) -- (n-1);

            \draw[
                decoration={
                    markings,
                    mark=at position 0.5 with {\arrow{>}}
                },
                postaction=decorate,
                double distance=1.5pt
            ] (n-1) -- (n);
        \end{tikzpicture}%
    }%
}

\newcommand{\ExtendedDynkinDiagramC}{%
    \dynkindiagram{%
        \begin{tikzpicture}
            \node[
                circle,
                draw,
                fill=white,
                inner sep=\numnode
            ] (0) at (0,0) {$1$};

            \node[
                circle,
                draw,
                fill=white,
                inner sep=\numnode
            ] (1) at (1,0) {$2$};

            \node[
                circle,
                draw,
                fill=white,
                inner sep=\numnode
            ] (2) at (2,0) {$2$};

            \node at (3,0) {$\cdots$};

            \node[
                circle,
                draw,
                fill=white,
                inner sep=\numnode
            ] (n-1) at (4,0) {$2$};

            \node[
                circle,
                draw,
                fill=white,
                inner sep=\numnode
            ] (n) at (5,0) {$1$};

            \node[below=\sp] at (0) {$v_0$};
            \node[below=\sp] at (1) {$v_1$};
            \node[below=\sp] at (2) {$v_2$};
            \node[below=\sp] at (n-1) {$v_{n-1}$};
            \node[below=\sp] at (n) {$v_n$};

            \draw[
                decoration={
                    markings,
                    mark=at position 0.5 with {\arrow{>}}
                },
                postaction=decorate,
                double distance=1.5pt
            ] (0) -- (1);

            \draw (1) -- (2);
            \draw (2) -- (2.5,0);
            \draw (3.5,0) -- (n-1);

            \draw[
                decoration={
                    markings,
                    mark=at position 0.5 with {\arrow{>}}
                },
                postaction=decorate,
                double distance=1.5pt
            ] (n) -- (n-1);
        \end{tikzpicture}%
    }%
}

\newcommand{\ExtendedDynkinDiagramDThree}{%
    \dynkindiagram{%
        \begin{tikzpicture}
            \node[
                circle,
                draw,
                fill=white,
                inner sep=\numnode
            ] (1) at (1,0) {$1$};

            \node[
                circle,
                draw,
                fill=white,
                inner sep=\numnode
            ] (2) at (2,0.75) {$1$};

            \node[
                circle,
                draw,
                fill=white,
                inner sep=\numnode
            ] (0) at (3,0) {$1$};

            \node[
                circle,
                draw,
                fill=white,
                inner sep=\numnode
            ] (3) at (2,-0.75) {$1$};

            \node[left=\sp] at (1) {$v_1$};
            \node[above=\sp] at (2) {$v_2$};
            \node[right=\sp] at (0) {$v_0$};
            \node[below=\sp] at (3) {$v_3$};

            \draw (1) -- (2);
            \draw (2) -- (0);
            \draw (0) -- (3);
            \draw (3) -- (1);
        \end{tikzpicture}%
    }%
}

\newcommand{\ExtendedDynkinDiagramD}{%
    \dynkindiagram{%
        \begin{tikzpicture}
            \node[
                circle,
                draw,
                fill=white,
                inner sep=\numnode
            ] (0) at (1,0.55) {$1$};

            \node[
                circle,
                draw,
                fill=white,
                inner sep=\numnode
            ] (1) at (1,-0.55) {$1$};

            \node[
                circle,
                draw,
                fill=white,
                inner sep=\numnode
            ] (2) at (2,0) {$2$};

            \node at (3,0) {$\cdots$};

            \node[
                circle,
                draw,
                fill=white,
                inner sep=\numnode
            ] (n-2) at (4,0) {$2$};

            \node[
                circle,
                draw,
                fill=white,
                inner sep=\numnode
            ] (n-1) at (5,0.55) {$1$};

            \node[
                circle,
                draw,
                fill=white,
                inner sep=\numnode
            ] (n) at (5,-0.55) {$1$};

            \node[above=\sp] at (0) {$v_0$};
            \node[below=\sp] at (1) {$v_1$};
            \node[below=\sp] at (2) {$v_2$};
            \node[below=\sp] at (n-2) {$v_{n-2}$};
            \node[above=\sp] at (n-1) {$v_{n-1}$};
            \node[below=\sp] at (n) {$v_n$};

            \draw (0) -- (2);
            \draw (1) -- (2);
            \draw (2) -- (2.5,0);
            \draw (3.5,0) -- (n-2);
            \draw (n-2) -- (n-1);
            \draw (n-2) -- (n);
        \end{tikzpicture}%
    }%
}
\newcommand{\positiveRootBlock}[1]{%
    \begin{minipage}[c]{\linewidth}
        \centering
        \setlength{\tabcolsep}{4pt}%
        \renewcommand{\arraystretch}{1.25}%
        \begin{tabular}{
            @{}
            >{\centering\arraybackslash}m{0.62\linewidth}
            >{\centering\arraybackslash}m{0.28\linewidth}
            @{}
        }
            #1
        \end{tabular}
    \end{minipage}%
}

\declaretheorem[style=plain,numberwithin=section,name=Theorem]{theorem}
\declaretheorem[style=plain,sibling=theorem,name=Lemma]{lemma}
\declaretheorem[style=plain,sibling=theorem,name=Proposition]{proposition}
\declaretheorem[style=plain,sibling=theorem,name=Corollary]{corollary}
\declaretheorem[style=plain,sibling=theorem,name=Claim]{claim}

\declaretheorem[style=definition,sibling=theorem,name=Example,qed=$\blacksquare$]{example}

\declaretheorem[style=definition,sibling=theorem,name=Remark,qed=$\blacksquare$]{remark}

\crefname{theorem}{Theorem}{Theorems}
\crefname{proposition}{Proposition}{Propositions}
\crefname{lemma}{Lemma}{Lemmas}
\crefname{exmp}{Example}{Examples}
\crefname{corollary}{Corollary}{Corollaryies}
\crefname{claim}{Claim}{Claims}
\crefname{remark}{Remark}{Remarks}
\crefname{section}{Section}{Sections}

\DeclarePairedDelimiterXPP{\sk}[1]{\mathcal{V}}{\lparen}{\rparen}{}{#1}

\newcommand{\subsp}[1]{{#1}_{\mathrm{sp}}}
\newcommand{\sgn}{\sigma}
\newcommand{\RS}{\Phi}
\newcommand{\simp}{\Delta}

\newcommand{\esimp}{\tilde{\simp}}
\newcommand{\Dynkin}[2]{{\textup{#1}}_{#2}}
\newcommand{\RSA}{\Phi_{\textup{A}_n}}
\newcommand{\RSB}{\Phi_{\textup{B}_n}}
\newcommand{\RSC}{\Phi_{\textup{C}_n}}

\newcommand{\RSD}{\Phi_{\textup{D}_n}}
\newcommand{\CC}{\Sigma}
\newcommand{\spCC}{{\Sigma}_{\mathrm{sp}}}
\newcommand{\euCC}{{\Sigma}_{\mathrm{Euc}}}

\newcommand{\Ldom}{\mathcal{V}}

\newcommand{\Mdom}{\mathcal{L}}

\newcommand{\calH}{\mathcal{H}}
\newcommand{\lincalH}{\mathcal{H}_{\mathrm{lin}}}
\newcommand{\affcalH}{\mathcal{H}_{\mathrm{aff}}}

\newcommand{\onevec}{\mathds{1}}
\newcommand{\MP}{\textup{MP}\xspace}

\newcommand{\Lpoly}[1]{\mathcal{L}_0\left[ #1 \right]}
\newcommand{\Lsub}[2]{{}_{0}\mathcal{L}\left[ #1 \to #2 \right]}
\newcommand{\qLfunc}[2]{\hat{\mathcal{L}}\left[ #1 \to #2 \right]}
\newcommand{\Lfunc}[2]{\mathcal{L}\left[ #1 \to #2 \right]}
\newcommand{\Mpoly}[1]{\mathcal{M}_0\left[ #1 \right]}
\newcommand{\Msub}[2]{{}_{0}\mathcal{M}\left[ #1 \to #2 \right]}
\newcommand{\qMfunc}[2]{\hat{\mathcal{M}}\left[ #1 \to #2 \right]}
\newcommand{\Mfunc}[2]{\mathcal{M}\left[ #1 \to #2 \right]}
\newcommand{\vdist}[1]{\check{\mathcal{T}}\left[ #1 \right]}
\newcommand{\tdist}[1]{\mathcal{T}\left[ #1 \right]}
\newcommand{\subm}[2]{\mathcal{S}_{#1}\left[ #2 \right]}
\newcommand{\chamber}[2]{C_{#1}^{#2}}

\DeclarePairedDelimiter{\rbra}{\lparen}{\rparen} 
\DeclarePairedDelimiter{\sbra}{\lbrack}{\rbrack} 
\DeclarePairedDelimiter{\floor}{\lfloor}{\rfloor} 
\DeclarePairedDelimiter{\ceil}{\lceil}{\rceil} 

\newcommand{\geo}[1]{|{#1}|}
\newcommand{\card}[1]{\# {#1}}
\newcommand{\ext}[1]{\overline{#1}}
\newcommand{\phext}[1]{\overline{#1}^{+}}

\newcommand{\cell}[2]{{#1}_{#2}}
\newcommand{\isc}[2]{{#1}\left[{#2}\right]}
\newcommand{\opnbor}[1]{\mathrm{N}^{\circ}_{#1}}
\newcommand{\nbor}[1]{\mathrm{N}_{#1}}

\newcommand{\St}[2]{\mathrm{St}_{#1}\left({#2}\right)}
\newcommand{\Lk}[2]{\mathrm{Lk}_{#1}\left({#2}\right)}

\newcommand{\B}[1]{\mathbf{B}( {#1} )}
\newcommand{\D}[1]{\mathbf{D}( {#1} )}
\newcommand{\bS}[1]{\mathbf{S}( {#1} )}
\DeclarePairedDelimiterX{\gen}[2]{\langle}{\rangle}{#1\;\delimsize\vert\;#2}

\newcommand{\face}[2]{ {#1}_{#2} }
\newcommand{\tRS}[2]{\RS_{#1}( {#2} )}
\newcommand{\bartRS}[2]{\bar{\RS}_{#1}( {#2} )}

\DeclarePairedDelimiterX{\inpr}[2]{\langle}{\rangle}{#1, #2}
\DeclarePairedDelimiterX{\set}[1]{\lbrace}{\rbrace}{\,#1\,}
\DeclarePairedDelimiterX{\Set}[2]{\lbrace}{\rbrace}{#1\;\delimsize\vert\;#2}

\newcommand{\Z}{\mathbb{Z}}

\newcommand{\R}{\mathbb{R}}
\newcommand{\ER}{\mathbb{R} \cup \{+\infty\}}
\newcommand{\EZ}{\mathbb{Z} \cup \{+\infty\}}
\newcommand{\eps}{\varepsilon}
\newcommand{\doms}[2]{{#1}\to{#2}}
\newcommand{\funcdoms}[3]{{#1}\vcentcolon\doms{#2}{#3}}

\DeclareMathOperator{\dom}{dom}
\DeclareMathOperator{\epi}{epi}

\DeclareMathOperator{\supp}{supp}
\DeclareMathOperator{\spn}{span}

\DeclareMathOperator*{\argmin}{argmin}

\DeclareMathOperator{\conv}{conv}
\DeclareMathOperator{\cl}{cl}
\DeclareMathOperator{\clconv}{\overline{conv}}
\DeclareMathOperator{\cone}{cone}

\DeclareMathOperator{\aff}{aff}
\DeclareMathOperator{\relint}{relint}
\DeclareMathOperator{\intr}{int}
\newcommand{\tcone}[2]{\mathrm{T}_{#2} ( #1 )}
\newcommand{\cvcoeff}[1]{\Lambda_{#1}}
\newcommand{\cncoeff}[1]{M_{#1}}
\DeclareMathOperator{\fan}{\Pi}
\DeclareMathOperator{\barfan}{\overline{\fan}}

\newcommand{\spt}[1]{\sigma \left[ #1 \right]}
\newcommand{\symdiff}{\mathbin{\triangle}}

\DeclareMathOperator{\OG}{O}

\makeatletter
\newsavebox{\@brx}
\newcommand{\llangle}[1][]{\savebox{\@brx}{\(\m@th{#1\langle}\)}%
  \mathopen{\copy\@brx\mkern2mu\kern-0.9\wd\@brx\usebox{\@brx}}}
\newcommand{\rrangle}[1][]{\savebox{\@brx}{\(\m@th{#1\rangle}\)}%
  \mathclose{\copy\@brx\mkern2mu\kern-0.9\wd\@brx\usebox{\@brx}}}
\makeatother

\DeclarePairedDelimiter{\even}{\llbracket}{\rrbracket}
\newcommand{\odd}[2][]{\llangle[#1] #2 \rrangle[#1]}
\newcommand{\qodd}[2][]{\llangle[#1] #2 \rrangle[#1]_{\mkern-2mu \ast}}
\newcommand{\qqodd}[2][]{\llangle[#1] #2 \rrangle[#1]_{\mkern-2mu \ast}^{\mkern-2mu \pm}}
\DeclarePairedDelimiterXPP{\qeven}[1]{}{\llbracket}{\rrbracket}{_{\ast}}{#1}
\DeclarePairedDelimiterXPP{\qqeven}[1]{}{\llbracket}{\rrbracket}{_{\ast}^{\pm}}{#1}
\DeclarePairedDelimiterXPP{\qfloor}[1]{}{\lfloor}{\rrbracket}{_{\ast}}{#1}
\DeclarePairedDelimiterXPP{\qceil}[1]{}{\lceil}{\rrbracket}{_{\ast}}{#1}
\newcommand{\qqfloor}[2][]{\llfloor #2 \rrangle[#1]_{\mkern-2mu \ast}}
\newcommand{\qqceil}[2][]{\llceil #2 \rrangle[#1]_{\mkern-2mu \ast}}

\makeatletter

\@addtoreset{equation}{section}
\makeatother
\usepackage{bm}
\allowdisplaybreaks

\begin{document}
	\title{Towards discrete convex analysis over classical root systems}
	\author{Yuni Iwamasa\thanks{Department of Mathematics, Graduate School of Science, Kobe University, Kobe, 657-8501, Japan.
			Email: \texttt{iwamasa@math.kobe-u.ac.jp}}}
	\date{\today}
	\maketitle
	
\begin{abstract}
  Discrete Convex Analysis (DCA) is a discrete analog of continuous convex analysis, 
  originally proposed as a unified theoretical framework for efficiently solvable combinatorial optimization problems. 
  Recently, DCA has proven to be a powerful tool across diverse fields, ranging from operations research to economics and pure mathematics. 
  
  Motivated by the broad applicability of DCA, this paper establishes a unified theory of discrete convex analysis over discrete structures arising from classical root systems, extending the usual setting of the integer lattice, which essentially corresponds to type~A.
  We adopt the vertex set of the Euclidean Coxeter complex as the primal discrete domain for L-convexity, 
  and the root lattice as the dual discrete domain for M-convexity. 
  Using the associated polyhedral structures, we formulate L- and M-convex
functions together with notions of integrality determined by the root system.
  We show that local optimality guarantees global optimality for these functions. 
  Furthermore, we establish that integral L-convex functions and integral M-convex functions correspond one-to-one via the discrete Fenchel--Legendre conjugate, thereby extending the conjugacy in the original DCA from type~A
to all classical root systems.
\end{abstract}
\begin{quote}
\end{quote}

\setcounter{tocdepth}{2}
\tableofcontents

\section{Introduction}\label{sec:intro}
\emph{Discrete Convex Analysis (DCA)}, initiated by Murota~\cite{Murota1998-yj,Murota2003-yb},
is a theory that combines convex analysis with submodular and matroid theory.
DCA originally emerged as a theoretical framework for efficiently solvable combinatorial optimization problems related to network flows and submodular/matroid optimization.
Recently, however,
it has appeared not only in the field of combinatorial optimization but also in various applied mathematical fields such as economics, game theory~\cite{Murota2016-kn}, and operations research~\cite{Simchi-Levi2014-jf,Chen2017-cy,Chen2021-xr,Chen2021-xa}, as well as in pure mathematics such as algebraic geometry (e.g., Lorentzian polynomials~\cite{Branden2020-nj}); see also~\cite{Murota2008-ej,Murota2024} for recent surveys.
DCA consists of two classes of convex functions defined on the integer lattice, called \emph{L-convex functions} and \emph{M-convex functions}.
The former is a generalization of \emph{submodular functions}~\cite{Fujishige2005-sp} and is defined by the so-called \emph{discrete midpoint convexity},
and the latter is a generalization of \emph{(valuated) matroids}~\cite{Murota2010-ar,Oxley2011-fw}
and is defined based on the simultaneous exchange property of matroids.
For these functions, local optimality guarantees global optimality.
Furthermore, 
an intriguing fact is that the integer-valued L- and M-convex functions (referred to as the \emph{integral} L- and M-convex functions) correspond one-to-one via the \emph{discrete Fenchel--Legendre conjugate}.

The purpose of this research is to establish a unified theory of
\emph{DCA over discrete structures arising from classical root systems}.
Specifically, in this paper, for classical root systems, 
we introduce the concepts of \emph{L-convexity over the vertex set of the Euclidean Coxeter complex} and \emph{M-convexity over the root lattice}. 
Furthermore, by formulating an appropriate notion of integrality
for each root system,
we show that integral L-convex functions and integral M-convex functions correspond one-to-one via the discrete Fenchel--Legendre conjugate.

Here, a \emph{root system} (see e.g.,~\cite{Humphreys1973-ly,Humphreys1992}) is a finite set consisting of (dual) vectors, called \emph{roots}, of a Euclidean space that possesses an integral structure capturing the symmetries of finite reflection groups (see \Cref{subsec:def:root-system} for the formal definition).
It was originally introduced in the late 19th century to classify complex semisimple Lie algebras,
and now serves as a fundamental mathematical structure across various disciplines, including representation theory~\cite{Humphreys1973-ly,Hall2015}, algebraic geometry~\cite{Slodowy1980-ub}, mathematical physics~\cite{Carter2005-jq}, and combinatorics~\cite{Bjorner2005-nz}. 
Irreducible root systems are completely classified.
The classification contains four infinite families, called the
\emph{classical} root systems, of types A, B, C, and D.

Numerous extensions of L- and M-convexity, including those of submodular functions and matroids, have been proposed so far, and most of them closely relate to classical root systems. 
The original L- and M-convexity can be viewed as those over discrete structures arising from the root system of type~A. 
As variants in terms of type~B, \emph{$\Delta$-matroids}~\cite{Bouchet1987-gl,Dress1986-oh,Chandrasekaran1988-wc} (see also~\cite[Chapter~4]{Borovik2011-io}) and \emph{BS-convex sets/functions}~\cite{Fujishige2014-bj} (see also \cite{Iwamasa2024-jw}) generalize matroids and M-convex sets/functions, respectively.
\emph{Bisubmodular functions}~\cite{Qi1988-wj,Bouchet1995-ba} are a generalization of submodular functions in terms of type~B or~C (see also~\cite{Ardila2020-ll}), 
and \emph{UJ-convex functions}~\cite{Fujishige2014-bj} are a generalization of L-convex functions in terms of type~C. 
\emph{Even $\Delta$-matroids}, which form a special subclass of $\Delta$-matroids, can be viewed as a generalization of matroids in terms of type~D (see~\cite[Chapter~4]{Borovik2011-io}). 
Furthermore, \emph{jump systems}~\cite{Bouchet1995-ba} were formulated by extending the exchange axioms of $\Delta$-matroids to the integer lattice. 
\emph{Valuated (even) $\Delta$-matroids}~\cite{Dress1991-ty,Wenzel1993-ee,Takazawa2014-rl} and \emph{M-convex functions on jump systems}~\cite{Murota2006-qi,Murota2021-hx} are valuated generalizations of the corresponding discrete structures,
which are defined based on the simultaneous exchange property.
In combinatorial optimization, these 
serve as mathematical structures capturing polynomial-time solvable combinatorial optimization problems that are not covered by standard DCA, including non-bipartite matchings and related problems;
see also, e.g., \cite{Cunningham2002-gf, Murota2006-qi,Kobayashi2009-rh, Kobayashi2012-cm,Kazda2018-ov, Murota2021-hx, Iwamasa2026-ms} and references therein. 
Related generalizations to arbitrary finite reflection groups include
\emph{Coxeter matroids}~\cite{Borovik2011-io} and, more recently,
\emph{Coxeter submodular functions}~\cite{Ardila2020-ll}.
In addition, L-convex functions particularly admit generalizations as discretizations of convex functions defined on metric spaces more general than Euclidean spaces, specifically \emph{CAT(0) spaces} (see~\cite{Bridson1999-yl}).
These include \emph{$k$-submodular functions}~\cite{Huber2012-dg}, \emph{L${}^\natural$-convex functions on rooted trees}~\cite{Kolmogorov2011-fg}, and \emph{L-convex functions on buildings}~\cite{Hirai2017-gx,Hirai2018-lw,Hamada2021-gw,Hirai2025-go},
which can also be interpreted as convex functions defined over discrete structures arising from root systems, particularly, \emph{Euclidean buildings}~\cite{Abramenko2010}.

The discrete version of convex analysis capturing the extended concepts associated with other root systems has not been well established. 
Indeed,
in the fields of discrete mathematics and combinatorics, 
although generalized discrete structures have been extensively investigated, 
these mathematical objects have rarely been analyzed from the viewpoint of convex analysis. 
Meanwhile, in the context of combinatorial optimization, the underlying discrete structures for the theory have conventionally been restricted to the integer lattice, although the integer lattice might not necessarily be the appropriate discrete structure for developing convex analysis for other types. 
Furthermore, quantitative extensions such as valuated $\Delta$-matroids and M-convex functions on jump systems in \cite{Dress1991-ty,Wenzel1993-ee,Takazawa2014-rl,Murota2006-qi,Murota2021-hx} are defined based on the simultaneous exchange property,
which enables us to easily guarantee that local optimality implies global optimality. 
However, this introduces a gap with the underlying combinatorial structures, i.e., $\Delta$-matroids and jump systems, which do not always satisfy this property (see~\cite{calvert2026characterisationsstrongdeltamatroids}). 
Consequently, in the context of combinatorial optimization, 
these valuated concepts are essentially restricted to functions defined on a special subclass of $\Delta$-matroids or jump systems where the simultaneous exchange property holds. 
This structural discrepancy between the general combinatorial structures and the restricted domains of their valuated counterparts has acted as a major barrier to constructing a theory that treats both the combinatorial structures and the valuated functions in a unified manner.

To establish such a unified framework, one must return to (continuous) convex analysis (see e.g.,~\cite{Rockafellar1996-wm}). 
In continuous optimization, convexity ensures that local optimality implies global optimality and establishes a general duality framework known as \emph{Fenchel duality}. 
A crucial role in Fenchel duality is played by the fact that proper closed convex functions on a Euclidean space $V$ and those on its dual space $V^*$ correspond one-to-one via the \emph{Fenchel--Legendre conjugate}. 
DCA successfully discretizes this conjugacy. 
It is worth noting that discretizing the theory of convex analysis is far from trivial. 
The continuous notion of convexity depends on continuous operations, such as the midpoint of two points, which generally fall outside a discrete domain; 
there is no immediately canonical way to define convexity on discrete structures. 
Furthermore, preserving the conjugacy requires carefully designing discrete structures both in the primal and dual spaces so that these are compatible with the Fenchel--Legendre conjugate.

Discrete structures arising from root systems---the \emph{Euclidean Coxeter complex} and the \emph{root lattice}---provide a suitable framework for developing the discrete Fenchel--Legendre conjugate. 
Recall that
continuous conjugacy connects a primal space $V$ and its dual space $V^*$
and
that an element in $V^*$ acts as a linear function on $V$, 
geometrically representing the normal vector of a hyperplane. 
To construct consistent discrete domains in both spaces, 
one can introduce a highly regular, locally finite arrangement of affine hyperplanes in $V$ by a root system. 
This arrangement partitions $V$ into a simplicial complex known as the \emph{Euclidean Coxeter complex}, 
whose vertex set serves as the primal discrete domain. 
Correspondingly, the normal vectors of these hyperplanes, which are the roots themselves, 
generate a lattice in $V^*$ called the \emph{root lattice}. 
By adopting the root lattice as the dual discrete domain, 
the geometric and algebraic correspondence between the hyperplanes in $V$ and their normal vectors in $V^*$ is maintained. 
This structural relationship provides the discrete primal-dual framework necessary to preserve conjugacy.

The main contribution of this paper is to establish DCA over discrete structures arising from classical root systems. 
Specifically, the details are summarized as follows:

\begin{itemize}
  \item
  As a geometric foundation for the theory in Euclidean spaces, we define
  \emph{L- and M-convex cones and polyhedra} and
  \emph{L- and M-sublinear functions} using the structure of root systems,
  together with appropriate notions of \emph{integrality}.
  Integral L-convex polyhedra correspond to the geometric
  realizations of convex subcomplexes of the Euclidean Coxeter
  complex, whereas the integrality of M-convex polyhedra is defined
  with respect to the root lattice.
  The integral variants of L- and M-sublinear functions are defined
  so that the correspondences between polyhedra and sublinear
  functions preserve integrality.

  \item
  We then formulate \emph{locally polyhedral L- and M-convex functions}
  in Euclidean spaces.
  They are defined by the condition that their minimizer sets
  under linear perturbations are L- or M-convex polyhedra, or
  equivalently, that their directional derivatives are L- or
  M-sublinear functions.
  Such a function is called \emph{primal-integral} if its minimizer
  sets are integral, \emph{dual-integral} if its directional
  derivatives are integral, and \emph{integral} if it satisfies both
  conditions.
  It is worth mentioning that the integrality is defined relative to the root system and need not coincide with integer-valuedness outside type A.
  Defining L- and M-convexity in terms of minimizer sets is similar to
the definitions of BS-convex
  functions~\cite{Fujishige2014-bj} and valuated
  $\Delta$-matroids in the sense of~\cite{cheung2025valuateddeltamatroidsprincipal}, whose underlying $\Delta$-matroids need not
  satisfy the simultaneous exchange property.
  It therefore allows M-convex functions to be defined without
  imposing the simultaneous exchange property.

  \item
  We define \emph{L-convex sets/functions on the vertex set of
  the Euclidean Coxeter complex} and \emph{M-convex sets/functions on the root lattice},
  which constitute the discrete L- and M-convexity studied in this paper.
  An L-convex set is the vertex set of a convex subcomplex of the
  Euclidean Coxeter complex, whereas an M-convex set is the set of
  lattice points of an integral M-convex polyhedron.
  The corresponding discrete functions are defined as the
  restrictions of primal-integral locally polyhedral L- and M-convex
  functions to the respective discrete domains.
  \Cref{prop:L-conv:unique,prop:M-conv:unique} show that the associated continuous
  extensions are uniquely determined by these restrictions.
  Moreover, \Cref{prop:tangent-cone} identifies the tangent cone of the
  convex hull of an M-convex set with the cone generated by the root directions
along which one can move one step within the set.
This relates the geometry of the convex hull to the discrete structure
of the M-convex set.

  \item
  We establish the fundamental optimality and conjugacy properties
  of the resulting discrete functions.
  Local optimality guarantees global optimality for both L- and
  M-convex functions
  (\Cref{thm:minimality:L,thm:local-optimality-M-convex}).
  We also prove that integral L-convex functions and integral
  M-convex functions correspond one-to-one through the \emph{discrete
  Fenchel--Legendre conjugate} (\Cref{thm:conjugacy}).
  This correspondence combines the uniqueness of the continuous
  extensions with the fact that continuous conjugacy exchanges the primal- and dual-integrality structures.

  \item
  We present intrinsic combinatorial characterizations of the discrete
  convex notions introduced above.
  L-convex sets and functions are characterized by the discrete midpoint
  convexity with respect to the Euclidean Coxeter complex
  (\Cref{thm:chara:L-convex-set,thm:chara:L-conv-func}),
  whereas M-convex sets and functions are characterized by the discrete
  tangent-cone property
  (\Cref{thm:chara:M-convex-set,thm:chara-M-conv-func}).
  These characterizations are stated in discrete terms, except for the
hole-freeness condition that appears in the characterizations of
M-convexity.
  Finally, \Cref{sec:examples} identifies the resulting classes with,
  or precisely compares them to, several existing notions associated
  with the classical root systems.
\end{itemize}

Some of the polyhedral notions introduced above recover existing
Coxeter-theoretic objects.
The L- and M-convex cones considered
here coincide with \emph{Coxeter cones}~\cite{Stembridge2008-zf} and
\emph{Coxeter root cones}~\cite{Ardila2020-ll}, respectively, while
the M-convex polytopes coincide with \emph{generalized Coxeter
permutahedra}~\cite{Ardila2020-ll}.
The contribution of the present paper is to introduce the discrete
domains and notions of integrality needed to define L- and M-convex
functions for classical root systems, and to organize these objects
into a common framework for DCA.

The assumption of classical type is needed only in a limited part of the theory.
The type-dependent arguments are essentially confined to
\Cref{prop:distance-function,prop:tangent-cone}.
In particular, most of the results on L-convexity, as well as the discrete
conjugacy theorem, remain valid for
arbitrary irreducible root systems, including the exceptional types;
see \Cref{rem:scope-root-systems}.

\paragraph{Organization.}

The rest of this paper is organized as follows. 
\Cref{sec:preliminaries} summarizes the preliminary concepts and notations concerning simplicial complexes, vector spaces, and continuous convex analysis. 
\Cref{sec:root-system} introduces the definitions and basic properties of root systems, including their classification into classical types. 
\Cref{sec:combinatorial-structures} details the combinatorial structures arising from root systems, such as root lattices and spherical/Euclidean Coxeter complexes, along with their geometric realizations and Lovász extensions. 
In \Cref{sec:DCA:continuous}, we formulate discrete convexity in
Euclidean spaces by defining L- and M-convex cones, polyhedra,
sublinear functions, and locally polyhedral convex functions.
Building upon these geometric concepts,
\Cref{sec:DCA:discrete} formally defines L- and M-convex sets and
functions over combinatorial objects and establishes the conjugacy
between integral L- and M-convex functions.
\Cref{sec:chara} provides combinatorial characterizations for these
discrete convexities.
These three sections constitute the main line of the paper.
To make this line of argument easier to follow, we defer to
\Cref{sec:deferred} the proofs of statements that are not
used in the arguments through \Cref{sec:chara}, together with
related discussions.
In \Cref{sec:examples},
we place our discrete convexity notions in relation to several existing classes,
including the standard DCA,
(valuated) $\Delta$-matroids, bisubmodular functions, BS-convex sets
and functions, UJ-convex functions, and Coxeter
matroids/submodular functions.
Finally, \Cref{sec:conclusion} concludes the paper with a summary of our results and a discussion of future research directions.

\section{Preliminaries}\label{sec:preliminaries}
Let $\Z$ and $\R$ denote the sets of integers and reals, respectively,
and $\Z_+$ and $\R_+$ their nonnegative subsets.
For a positive integer $c$,
we define $c\Z \coloneqq \Set{cx}{x \in \Z}$ and $\Z/c \coloneqq \Set{x/c}{x \in \Z}$.
In particular, $2\Z$ is the set of even integers and $\Z /2$ is the union of the sets of integers and of half-integers.
Let $2\Z + 1$ and $\Z + 1/2$ be the sets of odd integers and of half-integers,
respectively.

Let $S$ be a set.
The number of elements in $S$ is denoted by $\card{S}$,
which is in $\Z_+ \cup \{+\infty\}$.
For $s \in S$ and a tuple $x \in \R^S$ of reals indexed by $S$,
we denote by $x(s)$ its $s$th component.
Similarly,
for $x \in \R^n$ and $i \in [n]$,
we denote by $x(i)$
its $i$th component.
For $x \in \R^S$,
its \emph{support} $\supp{x}$ is the set of indices $s$ with $x(s) \neq 0$.
For nonnegative integers $m,n$ with $m \leq n$,
let $[m, n] \coloneqq \{ m, m+1, \dots, n \}$.
If $m = 1$ and $n \geq 1$, then we abbreviate $[m, n]$ as $[n]$.

\subsection{Simplicial complex}\label{subsec:simplicial_complex}
A \emph{partially ordered set} (or \emph{poset}) is a pair $(\CC, \preceq)$ of a set $\CC$ and a binary relation $\preceq$ over $\CC$ satisfying, for $x,y,z \in \CC$, that $x \preceq x$ (reflexivity), $x \preceq y$ and $y \preceq x$ imply $x = y$ (antisymmetry),
and $x \preceq y$ and $y \preceq z$ imply $x \preceq z$ (transitivity).
By $x \prec y$ we mean $x \preceq y$ and $x \neq y$.
By just $\CC$ we sometimes mean a poset if it is clear from the context.
For a finite set $S$,
the pair $(2^S, \subseteq)$ of the power set $2^S$ of $S$ and the inclusion relation $\subseteq$ forms a poset,
which is called the \emph{Boolean lattice}.

A nonempty poset $(\CC, \preceq)$ is called an \emph{(abstract) simplicial complex}
if all $A, B \in \CC$ have their greatest lower bound, denoted by $A \wedge B$, and
for all $A \in \CC$, there exists a unique nonnegative integer $r_A \in \Z_+$ such that the subposet induced by $\CC_{\preceq A} \coloneqq \Set{B \in \CC}{B \preceq A}$ is isomorphic to the Boolean lattice $(2^{[r_A]}, \subseteq)$.
A member $A$ of $\CC$ is called a \emph{simplex},
and its \emph{rank} is $r_A$.
In particular, a rank-one simplex is called a \emph{vertex}.
The rank of the lowest element of $\CC$, denoted by $\bot$,
is $0$.

Let $(\CC, \preceq)$ be a simplicial complex.
For $A \in \CC$,
a simplex $B$ with $B \preceq A$ is called a \emph{face} of $A$.
By a vertex of a simplex $A$ we mean a rank-one face of $A$.
Let $\sk{\CC}$ denote the set of vertices of $\CC$,
and $\sk{A}$ the set of vertices of a simplex $A$;
the number of the vertices of $A$ is equal to the rank of $A$.
Simplices $A_1, A_2, \dots, A_k \in \CC$ are said to be \emph{joinable}
if there is a simplex $A \in \CC$ such that all of $A_1, A_2, \dots, A_k$ are faces of $A$, or equivalently, there is an upper bound of $A_1, A_2, \dots, A_k$ in $\CC$.
If $A_1, A_2, \dots, A_k \in \CC$ are joinable,
then their least upper bound exists; we refer to it as the \emph{join} of $A_1, A_2, \dots, A_k$.
Two simplices $A, B$ are said to be \emph{disjoint} if $A \wedge B = \bot$.
Two maximal simplices $C, C'$ are said to be \emph{adjacent}
if they have a common corank-one face.
A sequence $(C_0, C_1, \dots, C_\ell)$ of maximal simplices is called a \emph{gallery}
if $C_{i-1}$ and $C_i$ are adjacent for each $i \in [\ell]$.
If $A \preceq C_0$ and $B \preceq C_\ell$ for two simplices $A, B \in \CC$,
then we refer to $(C_0, C_1, \dots, C_\ell)$ as
a \emph{gallery from $A$ to $B$}.
The \emph{length} of a gallery $(C_0, C_1, \dots, C_\ell)$ is $\ell$.
A gallery from $A$ to $B$ is said to be \emph{minimal}
if its length is minimum among all galleries from $A$ to $B$.
The \emph{distance} between $A$ and $B$,
denoted by $d(A, B)$,
is defined by the length of a minimal gallery from $A$ to $B$.

A nonempty subset $\CC' \subseteq \CC$ is called a \emph{subcomplex} of $\CC$
if
$A \in \CC'$ and $B \preceq A$ imply $B \in \CC'$,
i.e.,
$\CC'$ is closed under $\preceq$.
Note that $\CC'$ is also a simplicial complex.
For $A \in \CC$,
we define the subcomplexes $\St{\CC}{A}$ and $\Lk{\CC}{A}$ of $\CC$ by
\begin{align}
    \St{\CC}{A} &\coloneqq \Set{B \in \CC}{\text{$A$ and $B$ are joinable}},\\
    \Lk{\CC}{A} &\coloneqq \Set{B \in \CC}{\text{$A$ and $B$ are disjoint and joinable}}.
\end{align}
$\St{\CC}{A}$ is called the \emph{star complex}
and $\Lk{\CC}{A}$ the \emph{link complex}.
For a subset $D \subseteq \sk{\CC}$,
we denote by $\isc{\CC}{D}$ the subcomplex of $\CC$ induced by $D$, i.e.,
\begin{align}
    \isc{\CC}{D} \coloneqq \Set{A \in \CC}{\sk{A} \subseteq D}.
\end{align}
Note that, if $D = \emptyset$,
then $\isc{\CC}{D}$ consists of the lowest element $\bot$.

A simplicial complex $\CC$ is called a \emph{chamber complex}
if all maximal simplices have the same rank and any two maximal simplices are connected by a gallery.
Let $\CC$ be a chamber complex.
Then, the distance $d(A, B)$ is finite for any two cells $A, B \in \CC$.
A maximal simplex in $\CC$ is called a \emph{chamber}.
The \emph{rank} of $\CC$ is the rank of a chamber.

Suppose that the rank of $\CC$ is $n$.
Then, every chamber contains $n$ vertices.
Let $I$ be a set with $\card{I} = n$.
A function $\funcdoms{\tau}{\sk{\CC}}{I}$ is called a \emph{type function} on $\CC$
if $\tau(C)$ coincides with $I$ for any chamber $C \in \CC$,
where $\tau(A) \coloneqq \Set{\tau(x)}{x \in \sk{A}}$
for $A \in \CC$.
In other words,
a type function $\tau$ is a (proper) $n$-coloring of the $1$-skeleton of $\CC$ in the graphical sense.
We refer to $\tau(A)$ 
as the \emph{type} 
of a simplex $A$
with respect to $\tau$.
We remark that $\card{\tau(A)}$ coincides with the rank of $A$.
A chamber complex $\CC$ is said to be \emph{colorable} if $\CC$ admits a type function.
For any simplex $A$ of a colorable complex $\CC$ with type function $\funcdoms{\tau}{\sk{\CC}}{I}$,
the map $\CC_{\preceq A} \to 2^{\tau(A)}$,
given by $B \mapsto \tau(B)$ 
forms an isomorphism between the subposet $(\CC_{\preceq A}, \preceq)$ and the Boolean lattice $(2^{\tau(A)}, \subseteq)$.

Let $\CC$ be a colorable chamber complex of rank $n$
and $I$ a set of $n$ elements,
and choose an arbitrary chamber $C \in \CC$.
Then, any bijection $\funcdoms{\tau_C}{\sk{C}}{I}$ can be uniquely extended to a type function $\funcdoms{\tau}{\sk{\CC}}{I}$ satisfying $\tau(x) = \tau_C(x)$ for all $x \in \sk{C}$.
Indeed, since $\CC$ is a colorable chamber complex,
there is a type function $\funcdoms{\tau}{\sk{\CC}}{I}$ satisfying $\tau(x) = \tau_C(x)$ for all $x \in \sk{C}$.
For any chamber $D$ adjacent to $C$,
we have $|\sk{C} \setminus \sk{D}| = |\sk{D} \setminus \sk{C}| = 1$;
let $x_C$ and $x_D$ be the unique elements in $\sk{C} \setminus \sk{D}$ and in $\sk{D} \setminus \sk{C}$,
respectively.
Then, the type of $x_D$ must coincide with that of $x_C$, i.e., $\tau(x_D) = \tau(x_C) = \tau_C(x_C)$.
By a similar argument with the fact that there is a gallery from $C$ to any vertex $x \in \sk{\CC}$,
the type of every vertex $x$ is uniquely determined.

\subsection{Vector space}\label{subsec:vectorspace}
In this paper,
let $V$ denote a finite-dimensional vector space over $\R$ with an inner product $\funcdoms{\inpr{\cdot}{\cdot}}{V \times V}{\R}$.
Although we may identify $V$ with its dual vector space $V^*$ via $\inpr{\cdot}{\cdot}$,
we basically distinguish between $V$ and $V^*$.
On the other hand, we always identify $V$ with $V^{**}$ in a canonical way.
Both $V$ and $V^*$ are metric spaces induced from the norm $\|\cdot\| \coloneqq \inpr{\cdot}{\cdot}$.

For a vector $x \in V$, we denote by $x^*$ the corresponding dual vector in $V^*$ with respect to $\inpr{\cdot}{\cdot}$, that is, $x^*(y) = \inpr{x}{y}$ for $y \in V$.
Similarly, for a dual vector $p \in V^*$, we denote by $p^*$ the corresponding primal vector in $V$; $p(y) = \inpr{p^*}{y}$ for $y \in V$.
Let $\OG(V)$ and $\OG(V^*)$ denote the orthogonal groups over $V$ and $V^*$ with respect to $\inpr{\cdot}{\cdot}$,
respectively.
That is,
we have $\inpr{x}{y} = \inpr{wx}{wy}$ for any $x, y \in V$ and $w \in \OG(V)$.
In this paper,
we basically regard the inner product as the duality pairing $\funcdoms{\inpr{\cdot}{\cdot}}{V^* \times V}{\R}$ over $V^*$ and $V$ given by $\inpr{p}{x} \coloneqq p(x)$,
i.e.,
for $p \in V^*$ and $x \in V$,
we have $p(x) = \inpr{p^*}{x} = \inpr{p}{x}$.

By $U \leq V$ 
we mean that $U$ 
is a vector subspace of $V$. 
For $U \leq V$,
the orthogonal space of $U$
with respect to $\inpr{\cdot}{\cdot}$ is denoted by $U^\perp$,
that is,
$U^\perp \coloneqq \Set{p \in V^*}{\inpr{p}{x} = 0 \ (\forall x \in U)}$.
The equality $(U^\perp)^\perp = U$ 
holds.
Let $V/U$ denote the quotient space of $V$ by $U$.
Then, the duality pairing $\inpr{\cdot}{\cdot}$ over $V^*$ and $V$ naturally induces that over $U^\perp$ and $V/U$, since
$\inpr{p}{x + U} \coloneqq \inpr{p}{x}$ is well defined
for $p \in U^\perp$ and $x \in V$,
where $x + U$ denotes the equivalence class containing $x$.
Thus,
$(V/U)^*$ is isometric to $U^\perp$ for $U \leq V$,
and hence
we identify $(V/U)^*$ with $U^\perp$.
Similarly,
the duality pairing over $V/(W^\perp)$ and $W$ is naturally induced by $\inpr{\cdot}{\cdot}$ as
$\inpr{p}{x + W^\perp} \coloneqq \inpr{p}{x}$ for $p \in W$ and $x \in V$;
we also identify $V/(W^\perp)$ with $W^*$ for $W \leq V^*$.

For a subset $X \subseteq V$,
we denote by $\spn{X}$ the linear hull of $X$ in $V$,
and by $\aff{X}$ the affine hull of $X$ in $V$.
For any affine subspace $A$ of $V$,
there exists a unique linear subspace $A_0$ that is obtained by a parallel transformation of $A$,
i.e.,
$A = A_0 + x$ for $x \in A$.
For notational simplicity,
let $V / \aff{X}$ denote $V / (\aff{X})_0$,
where $(\aff{X})_0$ is the unique linear subspace of $V$ that is parallel to $\aff{X}$.
For a subset $X, Y \subseteq V$,
we denote by $X + Y$ and $X - Y$ the Minkowski sum and difference of $X$ and $Y$, respectively.
That is, $X + Y \coloneqq \Set{x + y}{x \in X,\ y \in Y}$
and $X - Y \coloneqq \Set{x - y}{x \in X,\ y \in Y}$.
In particular, if $X = \{x\}$ (resp. $Y = \{y\}$), then we use $x + Y$ (resp. $X - y$) instead of $\{x\} + Y$ (resp. $X - \{y\}$).
Let $\cl{X}$ and $\relint{X}$ denote the closure and the relative interior of $X$,
respectively.
For $x, y \in V$,
let $[x,y] \coloneqq \Set{\lambda x + (1-\lambda) y}{\lambda \in [0,1]}$ and $(x,y) \coloneqq \Set{\lambda x + (1-\lambda) y}{\lambda \in (0,1)}$ denote the closed and open line segments connecting $x$ and $y$, respectively.

In the case of $V = \R^n$,
we assume that the inner product $\inpr{\cdot}{\cdot}$ is given by $\inpr{x}{y} = \sum_{i=1}^n x(i) y(i)$ for $x, y \in V$,
and identify its dual $(\R^n)^*$ with $\R^n$.
Let $\onevec_n$ denote the all-one vector in $\R^n$.
The $i$th unit vector in $\R^n$ is denoted by $e_i$.

\subsection{Basics of convex analysis}\label{subsec:convex-analysis}
Let $\funcdoms{f}{V}{\R \cup \{ \pm \infty \}}$ be a function.
The sets
\begin{align}
    \dom{f} &\coloneqq \Set{x \in V}{-\infty < f(x) < +\infty},\\
    \epi{f} &\coloneqq \Set{(x,y) \in V \times \R}{f(x) \leq y}
\end{align}
are called the \emph{effective domain} and the \emph{epigraph} of $f$,
respectively.
If $\epi{f}$ is a closed set in $V \times \R$,
then $f$ is said to be \emph{closed}.
A set $C \subseteq V$ is called a \emph{cone} if $\mu x \in C$ holds for $x \in C$ and $\mu > 0$.
A function $f$ is said to be \emph{positively homogeneous} 
if $f(\mu x) = \mu f(x)$ holds for any $x \in V$ and $\mu > 0$,
or equivalently,
its epigraph $\epi{f}$ forms a cone.
For a function $\funcdoms{g}{V}{\R \cup \{\pm\infty\}}$,
we say that $f$ \emph{majorizes} $g$, or $g$ \emph{minorizes} $f$ if
$g \leq f$,
i.e.,
$g(x) \leq f(x)$ for all $x \in V$.
If $f$ majorizes some affine function,
then $f$ cannot take $-\infty$,
namely, $f$ is a function from $V$ to $\ER$.

A set $C \subseteq V$ is said to be \emph{convex}
if $x, y \in C$ implies $\lambda x + (1-\lambda)y \in C$ for all $\lambda \in (0,1)$.
A function $\funcdoms{f}{V}{\R \cup \{\pm \infty\}}$ is said to be \emph{convex}
if $\lambda f(x) + (1-\lambda) f(y) \geq f(\lambda x + (1-\lambda)y)$
holds for any distinct $x,y \in V$ and $\lambda \in (0,1)$.
It is well known that $f$ is a convex function if and only if its epigraph $\epi{f}$ is a convex set.
A convex function $f$ is said to be \emph{proper}
if $f$ never takes $-\infty$ and $\dom{f}$ is nonempty.
A convex and positively homogeneous function is referred to as a \emph{sublinear function},
which is characterized as a function whose epigraph is a convex cone.

Let $C \subseteq V$ be a convex set.
A convex subset $F \subseteq C$ is called a \emph{face} (see e.g.,~\cite[Section~18]{Rockafellar1996-wm}) of $C$ if $x, y \in C$ and $\lambda x + (1-\lambda)y \in F$ for some $\lambda \in (0,1)$ imply $x, y \in F$,
and called an \emph{exposed face} if there is $p \in V^*$ such that $F = \argmin_{x \in C} \inpr{p}{x}$.
Any exposed face is a face, but the converse does not necessarily hold.
A face $F$ of $C$ is said to be \emph{proper} if $F \neq \emptyset$ and $F \neq C$.
The face relation is transitive;
for a face $F$ of $C$,
a subset $F' \subseteq F$ is a face of $F$ if and only if $F'$ is a face of $C$.
The family consisting of the relative interiors of all faces of $C$ forms a partition of $C$~\cite[Theorem~18.2]{Rockafellar1996-wm}.
That is, for any point $x$ in a convex set $C$,
there exists a unique face $F$ of $C$ such that $x \in \relint{F}$.

For a closed convex set $C \subseteq V$ and a point $x \in C$,
the \emph{tangent cone} $\tcone{C}{x}$ of $C$ at $x$ is defined by
$\tcone{C}{x} \coloneqq \cl(\cone{\Set{y - x}{y \in C}})$.
We can easily derive the following by fundamental arguments in convex analysis.
\begin{lemma}\label{lem:clconv-tcone}
    For a nonempty closed convex set $C \subseteq V$,
    we have $C = \bigcap_{x \in C} (x + \tcone{C}{x})$.
\end{lemma}
\begin{proof}
    The inclusion $C\subseteq\bigcap_{x\in C}(x+\tcone{C}{x})$ is immediate.
    Conversely, let $y\notin C$.
    By the separation theorem (see e.g.,~\cite[Corollary~A.4.2.4]{Hiriart-Urruty2001-ck}), there are
    $\hat{x}\in C$ and $p\in V^*$ such that $\inpr{p}{y} < \inpr{p}{\hat{x}} \leq \inpr{p}{x}$ for all $x \in C$.
    Hence, $\inpr{p}{d}\geq0$ for every $d\in\tcone{C}{\hat{x}}$, and therefore $y\notin\hat{x}+\tcone{C}{\hat{x}}$.
\end{proof}

\subsubsection{Convex and conical hull}
For a set $S$, let $\cvcoeff{S}$ and $\cncoeff{S}$ denote the sets of all convex combination coefficients $(\lambda(s))_{s \in S}$ over $S$ and of all conical combination coefficients $(\mu(s))_{s \in S}$ over $S$, respectively. That is,
\begin{align}
    \cvcoeff{S} &\coloneqq \Set*{\lambda \in \R^S}{\lambda(s) \geq 0 \ (\forall s \in S),\ \card{\supp{\lambda}} < +\infty, \  \sum_{s \in S} \lambda(s) = 1},\\
    \cncoeff{S} &\coloneqq \Set*{\mu \in \R^S}{\mu(s) \geq 0 \ (\forall s \in S),\ \card{\supp{\mu}} < +\infty}.
\end{align}

For a subset $X \subseteq V$,
we denote by $\conv{X}$ and $\cone{X}$ the convex hull and conical hull of $X$, respectively:
\begin{align}
    \conv{X} \coloneqq \Set*{\sum_{x \in X} \lambda(x) x}{\lambda \in \cvcoeff{X}}, \qquad
    \cone{X} \coloneqq \Set*{\sum_{x \in X} \mu(x) x}{\mu \in \cncoeff{X}}.
\end{align}
Note that $\conv{\emptyset} = \emptyset$
and $\cone{\emptyset} = \{0\}$.
The \emph{closed convex hull}
of $X$,
denoted by $\clconv{X}$,
is the closure of $\conv{X}$.
The (closed) convex hull of $X$ can be characterized as the minimum (closed) convex set that contains $X$.

For a function $\funcdoms{f}{V}{\ER}$,
its \emph{convex hull} $\conv{f}$ and \emph{closed convex hull} $\clconv{f}$ are defined by
\begin{align}
    (\conv{f})(x) &\coloneqq \inf{\Set*{\sum_{y \in \dom{f}} \lambda(y) f(y)}{\lambda \in \cvcoeff{\dom{f}},\ x = \sum_{y \in V} \lambda(y) y}},\label{eq:func:conv} \\
    (\clconv{f})(x) &\coloneqq \sup{\Set{ \inpr{p}{x} - r}{p \in V^*,\  r \in \R,\  \inpr{p}{y} - r \leq f(y) \ (\forall y \in \dom{f})}}, \label{eq:func:clconv}
\end{align}
respectively.
If $f$ majorizes some affine function,
then the (closed) convex hull of $f$ can be characterized as the pointwise maximum (closed) convex function that minorizes $f$;
particularly,
$\clconv{f} \leq \conv{f} \leq f$ holds.

The following lemma shows the relationship between the effective domains of a function and its (closed) convex hulls,
in which the identity $\conv(\dom{f}) = \dom(\conv{f})$ immediately follows from the definition of the convex hull.
\begin{lemma}[{\cite[Corollary~7.4.1]{Rockafellar1996-wm}}]\label{lem:conv=clconv}
    Let $\funcdoms{f}{V}{\ER}$ be a function minorized by some affine function.
    Then, we have
    $\conv(\dom{f}) = \dom(\conv{f}) \subseteq \dom(\clconv{f}) \subseteq \clconv(\dom{f})$.
\end{lemma}

The \emph{conical hull}
of a function $\funcdoms{f}{V}{\ER}$, denoted by $\cone{f}$, is defined by
\begin{align}\label{eq:func:cone}
    (\cone{f})(x) \coloneqq \inf{\Set*{\sum_{y \in \dom{f}} \mu(y) f(y)}{\mu \in \cncoeff{\dom{f}},\ x = \sum_{y \in \dom{f}} \mu(y) y}} \qquad (x \in V);
\end{align}
note that 
$(\cone{f})(0) = 0$
if $\dom{f} = \emptyset$.
If $f$ majorizes some linear function,
then the conical hull of $f$ can be characterized as the pointwise maximum proper sublinear function that minorizes $f$.
If $f$ does not majorize any linear function,
then $(\cone{f})(x) = -\infty$ for all $x \in \cone(\dom{f})$.

\subsubsection{Fenchel--Legendre conjugate}
For a function $\funcdoms{f}{V}{\ER}$ that majorizes some affine function,
its \emph{Fenchel--Legendre conjugate} $f^*$, which is a function on $V^*$, is defined by
\begin{align}
  f^*(p) \coloneqq \sup{\Set{\inpr{p}{x} - f(x)}{x \in V}} \qquad (p \in V^*).\notag
\end{align}
The Fenchel--Legendre conjugate provides the one-to-one correspondence between the proper closed convex functions on $V$ and those on $V^*$, which plays a central role in convex analysis.
\begin{lemma}[{see e.g.,~\cite[Theorems~E.1.1.2 and~E.1.3.5]{Hiriart-Urruty2001-ck}}]\label{lem:conjugate}
  Let $\funcdoms{f}{V}{\ER}$ be a proper function that majorizes some affine function.
  \begin{enumerate}[label={\textup{(\arabic*)}}]
        \item $f^*$ is a closed convex function on $V^*$.
        \item $\clconv{f} = f^{**}$. In particular, if $f$ is a closed convex function, then $f = f^{**}$.
    \end{enumerate}
\end{lemma}

As a special case of the above,
we obtain the following one-to-one correspondence between the nonempty closed convex sets and the proper closed sublinear functions.
For a set $S \subseteq V$, its \emph{support function} $\spt{S}$ is the function on $V^*$ defined by
\begin{align}\label{eq:def:sptfunc}
  \spt{S}(p) \coloneqq \sup{\Set{\inpr{p}{x}}{x \in S}} \qquad (p \in V^*),
\end{align}
and for a function $\sigma$ on $V$,
we define the subset $\bS{\sigma}$ of $V^*$ by
\begin{align}
  \bS{\sigma} \coloneqq \Set{p \in V^*}{\inpr{p}{x} \leq \sigma(x) \ (\forall x \in V)}.
\end{align}
\begin{lemma}[{\cite[Theorem~C.2.2.2]{Hiriart-Urruty2001-ck}}]\label{lem:clconv-sublinear}
  For a nonempty closed convex set $S$ and a proper closed sublinear function $\sigma$,
  we have $S = \bS{\spt{S}}$ and $\sigma = \spt{\bS{\sigma}}$.
  In addition, 
  for any set $S$,
  its support function $\spt{S}$ is a closed sublinear function
  and $\clconv{S} = \bS{\spt{S}}$, i.e.,
  \begin{align}
    \clconv{S} = \Set{p \in V^*}{\inpr{p}{x} \leq \spt{S}(x) \ (\forall x \in V)}.
  \end{align}
\end{lemma}

\subsubsection{Directional derivative and subdifferential}
Let $f$ be a function on $V$ and $x \in \dom{f}$.
The \emph{directional derivative} of $f$ at $x$ with respect to $d \in V$, denoted by $f'(x; d)$, is the limit
\begin{align}
    f'(x; d) \coloneqq \lim_{t \downarrow 0} \frac{f(x + td) - f(x)}{t}
\end{align}
if it exists (it can take $+\infty$ or $-\infty$).
If $f'(x; d)$ exists for all $d \in V$, then we refer to the function $\funcdoms{f'(x; \cdot)}{V}{\R \cup \{+\infty, -\infty\}}$ as the \emph{directional derivative} of $f$ at $x$.
The \emph{subdifferential} of $f$ at $x$, denoted by $\partial f(x)$, is the subset of $V^*$ defined by
\begin{align}
    \partial f(x) \coloneqq&\; \Set{p \in V^*}{f(y) \geq f(x) + \inpr{p}{y-x} \ (\forall y \in V)}\\
    =&\; \Set{p \in V^*}{x \in \argmin(f- p)}.
\end{align}
A point $p \in \partial f(x)$ is called a \emph{subgradient} of $f$ at $x$.
If $\funcdoms{f}{V}{\ER}$ is a proper convex function,
then,
for each $x \in \dom{f}$, the directional derivative $f'(x; \cdot)$ (exists and) forms a sublinear function.
Specifically, the following relations hold~\cite[Theorem~23.2]{Rockafellar1996-wm}:
\begin{align}
    \clconv{(f'(x; \cdot))} &= \spt{\partial f(x)},\label{eq:derivative-support}\\
    \partial f(x) &= \bS{f'(x; \cdot)}.\label{eq:subdifferential}
\end{align}

\subsubsection{The minimizer sets with perturbations}

For a function $\funcdoms{f}{V}{\ER}$ and $x \in \dom{f}$,
we define
\begin{align}
    \fan(f) &\coloneqq \Set{\argmin(f-p)}{p \in V^*,\ \argmin(f-p) \neq \emptyset},\\
    \fan(f; x) &\coloneqq \Set{ S \in \fan(f)}{x \in S}.
\end{align}
We note that, if $f$ is a convex function, then any member of $\fan(f)$ can be viewed as the projection of an exposed face of $\epi{f} \subseteq V \times \R$ to $V$.
If $f$ is a closed convex function,
then every member of $\fan(f)$ is a closed convex set.
We also define the \emph{supports} of $\fan(f)$ and $\fan(f; x)$ by
\begin{align}
    \geo{\fan(f)} \coloneqq \bigcup_{S \in \fan(f)} S, \qquad
    \geo{\fan(f; x)} \coloneqq \bigcup_{S \in \fan(f; x)} S,
\end{align}
respectively.

We here summarize easy observations on $\fan(f)$ of a convex function $f$.
\begin{lemma}\label{lem:argmin-support}
    Let $\funcdoms{f}{V}{\ER}$ be a proper convex function.
    \begin{enumerate}[label={\textup{(\arabic*)}}]
        \item $\relint(\dom{f}) \subseteq \geo{\fan(f)} \subseteq \dom{f}$.
        In particular, if $\geo{\fan(f)}$ is closed,
    then $\geo{\fan(f)} = \dom{f}$.
        \item For $S, T \in \fan(f)$,
        we have that $S$ is a face of $T$ if and only if $(\relint{S}) \cap T \neq \emptyset$.
        \item If there is a subset $\fan' \subseteq \fan(f)$ such that $\dom{f} = \bigcup_{S \in \fan'} S$, then $\geo{\fan(f)} = \dom{f}$ and any member of $\fan(f)$ is a face of some member of $\fan'$.
        \item Let $x \in \dom{f}$ and $C \in \fan(f)$ such that $x \in \relint{C}$.
        For any $y \in C$, we have $\bS{f'(x; \cdot)} \subseteq \bS{f'(y; \cdot)}$.
        \item If $f$ is a proper closed convex function,
        then $\argmin(f^* - x) = \bS{f'(x; \cdot)}$ for $x \in \dom{f}$
        and $\argmin(f - p) = \bS{(f^*)'(p; \cdot)}$ for $p \in \dom{f^*}$.
        In particular, we have
        \begin{align}
          \fan(f^*) = \Set{ \partial f(x) }{x \in \dom{f},\ \partial f(x) \neq \emptyset} = \Set{ \bS{f'(x; \cdot)} }{ x \in \dom{f},\ \bS{f'(x; \cdot)} \neq \emptyset }.
        \end{align}
    \end{enumerate}
\end{lemma}
\begin{proof}
    (1).
    The inclusion $\geo{\fan(f)}\subseteq\dom{f}$ is immediate.
    By~\cite[Theorem~23.4]{Rockafellar1996-wm},
    every point in $\relint(\dom{f})$ admits a subgradient, and hence
    belongs to some member of $\fan(f)$.

    (2) follows from \cite[Theorem~18.1 and Corollary~18.1.2]{Rockafellar1996-wm}, applied to the exposed faces of $\epi{f}$ corresponding to the members of $\fan(f)$.
    
    (3) follows from the assertions~(1) and~(2).

    (4).
    If $p\in \partial f(x)$, i.e., $x \in \argmin(f-p) \in \fan(f)$, then $C$ is a face of $\argmin(f-p)$ by the assertion~(2).
    Thus, every $y\in C$ belongs to $\argmin(f-p)$, i.e., $p \in \partial f(y)$.
    Therefore, we obtain
    $\partial f(x)\subseteq\partial f(y)$,
    which is equivalent to $\bS{f'(x; \cdot)} \subseteq \bS{f'(y; \cdot)}$ by~\eqref{eq:subdifferential}.

    (5) follows from \cite[Theorem~23.5 and Corollary~23.5.1]{Rockafellar1996-wm}
    with the equation~\eqref{eq:subdifferential}.
\end{proof}

\subsubsection{Polyhedral structures}

A subset $P \subseteq V$ is called a \emph{polyhedron} (see e.g.,~\cite[Theorem~19.1]{Rockafellar1996-wm})
if it is representable as the intersection of finitely many closed half-spaces of $V$, i.e.,
\begin{align}\label{eq:poly-outer}
P = \bS{\theta} = \Set{x \in V}{\inpr{p}{x} \leq \theta(p) \ (p \in \dom{\theta})}
\end{align}
for some $\funcdoms{\theta}{V^*}{\ER}$ such that $\dom{\theta}$ is a finite set,
or equivalently,
it is representable as the Minkowski sum of the convex hull of finitely many points in $V$ and the conical hull of finitely many vectors in $V$, i.e.,
\begin{align}\label{eq:poly-inner}
P = \conv{X} + \cone{Y}
\end{align}
for some finite sets $X, Y \subseteq V$;
the second summand $\cone{Y}$ in~\eqref{eq:poly-inner} is uniquely determined by $P$ and is called the \emph{recession cone} (or \emph{characteristic cone}) of $P$.
We refer to the former representation of a polyhedron $P$ as an \emph{outer-representation}
and the latter an \emph{inner-representation}.
In the inner-representation~\eqref{eq:poly-inner},
$P \neq \emptyset$ if and only if $X \neq \emptyset$.
A bounded polyhedron is called a \emph{polytope}.
A \emph{polyhedral cone} is a cone that is also a polyhedron;
an inner-representation of a polyhedral cone is of the form $\cone{Y}$ for some finite set $Y \subseteq V$,
and an outer-representation is of the form $\Set{ x \in V }{ \inpr{p}{x} \leq 0 \ (\forall p \in F) }$
for some finite set $F \subseteq V^*$.

For a cone $K \subseteq V$, its \emph{polar cone} $K^* \subseteq V^*$ is
\begin{align}
    K^* \coloneqq \Set{p \in V^*}{\inpr{p}{x} \leq 0 \ (\forall x \in K)},
\end{align}
which forms a closed convex cone.
If $K$ is a closed convex cone then $K = K^{**}$,
and if $K$ is a polyhedral cone then so is $K^*$.
From an outer-representation (resp. an inner-representation) of a polyhedral cone,
we can easily construct an inner-representation (resp. an outer-representation) of
its polar cone as follows.
\begin{lemma}[{Farkas lemma; see e.g.,~\cite[Examples~A.3.2.2]{Hiriart-Urruty2001-ck} and~\cite[Corollary~22.3.1]{Rockafellar1996-wm}}]\label{lem:polar_polyhedral_cone}
    If $K$ is a polyhedral cone of the form $\Set{ x \in V }{\inpr{p}{x} \leq 0 \ (\forall p \in Q)}$ for some finite set $Q \subseteq V^*$,
then its polar cone $K^*$ is representable as $K^* = \cone{Q}$.
\end{lemma}

Let $P$ be a polyhedron.
For $x \in P$, its tangent cone $\tcone{P}{x}$ of $P$ at $x$ forms a polyhedral cone representable as $\tcone{P}{x} = \cone{\Set{ d \in V }{\exists \eps > 0.\  x + \eps d \in P}}$.
For $x \in P$, let $\face{P}{x}$ denote the unique face of $P$ such that $x \in \relint(\face{P}{x})$.
The following is a fundamental fact on the tangent cone of a polyhedron.
\begin{lemma}[{see e.g.,~\cite[Proposition~3.5.2]{Matthias-Beck2018-lk}}]\label{lem:poly:tcone}
    For a polyhedron $P$ and $x, y \in P$ with $\face{P}{x} = \face{P}{y}$,
    we have $\tcone{P}{x} = \tcone{P}{y}$.
\end{lemma}
Since the number of faces of $P$ is finite (see e.g.,~\cite[Section~8.3]{Schrijver2000-xk}),
so is
the set $\Set{\tcone{P}{x}}{x \in P}$ by \Cref{lem:poly:tcone}, which consists of polyhedral cones.
We can see that this property ensures a closed convex set to be a polyhedron.
\begin{lemma}\label{lem:tcone-poly-chara}
    For a closed convex set $C$,
    it is a polyhedron if and only if
    $\Set{ \tcone{C}{x} }{ x \in C }$ is a finite set consisting of polyhedral cones.
\end{lemma}
\begin{proof}
    We only show the if part; we have already seen the only-if part above.
    For $x \in C$,
    since $\tcone{C}{x}$ is a polyhedral cone,
    there is a finite set $S_x \subseteq V^*$ such that
    $\tcone{C}{x} = \Set{ y \in V }{ \inpr{p}{y} \leq 0 \ (\forall p \in S_x) }$;
    we may assume that $S_x = S_y$ if $\tcone{C}{x} = \tcone{C}{y}$.
    Then, since $\Set{ \tcone{C}{x} }{ x \in C }$ is finite,
    so is the set $S \coloneqq \bigcup_{x \in C} S_x$.
    For each $p \in S$ and $x \in C$,
    let $r_p(x) \coloneqq \inpr{p}{x}$ if $p \in S_x$,
    and $r_p(x) \coloneqq +\infty$ if $p \notin S_x$.
    Then, $x + \tcone{C}{x} = \Set{ y \in V }{ \inpr{p}{y} \leq r_p(x) \ (\forall p \in S) }$ holds.
    Since $C = \bigcap_{x \in C} (x + \tcone{C}{x})$ (\Cref{lem:clconv-tcone}),
    we obtain $C = \Set{ y \in V }{ \inpr{p}{y} \leq \inf_{x \in C} r_p(x)\ (\forall p \in S) }$.
    The finiteness of $S$ implies that $C$ is a polyhedron.
\end{proof}

A function is said to be \emph{polyhedral convex}
if its epigraph forms a polyhedron.
In particular, a \emph{polyhedral sublinear function} is a function whose epigraph forms a polyhedral cone.
We note that the effective domain of a polyhedral convex function (resp. polyhedral sublinear function) forms a polyhedron (resp. a polyhedral cone).

A proper polyhedral sublinear function $\sigma$ admits two kinds of representations similar to a polyhedron (see~\eqref{eq:poly-outer} and~\eqref{eq:poly-inner}); $\sigma$ is always representable as
\begin{align}
    \sigma(p) &= (\cone{\theta})(p) = \inf{\Set*{ \sum_{q \in \dom{\theta}} \mu(q) \theta(q)}{ p = \sum_{q \in \dom{\theta}} \mu(q) q, \ (\mu(q))_{q \in \dom{\theta}} \in \cncoeff{\dom{\theta}} }}\label{eq:poly-sublinear-inner}
\end{align}
for some function $\funcdoms{\theta}{V^*}{\ER}$ with $\dom{\theta}$ finite,
and as
\begin{align}
    \sigma(p) &=
    \begin{cases}
        \max{\Set{\inpr{p}{x} }{ x \in X }} & \text{if $p \in (\cone{Y})^* = \Set{ q \in V^* }{ \inpr{q}{y} \leq 0 \ (\forall y \in Y) }$},\\
        +\infty & \text{otherwise}\label{eq:poly-sublinear-outer}
    \end{cases}
\end{align}
for some finite sets $X, Y \subseteq V$.
The former and latter representations of $\sigma$ are referred to as an \emph{inner-} and \emph{outer-representation},
respectively.

The specialization of the one-to-one correspondence between the nonempty closed convex sets and the proper closed sublinear functions (\Cref{lem:clconv-sublinear}) to the polyhedral case provides the one-to-one correspondence between the nonempty polyhedra and the proper polyhedral sublinear functions~\cite[Theorem~19.2 and Corollary~19.2.1]{Rockafellar1996-wm}.
More specifically,
from an outer-representation (resp. an inner-representation) of a polyhedron,
we can easily construct an inner-representation (resp. an outer-representation) of
its support function,
and vice versa, as follows.
\begin{lemma}\label{lem:conjugate-poly-sublinear}
    Let $P$ be a nonempty polyhedron of the form~\eqref{eq:poly-outer} (resp.~\eqref{eq:poly-inner})
    and $\sigma$ a proper polyhedral sublinear function of the form~\eqref{eq:poly-sublinear-inner} (resp.~\eqref{eq:poly-sublinear-outer}).
    Then, we have $P = \bS{\sigma}$ and $\sigma = \spt{P}$.
    In addition, for $p \in \dom{\sigma}$,
    the set $\Set{ x \in P }{ \inpr{p}{x} = \sigma(p) }$ is a nonempty face of $P$.
\end{lemma}

A family $\fan$ of polyhedra is called a \emph{polyhedral complex} if it satisfies that
\begin{enumerate}[label={\textup{(\roman*)}}]
  \item if $P \in \fan$ and $P'$ is a face of $P$ then $P' \in \fan$; and
  \item the nonempty intersection $P \cap Q$ of two members $P, Q \in \fan$ is a face of $P$ and $Q$.
\end{enumerate}
A polyhedral complex consisting of polyhedral cones is particularly called a \emph{polyhedral fan}.
The \emph{support} of a polyhedral complex $\fan$,
denoted by $\geo{\fan}$, is defined as $\bigcup_{P \in \fan} P$.
For example, for a polyhedral convex function $f$ and a polyhedral sublinear function $\sigma$,
the families $\fan(f)$ and $\fan(\sigma)$ form a polyhedral complex whose support is $\dom{f}$ and a polyhedral fan whose support is $\dom{\sigma}$,
respectively (see e.g.,~\cite[Chapter~5]{Ziegler1994} and~\cite{Hirai2006-bx}).
We can easily obtain the following from \Cref{lem:conjugate-poly-sublinear},
which is a special case of \Cref{lem:argmin-support}~(5).
\begin{lemma}\label{lem:fan-support}
  Let $P \subseteq V$ be a nonempty polyhedron and $\funcdoms{\sigma}{V^*}{\ER}$ a proper polyhedral sublinear function such that $P = \bS{\sigma}$, namely, $\sigma = \spt{P}$.
  Then, we have $\fan(\sigma) = \Set{ (\tcone{P}{x})^* }{ x \in P }$ and $\tcone{P}{x} = (\argmin(\sigma - x))^*$ for $x \in P$.
\end{lemma}

A convex set $Q \subseteq V$ is said to be \emph{locally polyhedral} if, for any polytope $P$, the intersection $Q \cap P$ forms a polytope.
Similarly, a function is said to be \emph{locally polyhedral convex} if its epigraph forms a locally polyhedral set,
or equivalently, if for any polytope $P$, the restriction of $f$ to $P$ is a polyhedral convex function.
We can easily see that locally polyhedral convex sets and functions are closed.

The definition of locally polyhedral convex sets (resp. functions) states that a locally polyhedral convex set (resp. function) behaves locally like a polyhedron (resp. polyhedral convex function).
Hence, many local properties of polyhedral structures
can be generalized to the case of locally polyhedral structures.
We summarize basic (local) properties of a locally polyhedral convex function as follows.
\begin{lemma}\label{lem:quasi-polyhedral-convex-function}
    Let $\funcdoms{f}{V}{\ER}$ be a locally polyhedral convex function.
    \begin{enumerate}[label={\textup{(\arabic*)}}]
        \item $\dom{f} = \geo{\fan(f)}$.
        \item $f'(x; \cdot)$ is a proper polyhedral sublinear function and $\partial f(x)$ is a nonempty polyhedron for $x \in \dom{f}$.
        In addition,
        $\fan(f'(x; \cdot)) = \Set{ \tcone{Q}{x} }{ Q \in \fan(f; x) }$
        and
        $f'(x; d) = f(x + d) - f(x)$
        for $x \in \dom{f}$ and $d \in \geo{\fan(f; x)} - x$.
        \item If every member of $\fan(f)$ is a polyhedron,
        then $\fan(f)$ forms a polyhedral complex.
    \end{enumerate}
\end{lemma}
\begin{proof}
    (1) and (2).
    Fix any $x \in \dom{f}$ and polytope $P \subseteq V$ with $x \in \intr{P}$.
    Then, the restriction $f|_{P}$ of $f$ to $P$ is a polyhedral convex function such that
    $f'(x; \cdot) = (f|_{P})'(x; \cdot)$ (and hence $\partial f(x) = \partial f|_{P} (x)$).
    It follows from \cite[Theorem~23.10]{Rockafellar1996-wm} that $(f|_{P})'(x; \cdot)$ is a proper polyhedral convex function and $\partial f|_{P} (x)$ is a nonempty polyhedron.
    Hence so are $f'(x; \cdot)$ and $\partial f(x)$,
    which implies the assertion~(1) and the former statement of the assertion~(2).
    The latter statement follows from the linearity of $f$ and $f'(x; \cdot)$ on $Q \in \fan(f; x)$.

    (3).
    For any convex function $f$, the family $\fan(f)$ generally satisfies the condition~(ii) in the definition of a polyhedral complex.
    Indeed, for any $Q = \argmin(f-p), Q' = \argmin(f- p') \in \fan(f)$ with $Q \cap Q' \neq \emptyset$,
    we have $Q \cap Q' = \argmin(f - (p + p')/2)$,
    which implies that $Q \cap Q' \in \fan(f)$.
    Hence, it suffices to see that any nonempty face $F$ of $Q \in \fan(f)$ belongs to $\fan(f)$.

    Take any $Q \in \fan(f)$ and its nonempty face $F$.
    Let $P$ be a full-dimensional polytope in $V$ such that $P \cap \relint{F} \neq \emptyset$.
    Since $f$ is a locally polyhedral convex function,
    the restriction $f|_{P}$ of $f$ to $P$ is a polyhedral convex function,
    and hence, $\fan(f|_{P})$ is a polyhedral complex containing $P \cap Q$.
    Since $F$ is a face of $Q$ and $P \cap \relint{F} \neq \emptyset$,
    $P \cap F$ is a nonempty face of $P \cap Q$,
    which implies that $P \cap F \in \fan(f|_{P})$.
    Thus, there is $p \in V^*$ such that $P \cap F = \argmin(f|_{P} - p)$.
    For such $p$,
    we have $P \cap F \subseteq \argmin(f - p)$.
    Hence, $\argmin(f - p) \cap \relint{F} \neq \emptyset$ holds,
    which implies that $F$ is a face of $\argmin(f - p)$.
    Since $\dim(P \cap F) = \dim{F} = \dim(\argmin(f - p))$ by the full-dimensionality of $P$,
    we obtain $F = \argmin(f - p)$, that is, $F \in \fan(f)$.
    This completes the proof.
\end{proof}

\subsubsection{Integrality of polyhedra}
A subset $\Mdom \subseteq V^*$ is called a \emph{lattice}
if there is a basis $\{p_1, p_2, \dots, p_n\}$ of $V^*$
such that $\Mdom$ is representable as 
\begin{align}\label{eq:lattice}
    \Mdom = \Set*{ \sum_{i = 1}^n c_i p_i }{ c_1, c_2, \dots, c_n \in \Z }.
\end{align}
We particularly say that the lattice of the form~\eqref{eq:lattice} is \emph{generated from $\{p_1, p_2, \dots, p_n\}$}
and that $\{p_1, p_2, \dots, p_n\}$ is an \emph{$\Mdom$-basis} (or just a \emph{basis}).

A polyhedron $P \subseteq V^*$ is said to be \emph{$\Mdom$-integral}
if $P = \conv(P \cap \Mdom)$ holds.
There are several characterizations of the ($\Z^n$-)integrality of polyhedra,
which are well known in the field of combinatorial optimization (see e.g.,~\cite[Section~16.3]{Schrijver2000-xk}).
By adapting the proof there, we immediately obtain the following.
\begin{lemma}[{see e.g.,~\cite[Section~16.3]{Schrijver2000-xk}}]\label{lem:integral-polyhedron}
    Let $\Mdom$ be a lattice in $V^*$
    and $P \subseteq V^*$ a polyhedron.
    Then,
    $P$ is $\Mdom$-integral if and only if $P = \conv{X} + \cone{Y}$ for some finite sets $X, Y \subseteq \Mdom$.
    In addition, if $P$ is $\Mdom$-integral, then so is every nonempty face of $P$.
\end{lemma}

\section{Root system}\label{sec:root-system}
In this section, we review the fundamental concepts and basic properties of root systems. 
\Cref{subsec:def:root-system} provides the formal definition of a root system and introduces the notion of \emph{(extended) simple systems}. 
\Cref{subsec:Weyl-Dynkin} introduces \emph{Weyl groups} and \emph{(extended) Dynkin diagrams}, which capture the geometric and algebraic symmetries of root systems. 
Finally, \Cref{subsec:classification} reviews the classification of irreducible root systems, with a particular focus on the classical types A, B, C, and D.

\subsection{Definitions}\label{subsec:def:root-system}
For a nonzero dual vector $\alpha \in V^* \setminus \{ 0\}$, let $\alpha^{\vee}$ be a vector in $V$ defined by
\begin{align}\label{eq:coroot}
    \alpha^\vee \coloneqq \frac{2}{\inpr{\alpha}{\alpha^*}} \alpha^*;
\end{align}
we can easily see that $\inpr{\alpha}{\alpha^{\vee}} = 2$.
We similarly define $x^{\vee} \in V^*$ for a nonzero vector $x \in V \setminus \{0\}$.
Then, we have $(\alpha^{\vee})^{\vee} = \alpha$.
The \emph{reflection} $s_\alpha$ with respect to $\alpha$ is a linear map in $V$ defined by
\begin{align}\label{eq:reflection}
    s_\alpha(x) \coloneqq x -  \inpr{\alpha}{x} \alpha^{\vee} \qquad (x \in V).
\end{align}
That is, $s_\alpha$ is the reflection with respect to the hyperplane in $V$ orthogonal to $\alpha$.
By abuse of notation, we assume that $s_\alpha$ also represents the reflection in $V^*$ with respect to $\alpha^\vee$:
\begin{align}\label{eq:reflection-dual}
    s_\alpha(p) \coloneqq p -  \inpr{p}{\alpha^{\vee}} \alpha \qquad (p \in V^*).
\end{align}

A nonempty finite set $\RS \subseteq V^* \setminus \{0\}$ of nonzero dual vectors is called a \emph{root system}~\cite[Definition~8.1]{Hall2015} if it satisfies the following conditions \eqref{cond:R1}--\eqref{cond:R3}:
\begin{enumerate}[label={\textup{(R\arabic*)}},ref={\textup{R\arabic*}}]
    \item \label{cond:R1}
    $s_\alpha(\beta) \in \RS$ for any $\alpha, \beta \in \RS$;
    \item \label{cond:R2}
    $\Set{ \lambda \alpha}{ \lambda \in \R } \cap \RS = \{ \alpha, -\alpha \}$ for any $\alpha \in \RS$;
    \item \label{cond:R3}
    $\inpr{\alpha}{\beta^{\vee}} \in \Z$ for any $\alpha, \beta \in \RS$.
\end{enumerate}
A member $\alpha$ of $\RS$ is called a \emph{root} and $\alpha^{\vee}$ a \emph{coroot}.
Let $\RS^{\vee} \coloneqq \Set{\alpha^{\vee} }{ \alpha \in \RS } \subseteq V$.
Then, $\RS^{\vee}$ also satisfies \eqref{cond:R1}--\eqref{cond:R3}, and is called the \emph{coroot system}.
A root system $\RS$ is said to be \emph{essential} if $\spn{\RS} = V^*$.

It is known~\cite[Proposition~8.6]{Hall2015} that
for
any two distinct roots $\alpha, \beta \in \RS$ with $\| \alpha \| \geq \|\beta\|$ and $\alpha \neq -\beta$,
one of the following \eqref{cond:T0}--\eqref{cond:T3} holds:
\begin{enumerate}[label={\textup{(T\arabic*)}}, ref={\textup{T\arabic*}}]
\setcounter{enumi}{-1}
    \item \label{cond:T0}
    The angle between $\alpha$ and $\beta$ is $\pi/2$, or equivalently, $\inpr{\alpha}{\beta^{\vee}} = \inpr{\beta}{\alpha^\vee} = 0$;
    \item \label{cond:T1}
    $\| \alpha \| = \|\beta\|$ and the angle between $\alpha$ and $\beta$ is $\pi/3$ or $2\pi/3$, or equivalently, $\inpr{\alpha}{\beta^{\vee}} = \inpr{\beta}{\alpha^\vee} = 1$ or $\inpr{\alpha}{\beta^{\vee}} = \inpr{\beta}{\alpha^\vee} = -1$;
    \item \label{cond:T2}
    $\| \alpha \| = \sqrt{2}\| \beta \|$ and the angle between $\alpha$ and $\beta$ is $\pi/4$ or $3\pi/4$, or equivalently, $\inpr{\alpha}{\beta^{\vee}} = 2$ and $\inpr{\beta}{\alpha^\vee} = 1$, or $\inpr{\alpha}{\beta^{\vee}} = -2$ and $\inpr{\beta}{\alpha^\vee} = -1$;
    \item \label{cond:T3}
    $\| \alpha \| = \sqrt{3}\| \beta \|$ and the angle between $\alpha$ and $\beta$ is $\pi/6$ or $5\pi/6$, or equivalently, $\inpr{\alpha}{\beta^{\vee}} = 3$ and $\inpr{\beta}{\alpha^\vee} = 1$, or $\inpr{\alpha}{\beta^{\vee}} = -3$ and $\inpr{\beta}{\alpha^\vee} = -1$.
\end{enumerate}

If a root system $\RS$ admits a nontrivial partition $\{ \RS_1, \RS_2, \dots, \RS_k \}$
such that,
for any $\alpha \in \RS_i$ and $\beta \in \RS_j$ with distinct $i,j$,
we have $\inpr{\alpha}{\beta^{\vee}} = 0$,
i.e.,
$\alpha$ and $\beta$ satisfy (T0),
then each of $\RS_1, \RS_2, \dots, \RS_k$ is also a root system;
in this case, $\RS$ is called the \emph{direct sum} of $\RS_1, \RS_2, \dots, \RS_k$
and is denoted by $\RS_1 \oplus \RS_2 \oplus \cdots \oplus \RS_k$.
A root system $\RS$ is said to be \emph{reducible} if $\RS$ is the direct sum of at least two root systems,
and \emph{irreducible} if it is not reducible.
If $\RS = \RS_1 \oplus \RS_2 \oplus \cdots \oplus \RS_k$ for irreducible root systems $\RS_1, \RS_2, \dots, \RS_k$,
then each of $\RS_1, \RS_2, \dots, \RS_k$ is called an \emph{irreducible component} of $\RS$.
The set of irreducible components of $\RS$ is uniquely determined.
Two root systems $\RS_1$ and $\RS_2$ are said to be \emph{isomorphic}
if there exists an invertible linear map $\funcdoms{\varphi}{V^*}{V^*}$ such that
$\varphi$ induces a bijection between $\RS_1$ and $\RS_2$ and for all $\alpha \in \RS_1$ and $p \in V^*$, we have $\varphi(s_{\alpha}(p)) = s_{\varphi(\alpha)}(\varphi(p))$, i.e., the following diagram
\begin{equation}\label{eq:diagram}
    \begin{tikzcd}
    V^* \ar[r, "s_\alpha"] \ar[d,"\varphi"'] & V^* \ar[d,"\varphi"] \\
    V^* \ar[r, "s_{\varphi(\alpha)}"'] & V^*
  \end{tikzcd}
\end{equation}
commutes.
We write $\RS_1 \lesssim \RS_2$
if for each $\alpha \in \RS_1$,
there exist $\mu \in \R_+$ and $\alpha' \in \RS_2$ such that $\alpha' = \mu \alpha$.
If $\RS_1 \lesssim \RS_2$ and $\RS_1 \gtrsim \RS_2$,
then we say that $\RS_1$ and $\RS_2$ are \emph{equivalent},
and denoted by $\RS_1 \sim \RS_2$.

Let $\RS$ be a root system.
A subset $\Delta \subseteq \RS$ is a \emph{simple system} if
$\Delta$ is a basis of $\spn{\RS}$ as a vector space and $\RS = \left(\RS \cap \cone{\Delta}\right) \cup \left(\RS \cap \cone{(-\Delta)}\right)$.
We can easily see that, if $\RS = \RS_1 \oplus \RS_2 \oplus \cdots \oplus \RS_k$
and $\Delta_i$ is a simple system of $\RS_i$ for each $i$,
then $\Delta = \Delta_1 \cup \Delta_2 \cup \cdots \cup \Delta_k$ is a simple system of $\RS$.
Suppose that $n \coloneqq \dim (\spn{\RS})$ and that $\Delta = \{ \alpha_1, \alpha_2, \dots, \alpha_n \} \subseteq \RS$ is a simple system.
By the first condition in the definition of a root system, every vector $p \in \spn{\RS}$ is uniquely representable as the linear combination of $\alpha_1, \alpha_2, \dots, \alpha_n$; $p = \sum_{i = 1}^n p_i \alpha_i$ for some unique $p_1, p_2, \dots, p_n \in \R$.
We introduce a partial order $\leq_{\Delta}$ on $V^*$ as:
For $p = \sum_{i = 1}^n p_i \alpha_i, q = \sum_{i = 1}^n q_i \alpha_i \in V^*$,
$p \leq_{\Delta} q$ if and only if $p_i \leq q_i$ for all $i \in [n]$.
The second condition in the definition of a simple system says that every root $\alpha \in \RS$ is either $\alpha >_{\Delta} 0$ or $\alpha <_{\Delta} 0$.
The former root is said to be \emph{positive} and the latter \emph{negative}.
An intriguing fact is that
every root $\alpha \in \RS$ is representable as the linear combination of the simple roots with \emph{integral} coefficients:
\begin{lemma}[{\cite[Chapter~2.9]{Humphreys1992}}]\label{lem:root_integrality}
    Let $\RS$ be a root system and $\{ \alpha_1, \alpha_2, \dots, \alpha_n \}$ a simple system of $\RS$.
    For every root $\alpha \in \RS$, there (uniquely) exist integers $c_1, c_2, \dots, c_n \in \Z$ such that 
    $\alpha = \sum_{i = 1}^n c_i \alpha_i$.
\end{lemma}

Suppose that a root system $\RS$ is irreducible.
Then,
for each simple system $\Delta = \{ \alpha_1, \alpha_2, \dots, \alpha_n \}$,
there exists a unique maximal element $\tilde{\alpha} \in \RS$ with respect to $\leq_{\Delta}$,
which is called the \emph{highest root with respect to $\Delta$}.
The set $\esimp \coloneqq \{ \alpha_0, \alpha_1, \dots, \alpha_n \}$ that consists of the roots $\alpha_i$ in a simple system $\simp = \{ \alpha_1, \dots, \alpha_n \}$
and the negative $\alpha_0 \coloneqq - \tilde{\alpha}$ of the highest root $\tilde{\alpha}$ with respect to $\simp$
is called 
an \emph{extended simple system}.
In the following,
when we say that $\esimp = \{ \alpha_0, \alpha_1, \dots, \alpha_n \}$ is an extended simple system,
we assume that $\simp = \{\alpha_1, \dots, \alpha_n\}$ is a simple system of an irreducible root system and
$\alpha_0$ is the negative of the highest root with respect to $\simp$.

\subsection{Weyl groups and (extended) Dynkin diagrams}\label{subsec:Weyl-Dynkin}

Let $\Phi$ be a root system.
The group $W_\RS$ generated by the reflections $s_\alpha$ with respect to the roots $\alpha \in \RS$, i.e.,
\begin{align}\label{eq:def:Weyl-group}
    W_\RS \coloneqq \gen{s_\alpha}{\alpha \in \RS}
\end{align}
is called the \emph{Weyl group} with respect to $\RS$.
The Weyl group is viewed as the subgroup of $\OG(V)$ or $\OG(V^*)$ that depends on the context,
i.e.,
this group naturally acts on $V$ or $V^*$ as
$W_\RS \ni w \colon v \mapsto wv$.
We can easily see that, if $\RS = \RS_1 \oplus \RS_2 \oplus \cdots \oplus \RS_k$
then $W_\RS = W_{\RS_1} \oplus W_{\RS_2} \oplus \cdots \oplus W_{\RS_k}$,
and that if $\RS_1$ and $\RS_2$ are equivalent
then $W_{\RS_1} = W_{\RS_2}$.
The following lemma states that the Weyl group is generated by the reflections with respect to a simple system:
\begin{lemma}[{\cite[Theorem~1.69~(1)]{Abramenko2010} and \cite[Proposition~8.24]{Hall2015}}]\label{lem:Weyl-simple-system}
    For an arbitrary simple system $\Delta$ of a root system $\RS$,
    we have $W_\RS = \gen{ s_\alpha }{ \alpha \in \Delta }$.
\end{lemma}

For $S \subseteq V^*$ (or $S \subseteq V$) and $w \in W_{\RS}$,
we denote $\Set{ w s }{ s \in S }$ by $w S$.
The following is an intriguing fact on the Weyl group.
\begin{lemma}[{\cite[Theorem~1.69~(2)]{Abramenko2010}}]\label{lem:Weyl-group}
    The Weyl group $W_\RS$ acts simply transitively on the set of simple systems, i.e.,
    for any $w \in W_\RS$ and simple system $\Delta$,
    the set $w\Delta$
    is also a simple system,
    and for any two simple systems $\Delta, \Delta'$ of $\RS$,
    there exists a unique $w \in W_\RS$ such that $\Delta' = w \Delta$.
    The same holds for the extended simple systems.
\end{lemma}
Thus, we obtain the one-to-one correspondence between the set of simple systems and the Weyl group $W_\RS$, once we fix a simple system that corresponds to the identity element of $W_\RS$.

Let $\Delta = \{ \alpha_1, \alpha_2, \dots, \alpha_n \}$ be a simple system of $\RS$.
The \emph{Dynkin diagram} $D(\RS)$ of $\RS$ is a mixed graph (a graph that may have edges and arcs) such that
its vertex set is $\{ v_1,v_2,\dots, v_n \}$
and
for each pair of distinct vertices $v_i, v_j$,
\begin{enumerate}[label={\textup{(D\arabic*)}}, ref={\textup{D\arabic*}}]
\setcounter{enumi}{-1}
    \item \label{cond:D0}
    there are no edges and arcs if $\alpha_i$ and $\alpha_j$ satisfy~\eqref{cond:T0},
    i.e., $\inpr{\alpha_i}{\alpha_j^\vee} = \inpr{\alpha_j}{\alpha_i^\vee} = 0$;
    \item \label{cond:D1}
    there is an edge $\{v_i,v_j\}$ if $\alpha_i$ and $\alpha_j$ satisfy~\eqref{cond:T1},
    i.e., $\inpr{\alpha_i}{\alpha_j^\vee} = \inpr{\alpha_j}{\alpha_i^\vee} = -1$;
    \item \label{cond:D2}
    there are two parallel arcs of the form $(v_i,v_j)$ if $\alpha_i$ and $\alpha_j$ satisfy~\eqref{cond:T2},
    i.e., $\inpr{\alpha_i}{\alpha_j^{\vee}} = -2$ and $\inpr{\alpha_j}{\alpha_i^\vee} = -1$;
    \item \label{cond:D3}
    there are three parallel arcs of the form $(v_i,v_j)$ if $\alpha_i$ and $\alpha_j$ satisfy~\eqref{cond:T3},
    i.e., $\inpr{\alpha_i}{\alpha_j^{\vee}} = -3$ and $\inpr{\alpha_j}{\alpha_i^\vee} = -1$.
\end{enumerate}
Here, the inner product conditions in \eqref{cond:D0}--\eqref{cond:D3}
follow from the fact \cite[Proposition~8.13]{Hall2015} that $\inpr{\alpha}{\beta^\vee} \leq 0$ for any distinct $\alpha, \beta \in \simp$.

The \emph{extended Dynkin diagram} $\tilde{D}(\RS)$ of an irreducible root system $\RS$ is defined from an extended simple system $\esimp = \{ \alpha_0, \alpha_1, \dots, \alpha_n \}$ of $\RS$ similarly as the Dynkin diagram from a simple system $\Delta$
with the following additional rule:
\begin{enumerate}
    \item[(D$\infty$)] There is an edge $\{v_i,v_j\}$ with label $\infty$ if $\alpha_i = -\alpha_j$.
\end{enumerate}
The vertex set of $\tilde{D}(\RS)$ is $\{ v_0, v_1, \dots, v_n \}$
in which $v_i$ corresponds to $\alpha_i$
and the edge set is similarly defined as that of the Dynkin diagram.
For notational simplicity,
we also denote by $D(\RS)$ (resp. $\tilde{D}(\RS)$) the vertex set of the Dynkin diagram $D(\RS)$ (resp. the extended Dynkin diagram $D(\RS)$).

The graph structures of $D(\RS)$ and $\tilde{D}(\RS)$ are independent of the choice of a simple system by \Cref{lem:Weyl-group}.
In addition, they characterize the structure of a root system $\RS$.
\begin{lemma}[{\cite[Proposition~8.32]{Hall2015}}]\label{lem:root_integer}
    \begin{enumerate}[label={\textup{(\arabic*)}}]
        \item Root systems $\RS_1$ and $\RS_2$ are isomorphic if and only if $D(\RS_1)$ and $D(\RS_2)$ are isomorphic in the graphical sense.
        \item A root system is irreducible if and only if its Dynkin diagram is connected.
    \end{enumerate}
\end{lemma}

Next, we introduce the \emph{type} $\tau_\simp(\alpha)$ of a root $\alpha$ in a simple system $\simp$ as follows.
Suppose that $D(\RS)$, whose vertex set is $\{v_1, v_2, \dots, v_n\}$, is constructed from a simple system $\simp = \{ \alpha_1, \alpha_2, \dots, \alpha_n \}$ as in the definition of the Dynkin diagram.
By \Cref{lem:Weyl-group},
for any simple system $\simp'$,
there exists a unique $w \in W_{\RS}$ such that $\simp' = w\simp = \{ w\alpha_1, w\alpha_2, \dots, w\alpha_n \}$.
Hence,
we can uniquely correspond $w\alpha_i \in \simp'$ to $\alpha_i \in \simp$, and hence to the vertex $v_i \in D(\RS)$.
Then, we define $\tau_{\simp'}(w\alpha_i) \coloneqq v_i$,
which is referred to as the \emph{type of $w\alpha_i$ with respect to $\simp'$}.
Also, in the case of the extended Dynkin diagram $\tilde{D}(\RS)$,
we can provide the one-to-one correspondence between $w\alpha_0, w\alpha_1, w\alpha_2, \dots, w\alpha_n \in w\tilde{\Delta}$
and $v_0, v_1, v_2, \dots, v_n \in \tilde{D}(\RS)$ in a similar manner;
we define $\tau_{w\esimp}(w\alpha_i) \coloneqq v_i$,
which is referred to as the \emph{type of $w\alpha_i$ with respect to $w\esimp$}.

For a simple system $\simp$ and vertex subset $J \subseteq D(\RS)$,
let
\begin{align}
    \simp|_J &\coloneqq \Set{ \alpha \in \simp }{ \tau_\simp(\alpha) \in J },\\
    \simp|_{-J} &\coloneqq \simp \setminus \simp|_J.
\end{align}
For notational simplicity,
we denote $\simp|_{-\{v\}}$ 
by $\simp|_{-v}$
for a vertex $v \in D(\RS)$.
We similarly define $\esimp|_J$ and $\esimp|_{-J}$ for an extended simple system $\esimp$ and a nonempty and proper vertex subset $J \subseteq \tilde{D}(\RS)$,
and denote $\esimp|_{-\{v\}}$ by $\esimp|_{-v}$ for a vertex $v \in \tilde{D}(\RS)$.
We remark that $\esimp|_{-v_0} = \simp$ for an irreducible root system $\RS$ and its simple system $\simp$.
For a subset $\simp' \subseteq \simp$ or $\simp' \subseteq \esimp$,
let
\begin{align}
    \RS(\simp') &\coloneqq \bigcup_{ w \in \gen{s_\alpha}{\alpha \in \simp'}} w \simp'.
\end{align}
\begin{lemma}[{\cite[Corollary~2.6]{Humphreys1992}}; see also \Cref{tab:classical-Dynkin-diagrams}]\label{lem:contraction:root-system}
    \begin{enumerate}[label={\textup{(\arabic*)}}]
        \item Let $\RS$ be a root system, $\simp$ a simple system of $\RS$, and $J \subseteq D(\RS)$ a nonempty vertex subset of $D(\RS)$.
    Then, $\RS(\simp|_{J})$ is a root system whose Dynkin diagram is the graph obtained from $D(\RS)$ by restricting the vertex set to $J$.
    Moreover, the simple systems of $\RS(\simp|_{J})$ are exactly the sets $w \simp|_{J}$ for $w \in \gen{s_\alpha}{\alpha \in \simp|_{J}}$.
        \item Let $\RS$ be an irreducible root system, $\esimp$ an extended simple system of $\RS$, and $J \subseteq \tilde{D}(\RS)$ a nonempty and proper vertex subset of $\tilde{D}(\RS)$.
        Then, $\RS(\esimp|_{J})$ is a root system whose Dynkin diagram is the graph obtained from $\tilde{D}(\RS)$ by restricting the vertex set to $J$,
        and the simple systems of $\RS(\esimp|_{J})$ are exactly the sets $w \esimp|_{J}$ for $w \in \gen{s_\alpha}{\alpha \in \esimp|_{J}}$.
    \end{enumerate}
\end{lemma}

We then introduce the \emph{mark} of a vertex of $D(\RS)$ and $\tilde{D}(\RS)$.
Suppose that $\RS$ is an irreducible root system and $\simp = \{ \alpha_1, \alpha_2, \dots, \alpha_n \}$ is a simple system of $\RS$.
Then, $\simp$ admits the highest root $\tilde{\alpha} = \sum_{i = 1}^n a_i \alpha_i \in \RS$ with respect to $\simp$,
where each of the coefficients $a_i$ is a nonnegative integer by \Cref{lem:root_integrality};
the coefficients $a_i$ are actually positive integers by the maximality of $\tilde{\alpha}$
(see also \Cref{tab:classical-Dynkin-diagrams}).
For a vertex $v_i$ of $D(\RS)$ with $\tau_\simp(\alpha_i) = v_i$,
we refer to the coefficient $a_i$ as the \emph{mark} of $v_i$.
This is well defined,
since for any simple system $\simp'$ of $\RS$,
the coefficient of $\alpha' \in \simp'$ in the highest root with respect to $\simp'$ coincides with
that of $\alpha \in \simp$ if
$\tau_{\simp'}(\alpha') = \tau_{\simp}(\alpha)$.
Similarly,
we define the \emph{marks} of vertices $v_1, v_2, \dots, v_n$ of $\tilde{D}(\RS)$ as $a_1, a_2, \dots, a_n$, respectively.
In addition, we define the \emph{mark} $a_0$ of $v_0 \in \tilde{D}(\RS) \setminus D(\RS)$ (that corresponds to the negative of the highest root) as $1$.
Then, the equality $\sum_{i = 0}^n a_i\alpha_i = 0$ holds.
If $\RS = \RS_1 \oplus \RS_2 \oplus \cdots \oplus \RS_k$ is reducible,
where $\RS_1, \RS_2, \dots, \RS_k$ are irreducible components of $\RS$,
then $D(\RS)$ is the disjoint union of $D(\RS_1), D(\RS_2), \dots, D(\RS_k)$ (see \Cref{lem:root_integer}~(2));
we define the \emph{mark} of $v \in D(\RS_i) \subseteq D(\RS)$ as the mark of $v$ with respect to $D(\RS_i)$.

\subsection{Classification of irreducible root systems}\label{subsec:classification}
It is known (see e.g.,~\cite[Chapters~2 and~4]{Humphreys1992}) that irreducible root systems have been completely classified
as four infinite families $\textup{A}_n$, $\textup{B}_n$, $\textup{C}_n$ ($n \geq 1$), and $\textup{D}_n$ ($n \geq 3$) of root systems, called the \emph{classical root systems},
and five exceptional ones $\textup{E}_6$, $\textup{E}_7$, $\textup{E}_8$, $\textup{F}_4$, and $\textup{G}_2$, called the \emph{exceptional root systems}.
Here, the subscript of each of the letters A, B, C, D, E, F, and G describes the dimension of $\spn{\RS}$, or equivalently, the size of a simple system.

\Cref{tab:classical-Dynkin-diagrams} provides the (extended) Dynkin diagrams of the classical root systems of types $\textup{A}_n$, $\textup{B}_n$, $\textup{C}_n$, and $\textup{D}_n$.
The circles represent the vertices $(v_0,) v_1, \dots, v_n$ of the (extended) Dynkin diagram,
and a positive integer written inside a vertex represents its mark.
In particular, in the extended Dynkin diagrams,
the vertex $v_0$ corresponds to the negative of the highest root.
\begin{table}[tbp]
    \centering
    \small
    \setlength{\tabcolsep}{4pt}
    \renewcommand{\arraystretch}{1.5}

    \begin{tabular}{
        @{}
        >{\centering\arraybackslash}p{0.07\linewidth}
        >{\centering\arraybackslash}p{0.35\linewidth}
        >{\centering\arraybackslash}p{0.49\linewidth}
        @{}
    }
        \toprule
        Type
        &
        Dynkin diagram
        &
        Extended Dynkin diagram
        \\
        \midrule

        $\Dynkin{A}{n}$
        &
        \begin{minipage}[c]{\linewidth}
            \centering
            \DynkinDiagramA
        \end{minipage}
        &
        \begin{minipage}[c]{\linewidth}
            \centering
            \ExtendedDynkinDiagramAOne
            \par
            $n=1$
            \par\medskip
            \ExtendedDynkinDiagramA
            \par
            $n\geq2$
        \end{minipage}
        \\

        \specialrule{\lightrulewidth}{0.8em}{0.8em}

        $\Dynkin{B}{n}$
        &
        \begin{minipage}[c]{\linewidth}
            \centering
            \DynkinDiagramB
        \end{minipage}
        &
        \begin{minipage}[c]{\linewidth}
            \centering
            \ExtendedDynkinDiagramBTwo
            \par
            $n=2$
            \par\medskip
            \ExtendedDynkinDiagramB
            \par
            $n\geq3$
        \end{minipage}
        \\

        \specialrule{\lightrulewidth}{0.8em}{0.8em}

        $\Dynkin{C}{n}$
        &
        \begin{minipage}[c]{\linewidth}
            \centering
            \DynkinDiagramC
        \end{minipage}
        &
        \begin{minipage}[c]{\linewidth}
            \centering
            \ExtendedDynkinDiagramC
        \end{minipage}
        \\

        \specialrule{\lightrulewidth}{0.8em}{0.8em}

        $\Dynkin{D}{n}$
        &
        \begin{minipage}[c]{\linewidth}
            \centering
            \DynkinDiagramDThree
            \par
            $n=3$
            \par\medskip
            \DynkinDiagramD
            \par
            $n\geq4$
        \end{minipage}
        &
        \begin{minipage}[c]{\linewidth}
            \centering
            \ExtendedDynkinDiagramDThree
            \par
            $n=3$
            \par\medskip
            \ExtendedDynkinDiagramD
            \par
            $n\geq4$
        \end{minipage}
        \\

        \bottomrule
    \end{tabular}

    \caption{
        The Dynkin and extended Dynkin diagrams of the classical
        types.
    }
    \label{tab:classical-Dynkin-diagrams}
\end{table}

In this paper, we adopt the root systems $\RSA, \RSB, \RSC, \RSD$,
given in \Cref{tab:classical-coordinate-realizations},
as the concrete realizations for the classical root systems of type $\Dynkin{A}{n}, \Dynkin{B}{n}, \Dynkin{C}{n}, \Dynkin{D}{n}$, respectively.
It is worth noting that
a commonly used realization of a root system of type
$\Dynkin{C}{n}$ is $2\RSC = \Set{ \pm e_i \pm e_j \in \R^n }{  \text{$i,j \in [n]$ with $i \neq j$} } \cup \Set{ \pm 2e_i \in \R^n }{ i \in [n] }$.
We use the rescaled realization $\RSC$, since it makes the relations with UJ-convex and bisubmodular
functions transparent; see \Cref{tab:classical-coxeter-vertices} and \Cref{subsec:example:bisubmodular,subsec:example:UJ-convex}.
We can easily see that $\RSB \sim \RSC$.
\Cref{tab:classical-coordinate-realizations} also presents
the Weyl groups $W_{\RSA}, W_{\RSB}, W_{\RSC}, W_{\RSD}$ corresponding to $\RSA, \RSB, \RSC, \RSD$,
respectively.
Here,
for a positive integer $m$, let $S_m$ denote the
symmetric group on $[m]$.
Furthermore,
for $\pi\in S_{n+1}$, let $s_\pi$ be the linear map $\sum_{i=1}^{n+1} c_i e_i \mapsto \sum_{i=1}^{n+1}c_i e_{\pi(i)}$,
and
for $\pi\in S_n$ and a sign
$\funcdoms{\sgn}{[n]}{\{\pm 1\}}$, let $s_{(\pi,\sgn)}$ be the linear map $\sum_{i=1}^{n} c_i e_i \mapsto \sum_{i=1}^{n} \sgn(i)c_i e_{\pi(i)}$.
\begin{table}[tbp]
    \centering
    \small
    \setlength{\tabcolsep}{5pt}
    \renewcommand{\arraystretch}{1.25}

    \begin{tabularx}{\linewidth}{
    @{}
    >{\centering\arraybackslash}m{0.09\linewidth}
    >{\centering\arraybackslash}m{0.48\linewidth}
    >{\centering\arraybackslash}X
    @{}
}
        \toprule
        Type
        &
        Concrete realization
        &
        Weyl group
        \\
        \midrule

        $\Dynkin{A}{n}$
        &
        $\displaystyle
            \RSA
            \coloneqq
            \Set*{
                e_i-e_j\in\R^{n+1}
            }{
                i,j\in[n+1],\ i\neq j
            }
        $
        &
        $\displaystyle
            W_{\RSA}
            =
            \Set*{
                s_\pi
            }{
                \pi\in S_{n+1}
            }
        $
        \\
        \typesep

        $\Dynkin{B}{n}$
        &
        $\displaystyle
            \begin{aligned}
                \RSB
                \coloneqq{}&
                \Set*{
                    \pm e_i\pm e_j\in\R^n
                }{
                    i,j\in[n],\ i\neq j
                }
                \\
                &{}\cup
                \Set*{
                    \pm e_i\in\R^n
                }{
                    i\in[n]
                }
            \end{aligned}
        $
        &
        $\displaystyle
            W_{\RSB}
            =
            \Set*{
                s_{(\pi,\sgn)}
            }{
                \begin{gathered}
                    \pi\in S_n,\\
                    \sgn\in\{\pm1\}^{n}
                \end{gathered}
            }
        $
        \\
        \typesep

        $\Dynkin{C}{n}$
        &
        $\displaystyle
            \begin{aligned}
                \RSC
                \coloneqq{}&
                \Set*{
                    \pm\frac{1}{2} e_i
                    \pm\frac{1}{2} e_j
                    \in\R^n
                }{
                    i,j\in[n],\ i\neq j
                }
                \\
                &{}\cup
                \Set*{
                    \pm e_i\in\R^n
                }{
                    i\in[n]
                }
            \end{aligned}
        $
        &
        $\displaystyle
            \begin{aligned}
                W_{\RSC}
                &=
                W_{\RSB}
                \\
                &=
                \Set*{
                    s_{(\pi,\sgn)}
                }{
                    \begin{gathered}
                        \pi\in S_n,\\
                        \sgn\in\{\pm1\}^{n}
                    \end{gathered}
                }
            \end{aligned}
        $
        \\
        \typesep

        $\Dynkin{D}{n}$
        &
        $\displaystyle
            \RSD
            \coloneqq
            \Set*{
                \pm e_i\pm e_j\in\R^n
            }{
                i,j\in[n],\ i\neq j
            }
        $
        &
        $\displaystyle
        \begin{aligned}
            &W_{\RSD} =\\
            &\Set*{
                s_{(\pi,\sgn)}
            }{
                \begin{gathered}
                    \pi\in S_n,\quad
                    \sgn\in\{\pm1\}^{n},\\
                    \card{
                        \Set*{
                            i\in[n]
                        }{
                            \sgn(i)=-1
                        }
                    }
                    \in2\Z
                \end{gathered}
            }
            \end{aligned}
        $
        \\

        \bottomrule
    \end{tabularx}

    \caption{
        The concrete realizations chosen in this paper for the
        classical root systems and their Weyl groups.
    }
    \label{tab:classical-coordinate-realizations}
\end{table}

Additional properties of the classical root systems, which are used (only) in the proof of \Cref{prop:tangent-cone},
are summarized in \Cref{subsec:proof}.

\section{Combinatorial structures arising from root systems}\label{sec:combinatorial-structures}
In this section, we introduce the combinatorial structures arising from root systems, 
which serve as the underlying discrete domains for discrete convex analysis. 
\Cref{subsec:root-lattice} defines \emph{root lattices} and \emph{coroot lattices}. 
\Cref{subsec:complexes} details \emph{spherical and Euclidean Coxeter complexes}, including their spherical expressions, geometric realizations, (positively homogeneous) Lov\'{a}sz extensions, and the properties of their convex subcomplexes.

Let $\RS$ be a root system.
In this section,
for notational simplicity,
we assume that $\RS$ is essential,
i.e., $V^* = \spn{\RS}$;
if $\RS$ is not essential,
in the following argument,
we replace $V$ and $V^*$ with $V/(\spn{\RS})^\perp$ and $\spn{\RS}$, respectively; recall \Cref{subsec:vectorspace}.
Let $n \coloneqq \dim \spn{\RS}$, which coincides with $\dim V^* = \dim V$ by the assumption.

\subsection{(Co)root lattices}\label{subsec:root-lattice}

The \emph{root lattice} $\Mdom(\RS)$ of $\RS$ is the set of integral linear combinations of $\RS$, i.e.,
\begin{align}
    \Mdom(\RS) \coloneqq \Set*{ \sum_{\alpha \in \RS} k_\alpha \alpha }{ k_\alpha \in \Z \ (\forall \alpha \in \RS) }.
\end{align}
\Cref{lem:root_integrality} implies that $\Mdom(\RS)$ is actually a lattice in $V^*$,
which is generated by an arbitrary simple system $\Delta$ of $\RS$:
\begin{align}
    \Mdom(\RS) = \Set*{ \sum_{\alpha \in \Delta} k_\alpha \alpha }{ k_\alpha \in \Z \ (\forall \alpha \in \Delta) }.
\end{align}
Similarly,
the \emph{coroot lattice} $\Mdom^{\vee}(\RS)$ of $\RS$ is
\begin{align}
    \Mdom^{\vee}(\RS) \coloneqq\, &\Set*{ \sum_{\alpha \in \RS} k_{\alpha} \alpha^{\vee} }{ k_{\alpha} \in \Z \ (\forall \alpha \in \RS) }\\
    =\, &\Set*{ \sum_{\alpha \in \Delta} k_\alpha \alpha^{\vee} }{  k_\alpha \in \Z \ (\forall \alpha \in \Delta) },
\end{align}
which is a lattice in $V$ generated by an arbitrary simple system $\Delta^{\vee} \coloneqq \Set{\alpha^{\vee} }{ \alpha \in \Delta }$ of the coroot system $\RS^{\vee}$.
By the crystallographic condition~\eqref{cond:R3}, we obtain
\begin{align}\label{eq:px:Z}
  \inpr{p}{x}
  \in
  \Z
  \qquad
  \rbra*{
    p\in\Mdom(\RS),\
    x\in\Mdom^\vee(\RS)
  }.
\end{align}

\Cref{tab:root-coroot-lattices} summarizes the root and coroot lattices of the
classical root systems $\RSA$, $\RSB$, $\RSC$, and $\RSD$.
Here, $[x] \coloneqq x + \spn{\onevec_{n+1}}$ denotes the equivalence class of $\Z^{n+1} / {\spn{\onevec_{n+1}}}$ that contains $x \in \Z^{n+1}$;
note that $\sum_{i = 1}^{n+1} x(i) \pmod {n+1}$ is well-defined for $[x] \in \Z ^{n+1} / \spn{\onevec_{n+1}}$.
\begin{table}[t]
    \centering
    \small
    \renewcommand{\arraystretch}{1.5}
    \begin{tabularx}{\linewidth}{
        @{}
        c
        >{\centering\arraybackslash}p{0.32\linewidth}
        >{\centering\arraybackslash}X
        @{}
    }
        \toprule
        Root system
        &
        Root lattice $\Mdom(\RS)$
        &
        Coroot lattice $\Mdom^{\vee}(\RS)$
        \\
        \midrule
        $\RSA$
        &
        $\displaystyle
        \Set*{
            p\in\Z^{n+1}
        }{
            \sum_{i=1}^{n+1}p(i)=0
        }$
        &
        $\displaystyle
        \Set*{
            [x]\in
            \Z^{n+1}/\spn{\onevec_{n+1}}
        }{
            \sum_{i=1}^{n+1}x(i)
            \equiv 0\pmod{n+1}
        }$
        \\
        \typesep
        $\RSB$
        &
        $\Z^n$
        &
        $\displaystyle
        \Set*{
            x\in\Z^n
        }{
            \sum_{i=1}^{n}x(i)\in2\Z
        }$
        \\
        \typesep
        $\RSC$
        &
        $\displaystyle
        \Set*{
            p\in(\Z/2)^n
        }{
            \sum_{i=1}^{n}p(i)\in\Z
        }$
        &
        $(2\Z)^n$
        \\
        \typesep
        $\RSD$
        &
        $\displaystyle
        \Set*{
            p\in\Z^n
        }{
            \sum_{i=1}^{n}p(i)\in2\Z
        }$
        &
        $\displaystyle
        \Set*{
            x\in\Z^n
        }{
            \sum_{i=1}^{n}x(i)\in2\Z
        }$
        \\
        \bottomrule
    \end{tabularx}
    \caption{
        The root and coroot lattices of the classical root systems $\RSA$, $\RSB$, $\RSC$, and $\RSD$.
    }
    \label{tab:root-coroot-lattices}
\end{table}

Suppose that $\RS = \RS_1 \oplus \RS_2 \oplus \cdots \oplus \RS_k$ in which $\RS_1, \RS_2, \dots, \RS_k$ are irreducible components of $\RS$.
Since $\spn{\RS_i}$ and $\spn{\RS_j}$ are orthogonal for any two distinct irreducible components $\RS_i, \RS_j$,
we have
\begin{align}\label{eq:Mdom-sum}
    \Mdom(\RS) = \Mdom(\RS_1) + \Mdom(\RS_2) + \cdots + \Mdom(\RS_k),
\end{align}
where $+$ denotes the Minkowski sum.

\subsection{Spherical and Euclidean Coxeter complexes}\label{subsec:complexes}
This subsection is devoted to introducing \emph{spherical Coxeter complexes} and \emph{Euclidean Coxeter complexes}.
These are special cases of a more general concept called \emph{Coxeter complexes}.
In this paper,
we only introduce spherical/Euclidean Coxeter complexes and we adapt the arguments for general Coxeter complexes to these special cases.
See \cite{Humphreys1992,Abramenko2010,Hall2015} for details.

\subsubsection{Definitions}\label{subsubsec:def-CoxeterComplexes}

Let $\RS$ be a root system.
For $\alpha \in \RS$ and $k \in \Z$,
we define a hyperplane $H_{\alpha, k}^0$ in $V$ by
\begin{align}
    H_{\alpha, k}^0 &\coloneqq \Set{ x \in V }{ \inpr{\alpha}{x} = k },
\end{align}
and open half-spaces $H_{\alpha, k}^+$ and $H_{\alpha, k}^-$ by
\begin{align}
    H_{\alpha, k}^+ \coloneqq \Set{ x \in V }{ \inpr{\alpha}{x} > k }, \qquad
    H_{\alpha, k}^- \coloneqq \Set{ x \in V }{ \inpr{\alpha}{x} < k }.
\end{align}
Let $\lincalH(\RS)$ and $\affcalH(\RS)$ be the hyperplane arrangements defined by
\begin{align}
    \lincalH(\RS) \coloneqq \Set{ H_{\alpha, 0}^0 }{ \alpha \in \RS }, \qquad
    \affcalH(\RS) \coloneqq \Set{ H_{\alpha, k}^0 }{ \alpha \in \RS,\ k \in \Z },
\end{align}
respectively.
Note that each hyperplane $H$ in $\lincalH(\RS)$ (resp. in $\affcalH(\RS)$) corresponds to exactly two indices $\alpha, -\alpha$ in $\RS$ (resp. $(\alpha, k), (-\alpha, -k)$ in $\RS \times \Z$).

A nonempty subset $A \subseteq V$ is called an \emph{(open) cell with respect to $\lincalH(\RS)$}
if there is a sequence $\sigma(A) = (\sigma_{\alpha}(A))_{\alpha \in \RS} \in \{ +, 0, - \}^{\RS}$ such that
\begin{align}
    A = \bigcap_{\alpha \in \RS} H_{\alpha, 0}^{\sigma_{\alpha}(A)},
\end{align}
and called an \emph{(open) cell with respect to $\affcalH(\RS)$}
if there is a sequence $\sigma(A) = (\sigma_{(\alpha, k)}(A))_{(\alpha, k) \in \RS \times \Z} \in \{ +, 0, - \}^{\RS \times \Z}$ such that
\begin{align}
    A = \bigcap_{(\alpha, k) \in \RS \times \Z} H_{\alpha, k}^{\sigma_{(\alpha, k)}(A)}.
\end{align}
In both cases,
the sequence $\sigma(A)$ corresponding to $A$ is uniquely determined;
we call $\sigma(A)$ the \emph{sign sequence} of $A$.

We are ready to define spherical and Euclidean Coxeter complexes:
Examples are given in \Cref{tab:classical-coxeter-chambers,tab:classical-coxeter-vertices}.

The \emph{spherical Coxeter complex} of type $\RS$, denoted by $\spCC(\RS)$, is the collection of all open cells with respect to $\lincalH(\RS)$, endowed with the facial relation $\preceq$ as: 
$A \preceq B$ if and only if $\cl{A} \subseteq \cl{B}$.
In the spherical Coxeter complex $\spCC(\RS)$,
the least element $\bot$ is $\{0\}$.
If $\RS$ is a root system of type A, B, C, or D,
then we refer to $\spCC(\RS)$ as the spherical Coxeter complex of type A, B, C, or D,
respectively.
If $\RS_1 \lesssim \RS_2$,
then we have $\lincalH(\RS_1) \subseteq \lincalH(\RS_2)$,
which implies that $\spCC(\RS_1)$ is a coarse simplicial complex of $\spCC(\RS_2)$
in the sense that, for any cell $A_2 \in \spCC(\RS_2)$,
there exists a unique cell $A_1 \in \spCC(\RS_1)$ such that $A_1 \supseteq A_2$.
In particular, if two root systems $\RS_1$ and $\RS_2$ are equivalent,
then $\spCC(\RS_1) = \spCC(\RS_2)$.
Hence, the spherical Coxeter complexes of type $\Dynkin{B}{n}$ and of type $\Dynkin{C}{n}$ are the same; see~\Cref{tab:classical-coxeter-chambers}.

In considering the Euclidean Coxeter complex,
we assume that $\RS$ is irreducible.
The \emph{Euclidean Coxeter complex} of type $\RS$, denoted by $\euCC(\RS)$,
is the collection of all open cells with respect to $\affcalH(\RS)$ with the empty set $\emptyset$, endowed with the facial relation $\preceq$ as: 
$A \preceq B$ if and only if $\cl{A} \subseteq \cl{B}$.
In the Euclidean Coxeter complex $\euCC(\RS)$,
the least element $\bot$ is $\emptyset$.
We define
the sign sequence of the least element $\emptyset$ as $(0)_{(\alpha, k) \in \RS \times \Z}$.
If $\RS$ is a root system of type A, B, C, or D,
then we refer to $\euCC(\RS)$ as the Euclidean Coxeter complex of type A, B, C, or D,
respectively.
Unlike spherical Coxeter complexes,
even if two irreducible root systems $\RS_1$ and $\RS_2$ are equivalent,
the corresponding Euclidean Coxeter complexes $\euCC(\RS_1)$ and $\euCC(\RS_2)$ are different in general
(see \Cref{tab:classical-coxeter-chambers} again).

Both $\spCC(\RS)$ and $\euCC(\RS)$ are simplicial complexes, or more precisely, colorable chamber complexes~\cite[Theorem~3.5]{Abramenko2010},
whose ranks are $n$
and $n + 1$, respectively.
The following provides
the one-to-one correspondence between the chambers and the simple systems (with the members of the coroot lattice in the case of a Euclidean Coxeter complex):
\begin{lemma}[{\cite[Proposition~8.21]{Hall2015}} and {\cite[Chapter~10.1.8]{Abramenko2010}}]\label{lem:chamber}
\begin{enumerate}[label={\textup{(\arabic*)}}]
\item 
For each chamber $C$ of $\spCC(\RS)$,
    there exists a unique simple system $\simp$ such that
\begin{align}\label{eq:spherical-chamber}
    C = \Set{ x \in V }{ \inpr{\alpha}{x} > 0 \ (\forall \alpha \in \Delta) }.
\end{align}
Conversely, for each simple system $\Delta$,
the set $C$ defined as~\eqref{eq:spherical-chamber} is a chamber of $\spCC(\RS)$.
\item For each chamber $C$ of $\euCC(\RS)$,
    there uniquely exist a simple system $\simp$ and $x_0 \in \Mdom^{\vee}(\RS)$
    such that
    \begin{align}\label{eq:affine-chamber}
        C = x_0 + \Set{ x \in V }{ \text{$\inpr{\alpha_0}{x} > -1$ and $\inpr{\alpha}{x} > 0$  $(\forall \alpha \in \Delta)$} },
    \end{align}
    where $\alpha_0$ is the negative of the highest root with respect to the simple system $\Delta$.
    Conversely, for each simple system $\simp$ and $x_0 \in \Mdom^{\vee}(\RS)$,
    the set $C$ defined as~\eqref{eq:affine-chamber} is a chamber of $\euCC(\RS)$.
\end{enumerate}
\end{lemma}
The above lemma (\Cref{lem:chamber}) also says that
each chamber $C$ of $\spCC(\RS)$ (resp. $\euCC(\RS)$) can be characterized by the inequalities indexed by the simple system $\Delta$ (resp. the extended simple system $\tilde{\Delta} = \Delta \cup \{ \alpha_0 \}$) corresponding to $C$.
By using the integral vector $\tilde{d} \in \Z^{\tilde{\Delta}}$ defined by
$\tilde{d}(\alpha) \coloneqq -1$ if $\alpha = \alpha_0$ and $\tilde{d}(\alpha) \coloneqq 0$ if $\alpha \in \Delta$,
the formulation~\eqref{eq:affine-chamber} of a chamber $C$ of $\euCC(\RS)$ is representable as
\begin{align}
    C = x_0 + \Set{ x \in V }{ \inpr{\alpha}{x} > \tilde{d}(\alpha) \ (\forall \alpha \in \tilde{\Delta}) }.
\end{align}

Fix an arbitrary chamber $C \in \spCC(\RS)$
and let $\Delta$ be the simple system corresponding to $C$.
Then, for each vertex $x \in \sk{C}$,
there exists a unique root $\alpha \in \Delta$ such that $\inpr{\alpha}{x} > 0$.
This enables us to define the bijection $\funcdoms{\tau_C}{\sk{C}}{D(\RS)}$
by $\tau_C(x) \coloneqq \tau_\simp(\alpha)$,
where $\alpha$ is the unique root in $\simp$ satisfying $\inpr{\alpha}{x} > 0$.
We define the \emph{canonical} type function $\funcdoms{\tau_{\spCC(\RS)}}{\sk{\spCC(\RS)}}{D(\RS)}$ of $\spCC(\RS)$
by extending the bijection $\tau_C$ (see the last paragraph of \Cref{subsec:simplicial_complex}).
We can similarly define the \emph{canonical} type function $\funcdoms{\tau_{\euCC(\RS)}}{\sk{\euCC(\RS)}}{\tilde{D}(\RS)}$ of $\euCC(\RS)$.
Note that the canonical type function $\tau_{\spCC(\RS)}$ (resp. $\tau_{\euCC(\RS)}$) is independent of the choice of a chamber $C \in \spCC(\RS)$ (resp. $C \in \euCC(\RS)$).
The subscripts of $\tau_{\spCC(\RS)}$ and $\tau_{\euCC(\RS)}$ may be omitted if it is clear from the context.

For a simplex $A \in \spCC(\RS)$ and the simple system $\simp$ corresponding to a chamber $C \in \St{\spCC(\RS)}{A}$, i.e., a chamber $C$ of $\spCC(\RS)$ with $C \succeq A$,
we define
\begin{align}
  \simp|_{A} \coloneqq \simp|_{\tau(A)}, \qquad \simp|_{-A} \coloneqq \simp|_{-\tau(A)}.
\end{align}
Similarly, for a simplex $A \in \euCC(\RS)$ and the extended simple system $\esimp$ corresponding to a chamber $C \in \St{\euCC(\RS)}{A}$,
we define
\begin{align}
  \esimp|_{A} \coloneqq \esimp|_{\tau(A)}, \qquad \esimp|_{-A} \coloneqq \esimp|_{-\tau(A)}.
\end{align}
Then, by a general theory of Coxeter complexes and \Cref{lem:chamber},
each simplex $A$ of $\spCC(\RS)$ or $\euCC(\RS)$
admits an easy representation as follows.
\begin{corollary}[{see e.g., \cite[Chapter~3]{Abramenko2010}}]\label{cor:type-function}
    \begin{enumerate}[label={\textup{(\arabic*)}}]
        \item Let $\RS$ be a root system.
        For any simplex $A \in \spCC(\RS)$ and simple system $\simp$ corresponding to a chamber 
        $C \in \St{\spCC(\RS)}{A}$, 
        we have
        \begin{align}\label{eq:spherical:A}
            A = \Set*{ x \in V }{ \text{$\inpr{\alpha}{x} > 0$ $(\forall \alpha \in \simp|_{A})$ and $\inpr{\alpha}{x} = 0$ $(\forall \alpha \in \simp|_{-A})$} }.
        \end{align}
        \item Let $\RS$ be an irreducible root system.
        For any simplex $A \in \euCC(\RS)$ and the pair $(x_0, \esimp)$ of $x_0 \in \Mdom^{\vee}(\RS)$ and extended simple system $\esimp$ corresponding to a chamber 
        $C \in \St{\euCC(\RS)}{A}$,
        we have
        \begin{align}\label{eq:affine:A}
            A = x_0 + \Set*{ x \in V }{ \text{$\inpr{\alpha}{x} > d^*_\alpha$ $(\forall \alpha \in \esimp|_{A})$ and $\inpr{\alpha}{x} = \tilde{d}(\alpha)$ $(\forall \alpha \in \esimp|_{-A})$} },
        \end{align}
        where
        $\tilde{d} \in \Z^{\esimp}$ is the integral vector defined by
$\tilde{d}(\alpha) \coloneqq -1$ if $\alpha = \alpha_0$ and $\tilde{d}(\alpha) \coloneqq 0$ if $\alpha \in \simp$.
    \end{enumerate}
\end{corollary}

For a simplex $A \in \spCC(\RS)$, we define
\begin{align}
  \RS|_{-A} \coloneqq \RS(\simp|_{-A}),
\end{align}
where $\simp$ is the simple system corresponding to a chamber 
$C \in \St{\spCC(\RS)}{A}$.
It is well defined, since $\RS(\simp|_{-A}) = \RS(\simp'|_{-A})$ for any simple systems $\simp$ and $\simp'$ corresponding to chambers 
$C, C' \in \St{\spCC(\RS)}{A}$,
respectively.
Similarly, for a simplex $A \in \euCC(\RS) \setminus \{\emptyset\}$, we define
\begin{align}
  \tilde{\RS}|_{-A} \coloneqq \RS(\esimp|_{-A}),
\end{align}
where $\esimp$ is the extended simple system corresponding to a chamber 
$C \in \St{\euCC(\RS)}{A}$;
this is also well defined.
The following is an easy consequence of \Cref{cor:type-function}.
\begin{lemma}\label{lem:span}
    \begin{enumerate}[label={\textup{(\arabic*)}}]
        \item Let $\RS$ be a root system and $A \in \spCC(\RS)$.
        Then,
        $\spn{A} = (\spn{\RS|_{-A}})^\perp$ holds.
        In particular, $\spCC(\RS|_{-A})$
        is a spherical Coxeter complex that lies in $V/\spn{A}$.
        \item Let $\RS$ be an irreducible root system and $A \in \euCC(\RS) \setminus \{\emptyset\}$.
        Then, $(\aff{A})_0 = (\spn{\tilde{\RS}|_{-A}})^\perp$ holds.
        In particular, $\spCC(\tilde{\RS}|_{-A})$
        is a spherical Coxeter complex that lies in $V/\aff{A}$.
    \end{enumerate}
\end{lemma}
\begin{proof}
    The latter assertions of (1) and (2) follow from the former assertions,
    since $\spCC(\RS|_{-A})$ and $\spCC(\tilde{\RS}|_{-A})$ are simplicial complexes that lie in $V/(\spn{(\RS|_{-A})})^\perp$ and in $V/(\spn{(\tilde{\RS}|_{-A})})^\perp$, respectively, by the definitions.
    In the following, we prove the former assertions of (1) and (2).

    (1).
    By \Cref{cor:type-function}~(1), there is a simple system $\simp$
    such that $A$ is representable as~\eqref{eq:spherical:A}.
    Then, we obtain
    \begin{align}\label{eq:spanA}
        \spn{A} = \Set{ x \in V }{ \inpr{\alpha}{x} = 0 \ (\forall \alpha \in \simp|_{-A}) } = (\spn{\simp|_{-A}})^\perp = (\spn{\RS|_{-A}})^\perp.
    \end{align}
    Here, since $\simp|_{-A}$ is a simple system of the root system $\RS|_{-A}$ by \Cref{lem:contraction:root-system}~(1),
    the last equality follows.

    (2).
    By \Cref{cor:type-function}~(2), there is an extended simple system $\tilde{\Delta}$ and $x_0 \in \Mdom^{\vee}(\RS)$
    such that $A$ is representable as~\eqref{eq:affine:A}.
    Then, by a similar argument to the assertion~(1) above,
    we obtain $\aff{A} = x_0 + \Set{ x \in V }{ \inpr{\alpha}{x} = \tilde{d}(\alpha) \ (\forall \alpha \in \esimp|_{-A}) }$,
    which implies
    \begin{align}\label{eq:affA}
        (\aff{A})_0 = \Set{ x \in V }{ \inpr{\alpha}{x} = 0 \ (\forall \alpha \in \esimp|_{-A}) } = (\spn{\esimp|_{-A}})^\perp = (\spn{\tilde{\RS}|_{-A}})^\perp.
    \end{align}
    In particular, the last equality follows from \Cref{lem:contraction:root-system}~(2).
\end{proof}

For $\CC = \spCC(\RS)$ or $\CC = \euCC(\RS)$,
the \emph{mark} of a vertex $x \in \sk{\CC}$,
denoted by $a_\CC(x)$,
is defined as the mark of its type $\tau_\CC(x)$ (see the last paragraph of \Cref{subsec:Weyl-Dynkin} for the definition of the mark of the types);
we refer to $\funcdoms{a_\CC}{\sk{\CC}}{\Z_+}$ as the \emph{mark function} of $\CC$.
The subscript $\spCC(\RS)$ of $a_{\spCC(\RS)}$ (resp. $\euCC(\RS)$ of $a_{\euCC(\RS)}$) may be omitted if it is clear from the context.

We present explicit descriptions of the chambers of the spherical
and Euclidean Coxeter complexes associated with the classical root
systems.
These descriptions, together with the corresponding Euclidean
vertex sets and mark functions, are summarized in
\Cref{tab:classical-coxeter-chambers,tab:classical-coxeter-vertices}.
Here, recall that $[x]\coloneqq x+\spn{\onevec_{n+1}} \in\R^{n+1}/\spn{\onevec_{n+1}}$ for $x\in\R^{n+1}$.

\begin{table}[t]
    \centering
    \small
    \setlength{\tabcolsep}{6pt}
    \renewcommand{\arraystretch}{1.25}
    \begin{tabularx}{\linewidth}{
        @{}
        >{\centering\arraybackslash}m{0.12\linewidth}
        >{\centering\arraybackslash}X
        @{}
    }
        \toprule
        Root system
        &
        Chambers
        \\
        \midrule

        $\RSA$
        &
        \coxeterchambercell{
            $\displaystyle
            \chamber{\pi}{}
            \coloneqq
            \Set*{
                [x]\in
                \R^{n+1}/\spn{\onevec_{n+1}}
            }{
                x(\pi(1))
                >
                \cdots
                >
                x(\pi(n+1))
            }$
        }{
            $\displaystyle
            \chamber{\pi}{[x_0]}
            \coloneqq
            [x_0]
            +
            \Set*{
                [x]\in\chamber{\pi}{}
            }{
                x(\pi(n+1))
                >
                x(\pi(1))-1
            }$
        }
        \\
        \typesep

        $\RSB$
        &
        \coxeterchambercell{
            $\displaystyle
            \chamber{\pi,\sgn}{}
            \coloneqq
            \Set*{
                x\in\R^n
            }{
                \sgn(1)x(\pi(1))
                >
                \cdots
                >
                \sgn(n)x(\pi(n))
                >
                0
            }$
        }{
            $\displaystyle
            \chamber{\pi,\sgn}{x_0}
            \coloneqq
            x_0
            +
            \Set*{
                x\in\chamber{\pi,\sgn}{}
            }{
                \sgn(1)x(\pi(1))
                +
                \sgn(2)x(\pi(2))
                <
                1
            }$
        }
        \\
        \typesep

        $\RSC$
        &
        \coxeterchambercell{
            $\displaystyle
            \chamber{\pi,\sgn}{}
            \coloneqq
            \Set*{
                x\in\R^n
            }{
                \sgn(1)x(\pi(1))
                >
                \cdots
                >
                \sgn(n)x(\pi(n))
                >
                0
            }$
        }{
            $\displaystyle
            \chamber{\pi,\sgn}{x_0}
            \coloneqq
            x_0
            +
            \Set*{
                x\in\chamber{\pi,\sgn}{}
            }{
                \sgn(1)x(\pi(1))
                <
                1
            }$
        }
        \\
        \typesep

        $\RSD$
        &
        \coxeterchambercell{
            $\displaystyle
            \chamber{\pi,\sgn}{}
            \coloneqq
            \Set*{
                x\in\R^n
            }{
                \sgn(1)x(\pi(1))
                >
                \cdots
                >
                \sgn(n-1)x(\pi(n-1))
                >
                \lvert
                    \sgn(n)x(\pi(n))
                \rvert
            }$
        }{
            $\displaystyle
            \chamber{\pi,\sgn}{x_0}
            \coloneqq
            x_0
            +
            \Set*{
                x\in\chamber{\pi,\sgn}{}
            }{
                \sgn(1)x(\pi(1))
                +
                \sgn(2)x(\pi(2))
                <
                1
            }$
        }
        \\

        \bottomrule
    \end{tabularx}

    \caption{
      Chambers of the spherical and Euclidean Coxeter complexes
        associated with $\RSA$, $\RSB$, $\RSC$, and $\RSD$.
        In type $\RSA$, 
        $\pi$ and $[x_0]$ are taken over $S_{n+1}$ and $\Mdom^{\vee}(\RSA)$,
        respectively.
        In type $\RS \in \{\RSB, \RSC, \RSD\}$, 
        $\pi$ and $x_0$ are taken over $S_{n}$ and $\Mdom^{\vee}(\RS)$,
        respectively.
        In addition, $\sgn \in \{\pm1\}^{n}$ in types $\RSB$ and $\RSC$,
        and $\sgn \in \{\pm1\}^{n}$ with $\card{\Set{i \in [n]}{\sigma(i) = -1}} \in 2\Z$ in type $\RSD$.
    }
    \label{tab:classical-coxeter-chambers}
\end{table}
\begin{table}[t]
    \centering
    \small
    \setlength{\tabcolsep}{5pt}
    \renewcommand{\arraystretch}{1.5}
    \renewcommand{\tabularxcolumn}[1]{m{#1}}
    \begin{tabularx}{\linewidth}{
        @{}
        >{\centering\arraybackslash}m{0.14\linewidth}
        >{\centering\arraybackslash}X
        >{\centering\arraybackslash}X
        @{}
    }
        \toprule
        Root system
        &
        \shortstack[c]{
            Vertex set \\of the Euclidean Coxeter complex
        }
        &
        Vertex mark
        \\
        \midrule
        $\RSA$
        &
        $\Z^{n+1}/\spn{\onevec_{n+1}}$
        &
        $a([x])=1$
        \\
        \typesep

        $\RSB$
        &
        $\Set*{
            x\in(\Z/2)^n
        }{
            \card{\supp_{1/2}(x)}\neq1
        }$
        &
        $\displaystyle
        a(x)
        =
        \begin{cases}
            1 & \text{if $x\in\Z^n$},\\
            2 & \text{otherwise}
        \end{cases}$
        \\
        \typesep

        $\RSC$
        &
        $\Z^n$
        &
        $\displaystyle
        a(x)
        =
        \begin{cases}
            1
            &
            \text{if
                $x\in(2\Z)^n\cup(2\Z+1)^n$},\\
            2
            &
            \text{otherwise}
        \end{cases}$
        \\
        \typesep

        $\RSD$
        &
        $\Set*{
            x\in(\Z/2)^n
        }{
            \card{\supp_{1/2}(x)}
            \notin
            \{1,n-1\}
        }$
        &
        $\displaystyle
        a(x)
        =
        \begin{cases}
            1
            &
            \text{if
                $x\in\Z^n\cup(\Z+1/2)^n$},\\
            2
            &
            \text{otherwise}
        \end{cases}$
        \\
        \bottomrule
    \end{tabularx}
    \caption{
        Vertex sets and vertex marks of the Euclidean Coxeter
        complexes associated with $\RSA$, $\RSB$, $\RSC$, and $\RSD$.
        Let $\supp_{1/2}(x) \coloneqq \Set{ i \in [n] }{ x(i) \in \Z + 1/2 }$ for $x \in (\Z/2)^n$.
    }
    \label{tab:classical-coxeter-vertices}
\end{table}

In connection with the discrete midpoint convexity (see~\eqref{eq:midpointconvexity} below), it is also important,
for each classical type, to have an explicit description of the
simplex $\cell{A}{z}$ containing the midpoint
$z\coloneqq(x+y)/2$ of two vertices
$x,y\in\sk{\euCC(\RS)}$.
In types~$\RSA$ and~$\RSC$, these descriptions are also useful for
clarifying the relationship between the present framework and
existing notions (see~\Cref{subsubsec:example:L,subsec:example:bisubmodular}).
We collect these descriptions in \Cref{subsec:DMC}.

It is known~\cite[Proposition~3.16]{Abramenko2010} that the link complex $\Lk{\CC}{A}$ of a (general) Coxeter complex $\CC$ at any simplex $A \in \CC$ is also a Coxeter complex.
More specifically, the following holds for a spherical/Euclidean Coxeter complex.
\begin{lemma}\label{lem:complex:contraction}
    \begin{enumerate}[label={\textup{(\arabic*)}}]
        \item Let $\RS$ be a root system and
        $A$ a simplex of $\spCC(\RS)$.
        Then, the map $B \mapsto B + \spn{A}$
        forms an isomorphism from $\Lk{\spCC(\RS)}{A}$ to $\spCC(\RS|_{-A})$.
        \item
        Let $\RS$ be an irreducible root system and $A$ a nonempty simplex of $\euCC(\RS)$.
        Then, the map $B \mapsto \cone{(B - \hat{x})} + (\aff{A})_0$ forms an isomorphism from $\Lk{\euCC}{A}$ to $\spCC(\tilde{\RS}|_{-A})$,
        where $\hat{x}$ is a point in $A$.
    \end{enumerate}
\end{lemma}
\begin{proof}
    We omit the proof of~(2), since it can be obtained by a similar argument to that of~(1).

    (1).
    Take any $B \in \Lk{\spCC(\RS)}{A}$
    and let $C$ be a chamber of $\spCC(\RS)$ with $A \preceq C$ and $B \preceq C$
    and $\simp$ the corresponding simple system to $C$.
    By \Cref{cor:type-function}~(1), we have
    \begin{align}
        B = \Set*{ x \in V }{ \text{$\inpr{\alpha}{x} > 0$ $(\forall \alpha \in \simp|_{B})$ and $\inpr{\alpha}{x} = 0$ $(\forall \alpha \in \simp|_{-B})$} }.
    \end{align}
    By a general theory of Coxeter complexes (see e.g.,~\cite[Chapter~3]{Abramenko2010}), \Cref{cor:type-function}~(1), and \Cref{lem:span}~(1),
    the map from $B$ to
    \begin{align}\label{eq:B+spanA}
        \Set{ x \in V / \spn{A} }{ \text{$\inpr{\alpha}{x} > 0$ $(\forall \alpha \in \simp|_{B})$ and $\inpr{\alpha}{x} = 0$ $(\forall \alpha \in \simp|_{-A} \cap \simp|_{-B})$} }
    \end{align}
    forms an isomorphism from $\Lk{\spCC(\RS)}{A}$ to $\spCC(\RS|_{-A})$.
    Here, we note that $\tau(A)$ and $\tau(B)$ are disjoint,
    i.e., $\simp|_{B} \subseteq \simp|_{-A}$.
    By the equation~\eqref{eq:spanA} on $\spn{A}$,
    the set~\eqref{eq:B+spanA} is representable as $B + \spn{A}$.
\end{proof}
\Cref{lem:complex:contraction}~(1) enables us to introduce the canonical type function and the mark function of $\CC' \coloneqq \Lk{\spCC(\RS)}{A}$
via those of $\CC \coloneqq \spCC(\RS|_{-A})$ 
i.e.,
$\tau_{\CC'}(x) \coloneqq \tau_{\CC}(x + \spn{A})$
and
$a_{\CC'}(x) \coloneqq a_{\CC}(x + \spn{A})$.
It is worth mentioning that
$\tau_{\CC'}$ is simply the restriction of $\tau_{\spCC(\RS)}$ to $\sk{\CC'}$,
but this is not the case for $a_{\CC'}$.
Similarly,
by \Cref{lem:complex:contraction}~(2),
we can introduce the canonical type function and the mark function of $\CC' \coloneqq \Lk{\euCC(\RS)}{A}$
via those of $\CC \coloneqq \spCC(\tilde{\RS}|_{-A})$, 
i.e.,
$\tau_{\CC'}(x) \coloneqq \tau_{\CC}(\cone{(x - x_0)} + (\aff{A})_0)$
and
$a_{\CC'}(x) \coloneqq a_{\CC}(\cone{(x-x_0)} + (\aff{A})_0)$ for $x \in \sk{\CC'}$.

As we can see in \Cref{cor:type-function} (and \Cref{tab:classical-coxeter-chambers}),
while the closure of each simplex in $\euCC(\RS)$ is a geometric simplex in $V$,
that of each simplex in $\spCC(\RS)$ is a simplicial cone in $V$.
In particular,
while a vertex in $\euCC(\RS)$ is a point in $V$,
that in $\spCC(\RS)$ is an open half-line in $V$.
We sometimes use the letter $\ell$ to represent a vertex of $\spCC(\RS)$
to emphasize that a vertex of $\spCC(\RS)$ is \emph{not} a point in $V$.

\subsubsection{Spherical expressions of spherical Coxeter complexes}\label{subsubsec:spherical-expression}
Let $\RS$ be a root system.
Take an arbitrary point $x_\ell \in \ell$ for each vertex $\ell \in \sk{\spCC(\RS)}$,
and define $A^\circ$ as the relative interior of $\conv{\Set{ x_\ell }{\ell \in \sk{A} }}$
for each cell $A \in \spCC(\RS) \setminus \{0\}$.
Then, the \emph{spherical expression} $\CC^\circ$ of $\spCC(\RS)$ generated from $\Set{ x_\ell }{ \ell \in \sk{\spCC(\RS)} }$ is the poset defined by
\begin{align}
    \CC^\circ \coloneqq (\Set{ A^\circ }{ A \in \spCC(\RS) \setminus \{0\} } \cup \{ \emptyset \}, \preceq ),
\end{align}
where $\emptyset \preceq A^\circ$ for any $A^\circ$, and $A^\circ \preceq B^\circ$ if $\cl{A^\circ} \subseteq \cl{B^\circ}$.
The correspondence $A \mapsto A^\circ$ and $\{0\} \mapsto \emptyset$ forms an isomorphism between $\spCC(\RS)$ and its spherical expression
as a poset,
and the inverse map is given by $A^\circ \mapsto \cone{A^\circ}$.
Hence, a spherical expression is also a Coxeter complex that admits the canonical type function $x_\ell \mapsto \tau(\ell)$ and the mark function $x_\ell \mapsto a(\ell)$,
which are denoted by $\tau_{\CC^\circ}$ and $a_{\CC^\circ}$,
respectively.
Similarly, for $A^\circ \in \CC^\circ$, its link complex $\Lk{\CC^\circ}{A^\circ}$ is naturally isomorphic to $\Lk{\spCC(\RS)}{A}$,
which enables us to introduce the canonical type function and the mark function with respect to $\Lk{\CC^\circ}{A^\circ}$ as well.

Suppose further that $\RS = \RS_1 \oplus \RS_2 \oplus \cdots \oplus \RS_k$ in which $\RS_1, \RS_2, \dots, \RS_k$ are irreducible components of $\RS$.
Take any vertex $\ell \in \sk{\spCC(\RS)}$.
Let $C$ be a chamber with $\ell \preceq C$, i.e., $\ell \in \sk{C}$,
and $\Delta = \Delta_1 \cup \cdots \cup \Delta_k$ the simple system corresponding to $C$
in which $\Delta_1, \Delta_2, \dots, \Delta_k$ are simple systems of the irreducible component $\RS_1, \RS_2, \dots, \RS_k$ of $\RS$,
respectively.
Suppose that $\tau(\ell)$ coincides with the type of a root $\alpha$ in $\Delta_i$ with respect to $\Delta$.
Then, we define the nonzero vector $x^{\star}_\ell \in V \setminus \{0\}$ as the unique point in $\ell \subseteq V \setminus \{0\}$
that satisfies $\inpr{\alpha_0}{x^{\star}_\ell} = -1$,
where $\alpha_0$ is the negative of the highest root with respect to $\Delta_i$;
this is independent of the choice of a chamber $C$ with $\ell \preceq C$.
We refer to $x^{\star}_\ell$ as the \emph{standard point of $\ell$}.
The \emph{standard spherical expression} of $\spCC(\RS)$, denoted by $\spCC^{\star}(\RS)$,
is its spherical expression generated from the set $\Set{ x^{\star}_\ell }{ \ell \in \sk{\spCC(\RS)} }$ of standard points.

The definition of the standard spherical expression and the representation~\eqref{eq:affine:A} of a vertex of $\Lk{\euCC(\RS)}{0}$
immediately imply the following.
\begin{lemma}\label{lem:link-standard}
    For an irreducible root system $\RS$,
    we have $\spCC^\star(\RS) = \Lk{\euCC(\RS)}{0}$.
\end{lemma}

For a simple system $\simp = \{ \alpha_1, \alpha_2, \dots, \alpha_n \}$,
its \emph{fundamental coweights} $\omega_1^\vee, \omega_2^\vee, \dots, \omega_n^\vee$
are the points in $V$ that satisfy
\begin{align}\label{eq:fundamental-coweights}
    \inpr{\alpha_i}{\omega_j^\vee} =
    \begin{cases}
        1 & \text{if $i = j$},\\
        0 & \text{otherwise}
    \end{cases}
\end{align}
for $i,j \in [n]$.
The following lemma says that the standard points are the same as the fundamental coweights divided by their marks.
\begin{lemma}\label{lem:inpr}
    Let $\RS$ be a root system, $x \in \sk{\spCC^\star(\RS)}$, and $\simp$ the simple system of $\RS$ corresponding to a chamber $C \in \spCC^\star(\RS)$ with $x \in \sk{C}$,
    $\alpha \in \simp$ the root with $\tau(x) = \tau_\simp(\alpha)$.
    Then, we have
    \begin{align}
        \inpr{\beta}{x} =
        \begin{cases}
            1/a(x) & \text{if $\beta = \alpha$},\\
            0 & \text{if $\beta \in \simp \setminus \{ \alpha \}$}.
        \end{cases}
    \end{align}
\end{lemma}
\begin{proof}
    Suppose that $\RS = \RS_1 \oplus \RS_2 \oplus \cdots \oplus \RS_k$ in which $\RS_i$ is an irreducible component of $\RS$ for $i \in [k]$,
    and that
    $\simp = \simp_1 \cup \cdots \cup \simp_k$ the simple system corresponding to a chamber $C \in \spCC(\RS)$ with $x \in \sk{C}$
in which $\simp_i$ is a simple system of $\RS_i$ for $i \in [k]$.
    Let $\alpha$ be the root in $\simp$ such that $\tau_\simp(\alpha) = \tau(x)$.
    We may assume that $\alpha$ belongs to $\simp_1$.
    Then, by \Cref{cor:type-function}~(1) and the definition of standard points,
    the vertex $x$ is the unique point in $V$ that satisfies
    $\inpr{\alpha_0}{x} = -1$ and
    $\inpr{\beta}{x} = 0$ for $\beta \in \simp \setminus \{ \alpha \}$,
    where $\alpha_0 = -\sum_{\beta \in \simp_1} a_\beta \beta$ is the negative of the highest root with respect to $\simp_1$.
    Since $a_\alpha = a(x)$ by the definition of the mark,
    we obtain $\inpr{\alpha}{x} = \inpr{-\alpha_0-\sum_{\beta \in \simp_1 \setminus \{\alpha\}} a_\beta \beta}{x}/a(x) = -\inpr{\alpha_0}{x}/a(x) = 1/a(x)$.
    This completes the proof.
\end{proof}

\begin{corollary}\label{cor:inpr}
    \begin{enumerate}[label={\textup{(\arabic*)}}]
        \item Let $\RS$ be a root system, $x \in \sk{\spCC^\star(\RS)}$, and $\simp = \{ \alpha_1, \alpha_2, \dots, \alpha_n \}$ the simple system of $\RS$ corresponding to a chamber $C \in \spCC^\star(\RS)$ with $x \in \sk{C}$,
        $\alpha_k \in \simp$ the root with $\tau_\simp(\alpha_k) = \tau(x)$.
        Then,
        for $p = \sum_{i = 1}^n p_i \alpha_i \in V^*$,
        we have $\inpr{p}{x} = p_k/ a(x)$.
        In particular, if $p \in \Mdom(\RS)$,
        then $\inpr{p}{x} \in \Z/a(x)$.
        \item Let $\RS$ be an irreducible root system.
        For any $x \in \sk{\euCC(\RS)}$ and $p \in \Mdom(\RS)$,
        we have $\inpr{p}{x} \in \Z /a(x)$.
    \end{enumerate}
\end{corollary}
\begin{proof}
    (1).
    Since $p = \sum_{i = 1}^n p_i \alpha_i$,
    we have $\inpr{p}{x} = \sum_{i = 1}^n p_i \inpr{\alpha_i}{x} = p_k / a(x)$,
    where the last equality follows from \Cref{lem:inpr}.
    In addition, if $p \in \Mdom(\RS)$,
    then $p_i \in \Z$ for all $i \in [n]$.
    Thus, we obtain $\inpr{p}{x} \in \Z /a(x)$.

    (2).
  Take any $x \in \sk{\euCC(\RS)}$.
  Then, by \Cref{cor:type-function}~(2),
  there are a point $x_0 \in \Mdom^{\vee}(\RS)$ and a chamber $C$ with $\sk{C} = \{ 0, x_1, \dots, x_n \}$ such that $x = x_0 + x_k$,
  where $\tau(x) = \tau(x_k)$.
  Let $\simp = \{ \alpha_1, \alpha_2, \dots, \alpha_n \}$ be the simple system corresponding to $C$,
  in which
  $\tau_\simp(\alpha_k) = \tau(x_k)$.
  Here, we note that $\{ x_1, \dots, x_n \}$ is the vertex set of some chamber in $\Lk{\euCC(\RS)}{0} = \spCC^\star(\RS)$ (see \Cref{lem:link-standard}).

  Take any $p \in \Mdom(\RS)$,
  which is representable as $p = \sum_{i = 1}^n p_i \alpha_i$ for some integers $p_i \in \Z$.
  Since $x_0 \in \Mdom^{\vee}(\RS)$ and $p \in \Mdom(\RS)$,
  we have $\inpr{p}{x_0} \in \Z$ by~\eqref{eq:px:Z}.
  By the assertion~(1),
  we have $\inpr{p}{x_k} = p_k \inpr{\alpha_k}{x_k} = p_k/a(x_k) \in \Z / a(x_k) = \Z / a(x)$.
  Thus, we obtain $\inpr{p}{x} = \inpr{p}{x_0} + \inpr{p}{x_k } \in \Z / a(x)$.
\end{proof}

Let $\CC^\circ$ be a spherical expression of $\spCC(\RS)$.
For a simplex $A^\circ \in \CC^\circ$,
we define $\CC^\circ / A^\circ$ as
the spherical expression of $\spCC(\RS|_{-{\cone{A^\circ}}})$ generated from
$\Set{ x + \spn{A^\circ} }{ x \in \sk{\Lk{\CC^\circ}{A^\circ}} }$.
In the case where $\RS$ is irreducible,
for a simplex $A \in \euCC \coloneqq \euCC(\RS)$,
we define $\euCC / A$ as the spherical expression of $\spCC(\tilde{\RS}|_{-A})$ generated from $\Set{ x - \hat{x} + (\aff{A})_0 }{ x \in \sk{\Lk{\euCC}{A}} }$,
where $\hat{x}$ is a point in $A$.
These are well defined by \Cref{lem:complex:contraction}.
In general,
for a simplex $A^\star \in \spCC^\star(\RS)$ (resp. $A \in \euCC(\RS)$),
the simplicial complex $\spCC^\star(\RS) / A^\star$ (resp. $\euCC(\RS) / A$) is (isomorphic to but) different from $\spCC^\star(\RS|_{-\cone{A^\star}})$ (resp. $\spCC^\star(\tilde{\RS}|_{-A})$).
The following lemma gives a precise relation between them.
\begin{lemma}\label{lem:standard}
    \begin{enumerate}[label={\textup{(\arabic*)}}]
        \item Let $\RS$ be a root system, $\spCC^\star \coloneqq \spCC^\star(\RS)$, and $A^\star$ a simplex of $\spCC^\star$.
        Moreover, let $a$ and $a_{\mathrm{Lk}}$ denote the mark functions of $\spCC^\star$ and $\Lk{\spCC^\star}{A^\star}$,
    respectively.
        For $x \in \sk{\Lk{\spCC^\star}{A^\star}}$,
    let $a(x)$ and $a_{\mathrm{Lk}}(x)$ denote the marks of $x$ with respect to $\spCC^\star$ and $\Lk{\spCC^\star}{A^\star}$,
    respectively.
    Then, $\spCC^\star(\RS|_{-\cone{A^\star}})$
    is the spherical expression of $\spCC(\RS|_{-\cone{A^\star}})$ generated from
    \begin{align}
        \Set*{ \frac{a(x)}{a_{\mathrm{Lk}}(x)}x + \spn{A^\star} }{ \text{$x \in \sk{\Lk{\spCC^\star}{A^\star}}$} }.
    \end{align}
    \item Let $\RS$ be an irreducible root system, $\euCC \coloneqq \euCC(\RS)$, and $A$ a simplex of $\euCC$.
    Moreover, let $a$ and $a_{\mathrm{Lk}}$ denote the mark functions of $\euCC$ and $\Lk{\euCC}{A}$,
    respectively.
    Then, $\spCC^\star(\tilde{\RS}|_{-A})$
    is the spherical expression of $\spCC(\tilde{\RS}|_{-A})$ generated from
    \begin{align}
        \Set*{ \frac{a(x)}{a_{\mathrm{Lk}}(x)}(x - \hat{x}) + (\aff{A})_0 }{ \text{$x \in \sk{\Lk{\euCC}{A}}$} },
    \end{align}
    where $\hat{x}$ is a point in $A$.
    \end{enumerate}
\end{lemma}
\begin{proof}
    We omit the proof of~(2), since it can be obtained by a similar argument to that of~(1).

    (1).
    Let $A \coloneqq \cone{A^\star}$, which is the simplex of $\spCC(\RS)$ corresponding to $A^\star$.
    We first note that $\Lk{\spCC^\star(\RS)}{A^\star}$ is isomorphic to $\spCC^\star(\RS|_{-A})$,
    since $\Lk{\spCC^\star(\RS)}{A^\star} \simeq \Lk{\spCC(\RS)}{A} \simeq \spCC(\RS|_{-A}) \simeq \spCC^\star(\RS|_{-A})$.
    More precisely,
    the isomorphism is given by
    $\sk{\Lk{\spCC^\star(\RS)}{A^\star}} \ni x \mapsto \cone{x} \mapsto \cone{x} + \spn{A} \mapsto x^- \in \sk{\spCC^\star(\RS|_{-A})}$,
    where $x^-$ denotes the standard point of the vertex $\cone{x} + \spn{A}$ of $\spCC(\RS|_{-A})$ (see \Cref{lem:complex:contraction}~(1)).
    Hence, it suffices to show that
    $x^- = (a(x)/a_{\mathrm{Lk}}(x))x + \spn{A}$.

    Let
    $\tau$ be the type function of $\Lk{\spCC(\RS)}{A}$,
    $\simp$ a simple system corresponding to a chamber $C$ of $\St{\spCC(\RS)}{A}$,
    and $\alpha$ the unique root in $\simp$ such that $\tau_\simp(\alpha) = \tau(\cone{x})$.
    Since $\cone{x} \in \Lk{\spCC(\RS)}{A}$,
    we have $\alpha \in \simp|_{-A}$.
    By \Cref{lem:inpr},
    $x$ is the unique point in $V$ that satisfies
    \begin{align}\label{eq:x}
        \inpr{\alpha}{x} = 1/a(x) \qquad \text{and} \qquad \inpr{\beta}{x} = 0 \quad (\beta \in \simp \setminus \{ \alpha \}).
    \end{align}
    Similarly,
    $x^-$ is the unique point in $V/\spn{A}$ that satisfies
    \begin{align}\label{eq:x^-}
        \inpr{\alpha}{x^-} = 1/a_{\mathrm{Lk}}(x) \qquad \text{and} \qquad \inpr{\beta}{x^-} = 0 \quad (\beta \in \simp|_{-A} \setminus \{ \alpha \}).
    \end{align}
    By~\eqref{eq:x} and~\eqref{eq:x^-},
    we obtain $x^- = (a(x)/a_{\mathrm{Lk}}(x))x + \spn{A}$.
    This completes the proof.
\end{proof}

\subsubsection{Geometric realization and (positively homogeneous) Lov\'{a}sz extension}\label{subsubsec:geometric-realization}
Let $\CC$ be a spherical Coxeter complex, a Coxeter complex, or a spherical expression of the spherical Coxeter complex.
For a subset $\CC'$ of $\CC$,
its \emph{geometric realization} $\geo{\CC'}$
is defined by
\begin{align}
    \geo{\CC'} \coloneqq \bigcup_{A \in \CC'} A,
\end{align}
which is a subset of $V$.

Suppose that $\CC$ is a spherical or Euclidean Coxeter complex.
Then, we have $\geo{\CC} = V$.
Since any two distinct simplices are disjoint,
$\CC$ forms a partition of $V$.
Hence, for each $x \in \geo{\CC} = V$,
there exists a unique cell in $\CC$ that contains $x$,
which is denoted by $\cell{A}{x}$.
For a vertex subset $D \subseteq \sk{\CC}$,
its \emph{Lov\'{a}sz extension} $\ext{D}$ is defined as the geometric realization $\geo{\CC[D]}$ of the subcomplex of $\CC$ induced by $D$, i.e.,
\begin{align}
    \ext{D} \coloneqq \geo{\CC[D]} = \bigcup \Set{ A \in \CC }{ \sk{A} \subseteq D },
\end{align}
which is also a subset of $V$.
Note that a point $x \in V$ belongs to $\ext{D}$ if and only if $\sk{\cell{A}{x}} \subseteq D$.

We particularly focus on the case where $\CC$ is a Euclidean Coxeter complex $\euCC \coloneqq \euCC(\RS)$.
Since $\cell{A}{x}$ is a geometric simplex in $V$ that contains $x$ in its relative interior,
each point $x \in V$ is uniquely representable as the convex combination of the vertices in $\sk{\cell{A}{x}}$;
let $\cell{\lambda}{x} \in \cvcoeff{\sk{\euCC}}$ denote the corresponding convex combination coefficient, namely,
\begin{align}
    x = \sum_{v \in \sk{\euCC}} \cell{\lambda}{x}(v) v
\end{align}
with $\supp{\cell{\lambda}{x}} = \sk{\cell{A}{x}}$.
For $D \subseteq \sk{\euCC}$,
its Lov\'{a}sz extension $\ext{D}$ can be viewed as the linear extension of $D$ with respect to the Euclidean Coxeter complex $\euCC$, i.e.,
\begin{align}
    \ext{D} = \Set*{ \sum_{v \in \sk{A}} \lambda(v) v }{ A \in \euCC[D],\ \lambda \in \cvcoeff{\sk{A}} }.
\end{align}
For a function $f$ on $\sk{\euCC}$,
its \emph{Lov\'{a}sz extension} $f$ is the function on $\geo{\euCC} = V$
defined by
\begin{align}\label{eq:Lovasz-ext}
    \ext{f}(x) \coloneqq \sum_{v \in \sk{\euCC}} \cell{\lambda}{x}(v) f(v) = \sum_{v \in \sk{A_x}} \cell{\lambda}{x}(v) f(v)
\end{align}
for $x \in V$.

We then consider a spherical expression $\CC^\circ$ of a spherical Coxeter complex $\spCC = \spCC(\RS)$.
In this case,
$\geo{\CC^\circ}$ is homeomorphic to the $n$-sphere.
For any nonzero point $x \in V \setminus \{0\}$,
there exists a unique cell $A$ in $\CC^\circ$ such that $x$ belongs to the corresponding cell $\cone{A}$ of $\spCC(\RS)$ to $A$;
we denote the cell $A$ by $\cell{A}{x}$.
Then, each nonzero point $x \in V \setminus \{0\}$ is uniquely representable as the positive combination of the vertices in $\sk{\cell{A}{x}}$;
let $\cell{\mu}{x} \in \cncoeff{\sk{\CC^\circ}}$ denote the corresponding nonnegative combination coefficient, namely,
\begin{align}
    x = \sum_{v \in \sk{\CC^\circ}} \cell{\mu}{x}(v) v
\end{align}
with $\supp{\cell{\mu}{x}} = \sk{\cell{A}{x}}$.
We also define $\mu_0$ as the all-zero coefficient vector in $\cncoeff{\sk{\CC^\circ}}$.
For a vertex subset $D \subseteq \sk{\CC^\circ}$,
its \emph{positively homogeneous Lov\'{a}sz extension} $\phext{D}$ is defined as the Lov\'{a}sz extension $\ext{\subsp{D}}$ of the vertex subset $\subsp{D} \coloneqq \Set{ \cone(x) \in \sk{\spCC} }{ x \in D }$ of $\sk{\spCC}$
arising from $D$, or equivalently, the union of conical hulls $\cone{A}$ of all simplices $A$ in $\CC^\circ[D]$.
That is,
\begin{align}
    \phext{D} \coloneqq&\;\geo{\spCC[\subsp{D}]}\\
    =&\;\Set*{ \sum_{v \in \sk{A}} \mu(v) v }{ A \in \CC^\circ[D],\ \mu \in \cncoeff{\sk{A}}},
\end{align}
which forms a closed cone in $V$.
Note that a point $x \in V$ belongs to $\phext{D}$ if and only if $\sk{\cell{A}{x}} \subseteq D$.
Since $A_0 = \{0\}$ and $\sk{A_0} = \emptyset$,
we have $0 \in \phext{D}$ for any vertex subset $D \subseteq \sk{\CC^\circ}$.
Similarly, for a function $\rho$ on $\sk{\CC^\circ}$,
its \emph{positively homogeneous Lov\'{a}sz extension} $\phext{\rho}$ is the function on $\geo{\spCC} = V$ defined by
\begin{align}\label{eq:ph-Lovasz-ext}
    \phext{\rho}(x) \coloneqq \sum_{v \in \sk{\CC^\circ}} \mu_x(v) \rho(v) = \sum_{v \in \sk{A_x}} \mu_x(v) \rho(v) \qquad (x \in V),
\end{align}
which is a positively homogeneous function on $V$ satisfying $\phext{\rho}(x) = \rho(x)$ for all $x \in \sk{\CC^\circ}$ and $\phext{\rho}(0) = 0$.

\subsubsection{Convex subcomplexes}\label{subsubsec:convex-subcomplex}

In the following arguments of this subsection,
$(\calH,\CC)$ denotes $(\lincalH(\RS),\spCC(\RS))$ or $(\affcalH(\RS),\euCC(\RS))$;
this notation is used in the arguments that can be applied to both spherical and Euclidean Coxeter complexes.
We denote by $I$
the index set of a sign sequence,
namely, $I \coloneqq \RS$ when we consider the spherical Coxeter complex
and $I \coloneqq \RS \times \Z$ when we consider the Euclidean Coxeter complex.

For cells $A, B \in \CC$,
the \emph{product} $AB$ of $A$ and $B$ is the (unique) cell in $\CC$ such that its sign sequence $\sigma(AB)$ is
\begin{align}
    \sigma_{i}(AB) =
    \begin{cases}
        \sigma_{i}(A) & \text{if $\sigma_{i}(A) \neq 0$},\\
        \sigma_{i}(B) & \text{if $\sigma_{i}(A) = 0$}.
    \end{cases}
\end{align}
Equivalently, the product $AB$ can be defined as the cell
containing a point $(1-\eps)a + \eps b$ for $a \in A$, $b \in B$, and sufficiently small $\eps > 0$.
It is known (\cite[Corollary~3.110]{Abramenko2010}) that $\CC$ is a semigroup under the binary operation $(A,B) \mapsto AB$, i.e.,
for all $A,B,C \in \CC$, we have $(AB)C = A(BC)$.

A subcomplex $\CC'$ of $\CC$ is said to be \emph{convex} if $\CC'$ is a subsemigroup of $\CC$, i.e.,
$AB \in \CC'$ for any cells $A, B \in \CC'$.
The following lemma (\Cref{lem:convex_subcomplex}) says that
the convexity of a subcomplex $\CC'$ of $\CC$ can be characterized as the representability of $\CC'$ as the intersection of closed half-spaces arising from hyperplanes in $\calH$,
and as the convexity of the geometric realization $\geo{\CC'}$ in $V$.
\begin{lemma}\label{lem:convex_subcomplex}
    Let $\CC = \spCC(\RS)$ or $\CC = \euCC(\RS)$.
    For a subcomplex $\CC'$ of $\CC$,
    the following are equivalent:
    \begin{enumerate}[label={\textup{(\alph*)}}]
        \item $\CC'$ is a convex subcomplex.
        \item $\CC' = \bigcap_{i \in I'}\Set{ A \in \CC }{ A \subseteq H_i^{0} \cup H_i^-}$ for some index subset $I' \subseteq I$.
        \item The geometric realization $\geo{\CC'}$ of $\CC'$ is a convex set in $V$.
    \end{enumerate}
    In particular,
    the geometric realization $\geo{\CC'}$ of a convex subcomplex $\CC'$ of $\spCC(\RS)$ is of the form
    \begin{align}\label{eq:geo-sp}
        \geo{\CC'} = \Set{ x \in V }{ \inpr{\alpha}{x} \leq 0 \ (\forall \alpha \in I') }
    \end{align}
    for some index subset $I' \subseteq \RS$,
    and that of a convex subcomplex $\CC'$ of $\euCC(\RS)$ is of the form
    \begin{align}\label{eq:geo-aff}
        \geo{\CC'} = \Set{ x \in V }{ \inpr{\alpha}{x} \leq k \ (\forall (\alpha,k) \in I') }
    \end{align}
    for some index subset $I' \subseteq \RS \times \Z$.
\end{lemma}
\begin{proof}
    (a) $\Leftrightarrow$ (b) follows from \cite[Proposition~3.137]{Abramenko2010}.
    
    (b) $\Rightarrow$ (c) and the latter assertion.
    If a subcomplex $\CC'$ of $\CC$ satisfies (b),
    then $\geo{\CC'}$ is of the form~\eqref{eq:geo-sp} if $\CC'$ is a subcomplex of $\spCC(\RS)$,
    and $\geo{\CC'}$ is of the form~\eqref{eq:geo-aff} if $\CC'$ is a subcomplex of $\euCC(\RS)$.
    Thus, since $\geo{\CC'}$ is the intersection of closed half-spaces, it is a convex set in $V$

    (c) $\Rightarrow$ (a).
    Take any two cells $A, B \in \CC'$ and pick any points $x \in A$ and $y \in B$.
    Since $\geo{\CC'}$ is convex in $V$,
    we have $z \coloneqq (1-\eps)x + \eps y \in \geo{\CC'}$ for sufficiently small $\eps > 0$.
    Hence, $\cell{A}{z} = AB$ belongs to $\CC'$.
    This implies that $\CC'$ is a subsemigroup of $\CC$,
    i.e., a convex subcomplex of $\CC$.
\end{proof}

Since the star complex $\St{\CC}{A}$ is a convex subcomplex of $\CC$ for a simplex $A \in \CC$~\cite[Proposition~3.93 and Theorem~3.131]{Abramenko2010},
we immediately obtain the following by \Cref{lem:convex_subcomplex}.
\begin{lemma}\label{lem:star_complex}
    Let $\CC = \spCC(\RS)$ or $\CC = \euCC(\RS)$.
    For a simplex $A \in \CC$,
    the geometric realization $\geo{\St{\CC}{A}}$ of
    the star complex $\St{\CC}{A}$ is a convex set in $V$.
\end{lemma}

The following immediately follows from~\Cref{lem:convex_subcomplex},
which holds for a general Coxeter complex~\cite[Proposition~3.136]{Abramenko2010}.
\begin{lemma}[{\cite[Proposition~3.136]{Abramenko2010}}]\label{lem:chamber-complex}
    A convex subcomplex $\CC'$ of a Coxeter complex $\CC$ is a chamber complex.
    In particular, any maximal simplex of $\CC'$ has the same rank.
\end{lemma}

\subsubsection{Sign sequences and distances}
We say that
a hyperplane $H \in \calH$ \emph{strictly separates} two cells $A, B \in \CC$
if $A$ and $B$ belong to different open half-spaces arising from $H$,
i.e.,
$A \subseteq H_{i}^+$ and $B \subseteq H_{i}^-$ for some $i \in I$ with $H = H_{i}^0$.
It is known~\cite[Proposition~3.78]{Abramenko2010} that
the distance $d(A,B)$ between cells $A, B \in \CC$
is equal to the number of hyperplanes in $\calH$ that strictly separate $A$ and $B$.
That is, we have
\begin{align}
    d(A, B) &= \card{\Set{ H \in \calH }{ \text{$H$ strictly separates $A$ and $B$} }}\\
    &= \frac{1}{2} \cdot \card{\Set*{ i \in I }{ \{\sigma_{i}(A), \sigma_{i}(B)\} = \{+, -\} }}.\label{eq:dist}
\end{align}

We introduce a partial order $\prec$ on $\{+, 0, -\}$ as:
$+ \succ 0 \prec -$, and $+$ and $-$ are incomparable.
This partial order $\prec$ can be naturally extended to that on the sign sequences:
$\sigma(A) \preceq \sigma(B)$ if and only if $\sigma_{i}(A) \preceq \sigma_{i}(B)$ for all $i \in I$.
We say that $\sigma(A)$ and $\sigma(B)$ are \emph{compatible}
if there is no $i \in I$ such that $\{\sigma_{i}(A), \sigma_{i}(B)\} = \{+, -\}$.

We summarize the basic properties of the product and the distance.
\begin{lemma}\label{lem:sign-sequence}
    Let $A, A_1, \dots, A_k, B$ be cells in $\CC$.
    \begin{enumerate}[label={\textup{(\arabic*)}}]
        \item $A \preceq B$ if and only if $\sigma(A) \preceq \sigma(B)$.
        \item $ABA = AB$.
        \item $A_1, \dots, A_k$ are joinable if and only if $\sigma(A_1), \dots, \sigma(A_k)$ are mutually compatible.
        In addition, if $A_1, \dots, A_k$ are joinable, its join is the product $A_1 \cdots A_k$ of $A_1, \dots, A_k$.
        \item If $A_1$ and $A_2$ are joinable, then $d(A_1 A_2, B) \geq d(A_1, A_2 B)$.
        \item If $A_1 \preceq A_2$, then $d(A_1, B) \leq d(A_2, B)$.
        \item Let $x \in \sk{A}$ be a vertex of $A$.
        If $d(A, B) > 0$ and $d(A, B) = d(x, B)$, then $d(A', B) < d(A, B)$ for any $A' \prec A$ with $x \not\preceq A'$.
    \end{enumerate}
    In the following assertion, we consider the Euclidean Coxeter complex $\euCC(\RS)$.
    \begin{enumerate}[label={\textup{(\arabic*)}},resume]
      \item Let $x, y \in \sk{\euCC(\RS)}$ and $B \in \euCC(\RS)$ with $\sk{B} \ni y$.
      For $z \coloneqq (x+y)/2$, we have $d(x, B) \geq d(\cell{A}{z}, B)$.
    In addition, if the equality holds, then $x \preceq \cell{A}{z}$.
    \end{enumerate}
\end{lemma}
\begin{proof}
    The assertion (1) follows from $A \preceq B \iff \cl{A} \subseteq \cl{B} \iff \sigma(A) \preceq \sigma(B)$,
    and
    (2) follows from the definition of product.

    (3). If $A_1, A_2, \dots, A_k$ are joinable, then there is a cell $A \in \CC$ such that $A_1, A_2, \dots, A_k \preceq A$.
    Hence, by the assertion~(1), we have $\sigma(A_1), \sigma(A_2), \dots, \sigma(A_k) \preceq \sigma(A)$,
    which implies that $\sigma(A_1), \sigma(A_2), \dots, \sigma(A_k)$ are mutually compatible.
    Conversely, if $\sigma(A_1), \sigma(A_2), \dots, \sigma(A_k)$ are mutually compatible, then $\sigma(A_1), \sigma(A_2), \dots, \sigma(A_k) \preceq \sigma(A_1 A_2 \cdots A_k)$
    by the definition of the product,
    which implies that $A_1, A_2, \dots, A_k \preceq A_1 A_2 \cdots A_k$ by the assertion~(1).
    Hence, $A_1, A_2, \dots, A_k$ are joinable.
    Moreover, in this case, since there is no cell $A' \in \CC$ satisfying $A_1, A_2, \dots, A_k \preceq A' \prec A_1 A_2 \cdots A_k$,
    the product $A_1 A_2 \cdots A_k$ is the join of $A_1, A_2, \dots, A_k$.

    (4). By the equality~\eqref{eq:dist} on the distance function,
    we have
    \begin{align}
        2d(A_1A_2, B) \geq\ &\card{ \Set{ i \in I }{ \{\sigma_{i}(A_1), \sigma_{i}(B)\} = \{+, -\} }},\\
        2d(A_1, A_2B) =\ &\card{ \Set{ i \in I }{ \{\sigma_{i}(A_1), \sigma_{i}(A_2)\} = \{+, -\} }}\\
         &+ \card{\Set{ j \in I }{ \sigma_j(A_2) = 0 \text{ and } \{\sigma_{j}(A_1), \sigma_{j}(B)\} = \{+, -\} }}.\label{eq:d2}
    \end{align}
    Here, since $A_1$ and $A_2$ are joinable, the first summand in the RHS of~\eqref{eq:d2} is zero.
    Hence, we obtain $2d(A_1A_2, B) \geq 2d(A_1, A_2B)$.

    (5). By $A_1 \preceq A_2$ and the assertion~(1), if $\sigma_i(A_1) \neq 0$ then $\sigma_i(A_1) = \sigma_i(A_2)$.
    Hence, we have $\card{ \Set{ i \in I }{ \{\sigma_{i}(A_1), \sigma_{i}(B)\} = \{+, -\} }} \leq \card{ \Set{ i \in I }{ \{\sigma_{i}(A_2), \sigma_{i}(B)\} = \{+, -\} }}$,
    which implies $d(A_1, B) \leq d(A_2, B)$ by~\eqref{eq:dist}.

    (6).
    Take any minimal gallery $(C_0, C_1, \dots, C_\ell)$ from $A$ to $B$,
    where $\ell = d(A, B) > 0$.
    Since $d(A, B) = d(x, B)$, this is also a minimal gallery from $x$ to $B$.
    Hence, we have $x \not\preceq C_1$ and $A' \preceq C_1$ for any $A' \prec A$ with $x \notin \sk{A'}$,
    which implies that $(C_1, \dots, C_\ell)$ forms a gallery from $A'$ to $B$.
    Therefore, we obtain $d(A', B) \leq \ell - 1 < d(A, B)$ for any $A' \prec A$ with $x \not\preceq A'$.

    (7).
    We can observe that,
    for $(\alpha, k) \in \RS \times \Z$,
    if $\sigma_{(\alpha, k)}(x) \preceq \sigma_{(\alpha, k)}(y)$ (resp. $\sigma_{(\alpha, k)}(x) \succeq \sigma_{(\alpha, k)}(y)$)
    then $\sigma_{(\alpha, k)}(\cell{A}{z}) = \sigma_{(\alpha, k)}(y)$ (resp. $\sigma_{(\alpha, k)}(\cell{A}{z}) = \sigma_{(\alpha, k)}(x)$).
    Also, since $y \preceq B$,
    we have $\sigma(y) \preceq \sigma(B)$ by the assertion~(1).
    Hence,
    for each $(\alpha, k) \in \RS \times \Z$, we have $\sigma_{(\alpha, k)}(\cell{A}{z}) \preceq \sigma_{(\alpha, k)}(x)$ or $\sigma_{(\alpha, k)}(\cell{A}{z}) \preceq \sigma_{(\alpha, k)}(y) \preceq \sigma(B)$.
    Therefore, if $\{\sigma_{(\alpha, k)}(\cell{A}{z}), \sigma_{(\alpha, k)}(B)\} = \{+,-\}$,
    then $\sigma_{(\alpha, k)}(\cell{A}{z}) = \sigma_{(\alpha, k)}(x)$,
    which implies $\{\sigma_{(\alpha, k)}(x), \sigma_{(\alpha, k)}(B)\} = \{+,-\}$.
    Thus, we obtain $d(x, B) \geq d(\cell{A}{z}, B)$.

    Suppose that the equality $d(x, B) = d(\cell{A}{z}, B)$ holds.
    Take any $(\alpha, k) \in \RS \times \Z$ with $\sigma_{(\alpha, k)}(x) \neq 0$.
    As observed above, if $\sigma_{(\alpha, k)}(x) \succeq \sigma_{(\alpha, k)}(y)$,
    we have $\sigma_{(\alpha, k)}(\cell{A}{z}) = \sigma_{(\alpha, k)}(x)$.
    Furthermore, even if $\sigma_{(\alpha, k)}(x) \not\succeq \sigma_{(\alpha, k)}(y)$,
    i.e.,  $\{\sigma_{(\alpha, k)}(x), \sigma_{(\alpha, k)}(B)\} = \{+,-\}$,
    we have $\sigma_{(\alpha, k)}(\cell{A}{z}) = \sigma_{(\alpha, k)}(x)$ by the equality $d(x, B) = d(\cell{A}{z}, B)$.
    Hence, $\sigma_{(\alpha, k)}(x) \preceq \sigma_{(\alpha, k)}(\cell{A}{z})$ holds,
    which implies $x \preceq \cell{A}{z}$ by the assertion~(1).
\end{proof}

\section{Discrete convexity in Euclidean spaces}\label{sec:DCA:continuous}
This section formulates discrete convexity in Euclidean spaces to provide a geometric foundation for the discrete theory. 
\Cref{subsec:cone} defines L- and M-convex cones. 
\Cref{subsec:L-poly:M-sub,subsec:M-poly-L-sub} formulate \emph{L- and M-convex polyhedra}, \emph{L- and M-sublinear functions}, and their integrality. 
Finally, \Cref{subsec:locally-polyhedral} introduces \emph{locally polyhedral L- and M-convex functions} and reveals the conjugacy between them.

Throughout
\Cref{sec:DCA:continuous,sec:DCA:discrete,sec:chara},
we assume that $\RS$ is a root system such that its irreducible components are of classical types.
Note that any subsystem of a classical root system satisfies this assumption (see \Cref{lem:Dynkin-shape}~(2)).
Let
$\spCC \coloneqq \spCC(\RS)$,
$\spCC^\star \coloneqq \spCC^\star(\RS)$,
and
$\euCC \coloneqq \euCC(\RS)$
denote
the spherical Coxeter complex,
the standard spherical expression of the spherical Coxeter complex,
and Euclidean Coxeter complexes of type $\RS$,
respectively;
in considering the Euclidean Coxeter complex,
we assume that $\RS$ is irreducible.
Let $\Ldom \coloneqq \sk{\euCC}$ denote the vertex set of $\euCC$,
and $\Mdom \coloneqq \Mdom(\RS)$ the root lattice of $\RS$.
All notions involving L- or M-convexity are defined relative
to the fixed root system $\RS$.
Accordingly, the full names of these notions include the qualifier
``of type $\RS$.''
For notational simplicity, we omit this qualifier whenever the
underlying root system is clear from the context.

\begin{remark}\label{rem:scope-root-systems}
The above assumption of classical type is imposed so that all statements in
\Cref{sec:DCA:continuous,sec:DCA:discrete,sec:chara}
can be presented under a common hypothesis.
Most of the definitions and arguments in these sections do not use
the classification of irreducible root systems.
The definitions of L- and M-convex cones, polyhedra, sublinear
functions, locally polyhedral convex functions, and the corresponding discrete notions can be formulated for
arbitrary irreducible root systems.

The arguments that depend on the classical types are essentially
confined to \Cref{prop:distance-function}, which gives an explicit
type-by-type characterization of tight distance functions, and
\Cref{prop:tangent-cone}.
In particular, the results on L-convexity in
\Cref{prop:L-conv:unique,thm:minimality:L,thm:chara:L-convex-set,thm:chara:L-conv-func}
and the discrete conjugacy theorem (\Cref{thm:conjugacy})
remain valid for arbitrary irreducible root systems, including the
exceptional types.
By contrast, the M-convex local optimality and intrinsic
characterizations that use \Cref{prop:tangent-cone} are established
here under the classical-type assumption.
\end{remark}

\subsection{L- and M-convex cones}\label{subsec:cone}
A polyhedral cone $K \subseteq V$ is said to be \emph{L-convex of type $\RS$}
if it admits an outer-representation
\begin{align}\label{eq:L-convex-cone}
    K = \Set{ x \in V }{ \inpr{\alpha}{x} \leq 0 \ (\forall \alpha \in \RS') }
\end{align}
for some $\RS' \subseteq \RS$.
Similarly,
a polyhedral cone $K \subseteq V^*$
is said to be \emph{M-convex of type $\RS$}
if it admits an inner-representation
\begin{align}\label{eq:M-convex-cone}
    K = \cone{\RS'}
\end{align}
for some $\RS' \subseteq \RS$.
We note that,
for a root system $\RS_-$ with $\RS_- \precsim \RS$,
an L-convex (resp. M-convex) cone of type $\RS_-$ is also L-convex (resp. M-convex) of type $\RS$.

The following one-to-one correspondence between the L- and M-convex cones is a direct consequence of \Cref{lem:polar_polyhedral_cone}.
\begin{theorem}\label{thm:L-M-polar}
    The polar cone of an L-convex cone is an M-convex cone, and vice versa.
\end{theorem}
\Cref{lem:convex_subcomplex} (particularly, the inequality expression~\eqref{eq:geo-sp}) and the definition of L-convex cones~\eqref{eq:L-convex-cone} immediately imply that
the L-convex cones are exactly the geometric realizations of nonempty convex subcomplexes of $\spCC$.
For later reference, we summarize this characterization of L-convex cones.
\begin{lemma}\label{lem:chara-L-convex-cone}
    A set $K \subseteq V$ is an L-convex cone if and only if it is the geometric realization of some nonempty convex subcomplex of $\spCC$.
\end{lemma}

We note that L-convex cones and M-convex cones are exactly the same as \emph{Coxeter cones}~\cite{Stembridge2008-zf} and \emph{Coxeter root cones}~\cite{Ardila2020-ll},
respectively.

\subsection{L-convex polyhedra and M-sublinear functions}\label{subsec:L-poly:M-sub}
We refer to a function $\funcdoms{\gamma}{\RS}{\ER}$ as a \emph{distance function of type $\RS$}.
For consistency with existing notation in DCA~\cite[Chapter~5.2]{Murota2003-yb},
we denote the polyhedron $\bS{\gamma}$ by $\D{\gamma}$, i.e.,
\begin{align}
    \D{\gamma} := \bS{\gamma} = \Set{ x \in V }{ \inpr{\alpha}{x} \leq \gamma(\alpha) \ (\forall \alpha \in \RS) }.
\end{align}
Similarly, we denote $\bS{\cone{\gamma}}$ by $\D{\cone{\gamma}}$;
recall that $\D{\cone{\gamma}} = \D{\gamma}$ (\Cref{lem:conjugate-poly-sublinear}).
We say that a distance function $\funcdoms{\gamma}{\RS}{\ER}$ is \emph{valid}
if $\D{\gamma}$ is a nonempty polyhedron, or equivalently, $\cone{\gamma}$ is a proper polyhedral sublinear function.
A valid distance function $\funcdoms{\gamma}{\RS}{\ER}$ is said to be \emph{integral}
if there is an integer-valued distance function $\funcdoms{\gamma'}{\RS}{\EZ}$ such that $\D{\gamma} = \D{\gamma'}$, or equivalently, $\cone{\gamma} = \cone{\gamma'}$.

A set $P \subseteq V$ is called an \emph{L-convex polyhedron of type $\RS$}
if $P = \D{\gamma}$ for some valid distance function $\gamma$.
In addition, an L-convex polyhedron $P$ is \emph{integral} if $P = \D{\gamma}$
for some integral valid distance function $\gamma$.
We denote by $\Lpoly{\R}$ (resp. $\Lpoly{\Z|\R}$) the family of L-convex polyhedra (resp. integral L-convex polyhedra).
We note that
an L-convex cone is an integral L-convex polyhedron,
since it is representable as $\D{\gamma}$ for some $\funcdoms{\gamma}{\RS}{\{0,+\infty\}}$.

A function $\funcdoms{\sigma}{V^*}{\ER}$ is called an \emph{M-sublinear function of type $\RS$}
if $\sigma = \cone{\gamma}$ for some valid distance function $\gamma$.
In addition, an M-sublinear function is said to be \emph{integral} if $\sigma = \cone{\gamma}$ for some integral valid distance function $\gamma$.
We denote by $\Msub{\R}{\R}$ (resp. $\Msub{\R}{\Z|\R}$) the family of M-sublinear functions (resp. integral M-sublinear functions).

When we emphasize the integrality of an L-convex polyhedron (resp. an M-sublinear function) of type $\RS$,
we specifically use the term \emph{$\RS$-integral}.

For a root system $\RS_-$ with $\RS_- \precsim \RS$,
although an L-convex polyhedron (resp. an M-sublinear function) of type $\RS_-$ is L-convex (resp. M-sublinear) of type $\RS$,
a $\RS_-$-integral L-convex polyhedron (resp. a $\RS_-$-integral M-sublinear function) is not necessarily $\RS$-integral in general.

We immediately obtain
the following one-to-one correspondence between L-convex polyhedra and M-sublinear functions
from \Cref{lem:conjugate-poly-sublinear} and the definitions of L-convex polyhedra and M-sublinear functions.
\begin{theorem}\label{thm:conjugate-Lpoly-Msub}
    The support function $\spt{P}$ of an (integral) L-convex polyhedron $P$ is an (integral) M-sublinear function, and the set $\D{\sigma}$ arising from an (integral) M-sublinear function $\sigma$ is an (integral) L-convex polyhedron.
\end{theorem}

\Cref{lem:convex_subcomplex} (particularly, the inequality expression~\eqref{eq:geo-aff}) and the definition of integral L-convex polyhedra immediately imply that,
if $\RS$ is an irreducible root system,
then the integral L-convex polyhedra are exactly the geometric realizations of nonempty convex subcomplexes of $\euCC$.
For later reference, we explicitly describe this characterization of integral L-convex polyhedra.
\begin{lemma}\label{lem:chara-integral-L-convex-poly}
    Suppose that $\RS$ is irreducible.
    Then,
    a set $P \subseteq V$ is an integral L-convex polyhedron if and only if it is the geometric realization of some nonempty convex subcomplex of $\euCC$.
\end{lemma}

The following lemma states that the M-sublinearity of a polyhedral sublinear function $\sigma$ can be characterized as the M-convexity of the polyhedral cones in $\fan(\sigma)$,
which is the counterpart of the definition of an L-sublinear function; see \Cref{subsec:M-poly-L-sub} below.
\begin{lemma}\label{lem:M-sublinear}
  A proper polyhedral sublinear function $\funcdoms{\sigma}{V^*}{\ER}$ is M-sublinear if and only if every member of $\fan(\sigma)$ is an M-convex cone.
\end{lemma}
\begin{proof}
    (Only-if part).
    Since $\sigma$ is M-sublinear, there is a valid $\funcdoms{\gamma}{\RS}{\ER}$ such that $\sigma = \cone{\gamma}$.
    Take any $K \in \fan(\sigma)$.
    Then, for some $x \in V$,
    we have
    $K = \argmin(\sigma - x) = \argmin(\cone{\gamma} - x)$.
    Since $K$ is nonempty, we have $\gamma(\alpha) - \inpr{\alpha}{x} \geq 0$ for any $\alpha \in \dom{\gamma}$.
    Thus, $K = \cone{\Set{ \alpha \in \RS }{ \gamma(\alpha) = \inpr{\alpha}{x} }}$,
    which implies that $K$ is an M-convex cone.

    (If part).
    Since every member of $\fan(\sigma)$ is an M-convex cone, for each $K \in \fan(\sigma)$, there is $\RS_K \subseteq \RS$ such that $K = \cone{\RS_K}$.
    Let $\funcdoms{\gamma}{\RS}{\ER}$ be the restriction of $\sigma$ to $\RS$, i.e., $\gamma(\alpha) \coloneqq \sigma(\alpha)$ for $\alpha \in \RS$.
    Then, for $K \in \fan(\sigma)$ and $p \in K$,
    we have
    $\sigma(p) = \sum_{\alpha \in \RS_K} \mu(\alpha)\gamma(\alpha) \geq (\cone{\gamma})(p)$ by the definition of $\cone{\gamma}$~\eqref{eq:func:cone},
    where $\mu \in \cncoeff{\RS_K}$ satisfying $\sum_{\alpha \in \RS_K} \mu(\alpha) \alpha = p$.
    Thus, we obtain $\sigma \geq \cone{\gamma}$.
    On the other hand, since $\sigma$ is a proper sublinear function satisfying $\sigma(\alpha) = \gamma(\alpha)$ for $\alpha \in \RS$
    and $\cone{\gamma}$ is the pointwise maximum proper sublinear function satisfying $(\cone{\gamma})(\alpha) \leq \gamma(\alpha)$ for all $\alpha \in \RS$,
    we have $\sigma \leq \cone{\gamma}$.
    This implies that $\sigma = \cone{\gamma}$, i.e., $\sigma$ is M-sublinear.
\end{proof}

\subsubsection{Tight distance functions}\label{subsubsec:tight-distance}
A distance function $\funcdoms{\gamma}{\RS}{\ER}$ is said to be \emph{tight} if $\gamma$ is the restriction of an M-sublinear function to $\RS$, i.e.,
$\gamma$ is a valid distance function with $\gamma = (\cone{\gamma})|_{\RS}$.
Moreover, a tight distance function is said to be \emph{integral} if it is the restriction of an integral M-sublinear function.
Let $\vdist{\R}$ denote the family of all valid distance functions,
and $\tdist{\R}$ (resp. $\tdist{\Z}$) the family of tight (resp. integral tight) distance functions.
See the commutative diagram~\eqref{eq:diagram-Lset} below for the relations among
$\Lpoly{\R}$, $\Lpoly{\Z|\R}$, $\Msub{\R}{\R}$, $\Msub{\R}{\Z|\R}$, $\vdist{\R}$, $\tdist{\R}$, and $\tdist{\Z}$.

We conclude \Cref{subsubsec:tight-distance} with a ``triangle-inequality-like'' characterization of tight distance functions for each type of root system;
note that the condition for type~A is exactly the same as that presented in \cite[Sections~5.2 and 5.6]{Murota2003-yb}.
Unlike the preceding results, the following explicit
characterization uses the coordinate descriptions of the classical
root systems.
The definition of tight distance functions and the general
correspondences developed above do not require this restriction.
Here, for notational simplicity, we consider $2\RSC$ instead of $\RSC$ in the case of Type~C.
\begin{proposition}\label{prop:distance-function}
    A distance function $\funcdoms{\gamma}{\RS}{\ER}$ is tight
    if and only if
    $\gamma(\alpha) + \gamma(-\alpha) \geq 0$ holds for all $\alpha \in \RS$
    and the following conditions hold for each case of $\RS = \RSA$ (Type~A), $\RS =\RSB$ (Type~B), $\RS =2\RSC$ (Type~C), and $\RS =\RSD$ (Type~D):
    \begin{description}[font=\normalfont]
        \item[(Type~A)] $\gamma(\alpha_1) + \gamma(\alpha_2) \geq \gamma(\alpha_1+ \alpha_2)$ for any distinct $\alpha_1, \alpha_2 \in \RS$ with $\alpha_1 + \alpha_2 \in \RS$,
        which is equivalent to
        $\gamma(e_i - e_j) + \gamma(e_j - e_k) \geq \gamma(e_i - e_k)$ for any distinct $i, j, k \in [n+1]$.
    \item[(Type~B)] $\gamma(\alpha_1) + \gamma(\alpha_2) \geq \gamma(\alpha_1+ \alpha_2)$ for any distinct $\alpha_1, \alpha_2 \in \RS$ with $\alpha_1 + \alpha_2 \in \RS$,
    and $\gamma(\alpha_1) + \gamma(\alpha_2) \geq 2\gamma((\alpha_1 + \alpha_2)/2)$
    for any distinct $\alpha_1, \alpha_2 \in \RS$ with $(\alpha_1 + \alpha_2)/2 \in \RS$,
    which are equivalent to
    \begin{itemize}
      \item $\gamma(\sigma_i e_i + \sigma_j e_j) + \gamma(- \sigma_j e_j + \sigma_k e_k) \geq \gamma(\sigma_i e_i + \sigma_k e_k)$,
        \item $\gamma(\sigma_i e_i) + \gamma(\sigma_j e_j) \geq \gamma(\sigma_i e_i + \sigma_j e_j)$,
        \item $\gamma(\sigma_i e_i) + \gamma(-\sigma_i e_i +\sigma_j e_j) \geq \gamma(\sigma_j e_j)$, and 
        \item $\gamma(\sigma_i e_i + \sigma_j e_j) + \gamma(\sigma_i e_i - \sigma_j e_j) \geq 2\gamma(\sigma_i e_i)$
    \end{itemize}
    for any distinct $i, j, k \in [n]$ and any $\sigma_i, \sigma_j, \sigma_k \in \{\pm1\}$ in the first bullet, and for any distinct $i, j \in [n]$ and any $\sigma_i, \sigma_j \in \{\pm1\}$ in the other bullets.
    \item[(Type~C)] $\gamma(\alpha_1) + \gamma(\alpha_2) \geq \gamma(\alpha_1+ \alpha_2)$ for any distinct $\alpha_1, \alpha_2 \in \RS$ with $\alpha_1 + \alpha_2 \in \RS$,
    and $\gamma(\alpha_1) + \gamma(\alpha_2) \geq 2\gamma((\alpha_1 + \alpha_2)/2)$
    for any distinct $\alpha_1, \alpha_2 \in \RS$ with $(\alpha_1 + \alpha_2)/2 \in \RS$,
    which are equivalent to
    \begin{itemize}
      \item $\gamma(\sigma_i e_i + \sigma_j e_j) + \gamma(- \sigma_j e_j + \sigma_k e_k) \geq \gamma(\sigma_i e_i + \sigma_k e_k)$,
        \item $\gamma(\sigma_i e_i + \sigma_j e_j) + \gamma(\sigma_i e_i - \sigma_j e_j) \geq \gamma(2\sigma_i e_i)$,
        \item $\gamma(2\sigma_i e_i) + \gamma(-\sigma_i e_i + \sigma_j e_j) \geq \gamma(\sigma_i e_i +\sigma_j e_j)$, and
        \item $\gamma(2\sigma_i e_i) + \gamma(2\sigma_j e_j) \geq 2\gamma(\sigma_i e_i + \sigma_j e_j)$
    \end{itemize}
    for any distinct $i, j, k \in [n]$ and any $\sigma_i, \sigma_j, \sigma_k \in \{\pm1\}$ in the first bullet, and for any distinct $i, j \in [n]$ and any $\sigma_i, \sigma_j \in \{\pm1\}$ in the other bullets.
    \item[(Type~D)] 
    $\gamma(\alpha_1) + \gamma(\alpha_2) \geq \gamma(\alpha_1+ \alpha_2)$ for any distinct $\alpha_1, \alpha_2 \in \RS$ with $\alpha_1 + \alpha_2 \in \RS$,
    and $\gamma(\alpha_1) + \gamma(\alpha_2) + \gamma(\alpha_3) + \gamma(\alpha_4) \geq 2\gamma((\alpha_1 + \alpha_2 + \alpha_3 + \alpha_4)/2)$
    for any distinct $\alpha_1, \alpha_2, \alpha_3, \alpha_4 \in \RS$ with $(\alpha_1 + \alpha_2 + \alpha_3 + \alpha_4)/2 \in \RS$,
    which are equivalent to
    \begin{itemize}
      \item $\gamma(\sigma_i e_i + \sigma_j e_j) + \gamma(- \sigma_j e_j + \sigma_k e_k) \geq \gamma(\sigma_i e_i + \sigma_k e_k)$ and
        \item $\gamma(\sigma_i e_i + \sigma_{i'} e_{i'}) + \gamma(\sigma_i e_i -\sigma_{i'} e_{i'}) + \gamma(\sigma_j e_j + \sigma_{j'} e_{j'}) + \gamma(\sigma_j e_j -\sigma_{j'} e_{j'}) \geq 2\gamma(\sigma_i e_i + \sigma_j e_j)$
    \end{itemize}
    for any distinct $i,j,k \in [n]$ and any $\sigma_i, \sigma_j, \sigma_k \in \{\pm1\}$ in the first bullet,
    and
    for any distinct $i,i', j, j' \in [n]$ and any $\sigma_i, \sigma_{i'}, \sigma_j, \sigma_{j'} \in \{\pm1\}$ in the second bullet.
    \end{description}
\end{proposition}

The above proposition is of independent interest, 
but lies outside the main line of development and is not used in the proofs of our other results.
We therefore defer its proof to \Cref{sec:proof:prop:distance-function}.

\subsubsection{Basic properties of L-convex polyhedra}
We here summarize basic properties of L-convex polyhedra.
\begin{lemma}\label{lem:face:L-poly}
    Every nonempty face of an (integral) L-convex polyhedron is an (integral) L-convex polyhedron.
\end{lemma}
\begin{proof}
    Let $P$ be an L-convex polyhedron
    and suppose that $P$ admits an outer-representation $P = \D{\gamma}$
    for a distance function $\funcdoms{\gamma}{\RS}{\ER}$.
    Then, by a well-known fact (see e.g.,~\cite[Chapter~8.3]{Schrijver2000-xk}) on the faces of a polyhedron,
    for any nonempty face $F$ of $P$, there is a subset $\RS' \subseteq \dom{\gamma}$
    such that
    \begin{align}
        F = \Set{ x \in V }{ \inpr{\alpha}{x} = \gamma(\alpha) \ (\forall \alpha \in \RS'),\ \inpr{\alpha}{x} \leq \gamma(\alpha) \ (\forall \alpha \in \RS \setminus \RS') }.
    \end{align}
    Hence, $F$ is representable as $F = \D{\gamma'}$ for some distance function $\gamma'$, implying that $F$ is an L-convex polyhedron.
    In particular, if $P$ is integral, then we can take $\gamma'$ as integer-valued.
    Therefore, $F$ is also integral.
\end{proof}

The definition of L-convex polyhedra immediately implies the following:
\begin{lemma}\label{lem:intersection-L-convex}
    The intersection of arbitrarily many (integral) L-convex polyhedra is an (integral) L-convex polyhedron if it is nonempty.
\end{lemma}
\begin{proof}
    Let $\{ P_\lambda \}_{\lambda \in \Lambda}$ be a family of L-convex polyhedra, where $\Lambda$ is an arbitrary index set,
    and $P \coloneqq \bigcap_{\lambda \in \Lambda} P_\lambda$.
    Suppose that $P$ is nonempty.
    For each $\lambda \in \Lambda$, there is a distance function $\funcdoms{\gamma_\lambda}{\RS}{\ER}$ such that $P_\lambda = \D{\gamma_\lambda}$.
    Then, we have $P = \bigcap_{\lambda \in \Lambda} \D{\gamma_\lambda} = \D{\gamma}$, where $\gamma \coloneqq \inf_{\lambda \in \Lambda} \gamma_\lambda$.
    Since $P$ is nonempty, $\gamma$ never takes $-\infty$ and $\gamma$ is valid.
    Thus, $P$ is an L-convex polyhedron.

    If each $P_\lambda$ is integral, then we can take $\gamma_\lambda$ to be integer-valued for each $\lambda \in \Lambda$.
    Hence, $\gamma \coloneqq \inf_{\lambda \in \Lambda} \gamma_\lambda$ is also integer-valued, which implies that $P$ is an integral L-convex polyhedron.
\end{proof}

The following provides the characterization of L-convex polyhedra as the L-convexity of its tangent cone,
which is the counterpart of the definition of M-convex polyhedra (see \Cref{subsec:M-poly-L-sub} below).
\begin{lemma}\label{lem:L-convex-tcone}
    A nonempty closed convex set $C$ is an L-convex polyhedron if and only if, for any $x \in C$,
    the tangent cone $\tcone{C}{x}$ of $C$ at $x$ is an L-convex cone.
\end{lemma}
\begin{proof}
    (Only-if part).
    Since $C$ is an L-convex polyhedron, there exists a valid function $\funcdoms{\gamma}{\RS}{\ER}$ such that $C = \D{\gamma}$.
    Take any $x \in C$.
    It follows from \cite[Proposition~3.5.2]{Matthias-Beck2018-lk} that
    the tangent cone $\tcone{C}{x}$ is of the form $\Set{ y \in V }{ \inpr{\alpha}{y} \leq 0 \ (\forall \alpha \in \RS') }$,
    where $\RS' \coloneqq \Set{ \alpha \in \dom{\gamma} }{ \inpr{\alpha}{x} = \gamma(\alpha) }$.
    Thus, $\tcone{C}{x}$ is an L-convex cone.

    (If part).
    For each $x \in C$,
    the set $x + \tcone{C}{x}$ forms an L-convex polyhedron,
    since $\tcone{C}{x}$ is an L-convex cone.
    By \Cref{lem:clconv-tcone},
    we have $C = \bigcap_{x \in C} (x+\tcone{C}{x})$.
    Thus, $C$ is an L-convex polyhedron by \Cref{lem:intersection-L-convex}.
\end{proof}

\begin{lemma}\label{lem:union-closed:L}
  The union of arbitrarily many integral L-convex polyhedra is closed.
\end{lemma}
\begin{proof}
  Let $\{\D{\gamma_\lambda}\}_{\lambda \in \Lambda}$ be an arbitrary family of integral L-convex polyhedra,
  where $\funcdoms{\gamma_\lambda}{\RS}{\Z}$ is an integer-valued distance function for each $\lambda \in \Lambda$,
  and $L \coloneqq \bigcup_{\lambda \in \Lambda} \D{\gamma_\lambda}$.
  Take any $x \in V \setminus L$.
  Then, for all $\lambda \in \Lambda$,
  there is a root $\alpha_\lambda \in \RS$ such that $\gamma_\lambda(\alpha_\lambda) \in \Z$ and $\inpr{\alpha_\lambda}{x} > \gamma_\lambda(\alpha_\lambda)$.
  Since $\Set{ \alpha_\lambda }{ \lambda \in \Lambda } \subseteq \RS$ is a finite set,
  we have $\card{\Set{ \inpr{\alpha_\lambda}{x} }{ \lambda \in \Lambda }} < +\infty$.
  By $\gamma_\lambda(\alpha_\lambda) \in \Z$,
  there is $\delta > 0$ such that $\inpr{\alpha_\lambda}{x} - \delta > \gamma_\lambda(\alpha_\lambda)$ for any $\lambda \in \Lambda$.
  Therefore,
  there is a sufficiently small $\eps > 0$
  such that, for all $d \in V$ with $\|d\| < \eps$ and $\lambda \in \Lambda$,
  we have $\inpr{\alpha_\lambda}{x + d} > \gamma_\lambda(\alpha_\lambda)$.
  This implies that the open ball $B(x; \eps)$ centered at $x$ with radius $\eps$ is contained in $V \setminus L$.
  Hence, $L$ is closed.
\end{proof}

\subsection{M-convex polyhedra and L-sublinear functions}\label{subsec:M-poly-L-sub}
A nonempty set $P \subseteq V^*$ is called an \emph{M-convex polyhedron of type $\RS$}
if it is a closed convex set such that, for any $p \in P$,
the tangent cone $\tcone{P}{p}$ of $P$ at $p$ is an M-convex cone (cf. \Cref{lem:L-convex-tcone}).
An M-convex polyhedron $P$ is actually a polyhedron.
Indeed, since the number of M-convex cones of type $\RS$ is bounded by that of subsets of $\RS$, particularly, a finite value,
the set $\Set{\tcone{P}{p} }{ p \in P }$ is a finite set,
which implies that $P$ is a polyhedron by \Cref{lem:tcone-poly-chara}.
An M-convex polyhedron $P$ is said to be \emph{integral} 
if it is $\Mdom$-integral, i.e., $P = \conv(P \cap \Mdom)$.
We note that an M-convex cone is an integral M-convex polyhedron.
Indeed, for an M-convex cone $K$, which is representable as $K = \cone{\RS'}$ for some $\RS' \subseteq \RS$, and $p = \sum_{i = 1}^k \mu_i \alpha_i \in K$ with $\alpha_1, \alpha_2, \dots, \alpha_k \in \RS'$ and $\mu_1, \mu_2, \dots, \mu_k > 0$,
we have $\tcone{K}{p} = \cone(\RS' \cup \{ -\alpha_1, -\alpha_2, \dots, -\alpha_k \})$;
this implies that $K$ is an M-convex polyhedron.
Moreover, since $K = \conv\{0\} + \cone{\RS'}$,
it is $\Mdom$-integral by \Cref{lem:integral-polyhedron}.
We denote by $\Mpoly{\R}$ (resp. $\Mpoly{\Z|\R}$) the family of M-convex polyhedra (resp. integral M-convex polyhedra).

A polyhedral sublinear function $\funcdoms{\sigma}{V}{\ER}$ is called an \emph{L-sublinear function of type $\RS$} if every member of $\fan(\sigma)$ is an L-convex cone (cf. \Cref{lem:M-sublinear}).
An L-sublinear function $\sigma$ is said to be \emph{integral}
if there is a finite subset $B \subseteq \Mdom$ such that
$\sigma$ is representable as
\begin{align}\label{eq:integral-Lsub}
    \sigma(x) =
        \max{\Set{ \inpr{p}{x} }{ p \in B }}
\end{align}
for $x \in \dom{\sigma}$.
In other words, a polyhedral sublinear function $\funcdoms{\sigma}{V}{\ER}$ is an integral L-sublinear function if and only if every member of $\fan(\sigma)$ is an L-convex cone
and, for any maximal $K \in \fan(\sigma)$, there is $p \in \Mdom$ such that $K = \argmin(\sigma - p)$.
We denote by $\Lsub{\R}{\R}$ (resp. $\Lsub{\R}{\Z|\R}$) the family of L-sublinear functions (resp. integral L-sublinear functions).

When we emphasize the integrality of an M-convex polyhedron (resp. an L-sublinear function) of type $\RS$,
we specifically use the term \emph{$\RS$-integral}.

For a root system $\RS_-$ with $\RS_- \precsim \RS$,
although an M-convex polyhedron (resp. an L-sublinear function) of type $\RS_-$ is M-convex (resp. L-sublinear) of type $\RS$,
a $\RS_-$-integral M-convex polyhedron (resp. a $\RS_-$-integral L-sublinear function) is not necessarily $\RS$-integral.

As in the case of L-convex polyhedra and M-sublinear functions (\Cref{thm:conjugate-Lpoly-Msub}),
we establish the one-to-one correspondence between the (integral) M-convex polyhedra and the (integral) L-sublinear functions as follows.
Here,
for an L-sublinear function $\funcdoms{\sigma}{V}{\ER}$,
we denote the polyhedron $\bS{\sigma}$ by $\B{\sigma}$ (cf.~\cite[Chapter~4.3]{Murota2003-yb}).
\begin{theorem}\label{thm:conjugate-Mpoly-Lsub}
    The support function $\spt{P}$ of an (integral) M-convex polyhedron $P$ is an (integral) L-sublinear function, and the set $\B{\sigma}$ arising from an L-sublinear function $\sigma$ is an (integral) M-convex polyhedron.
\end{theorem}

The following lemma is an easy observation on L-sublinear functions,
which is used in the proof of \Cref{thm:conjugate-Mpoly-Lsub}.
\begin{lemma}\label{lem:L-sublinear-effective-domain}
    The effective domain of an L-sublinear function is an L-convex cone.
\end{lemma}
\begin{proof}
    Suppose that $\funcdoms{\sigma}{V}{\ER}$ is an L-sublinear function.
    Then, we have $\dom{\sigma} = \geo{\fan(\sigma)} = \bigcup_{K \in \fan(\sigma)} K$.
    Since each $K \in \fan(\sigma)$ is an L-convex cone,
    it is the geometric realization of a convex subcomplex $\CC_K$ of $\spCC$ by \Cref{lem:chara-L-convex-cone}.
    Then, $\CC' \coloneqq \bigcup_{K \in \fan(\sigma)} \CC_K$ is also a convex subcomplex of $\spCC$ by \Cref{lem:convex_subcomplex}, since $\dom{\sigma} = \bigcup_{K \in \fan(\sigma)} \geo{\CC_K} = \geo{\CC'}$ is convex.
    Thus, $\dom{\sigma}$ is the geometric realization of a convex subcomplex of $\spCC$, which implies that $\dom{\sigma}$ is an L-convex cone by \Cref{lem:chara-L-convex-cone}.
\end{proof}

We are ready to show \Cref{thm:conjugate-Mpoly-Lsub}.
\begin{proof}[Proof of \Cref{thm:conjugate-Mpoly-Lsub}]
    If $P$ is an M-convex polyhedron,
    then we have $\fan(\spt{P}) = \Set{ (\tcone{P}{p})^* }{ p \in P }$ by \Cref{lem:fan-support}.
    Since $\tcone{P}{p}$ is an M-convex cone for any $p \in P$ by the definition,
    its polar cone $(\tcone{P}{p})^*$ is an L-convex cone by \Cref{thm:L-M-polar}.
    Thus, $\fan(\spt{P})$ consists of L-convex cones, which implies that $\spt{P}$ is an L-sublinear function.
    Conversely, if $\sigma$ is an L-sublinear function,
    then $\B{\sigma}$ is a nonempty polyhedron.
    By \Cref{lem:fan-support},
    for any $p \in \B{\sigma}$
    we have $\tcone{\B{\sigma}}{p} = (\argmin(\sigma-p))^*$.
    Since $\argmin(\sigma-p) \in \fan(\sigma)$ is an L-convex cone, its polar cone $\tcone{\B{\sigma}}{p}$ is an M-convex cone by \Cref{thm:L-M-polar}.
    Thus, $\B{\sigma}$ is an M-convex polyhedron.

    By the above correspondence,
    the recession cone of an M-convex polyhedron is an M-convex cone.
    Indeed, for an M-convex polyhedron $P$ and its corresponding L-sublinear function $\sigma$,
    the recession cone of $P$ is $(\dom{\sigma})^*$
    by \Cref{lem:conjugate-poly-sublinear},
    which is an M-convex cone by \Cref{thm:L-M-polar} and \Cref{lem:L-sublinear-effective-domain}.
    
    We then see that the one-to-one correspondence between M-convex polyhedra and L-sublinear functions preserves integrality.
    If $P$ is an integral M-convex polyhedron,
    then $P$ admits an inner-representation $P = \conv{B} + \cone{\RS'}$
    for some finite sets $B \subseteq \Mdom$ and $\RS' \subseteq \RS \subseteq \Mdom$ by \Cref{lem:integral-polyhedron}.
    Here, we have already seen that $\spt{P}$ is L-sublinear.
    By \Cref{lem:conjugate-poly-sublinear},
    we obtain $\spt{P}(x) = \max{\Set{ \inpr{p}{x} }{ p \in B }}$
    for $x \in \dom(\spt{P})$.
    Thus, $\spt{P}$ is integral.
    Conversely, if $\sigma$ is an integral L-sublinear function,
    then there is a finite set $B \subseteq \Mdom$ such that $\sigma(x) = \max{\Set{ \inpr{p}{x} }{ p \in B }}$ for $x \in \dom{\sigma}$.
    Here, we have already seen that $\B{\sigma}$ is an M-convex polyhedron.
    By \Cref{lem:conjugate-poly-sublinear},
    we obtain $\B{\sigma} = \conv{B} + \cone{\RS'}$, which implies that $\B{\sigma}$ is integral by \Cref{lem:integral-polyhedron}.
\end{proof}

\subsubsection{(Spherical) submodular functions}\label{subsubsec:spherical-submodular}
Fix an arbitrary spherical expression $\CC^\circ$ of $\spCC$.
A function $\funcdoms{\rho}{\sk{\CC^\circ}}{\ER}$ is said to be \emph{spherical submodular of type $\RS$}
if $\rho$ is the restriction of some L-sublinear function $\sigma$ to $\sk{\CC^\circ}$.
In addition, a spherical submodular function $\rho$ is said to be \emph{integral} if there is an integral L-sublinear function $\sigma$ of type $\RS$ such that $\rho$ is the restriction of $\sigma$ to $\sk{\CC^\circ}$.
For consistency with existing notation in DCA~\cite[Chapter~4.3]{Murota2003-yb},
for a spherical submodular function $\funcdoms{\rho}{\sk{\CC^\circ}}{\ER}$,
we denote the polyhedron $\bS{\rho}$ by $\B{\rho}$, i.e.,
\begin{align}
    \B{\rho} \coloneqq \bS{\rho} = \Set{ p \in V^* }{ \inpr{p}{x} \leq \rho(x) \ (\forall x \in \dom{\rho}) }.
\end{align}
We particularly refer to $\B{\rho}$ as the \emph{base polyhedron} of $\rho$.

The following proposition states that the spherical submodular functions play a role as the tight distance functions in \Cref{subsec:L-poly:M-sub}
and that they can be characterized as the convexity of its positively homogeneous Lov\'{a}sz extension.
\begin{proposition}\label{prop:spherical-submodular-cone}
    Let $\CC^\circ$ be a spherical expression of $\spCC$ and $\funcdoms{\rho}{\sk{\CC^\circ}}{\ER}$ a function.
    Then, the following are equivalent:
    \begin{enumerate}[label={\textup{(\alph*)}}]
        \item $\rho$ is spherical submodular.
        \item $\phext{\rho}$ is convex.
        \item $\cone{\rho}$ is an L-sublinear function and $\rho = (\cone{\rho})|_{\CC^\circ}$.
    \end{enumerate}
    In addition, for a spherical submodular function $\rho$,
    we have $\cone{\rho} = \phext{\rho}$,
    which forms the unique L-sublinear function whose restriction to $\CC^\circ$ is $\rho$.
\end{proposition}
\begin{proof}
    The implication (c) $\Rightarrow$ (a) is clear.

    (a) $\Rightarrow$ (b).
    Since $\rho$ is a spherical submodular function, there is an L-sublinear function $\funcdoms{\sigma}{V}{\ER}$ such that $\rho$ is the restriction of $\sigma$ to $\sk{\CC^\circ}$.
    Our aim is to show that $\phext{\rho} = \sigma$,
    which implies that $\phext{\rho}$ is convex (and $\sigma$ is the unique L-sublinear function whose restriction to $\sk{\CC^\circ}$ is $\rho$).
    Since $\dom{\phext{\rho}} = \phext{\dom{\rho}} \subseteq \cone(\dom{\rho}) \subseteq \dom{\sigma}$,
    it suffices to show that $\phext{\rho}(x) = \sigma(x)$ for any $x \in \dom{\sigma}$.

    Take any $x \in \dom{\sigma}$.
    Recall that $\cell{A}{x}$ denotes the unique cell of $\CC^\circ$ such that the corresponding simplex $\cone{\cell{A}{x}}$ of $\spCC$ contains $x$.
    Let $K \in \fan(\sigma)$ with $x \in K$.
    Since $K$ is an L-convex cone, there is a convex subcomplex $\CC_K$ of $\spCC$ such that $K = \geo{\CC_K}$ by \Cref{lem:chara-L-convex-cone}.
    In particular, the cell $\cone{\cell{A}{x}}$ belongs to $\CC_K$.
    Since $\sigma$ is linear on $K$,
    we have $\sigma(x) = \sum_{v \in \sk{\cell{A}{x}}} \mu_x(v) \sigma(v) = \sum_{v \in \sk{\cell{A}{x}}} \mu_x(v) \rho(v) = \phext{\rho}(x)$,
    where the second equality follows from the fact that $\rho$ is the restriction of $\sigma$ to $\sk{\CC^\circ}$.

    (b) $\Rightarrow$ (c).
    Suppose that $\phext{\rho}$ is convex.
    It follows from the definitions~\eqref{eq:func:cone} and~\eqref{eq:ph-Lovasz-ext} of $\cone{\rho}$ and $\phext{\rho}$, respectively,
    that $\phext{\rho} \geq \cone{\rho}$ and $\phext{\rho}(x) = \rho(x)$ for $x \in \sk{\CC^\circ}$.
    Since $\cone{\rho}$ is the pointwise maximum proper sublinear function satisfying $\cone{\rho}(x) \leq \rho(x)$ for all $x \in \sk{\CC^\circ}$,
    we have $\cone{\rho} \geq \phext{\rho}$, which implies that $\phext{\rho} = \cone{\rho}$.
    Thus, it suffices to see that every member of $\fan(\phext{\rho})$ is an L-convex cone.

    Take any $K \in \fan(\phext{\rho})$ and $x \in K$.
    By the definition of $\phext{\rho}$,
    we have $\phext{\rho}(x) = \sum_{v \in \sk{\cell{A}{x}}} \cell{\mu}{x}(v) \rho(v)$.
    Suppose that $K = \argmin(\phext{\rho} - p)$ for $p \in V^*$.
    If $\sk{\cell{A}{x}} \not\subseteq K$,
    then $\phext{\rho}(x) - \inpr{p}{x} > \min(\phext{\rho} - p)$,
    which contradicts that $x \in K$.
    Thus, $\sk{\cell{A}{x}} \subseteq K$, i.e., $\cone{\cell{A}{x}} \subseteq K$ holds.
    Hence, $K$ is representable as $K = \bigcup_{x \in K} \cone{\cell{A}{x}}$,
    which implies that $K$ is the geometric realization of a subcomplex of $\spCC$.
    Furthermore, since $K$ is a convex set,
    it is the geometric realization of a convex subcomplex by \Cref{lem:convex_subcomplex}.
    Thus, $K$ is an L-convex cone by \Cref{lem:chara-L-convex-cone}.
\end{proof}

The following corollary is a direct consequence of \Cref{thm:conjugate-Mpoly-Lsub} and \Cref{prop:spherical-submodular-cone}.
\begin{corollary}\label{cor:Mpoly-subm}
    A subset $P \subseteq V^*$ is an (integral) M-convex polyhedron if and only if
    $P = \B{\rho}$ for some (integral) spherical submodular function $\rho$ on a spherical expression of $\spCC$.
\end{corollary}

Let $\subm{\CC^\circ}{\R}$ and $\subm{\CC^\circ}{\Z}$ denote the families of spherical submodular functions and integral spherical submodular functions, respectively.
The commutative diagram~\eqref{eq:diagram-Mset} below summarizes the relations among $\Mpoly{\R}$, $\Mpoly{\Z}$, $\Lsub{\R}{\R}$, $\Lsub{\R}{\Z|\R}$, $\subm{\CC^\circ}{\R}$, and $\subm{\CC^\circ}{\Z}$,
which is the counterpart of the diagram~\eqref{eq:diagram-Lset} for L-convex polyhedra and M-sublinear functions.

Suppose that $\RS$ is irreducible.
For a vertex $x \in \Ldom$,
we refer to a vertex of $\St{\euCC}{x}$ as a \emph{neighbor} of $x$
and denote by $\nbor{x}(\RS)$ the set of neighbors of $x$, that is,
$\nbor{x}(\RS) \coloneqq \sk{\St{\euCC}{x}}$;
we write just $\nbor{x}$ if it is clear from the context.
A vertex in $\sk{\Lk{\euCC}{x}} = \nbor{x} \setminus \{x\}$
is called a \emph{proper neighbor} of $x$
and we define $\opnbor{x}(\RS) = \opnbor{x} \coloneqq \sk{\Lk{\euCC}{x}}$.
In particular,
we have $\opnbor{0} = \sk{\spCC^\star}$
and $\nbor{0} = \sk{\spCC^\star} \cup \{0\}$
by \Cref{lem:link-standard}.
Recall that $\Lk{\euCC}{x} - x$ is the spherical expression of $\spCC(\tilde{\RS}|_{-x})$ generated from $\opnbor{x} - x$ (see \Cref{lem:standard}~(2)).

Fix $x_* \in \Ldom$.
A function $\funcdoms{\rho}{\nbor{x_*}}{\ER}$ with $\rho(x_*) < +\infty$ is said to be \emph{submodular}
if the function $\rho^\circ$ on $\opnbor{x_*} - x_*$ defined by 
\begin{align}\label{eq:rho-circ}
    \rho^\circ(x - x_*) \coloneqq \rho(x) - \rho(x_*) \qquad (x \in \opnbor{x_*})
\end{align}
is spherical submodular (of type $\tilde{\RS}_{-x_*}$);
we particularly refer to a submodular function on $\nbor{0} = \nbor{0}(\RS)$ as a \emph{submodular function of type $\RS \;(= \tilde{\RS}_{-0})$}.
In addition, a submodular function $\rho$ is \emph{integral} if $\rho^\circ$ is integral.
For a submodular function $\funcdoms{\rho}{\nbor{x_*}}{\ER}$ with $\rho(x_*) = 0$,
we define its \emph{base polyhedron} $\B{\rho}$ by
\begin{align}
    \B{\rho} \coloneqq \Set{ p \in V^* }{ \inpr{p}{x - x_*} \leq \rho(x) \ (\forall x \in \nbor{x_*}) },
\end{align}
which coincides with the base polyhedron $\B{\rho^\circ}$ of the spherical submodular function $\rho^\circ$ defined in~\eqref{eq:rho-circ}.

We can easily see that the submodularity can be characterized as the convexity of its Lov\'{a}sz extension by the definition of submodularity and \Cref{prop:spherical-submodular-cone}.
Moreover, the following \emph{submodular inequality} also characterizes the submodularity.
We say that a function $\funcdoms{\rho}{\nbor{x_*}}{\ER}$ with $\rho(x_*) < +\infty$ satisfies the \emph{submodular inequality} if
\begin{align}\label{eq:subm-ineq}
    \rho(x) + \rho(y) \geq 2\sum_{v \in \nbor{x_*}} \cell{\lambda}{z}(v) \rho(v) \qquad (x, y \in \nbor{x_*}).
\end{align}
We also refer to the inequalities~\eqref{eq:subm-ineq} as the \emph{discrete midpoint convexity}; see also~\eqref{eq:midpointconvexity}.
Note that the inequality~\eqref{eq:subm-ineq} is well-defined.
Indeed,
by \Cref{lem:star_complex},
the geometric realization $\geo{\St{\euCC}{x_*}}$ of $\St{\euCC}{x_*}$ is a convex set in $V$.
Thus, for any $x, y \in \nbor{x_*} = \sk{\St{\euCC}{x_*}} \subseteq \Ldom$,
their midpoint $z \coloneqq (x+y)/2$ belongs to $\geo{\St{\euCC}{x_*}}$,
which implies $\cell{A}{z} \in \St{\euCC}{x_*}$, i.e., $\supp{\cell{\lambda}{z}} = \sk{\cell{A}{z}} \subseteq \nbor{x_*}$.
\begin{theorem}\label{thm:submodular-chara}
    Let $\RS$ be an irreducible root system and $x_* \in \Ldom$.
    For a function $\funcdoms{\rho}{\nbor{x_*}}{\ER}$ with $\rho(x_*) < +\infty$,
    the following are equivalent:
    \begin{enumerate}[label={\textup{(\alph*)}}]
        \item $\rho$ is submodular.
        \item $\ext{\rho}$ is convex.
        \item $\rho$ satisfies the discrete midpoint convexity.
    \end{enumerate}
\end{theorem}
The following only provides the proof of (a) $\Leftrightarrow$ (b) $\Rightarrow$ (c).
The proof of (c) $\Rightarrow$ (b) can be obtained from a more general statement (\Cref{thm:chara:L-conv-func}) below; see \Cref{rmk:proof:subm}.
\begin{proof}
    (a) $\Leftrightarrow$ (b) immediately follows from the definitions of (spherical) submodularity
    and (positively homogeneous) Lov\'{a}sz extension.

    (b) $\Rightarrow$ (c).
    For any $x, y \in \nbor{x_*}$,
    we have $\rho(x) + \rho(y) = \ext{\rho}(x) + \ext{\rho}(y) \geq 2\ext{\rho}(z) = 2\sum_{v \in \nbor{x_*}} \cell{\lambda}{z}(v) \rho(v)$,
    where $z \coloneqq (x+y)/2$.
    Here, the inequality follows from the convexity of the Lov\'{a}sz extension $\ext{\rho}$
    and the last equality follows from the definition of the Lov\'{a}sz extension and $\supp{\cell{\lambda}{z}} = \sk{\cell{A}{z}} \subseteq \nbor{x_*}$.
\end{proof}

Recall that, for a root system $\RS_-$ with $\RS_- \precsim \RS$,
if $\cone{\rho}$ is an L-sublinear function of type~$\RS_-$,
then it is also of type~$\RS$.
That is, a spherical submodular function $\funcdoms{\rho}{\sk{\CC_-}}{\ER}$ of type $\RS_-$ can also be viewed as the restriction of the L-sublinear function of type $\RS$ to $\sk{\CC_-}$,
where $\CC_-$ is a spherical expression of $\spCC(\RS_-)$.
Thus, we can define the \emph{$\RS$-integrality} for spherical submodular functions of type $\RS_-$:
For a root system $\RS_-$,
a spherical submodular function $\funcdoms{\rho}{\sk{\CC_-}}{\ER}$ of type $\RS_-$
is said to be \emph{$\RS$-integral}
if $\cone{\rho}$ is a $\RS$-integral L-sublinear function.
We also define the \emph{$\RS$-integrality} for submodular functions in a similar way.

We can characterize the $\RS$-integrality of (spherical) submodular functions as that of their function values in the following sense.
\begin{proposition}\label{prop:integral-subm}
    \begin{enumerate}[label={\textup{(\arabic*)}}]
        \item A spherical submodular function $\funcdoms{\rho}{\sk{\spCC^\star}}{\ER}$ is integral if and only if $\rho(x) \in \Z/a(x)$
    for any $x \in \dom{\rho}$.
        \item Suppose that $\RS$ is an irreducible root system
        and that $\funcdoms{a}{\Ldom}{\Z}$ is the mark function of $\euCC$.
        Let $x_* \in \Ldom$
    and $\funcdoms{\rho}{\nbor{x_*}}{\ER}$ with $\rho(x_*) < +\infty$ a submodular function.
    If there is $r \in \R$ such that
    $\rho(x) - r \in \Z/a(x)$ for all $x \in \dom{\rho}$,
    then $\rho$ is $\RS$-integral.
    In addition, the converse also holds if, for each maximal simplex $A$ of $\euCC[\dom{\rho}]$ and $x \in \sk{A} \setminus \{x_*\}$, there is a vertex $x' \in \sk{A} \setminus \{x\}$ such that $a(x') = 1$.
    \end{enumerate}
\end{proposition}
\Cref{prop:integral-subm}~(2) (and its consequences \Cref{cor:integral-subm} and \Cref{thm:integral:L-func}~(2)) are basically used only in \Cref{sec:examples} to
clarify the relationship between the present notion of integrality and
previously studied one.
We therefore defer its proof to \Cref{subsec:proof:prop:integral-subm} and prove only \Cref{prop:integral-subm}~(1)
here.
\begin{proof}[Proof of \Cref{prop:integral-subm}~(1)]
    Let $\funcdoms{\rho}{\sk{\spCC^\star}}{\ER}$ be a spherical submodular function
    and $\sigma$ the (unique) L-sublinear function whose restriction to $\sk{\spCC^\star}$ is $\rho$.

    (Only-if part).
    Since $\rho$ is integral,
    so is $\sigma$.
    Take any $x \in \dom{\rho}$ and maximal L-convex cone $K \in \fan(\sigma)$ with $x \in K$.
    Then,
    there is $p \in \Mdom$ such that
    $K = \argmin(\sigma - p)$.
    For such $p$,
    we have $\rho(x) = \sigma(x) = \inpr{p}{x} \in \Z / a(x)$ by \Cref{cor:inpr}~(1).

    (If part).
    Let $K$ be an arbitrary maximal L-convex cone in $\fan(\sigma)$;
    our aim is to show that there is $p_K \in \Mdom$ such that $K = \argmin(\sigma - p_K)$.

    Let $\CC_K$ be the convex subcomplex of $\spCC$ whose geometric realization is $K$ (see \Cref{lem:chara-L-convex-cone}).
    Fix an arbitrary maximal simplex $A_K$ of $\CC_K$
    and let $A^\star$ be the simplex of $\spCC^\star$ corresponding to $A_K$.
    Then, we obtain
    \begin{align}\label{eq:dim:=}
        \dim(\dom{\sigma}) = \dim{K} = \dim{A_K} = \dim(\cone{A^\star}),
    \end{align}
    where the second identity follows from the fact that $\CC_K$ is a chamber complex (\Cref{lem:chamber-complex}).

    Suppose that $\simp$ is the simple system of $\RS$ corresponding to some chamber $C$ of $\spCC$ with $A_K \preceq C$.
    Then, we define the vector $p_K$ by
    \begin{align}
        p_K \coloneqq \sum_{x \in \sk{A^\star}} a(x)\rho(x)\alpha_x,
    \end{align}
    where, for $x \in \sk{A^\star}$, we denote by $\alpha_x$ the (unique) root in $\simp$ satisfying $\tau_{\simp}(\alpha_x) = \tau(x)$.

    Since $\rho(x) \in \Z/a(x)$ for $x \in \sk{A^\star}$ by the assumption,
    we obtain $p_K \in \Mdom$.
    Furthermore, it follows from \Cref{cor:inpr}~(1) that $\inpr{p_K}{x} = \rho(x)$ for all $x \in \sk{A^\star}$.
    Since $\sigma = \phext{\rho}$ by \Cref{prop:spherical-submodular-cone},
    $\sigma$ is linear on $K$, and $\dim{K} = \dim(\cone{A^\star})$ by~\eqref{eq:dim:=},
    we obtain $\inpr{p_K}{x} = \sigma(x)$ for all $x \in K$,
    which further implies that $K = \argmin(\sigma - p_K)$,
    since $\dim{K} = \dim(\dom{\sigma})$ by~\eqref{eq:dim:=} again and $\sigma$ is convex function.
    This completes the proof.
\end{proof}

The following is a direct consequence of \Cref{prop:integral-subm}~(2).
\begin{corollary}\label{cor:integral-subm}
    Suppose that $\RS$ is an irreducible root system.
    Let $x_* \in \Ldom$
    and $\funcdoms{\rho}{\nbor{x_*}}{\ER}$ with $\rho(x_*) < +\infty$ a submodular function.
    Then, $\rho$ is $\RS$-integral if and only if there is $r \in \R$ such that $\rho(x) -r \in \Z/a(x)$
    for any $x \in \dom{\rho}$,
    provided that one of the following conditions holds:
    \begin{enumerate}[label={\textup{(\alph*)}}]
        \item $a(x_*) = 1$.
        \item $\dom{\rho}$ is full-dimensional, i.e., there is a chamber $C$ of $\St{\euCC}{x_*}$ such that $\sk{C} \subseteq \dom{\rho}$.
    \end{enumerate}
\end{corollary}

It remains unclear whether the additional assumptions in
\Cref{prop:integral-subm}~(2) and
\Cref{cor:integral-subm} can be omitted, that is, whether
$\RS$-integrality is characterized solely by the existence of
$r\in\R$ such that
$\rho(x)-r\in\Z/a(x)\cup\{+\infty\}$ for all
$x\in\nbor{x_*}$.

\subsubsection{Basic properties of M-convex polyhedra}
We here summarize basic properties of M-convex polyhedra.
\begin{lemma}\label{lem:face:M-poly}
    Every nonempty face of an (integral) M-convex polyhedron is an (integral) M-convex polyhedron.
\end{lemma}
\begin{proof}
  Let $F$ be a nonempty face of an M-convex polyhedron $P$.
  Take any $p \in F$; our aim is to show that $\tcone{F}{p}$ is an M-convex cone.
  Then, $\tcone{F}{p}$ is a face of $\tcone{P}{p}$.
  Since $\tcone{P}{p}$ is an M-convex cone,
  there is $\RS' \subseteq \RS$ such that $\tcone{P}{p} = \cone{\RS'}$.
  Thus, the face $\tcone{F}{p}$ is of the form $\cone{\RS''}$ for some $\RS'' \subseteq \RS'$,
  which implies that $\tcone{F}{p}$ is an M-convex cone.

  In addition, since any face of an $\Mdom$-integral polyhedron is $\Mdom$-integral by \Cref{lem:integral-polyhedron},
  we can conclude that any nonempty face of an integral M-convex polyhedron is an integral M-convex polyhedron.
\end{proof}

\begin{lemma}\label{lem:union-closed:M}
  The union of arbitrarily many integral M-convex polyhedra is closed.
\end{lemma}
\begin{proof}
    By \Cref{cor:Mpoly-subm},
    for any integral M-convex polyhedron $P$,
    there is an integral spherical submodular function $\funcdoms{\rho}{\sk{\spCC^\star}}{\ER}$.

    Let $\{ \B{\rho_\lambda} \}_{\lambda \in \Lambda}$ be an arbitrary family of integral M-convex polyhedra,
    where $\rho_\lambda$ is an integral spherical submodular function on $\spCC^\star$.
    For any $\lambda \in \Lambda$ and $x \in \dom{\rho_\lambda}$,
    we have $\rho_\lambda(x) \in \Z/a(x)$ by \Cref{prop:integral-subm}~(1).
    By a similar argument to the proof of \Cref{lem:union-closed:L},
    we can show that $M \coloneqq \bigcup_{\lambda \in \Lambda} \B{\rho_\lambda}$ is closed.
\end{proof}

We then provide a characterization of M-convex polyhedra in addition to \Cref{cor:Mpoly-subm},
which is closely related to the definition of \emph{Coxeter matroids} (see \Cref{subsec:Coxeter-matroids}).
The \emph{maximality property \eqref{cond:MP}} for a nonempty subset $S \subseteq V^*$ is as follows: 
\begin{enumerate}[label=(\MP),ref=\MP]
\item \label{cond:MP} 
For any $p, q \in S$ and simple system $\simp = \{ \alpha_1, \alpha_2, \dots, \alpha_n \}$,
there exists $r \in S$ with $p \leq_\simp r$ and $q \leq_\simp r$.
\end{enumerate}

\begin{theorem}\label{thm:chara:M-poly}
    \begin{enumerate}[label={\textup{(\arabic*)}}]
        \item A set $P \subseteq V^*$ is an M-convex polyhedron
    if and only if
    $P = \clconv{S}$ for some nonempty set $S \subseteq V^*$ having~\eqref{cond:MP}.
        \item A set $P \subseteq V^*$ is an integral M-convex polyhedron
    if and only if
    $P = \conv{B}$ for some nonempty set $B \subseteq \Mdom$ having~\eqref{cond:MP}.
    \end{enumerate}
\end{theorem}
The proof of
the only-if part of (2) is deferred to \Cref{rmk:Mpoly:proof}.
\begin{proof}
    We first show the if parts of (1) and (2).
    Let $S \subseteq V^*$ be a nonempty set having~\eqref{cond:MP}.
    The following claim is essential in our proof.
    \begin{claim}\label{cl:MP}
        The support function $\spt{S}$ of $S$ is an L-sublinear function.
        In addition, if $S \subseteq \Mdom$, then $\spt{S}$ is integral.
    \end{claim}
    \begin{proof}[Proof of \Cref{cl:MP}]
        Let $\rho$ be the restriction of $\spt{S}$ to $\sk{\spCC^\star}$.
    It suffices to show that
    $\spt{S} = \phext{\rho}$.
    Indeed, if $\spt{S} = \phext{\rho}$,
    then $\phext{\rho}$ is convex,
    which implies that $\rho$ is spherical submodular 
    and $\spt{S}$ is L-sublinear by \Cref{prop:spherical-submodular-cone}.
    In addition, if $S \subseteq \Mdom$,
    then, for $x \in \sk{\spCC^\star}$,
    we have $\rho(x) = \spt{S}(x) = \sup{\Set{ \inpr{p}{x} }{ p \in S }} \in \Z/a(x) \cup \{+\infty\}$ by \Cref{cor:inpr}~(1).
    Thus, $\rho$ is an integral spherical submodular function by \Cref{prop:integral-subm}~(1),
    which implies that $\spt{S}$ is integral.

    Take any $x \in V$ and chamber $C \in \spCC^\star$ with $x \in \cl(\cone{C})$,
    where we recall that $\cone{C}$ is the chamber of $\spCC$ corresponding to $C$; see \Cref{subsubsec:spherical-expression}.
    Suppose that $\sk{C} = \{ x_1, x_2, \dots, x_n \}$, in which $\tau(x_i) = v_i$ for $i \in [n]$.
    We may assume that $\sk{A_x} = \{ x_1, \dots, x_k \}$,
    namely,
    $x = \sum_{i = 1}^k \mu_i x_i$ for some $\mu_i > 0$ for $i \in [k]$.
    Our aim is to show that $\spt{S}(x) = \sum_{i = 1}^k \mu_i \rho(x_i)$,
    which implies $\spt{S}(x) = \phext{\rho}(x)$.

    By the definition of $\spt{S}$ and $\rho$,
    we obtain
    \begin{align}
        \spt{S}(x) &= \sup{\Set{ \inpr{p}{x} }{ p \in S }}\\
        &= \sup{\Set*{ \sum_{i = 1}^k \mu_i \inpr{p}{x_i}}{ p \in S }}\\
        &\leq \sum_{i = 1}^k \mu_i \sup{\Set{ \inpr{p}{x_i} }{ p \in S }}\\
        &= \sum_{i = 1}^k \mu_i \rho(x_i).
    \end{align}
    Let $\simp = \{ \alpha_1, \alpha_2, \dots, \alpha_n \} \subseteq \RS$ be the simple system corresponding to $C$ with $\tau_{\simp}(\alpha_i) = \tau(x_i)$
    and $\hat{p}_i \coloneqq \sup\left\{ p_i \ \middle|\ \sum_{i = 1}^n p_i \alpha_i \in S \right\} \in \ER$.
    Then, for $i \in [n]$, we have
    \begin{align}
        \rho(x_i) = \sup{\Set{ \inpr{p}{x_i} }{ p \in S }} = \inpr{\alpha_i}{x_i}\sup{\Set*{ p_i }{ \sum_{i = 1}^n p_i \alpha_i \in S }} = \inpr{\alpha_i}{x_i} \hat{p}_i,
    \end{align}
    where the second equality follows from \Cref{lem:inpr}.
    Thus,
    \begin{itemize}
        \item if $\rho(x_i) < +\infty$, i.e., $\hat{p}_i < +\infty$,
        then for any sufficiently small $\eps > 0$, there is $q^i = \sum_{j = 1}^n q_j \alpha_j \in S$ such that $q_i \geq \hat{p}_i - \eps/ (\mu_i \inpr{\alpha_i}{x_i})$; and
        \item if $\rho(x_i) = +\infty$, i.e., $\hat{p}_i = +\infty$, then for any sufficiently large $M > 0$, there is $q^i = \sum_{j = 1}^n q_j \alpha_j \in S$ such that $q_i \geq M / (\mu_i \inpr{\alpha_i}{x_i})$.
    \end{itemize}
    Since $S$ has~\eqref{cond:MP}, for any sufficiently small $\eps > 0$ and large $M > 0$, there is $r_{\eps, M} = \sum_{i = 1}^n r_i \alpha_i \in S$ satisfying $r_{\eps, M} \geq_{\Delta} q^i$ for $q^i \in S$ defined as above for $i \in [k]$,
    which implies
    \begin{align}
        r_i \geq
        \begin{cases}
            \hat{p}_i - \eps/ (\mu_i \inpr{\alpha_i}{x_i}) & \text{if $\rho(x_i) < + \infty$},\\
            M / (\mu_i \inpr{\alpha_i}{x_i}) & \text{if $\rho(x_i) = + \infty$}.
        \end{cases}
    \end{align}
    Hence, we obtain
    \begin{align}
        \spt{S}(x) &\geq \inpr{r_{\eps, M}}{x} = \sum_{i = 1}^k \mu_i r_i \inpr{\alpha_i}{x_i}\\
        &\geq 
        \begin{dcases}
            \sum_{i = 1}^k \mu_i \rho(x_i) - k \eps & \text{if $\rho(x_i) < +\infty$ for all $i \in [k]$},\\
            M & \text{if $\rho(x_i) = +\infty$ for some $i \in [k]$}
        \end{dcases}
    \end{align}
    for any sufficiently small $\eps > 0$ and large $M > 0$.
    This implies that $\spt{S}(x) = \sum_{i = 1}^k \mu_i \rho(x_i)$.
    \end{proof}

    Let us return to the proof of \Cref{thm:chara:M-poly}.

    (If part of~(1)).
    By \Cref{lem:clconv-sublinear},
    we have $\clconv{S} = \B{\spt{S}}$,
    where we note that $\spt{S}$ is L-sublinear by \Cref{cl:MP}.
    Thus, by \Cref{thm:conjugate-Mpoly-Lsub},
    we can conclude that $\clconv{S}$ forms an M-convex polyhedron.

    (If part of~(2)).
    Suppose that $S \subseteq \Mdom$.
    Since $\spt{S}$ is an integral L-sublinear function by \Cref{cl:MP},
    $\B{\spt{S}}$ is an integral M-convex polyhedron by \Cref{thm:conjugate-Mpoly-Lsub}.
    Thus, it suffices to prove that $\conv{S} = \B{\spt{S}}$,
    which implies that $\conv{S}$ is an integral M-convex polyhedron.
    Since $\conv{S} \subseteq \clconv{S}$ holds in general and we have already shown $\clconv{S} = \B{\spt{S}}$ above,
    we prove $\conv{S} \supseteq \B{\spt{S}}$ in the following.

    Let $\simp = \{ \alpha_1, \alpha_2, \dots, \alpha_n \}$ be an arbitrary simple system
    and $C$ be the chamber of $\spCC^\star$ corresponding to $\simp$,
    where $\sk{C} = \{ x_1, x_2, \dots, x_n \}$ and $\tau_{\simp}(\alpha_i) = \tau(x_i)$ for $i \in [n]$.
    Take any $p = \sum_{i = 1}^n p_i \alpha_i \in \B{\spt{S}}$.
    Then,
    we have $\inpr{p}{x_i} = p_i\inpr{\alpha_i}{x_i} \leq \spt{S}(x_i)$ for each $i \in [n]$,
    in which the first equality follows from \Cref{lem:inpr}.
    On the other hand,
    since $\spt{S}(x) = \sup{\Set{ \inpr{p}{x} }{ p \in S }}$,
    for each $i \in [n]$, there is $q^i = \sum_{j = 1}^n q^i_j \alpha_j \in S$ such that
    \begin{align}
        q^i_i
        \begin{cases}
            = \spt{S}(x_i)/\inpr{\alpha_i}{x_i} (\geq p_i) & \text{if $x_i \in \dom{\spt{S}}$},\\
            \geq p_i & \text{if $x_i \notin \dom{\spt{S}}$}.
        \end{cases}
    \end{align}
    Since $S$ has~\eqref{cond:MP}, there exists $r^\simp \in S$ such that
    $r^\simp \geq_{\simp} q^i$ for each $i \in [n]$;
    this $r^\simp$ also satisfies $r^\simp \geq_{\simp} p$.

    Let $R \coloneqq \Set{ r^\simp - p }{ \text{$\simp:$ simple system} }$.
    Take any $x \in V$ and chamber $C' \in \spCC^\star$ with $x \in \cl(\cone{C'})$,
    where we suppose that $\sk{C'} = \{ x_1', x_2', \dots, x_n' \}$.
    Let $\simp' = \{ \alpha_1', \alpha_2', \dots, \alpha_n' \} \subseteq \RS$ be the simple system corresponding to $C'$ with $\tau_{\simp'}(\alpha_i') = \tau(x_i')$ for $i \in [n]$.
    Then we have $\spt{R}(x) \geq \inpr{r^{\simp'} - p}{x} = \sum_{i = 1}^n \cell{\mu}{x}(x_i) \inpr{r^{\simp'} - p}{x_i} \geq 0$,
    where the last inequality follows from $p \leq_{\simp'} r^{\simp'}$.
    Therefore, we obtain
    \begin{align}
        0 \in \Set{ q \in V^* }{ \inpr{q}{x} \leq \spt{R}(x) \ (\forall x \in V) } = \clconv{R} = \conv{R},
    \end{align}
    where the first equality follows from \Cref{lem:clconv-sublinear}
    and the second follows from the fact that $R$ is a finite set.
    This implies $p \in \conv{\Set{ r^\simp }{ \text{$\simp:$ simple system}}} \subseteq \conv{S}$.

    (Only-if part of (1)).
    Let $P \subseteq V^*$ be an M-convex polyhedron.
    Take any $p, q \in P$ and simple system $\simp = \{ \alpha_1, \alpha_2, \dots, \alpha_n \}$.
    Let $\omega_1^\vee, \omega_2^\vee, \dots, \omega_n^\vee$ be the fundamental coweights with respect to $\alpha_1, \alpha_2, \dots, \alpha_n$ (recall the definition~\eqref{eq:fundamental-coweights}).
    It suffices to find $r \in P$ such that $r \geq_{\simp} p$ and $r \geq_{\simp} q$,
    or equivalently, $\inpr{r}{\omega_i^\vee} \geq \max\{ \inpr{p}{\omega_i^\vee}, \inpr{q}{\omega_i^\vee} \}$ for all $i \in [n]$.

    Suppose that $p = \sum_{i = 1}^n p_i \alpha_i$ and $q = \sum_{i = 1}^n q_i \alpha_i$.
    Take any $j \in [n]$ with $p_j < q_j$ (if there is no such $j$, then $r = p$ is a desired point).
    Then we consider the following LP:
    \begin{align}
    \begin{array}{ll}
	\text{Maximize} & \displaystyle \inpr{p'}{\omega_j^\vee}\\
	\text{subject to} & \displaystyle p' \in P \cap \left( p + \cone{\{\alpha_1, \alpha_2, \dots, \alpha_n\}}\right),
	\end{array}    
    \end{align}
    where we note that $p + \cone{\{\alpha_1, \alpha_2, \dots, \alpha_n\}}$ is the set of points $p'$ with $p \preceq_{\simp} p'$.
    Then the optimal value is at least $q_j$.
    Indeed, suppose to the contrary that the optimal value is smaller than $q_j$.
    Let $\hat{p} = \sum_{i = 1}^n \hat{p}_i \alpha_i$ be an optimal solution.
    Then, $\inpr{\hat{p}}{\omega_j^\vee} = \hat{p}_j < q_j$ by the assumption.
    Since $P$ is an M-convex polyhedron and $\hat{p}, q \in P$, there is $\mu \in \cncoeff{\RS}$ such that
    $\supp{\mu} \subseteq \tcone{P}{\hat{p}} \cap \RS$ and
    $\hat{p} + \sum_{\alpha \in \RS} \mu(\alpha) \alpha = q$.
    In particular, by $\hat{p}_j < q_j$,
    there is $\beta = \sum_{i = 1}^n b_i \alpha_i \in \supp{\mu}$ such that $b_j > 0$,
    which implies that $\beta$ is a positive root.
    Since $\beta \in \tcone{P}{\hat{p}}$, there is $\eps > 0$ such that $\hat{p} + \eps \beta \in P$, which improves the value of the objective function, i.e., $\inpr{\hat{p} + \eps \beta}{\omega_j^\vee} > \inpr{\hat{p}}{\omega_j^\vee}$.
    This contradicts the optimality of $\hat{p}$.

    Take any feasible solution $p'$ of the above LP with $p'_j \geq q_j$
    and update $p \leftarrow p'$.
    By repeating this update (at most $n$ times),
    we finally obtain $r = \sum_{i = 1}^n r_i \alpha_i \in P$ such that $r_i \geq \max\{p_i, q_i\}$ for each $i \in [n]$,
    which implies $r \geq_{\simp} p$ and $r \geq_{\simp} q$.
\end{proof}

\begin{corollary}\label{cor:M-convex-set:MP}
    A nonempty polytope $P \subseteq V^*$ is M-convex
    if and only if,
    for any simple system $\simp$,
    there exists a unique maximal member $p \in P$ with respect to $\leq_\simp$.
\end{corollary}

\subsection{Locally polyhedral L- and M-convex functions}\label{subsec:locally-polyhedral}
In this subsection,
we introduce the L- and M-convexity of locally polyhedral convex functions.
For a locally polyhedral L-convex (resp. M-convex) function $\hat{h}$, one might naturally consider two definitions: The directional derivative is L-sublinear (resp. M-sublinear) at each point in $\dom{\hat{h}}$, or each member $P$ of $\fan(\hat{h})$ is an L-convex (resp. M-convex) polyhedron.
The following proposition ensures that it does not matter which definition we adopt, as they are equivalent.
\begin{proposition}\label{prop:locally-poly-LM}
  \begin{enumerate}[label={\textup{(\arabic*)}}]
        \item For a proper locally polyhedral convex function $\funcdoms{\hat{f}}{V}{\ER}$,
  $\fan(\hat{f})$ consists of L-convex polyhedra
  if and only if $\hat{f}'(x; \cdot)$ forms an L-sublinear function
  for any $x \in \dom{\hat{f}}$.
        \item For a proper locally polyhedral convex function $\funcdoms{\hat{g}}{V^*}{\ER}$,
    $\fan(\hat{g})$ consists of M-convex polyhedra
    if and only if $\hat{g}'(p; \cdot)$ forms an M-sublinear function
    for any $p \in \dom{\hat{g}}$.
    \end{enumerate}
\end{proposition}
\begin{proof}
  We only show the assertion~(1).
  The proof for the assertion~(2) is given by replacing ``the definition'' and ``\Cref{lem:L-convex-tcone}'' with ``\Cref{lem:M-sublinear}'' and ``the definition,''
  respectively, in the following proof.

    (Only-if part).
    Fix an arbitrary $x \in \dom{\hat{f}}$.
    By \Cref{lem:quasi-polyhedral-convex-function}~(2),
    we have $\fan(\hat{f}'(x; \cdot)) = \Set{ \tcone{P}{x} }{ P \in \fan(\hat{f}; x) }$.
    Since every $P \in \fan(\hat{f}; x)$ is an L-convex polyhedron by the assumption,
    its tangent cone $\tcone{P}{x}$ at $x$ is an L-convex cone by \Cref{lem:L-convex-tcone}.
    Since $\fan(\hat{f}'(x; \cdot))$ consists of L-convex cones,
    $\hat{f}'(x; \cdot)$ is L-sublinear by the definition.

    (If part).
    Take any $P \in \fan(\hat{f})$ and $x \in P$.
    Since $P$ is a closed convex set,
    it suffices to see that $\tcone{P}{x}$ is an L-convex cone by \Cref{lem:L-convex-tcone}.
    By \Cref{lem:quasi-polyhedral-convex-function}~(2),
    we have $\tcone{P}{x} \in \fan(\hat{f}'(x; \cdot))$.
    Since
    $\hat{f}'(x; \cdot)$ is L-sublinear by the assumption,
    $\tcone{P}{x}$ forms an L-convex cone by the definition.
\end{proof}

Hence, we say that a proper locally polyhedral convex function $\funcdoms{\hat{f}}{V}{\ER}$ is a \emph{locally polyhedral L-convex function of type $\RS$}
if $\fan(\hat{f})$ consists of L-convex polyhedra of type $\RS$,
or equivalently,
$\hat{f}'(x; \cdot)$ forms an L-sublinear function of type $\RS$
for any $x \in \dom{\hat{f}}$.
In addition, it is said to be \emph{primal-integral} and \emph{dual-integral}
if every member of $\fan(\hat{f})$ is an integral L-convex polyhedron
and $\hat{f}'(x; \cdot)$ forms an integral L-sublinear function
for any $x \in \dom{\hat{f}}$,
respectively.
If a locally polyhedral L-convex function is primal- and dual-integral,
then it is said to be just \emph{integral}.
The M-convex counterparts,
namely, a \emph{locally polyhedral M-convex function $\funcdoms{\hat{g}}{V^*}{\ER}$ of type $\RS$} and its \emph{primal-integral}, \emph{dual-integral}, and \emph{integral} variants,
are defined analogously by replacing $\hat{f}$, $V$, and ``L-convex'' with $\hat{g}$, $V^*$, and ``M-convex'', respectively.

The following theorem (\Cref{thm:conjugate-loc-poly-L-M-convex}) describes the relations among the (primal-integral/dual-integral/integral) locally polyhedral L-/M-convex functions,
particularly, the one-to-one correspondences
\begin{itemize}
    \item between the dual-integral locally polyhedral L-convex (resp. M-convex) functions and the primal-integral locally polyhedral M-convex (resp. L-convex) functions; and
    \item between the integral locally polyhedral L- and M-convex functions
\end{itemize}
via the Fenchel--Legendre conjugates.
The proof is given later.
\begin{theorem}\label{thm:conjugate-loc-poly-L-M-convex}
  \begin{enumerate}[label={\textup{(\arabic*)}}]
    \item If $\hat{f}$ (resp. $\hat{g}$) is a (primal-integral) locally polyhedral L-convex (resp. M-convex) function and $\hat{f}^*$ (resp. $\hat{g}^*$) is a locally polyhedral convex function, then $\hat{f}^*$ (resp. $\hat{g}^*$) is a (dual-integral) locally polyhedral M-convex (resp. L-convex) function.
    \item If $\hat{f}$ (resp. $\hat{g}$) is a dual-integral locally polyhedral L-convex (resp. M-convex) function, then $\hat{f}^*$ (resp. $\hat{g}^*$) is a primal-integral locally polyhedral M-convex (resp. L-convex) function.
    \item If $\hat{f}$ (resp. $\hat{g}$) is an integral locally polyhedral L-convex (resp. M-convex) function, then $\hat{f}^*$ (resp. $\hat{g}^*$) is an integral locally polyhedral M-convex (resp. L-convex) function.
  \end{enumerate}
\end{theorem}

Let $\qLfunc{\R}{\R}$, $\qLfunc{\Z|\R}{\R}$, $\qLfunc{\R}{\Z|\R}$, and $\qLfunc{\Z|\R}{\Z|\R}$ denote the classes of locally polyhedral L-convex functions, and their primal-integral, dual-integral, and integral variants, respectively.
The corresponding classes of M-convex functions are denoted by $\qMfunc{\R}{\R}$, $\qMfunc{\Z|\R}{\R}$, $\qMfunc{\R}{\Z|\R}$, and $\qMfunc{\Z|\R}{\Z|\R}$, respectively.
Furthermore,
the corresponding unhatted symbols (e.g., $\Lfunc{\R}{\R}$, $\Mfunc{\R}{\R}$, etc.) denote their respective subclasses consisting of functions whose conjugates are also locally polyhedral convex.
\Cref{thm:conjugate-loc-poly-L-M-convex}~(2) and~(3)
imply that
the conjugate of
any dual-integral/integral locally polyhedral L- or M-convex function is a locally polyhedral function.
Thus, we obtain
$\qLfunc{\R}{\Z|\R} = \Lfunc{\R}{\Z|\R}$, $\qMfunc{\R}{\Z|\R} = \Mfunc{\R}{\Z|\R}$,
$\qLfunc{\Z|\R}{\Z|\R} = \Lfunc{\Z|\R}{\Z|\R}$, and $\qMfunc{\Z|\R}{\Z|\R} = \Mfunc{\Z|\R}{\Z|\R}$.
The commutative diagram~\eqref{eq:diagram-LM-func} below includes the relationships among these classes.

The following is a sufficient condition for a given convex function to be primal-integral locally polyhedral L- or M-convex function,
which is used in the proof of \Cref{thm:conjugate-loc-poly-L-M-convex}~(2) and several statements in \Cref{sec:DCA:discrete,sec:chara}.
\begin{proposition}\label{prop:loc-poly-LM-convex-sufficient-condition}
    \begin{enumerate}[label={\textup{(\arabic*)}}]
        \item Let $\funcdoms{\hat{f}}{V}{\ER}$ be a proper convex function.
    If there is a subset $\fan' \subseteq \fan(\hat{f})$ such that
    every member of $\fan'$ is an integral L-convex polyhedron
    and
    $\dom{\hat{f}} = \bigcup_{P \in \fan'} P$,
    then $\hat{f}$ is a primal-integral locally polyhedral L-convex function.
        \item Let $\funcdoms{\hat{g}}{V^*}{\ER}$ be a proper convex function.
    If there is a subset $\fan' \subseteq \fan(\hat{g})$ such that
    every member of $\fan'$ is an integral M-convex polyhedron
    and
    $\dom{\hat{g}} = \bigcup_{P \in \fan'} P$,
    then $\hat{g}$ is a primal-integral locally polyhedral M-convex function.
    \end{enumerate}
\end{proposition}
\begin{proof}
    We only show the assertion~(1).
  The proof for the assertion~(2) is basically the same as that for the assertion~(1).

  We first observe that every member of $\fan(f)$ is an integral L-convex polyhedron. 
  Indeed,
  by \Cref{lem:argmin-support}~(3),
  for any $S \in \fan(\hat{f})$,
  there is $P \in \fan'$ such that $S$ is a face of $P$.
  Since $P$ is an integral L-convex polyhedron by the assumption,
  so is $S$ by \Cref{lem:face:L-poly}.

By the above observation, it suffices to show that $\hat{f}$ is a locally polyhedral convex function.
To this end, we see the following claim.
    Here, for $S \subseteq V$,
    let $\fan'(S)$ denote the set of polyhedra in $\fan'$ that intersect $S$.
If $S$ is a singleton $\{x\}$, we simply write $\fan'(x)$ instead of $\fan'(\{x\})$.
  \begin{claim}\label{cl:finite}
    \begin{enumerate}[label={\textup{(\arabic*)}}]
        \item For any compact set $C \subseteq V$, $\card{\fan'(C)}$ is finite.
        \item For any polytope $P \subseteq V$, the intersection $P \cap \dom{\hat{f}}$ is also a polytope.
    \end{enumerate}
\end{claim}
\begin{proof}[Proof of \Cref{cl:finite}]
  (1).
  Fix an arbitrary $x \in \dom{\hat{f}}$.
  Since every $P \in \fan'(x)$ is an (integral) L-convex polyhedron,
  we have $\tcone{P}{x} = (\cone{\RS'})^*$ for some $\RS' \subseteq \RS$
  by \Cref{lem:L-convex-tcone}.
  Furthermore, $\tcone{P}{x} \neq \tcone{P'}{x}$ holds for distinct $P, P' \in \fan'(x)$.
  Indeed, if $\tcone{P}{x} = \tcone{P'}{x}$ holds for $P, P' \in \fan'(x)$ then $(\relint{P}) \cap (\relint{P'}) \neq \emptyset$,
  which implies $P = P'$ by \Cref{lem:argmin-support}~(2).
  Thus, $\card{\fan'(x)}$ is bounded by the number of subsets of $\RS$, which is finite.

  Let $E_x \coloneqq \bigcup_{P \in \fan' \setminus \fan'(x)} P$.
  Since $E_x$ is closed by \Cref{lem:union-closed:L} and $x \notin E_x$,
  we can take a sufficiently small $\eps_x > 0$ so that
  $B(x; \eps_x) \cap E_x = \emptyset$,
  where $B(x; \eps_x)$ denotes the open ball centered at $x$ with radius $\eps_x$.
  For this $\eps_x$, we have $\fan'(x) = \fan'(B(x; \eps_x))$.
  In particular, since $\dom{\hat{f}} = \bigcup_{P \in \fan'} P$,
  we have $B(x; \eps_x) \cap \dom{\hat{f}} \subseteq \bigcup_{P \in \fan'(B(x; \eps_x))} P = \bigcup_{P \in \fan'(x)} P$.

    Take any compact set $C \subseteq V$.
    Since $\dom{\hat{f}} = \bigcup_{P \in \fan'} P$ is closed by \Cref{lem:union-closed:L},
    the intersection $C \cap \dom{\hat{f}}$ is compact.
    Hence, there are finitely many points $x_1, x_2, \dots, x_k$ in $C \cap \dom{\hat{f}}$ such that $C \cap \dom{\hat{f}} \subseteq \bigcup_{i = 1}^k B(x_i; \eps_{x_i})$.
    Then, we have $\fan'(C) = \bigcup_{i = 1}^k \fan'(B(x_i; \eps_{x_i})) = \bigcup_{i = 1}^k \fan'(x_i)$,
    which is a finite set.

    (2).
    By the assertion~(1), $\fan'(P)$ is a finite set;
    suppose that $\fan'(P) = \{ P_1, P_2, \dots, P_\ell \}$.
    For each $P_i \in \fan'(P)$,
    since $P \cap P_i$ is a polytope,
    there is a finite set $F_i \subseteq V$ such that $P \cap P_i = \conv{F_i}$.
    Thus, $\bigcup_{i = 1}^\ell F_i \subseteq P \cap \dom{\hat{f}} = \bigcup_{i = 1}^\ell (P \cap P_i) \subseteq \conv\left(\bigcup_{i = 1}^\ell F_i\right)$ holds.
    Since $P \cap \dom{\hat{f}}$ is a convex set,
    we obtain $P \cap \dom{\hat{f}} = \conv\left(\bigcup_{i = 1}^\ell F_i\right)$.
    Thus,
    $P \cap \dom{\hat{f}}$ is a polytope by the finiteness of $\bigcup_{i = 1}^\ell F_i$.
\end{proof}

We are ready to show that $\hat{f}$ is a locally polyhedral convex function.
Take any polytope $P \subseteq V$.
By \Cref{cl:finite}~(1),
$\fan'(P) = \{P_1, P_2, \dots, P_\ell\}$ is a finite set.
For each $i \in [\ell]$,
we pick one $p_i \in V^*$ such that $P_i = \argmin(\hat{f}-p_i)$.
Then,
we have
\begin{align}
    (\hat{f}|_P)(x) =
    \begin{cases}
        \max{\Set*{ \inpr{p_i}{x} + \min(\hat{f} - p_i) }{ i \in [\ell] }} & \text{if $x \in P \cap \dom{\hat{f}}$},\\
        +\infty & \text{otherwise}.
    \end{cases}
\end{align}
Since $P \cap \dom{\hat{f}}$ is a polytope by \Cref{cl:finite}~(2),
the restriction of $\hat{f}$ to $P$ is a polyhedral convex function.
Thus, $\hat{f}$ is a locally polyhedral convex function.
\end{proof}

\begin{proof}[Proof of \Cref{thm:conjugate-loc-poly-L-M-convex}]
    We only show the assertions for a (primal-integral/dual-integral/integral) locally polyhedral L-convex function $\hat{f}$.
    The proofs for a (primal-integral/dual-integral/integral) locally polyhedral M-convex function $\hat{g}$ are similar.

  (1).
  Since $\hat{f}^*$ is a locally polyhedral convex function by the assumption,
  the directional derivative $(\hat{f}^*)'(p; \cdot)$ is a polyhedral sublinear function by \Cref{lem:quasi-polyhedral-convex-function}~(2),
  particularly, a closed convex function.
  Thus,
  $(\hat{f}^*)'(p; \cdot)$ is the support function of $\partial \hat{f}^* (p)$; recall~\eqref{eq:derivative-support}.
  Since $\partial \hat{f}^* (p) = \argmin(\hat{f} - p)$ by \Cref{lem:argmin-support}~(5)
  and $\argmin(\hat{f} - p)$ is an (integral) L-convex polyhedron by the assumption,
  $(\hat{f}^*)'(p; \cdot)$ is an (integral) M-sublinear function by \Cref{thm:conjugate-Lpoly-Msub}.

  (2).
  Clearly, $\hat{f}^*$ is a proper convex function.
  By \Cref{lem:argmin-support}~(5),
  we have $\fan(\hat{f}^*) = \Set{\partial \hat{f}(x) }{ x \in \dom{\hat{f}},\ \partial \hat{f}(x) \neq \emptyset }$.
  For any $x \in \dom{\hat{f}}$,
  since $\hat{f}'(x; \cdot)$ is an integral L-sublinear function,
  the set $\partial \hat{f}(x) = \B{\hat{f}'(x; \cdot)}$ is an integral M-convex polyhedron by \Cref{thm:conjugate-Mpoly-Lsub}.
  Thus,
  $\fan(\hat{f}^*)$ consists of integral M-convex polyhedra,
  which implies that
  $\geo{\fan(\hat{f}^*)}$ is a closed set by \Cref{lem:union-closed:M}.
  Moreover, the closedness implies $\geo{\fan(\hat{f}^*)} = \dom{\hat{f}^*}$ by \Cref{lem:argmin-support}~(1).
  Therefore,
  by \Cref{prop:loc-poly-LM-convex-sufficient-condition}~(2),
  $\hat{f}^*$ is a primal-integral locally polyhedral M-convex function.

  (3) immediately follows from the assertions~(1) and (2).
\end{proof}

As one may expect,
the effective domain of a primal-integral locally polyhedral L-/M-convex function forms an (integral) L-/M-convex polyhedron.
\begin{proposition}\label{prop:dom:LM-convex}
  The effective domain of a primal-integral locally polyhedral L-convex function (resp. M-convex function) is an integral L-convex polyhedron (resp. M-convex polyhedron).
\end{proposition}
\begin{proof}
  We first show the assertion for a primal-integral locally polyhedral L-convex function $\hat{f}$.
  Since each $P \in \fan(\hat{f})$ is an integral L-convex polyhedron,
    it is the geometric realization of a convex subcomplex $\CC_P$ of $\euCC$ by \Cref{lem:chara-integral-L-convex-poly}.
    Then, $\CC' \coloneqq \bigcup_{P \in \fan(\hat{f})} \CC_P$ is also a convex subcomplex of $\euCC$ by \Cref{lem:convex_subcomplex}, since $\dom{\hat{f}} = \bigcup_{P \in \fan(\hat{f})} \geo{\CC_P} = \geo{\CC'}$ is convex.
    Thus, $\dom{\hat{f}}$ is the geometric realization of a convex subcomplex of $\euCC$, which implies that $\dom{\hat{f}}$ is an integral L-convex polyhedron by \Cref{lem:chara-integral-L-convex-poly}.

  We next see the assertion for a primal-integral locally polyhedral M-convex function $\hat{g}$.
  It follows from $\dom{\hat{g}} = \geo{\fan(\hat{g})}$ and \Cref{lem:union-closed:M}
  that $\dom{\hat{g}}$ is a closed convex set.
  For any $p \in \dom{\hat{g}}$,
  we have $\tcone{\dom{\hat{g}}}{p} = \dom(\hat{g}'(p; \cdot))$,
  which forms an M-convex cone by the M-sublinearity of $\hat{g}'(p; \cdot)$.
  Thus, by the definition,
  the effective domain $\dom{\hat{g}}$ is an M-convex polyhedron.

  Hence, it suffices to show $\dom{\hat{g}} = \conv(\dom{\hat{g}} \cap \Mdom)$ for the integrality of $\dom{\hat{g}}$.
  The inclusion $\dom{\hat{g}} \supseteq \conv(\dom{\hat{g}} \cap \Mdom)$ clearly holds.
  We can obtain the reverse inclusion as
  \begin{align}
    \conv(\dom{\hat{g}} \cap \Mdom) &= \conv\left( \left(\bigcup_{P \in \fan(\hat{g})} P\right) \cap \Mdom \right) = \conv\left( \bigcup_{P \in \fan(\hat{g})} (P \cap \Mdom) \right)\\
    &\supseteq \bigcup_{P \in \fan(\hat{g})} \conv(P \cap \Mdom)
    = \bigcup_{P \in \fan(\hat{g})} P = \dom{\hat{g}},
  \end{align}
  where the penultimate equality follows from the fact that $P \in \fan(\hat{g})$ is an integral M-convex polyhedron.
\end{proof}

\section{Discrete convexity over combinatorial objects}\label{sec:DCA:discrete}
Building upon the geometric formulations in \Cref{sec:DCA:continuous}, 
this section formally defines discrete convexity over combinatorial objects. 
\Cref{subsec:L-convex,subsec:M-convex} introduce \emph{L- and M-convex sets and functions} over the vertex sets of Euclidean Coxeter complexes and root lattices, respectively, and their integral ones.
\Cref{subsec:conjugacy} establishes the \emph{discrete Fenchel--Legendre conjugacy} between integral L-convex functions and integral M-convex functions. 

\subsection{Definitions of L-convex sets and functions}\label{subsec:L-convex}

In this subsection, we assume that $\RS$ is an irreducible root system.
Recall that $P \subseteq V$ is an integral L-convex polyhedron of type $\RS$ if and only if
$P = \geo{\CC'}$ for some nonempty convex subcomplex $\CC'$ of $\euCC$ (\Cref{lem:chara-integral-L-convex-poly}).

A nonempty subset $D \subseteq \Ldom$ is said to be \emph{L-convex of type $\RS$} 
if it is the vertex set of some nonempty convex subcomplex $\CC'$ of $\euCC$,
or equivalently,
if there exists an integral L-convex polyhedron $P$ such that $D = P \cap \Ldom$.
This $P$ is uniquely determined as $P = \ext{D} = \conv{D} = \clconv{D}$,
since $P = \geo{\CC'} = \ext{D} \subseteq \conv{D} \subseteq \clconv{D} \subseteq P$,
where the last inclusion follows from the fact that $\clconv{D}$ is the inclusion-wise minimum closed convex set that contains $D$.
For later reference,
we summarize this fact.
\begin{corollary}\label{cor:ext=conv=clconv}
For an L-convex set $D \subseteq \Ldom$,
we have $\ext{D} = \conv{D} = \clconv{D}$,
which forms an integral L-convex polyhedron.
\end{corollary}
By combining the above arguments with those in~\Cref{subsubsec:tight-distance}, we obtain the following commutative diagram, describing relations among $\Lpoly{\R}$, $\Lpoly{\Z|\R}$, $\Msub{\R}{\R}$, $\Msub{\R}{\Z|\R}$, $\vdist{\R}$, $\tdist{\R}$, and $\tdist{\Z}$ with the family $\Lpoly{\Z}$ of L-convex sets.
\begin{align}\label{eq:diagram-Lset}
\begin{tikzcd}[row sep=3em, column sep=6em, ampersand replacement=\&, crossing over clearance=0.8ex]
\& \& \& |[alias=vdistR]| \vdist{\R} \& \\
\& \& \& |[alias=tdistR]| \tdist{\R} \& \\
\& |[alias=LpolyR]| \Lpoly{\R} \& \& \& |[alias=MsubRR]| \Msub{\R}{\R} \\
|[alias=Lset]| \Lpoly{\Z} \& \& \& |[alias=tdistZ]| \tdist{\Z}\& \\
\& |[alias=LpolyZR]| \Lpoly{\Z|\R} \& \& \& |[alias=MsubRZR]| \Msub{\R}{\Z|\R}
\arrow[from=tdistR, to=vdistR, hook]
\arrow[from=tdistZ, to=tdistR, hook]
\arrow[from=MsubRZR, to=MsubRR, hook]
\arrow[from=vdistR, to=LpolyR, "\D{\cdot}"{sloped, pos=0.5, above}]
\arrow[from=vdistR, to=MsubRR, "\cone(\cdot)"{sloped, pos=0.5, above}] 
\arrow[from=LpolyR, to=MsubRR, shift left=0.4ex, "\spt{\cdot}"{pos=0.5, above}, crossing over]
\arrow[from=MsubRR, to=LpolyR, shift left=0.4ex, "\D{\cdot}"{pos=0.5, below}, crossing over]
\arrow[from=tdistR, to=LpolyR, shift right=0.4ex, "\D{\cdot}"{sloped, pos=0.5, above}]
\arrow[from=LpolyR, to=tdistR, shift right=0.4ex, "(\spt{\cdot})|_{\RS}"{sloped, pos=0.5, below}]
\arrow[from=tdistR, to=MsubRR, shift left=0.4ex, "\cone(\cdot)"{sloped, pos=0.5, above}]
\arrow[from=MsubRR, to=tdistR, shift left=0.4ex, "(\cdot)|_{\RS}"{sloped, pos=0.5, below}]
\arrow[from=tdistZ, to=LpolyZR, shift right=0.4ex, "\D{\cdot}"{sloped, pos=0.5, above}]
\arrow[from=LpolyZR, to=tdistZ, shift right=0.4ex, "(\spt{\cdot})|_{\RS}"{sloped, pos=0.5, below}]
\arrow[from=tdistZ, to=MsubRZR, shift left=0.4ex, "\cone(\cdot)"{sloped, pos=0.5, above}]
\arrow[from=MsubRZR, to=tdistZ, shift left=0.4ex, "(\cdot)|_{\RS}"{sloped, pos=0.5, below}]
\arrow[from=LpolyZR, to=MsubRZR, shift left=0.4ex, "\spt{\cdot}"{pos=0.5, above}]
\arrow[from=MsubRZR, to=LpolyZR, shift left=0.4ex, "\D{\cdot}"{pos=0.5, below}]
\arrow[from=Lset, to=LpolyZR, shift right=0.4ex, "\ext{(\cdot)} = \conv(\cdot) = \clconv(\cdot)"{sloped, pos=0.5, below}]
\arrow[from=LpolyZR, to=Lset, shift right=0.4ex, "(\cdot) \cap \Ldom"{sloped, pos=0.5, above}]
\arrow[from=Lset, to=tdistZ, shift left=0.4ex, "(\spt{\cdot})|_{\RS}"{sloped, pos=0.5, above}]
\arrow[from=tdistZ, to=Lset, shift left=0.4ex, "\D{\cdot} \cap \Ldom"{sloped, pos=0.5, below}]
\arrow[from=LpolyZR, to=LpolyR, hook, crossing over]
\end{tikzcd}
\end{align}

A function $\funcdoms{f}{\Ldom}{\ER}$ is said to be \emph{L-convex of type $\RS$} if there is a primal-integral locally polyhedral L-convex function $\funcdoms{\hat{f}}{V}{\ER}$ such that $f$ is the restriction of $\hat{f}$ to $\Ldom$.
An L-convex function $\funcdoms{f}{\Ldom}{\ER}$ is said to be \emph{integral}
if it is the restriction of an integral locally polyhedral L-convex function.
\begin{proposition}\label{prop:L-conv:unique}
    For an L-convex function $\funcdoms{f}{\Ldom}{\ER}$,
    we have $\ext{f} = \conv{f} = \clconv{f}$,
    which forms the unique primal-integral locally polyhedral L-convex function whose restriction to $\Ldom$ is $f$.
\end{proposition}
\begin{proof}
    Let $\funcdoms{f}{\Ldom}{\ER}$ be an L-convex function
    and suppose that $\funcdoms{\hat{f}}{V}{\ER}$ is a primal-integral locally polyhedral L-convex function satisfying $f = \hat{f}|_{\Ldom}$.
    Since $\hat{f}$ is a closed convex function that minorizes $f$,
    we have $\hat{f} \leq \clconv{f} \leq \conv{f} \leq \ext{f}$,
    where the last inequality follows from the definitions of $\conv{f}$ and $\ext{f}$.
    Thus, it suffices to show $\ext{f} = \hat{f}$.

    We first see that $\dom{\hat{f}} = \dom{\ext{f}}$.
    Since $\dom{\hat{f}}$ is an integral L-convex polyhedron by \Cref{prop:dom:LM-convex},
    there is a nonempty convex subcomplex $\CC_f$ of $\euCC$ with $\dom{\hat{f}} = \geo{\CC_f}$ by \Cref{lem:chara-integral-L-convex-poly}.
    Since $\dom{f} = \dom{\hat{f}} \cap \Ldom$,
    we obtain $\CC_f = \euCC[\dom{f}]$.
    Thus, we have $\dom{\hat{f}} = \geo{\euCC[\dom{f}]} = \dom{\ext{f}}$.

    Take any $x \in \dom{\hat{f}} = \dom{\ext{f}}$;
    our aim is to show that $\hat{f}(x) = \ext{f}(x)$.
    Since $\dom{\hat{f}} = \geo{\fan(\hat{f})}$,
    there is an integral L-convex polyhedron $P \in \fan(\hat{f})$ with $x \in P$,
    which is the geometric realization of a certain nonempty convex subcomplex $\CC_P$ of $\euCC$ by \Cref{lem:chara-integral-L-convex-poly}.
    Since $x \in \geo{\CC_P}$,
    we have $\cell{A}{x} \in \CC_P$, implying that $\sk{\cell{A}{x}} \subseteq \sk{\CC_P} \subseteq \dom{f}$.
    Since $\hat{f}$ is affine on $P = \geo{\CC_P}$,
    we have
    \begin{align}
        \hat{f}(x) = \sum_{v \in \sk{\cell{A}{x}}} \cell{\lambda}{x}(v) \hat{f}(v) = \sum_{v \in \sk{\cell{A}{x}}} \cell{\lambda}{x}(v) f(v) = \ext{f}(x).
    \end{align}
    This completes the proof.
\end{proof}

For a function $\funcdoms{f}{\Ldom}{\ER}$
and $x \in \dom{f}$,
we define 
$\funcdoms{\rho_{f, x}}{(\opnbor{x} - x)}{\ER}$ 
by
\begin{align}
  \rho_{f, x}(d) \coloneqq f(x+d) - f(x) \qquad (d \in \opnbor{x} - x).
\end{align}
Intuitively, $\rho_{f, x}$ plays a role as the directional derivative of $f$ at $x$.
\begin{proposition}\label{prop:directional:derivative:L}
  Suppose that $\funcdoms{f}{\Ldom}{\ER}$ is an L-convex function.
  Then, $\rho_{f, x}$ is the restriction of $\hat{f}'(x; \cdot)$ to $\opnbor{x} - x$ for any $x \in \dom{f}$,
  where $\hat{f}$ denotes the primal-integral locally polyhedral L-convex function whose restriction to $\Ldom$ is $f$.
  In particular, $\rho_{f,x}$ is a spherical submodular function for any $x \in \dom{f}$.
\end{proposition}
\begin{proof}
  We only show the former assertion; the latter assertion immediately follows from the former.

  Let $\funcdoms{\hat{f}}{V}{\ER}$ be the (unique) primal-integral locally polyhedral L-convex function such that $f = \hat{f}|_{\Ldom}$.
  By \Cref{prop:L-conv:unique},
  we have $\hat{f} = \ext{f}$.
  Take an arbitrary $x \in \dom{f} \subseteq \dom{\hat{f}}$.
  By the convexity of $\ext{f}$ and the definitions of (positively homogeneous) Lov\'{a}sz extension,
  $\ext{f}'(x; \cdot)$ coincides with the positively homogeneous Lov\'{a}sz extension of $\rho_{f, x}$.
  Hence,
  we obtain $\hat{f}'(x; \cdot) = \ext{f}'(x; \cdot) = \phext{\rho_{f, x}}$,
  which implies that $\rho_{f,x}$ is the restriction of $\hat{f}'(x; \cdot)$ to $\opnbor{x} - x$.
\end{proof}

\begin{theorem}\label{thm:integral:L-func}
    Let $\funcdoms{f}{\Ldom}{\ER}$ be an L-convex function.
    \begin{enumerate}[label={\textup{(\arabic*)}}]
        \item $f$ is integral if and only if $\rho_{f, x}$ is $\RS$-integral for all $x \in \dom{f}$.
        \item If there is $r \in \R$ such that $f(x) - r \in \Z/a(x)$ for all $x \in \dom{f}$, then $f$ is integral.
        In addition, the converse also holds if $\RS$ is of type A or $\dom{f}$ is full-dimensional, i.e., there is a chamber $C$ of $\euCC$ with $\sk{C} \subseteq \dom{f}$.
    \end{enumerate}
\end{theorem}
\begin{proof}
  (1).
  Let $\hat{f}$ be the primal-integral locally polyhedral L-convex function whose restriction to $\Ldom$ is $f$.
  By \Cref{prop:directional:derivative:L},
  for any $x \in \dom{f}$,
  the function $\rho_{f,x}$ is the restriction of $\hat{f}'(x; \cdot)$ to $\opnbor{x} - x$.

  If $f$ is an integral L-convex function, i.e., $\hat{f}$ is an integral locally polyhedral L-convex function,
  then $\hat{f}'(x; \cdot)$ is an integral L-sublinear function for $x \in \dom{f}$.
  Thus, $\rho_{f,x}$ is $\RS$-integral,
  which implies the only-if part.

  We then show the if part.
  It suffices to prove that ${\hat{f}}^*$ is a primal-integral locally polyhedral M-convex function.
  Indeed,
  since $\hat{f}$ is a primal-integral locally polyhedral L-convex function,
  by \Cref{thm:conjugate-loc-poly-L-M-convex}~(1),
  its conjugate $\hat{f}^*$ forms a dual-integral locally polyhedral M-convex function provided that it is a locally polyhedral convex function.
  Thus, if $\hat{f}^*$ is a primal-integral locally polyhedral M-convex function,
  then it is an integral locally polyhedral M-convex function,
  implying that $\hat{f}^{**} = \hat{f}$ is an integral locally polyhedral L-convex function by \Cref{thm:conjugate-loc-poly-L-M-convex}~(3).
  Thus, we can conclude that $f$ is an integral L-convex function.

  Let $\Pi \coloneqq \Set{ \B{\rho_{f,x}} }{ x \in \dom{f} }$.
  Since $\rho_{f,x}$ is $\RS$-integral for any $x \in \dom{f}$,
  $\Pi$ consists of integral M-convex polyhedra of type $\RS$.
  Furthermore, since $\fan(\hat{f}^*) = \Set{ \B{\hat{f}'(y; \cdot)} }{ y \in \dom{\hat{f}} }$ by \Cref{lem:argmin-support}~(5)
  and $\B{\rho_{f,x}} = \B{\hat{f}'(x; \cdot)}$,
  we have $\Pi \subseteq \fan(\hat{f}^*)$.

  Take any $x \in \dom{\hat{f}}$
  and let $P$ be the member of $\fan(\hat{f})$ such that $x \in \relint{P}$.
  We note that, since $\fan(\hat{f})$ forms a polyhedral complex and $\dom{\hat{f}} = \geo{\fan(\hat{f})}$,
  such $P$ uniquely exists.
  Since $P$ is an integral L-convex polyhedron, we have $P \cap \Ldom \neq \emptyset$.
  For any $y \in P \cap \Ldom \subseteq \dom{f}$,
  we have $\B{\hat{f}'(x; \cdot)} \subseteq \B{\hat{f}'(y; \cdot)} = \B{\rho_{f,y}}$
  by \Cref{lem:argmin-support}~(4).
  Thus, we obtain that $\bigcup_{P \in \Pi} P = \geo{\fan(\hat{f}^*)}$.
  By \Cref{prop:loc-poly-LM-convex-sufficient-condition}~(2),
  $\hat{f}^*$ is a primal-integral locally polyhedral M-convex function.

  (2) immediately follows from the assertion~(1) and \Cref{cor:integral-subm}.
\end{proof}

We can see that local optimality guarantees global optimality.
\begin{theorem}\label{thm:minimality:L}
  For an L-convex function $\funcdoms{f}{\sk{\euCC}}{\ER}$
  and $x \in \dom{f}$,
  we have that
  $x \in \argmin{f}$
  if and only if
  $f(x) \leq f(y)$ for all $y \in \opnbor{x}$.
\end{theorem}
\begin{proof}
    The only-if part is clear.
    We show the if part.
    Suppose that $f(x) \leq f(y)$ for all $y \in \opnbor{x}$,
    or equivalently, $\rho_{f,x} \geq 0$.
    Let $\hat{f}$ be the primal-integral locally polyhedral L-convex function satisfying $f = \hat{f}|_{\Ldom}$.
    By \Cref{prop:directional:derivative:L},
    we have $\hat{f}'(x; \cdot) = \cone{\rho_{f,x}}$,
    which implies that $\hat{f}'(x; \cdot) \geq 0$.
    Thus, for any $y \in \dom{f}$,
    we have
    \begin{align}
        f(y) = \hat{f}(y) \geq \hat{f}(x) + \hat{f}'(x; y -x) \geq \hat{f}(x) = f(x),
    \end{align}
    where the first inequality follows from a fundamental fact in convex analysis; see e.g.,~\cite[Proposition~4.47]{Mordukhovich2023-hw}.
\end{proof}

\subsection{Definitions of M-convex sets and functions}\label{subsec:M-convex}
Recall that an integral M-convex polyhedron $P$ is an M-convex polyhedron having the $\Mdom$-integrality, i.e., $P = \conv(P \cap \Mdom)$; see \Cref{subsec:M-poly-L-sub}.

A nonempty subset $B \subseteq \Mdom$ is said to be \emph{M-convex} if
$B$ is the set of lattice points in some integral M-convex polyhedron $P$,
i.e.,
$B = P \cap \Mdom$ for some integral M-convex polyhedron $P$.
This is equivalent to the condition that
$\conv{B}$ is an integral M-convex polyhedron and $B = (\conv{B}) \cap \Mdom$.

For an M-convex set $B \subseteq \Mdom$,
we have $\clconv{B} = \conv{B}$,
since $\conv{B}$ forms a polyhedron.
Thus, with the consideration of \Cref{subsubsec:spherical-submodular}, 
we obtain the following commutative diagram describing the relations among
$\Mpoly{\R}$, $\Mpoly{\Z}$, $\Lsub{\R}{\R}$, $\Lsub{\R}{\Z|\R}$, $\subm{\CC^\circ}{\R}$, $\subm{\CC^\circ}{\Z}$
with the family $\Mpoly{\Z}$ of M-convex sets.
\begin{align}\label{eq:diagram-Mset}
\begin{tikzcd}[row sep=3em, column sep=6em, ampersand replacement=\&, crossing over clearance=0.8ex]
\& |[alias=submCR]| \subm{\CC^\circ}{\R} \& \& \& \\
|[alias=LsubRR]| \Lsub{\R}{\R} \& \& \& |[alias=MpolyR]| \Mpoly{\R} \& \\
\& |[alias=submCZ]| \subm{\CC^\circ}{\Z} \& \& \& |[alias=Mset]| \Mpoly{\Z} \\
|[alias=LsubRZR]| \Lsub{\R}{\Z|\R} \& \& \& |[alias=MpolyZR]| \Mpoly{\Z|\R} \&
\arrow[from=submCZ, to=submCR, hook]
\arrow[from=LsubRZR, to=LsubRR, hook]
\arrow[from=LsubRR, to=MpolyR, shift left=0.4ex, "\B{\cdot}"{pos=0.5, above}, crossing over]
\arrow[from=MpolyR, to=LsubRR, shift left=0.4ex, "\spt{\cdot}"{pos=0.5, below}, crossing over]
\arrow[from=submCR, to=LsubRR, shift right=0.4ex, "\cone(\cdot) = \phext{(\cdot)}"{sloped, pos=0.5, above}]
\arrow[from=LsubRR, to=submCR, shift right=0.4ex, "(\cdot)|_{\CC^\circ}"{sloped, pos=0.5, below}]
\arrow[from=submCR, to=MpolyR, shift left=0.4ex, "\B{\cdot}"{sloped, pos=0.5, above}]
\arrow[from=MpolyR, to=submCR, shift left=0.4ex, "(\spt{\cdot})|_{\CC^\circ}"{sloped, pos=0.5, below}]
\arrow[from=submCZ, to=LsubRZR, shift right=0.4ex, "\cone(\cdot) = \phext{(\cdot)}"{sloped, pos=0.5, above}]
\arrow[from=LsubRZR, to=submCZ, shift right=0.4ex, "(\cdot)|_{\CC^\circ}"{sloped, pos=0.5, below}]
\arrow[from=submCZ, to=MpolyZR, shift left=0.4ex, "\B{\cdot}"{sloped, pos=0.5, above}]
\arrow[from=MpolyZR, to=submCZ, shift left=0.4ex, "(\spt{\cdot})|_{\CC^\circ}"{sloped, pos=0.5, below}]
\arrow[from=LsubRZR, to=MpolyZR, shift left=0.4ex, "\B{\cdot}"{pos=0.5, above}]
\arrow[from=MpolyZR, to=LsubRZR, shift left=0.4ex, "\spt{\cdot}"{pos=0.5, below}]
\arrow[from=Mset, to=MpolyZR, shift left=0.4ex, "\conv(\cdot) = \clconv(\cdot)"{sloped, pos=0.5, below}]
\arrow[from=MpolyZR, to=Mset, shift left=0.4ex, "(\cdot) \cap \Mdom"{sloped, pos=0.5, above}]
\arrow[from=submCZ, to=Mset, shift left=0.4ex, "\B{\cdot} \cap \Mdom"{pos=0.5, above}]
\arrow[from=Mset, to=submCZ, shift left=0.4ex, "(\spt{\cdot})|_{\CC^\circ}"{pos=0.5, below}]
\arrow[from=MpolyZR, to=MpolyR, hook, crossing over]
\end{tikzcd}
\end{align}

A function $\funcdoms{g}{\Mdom}{\ER}$ is said to be \emph{M-convex} if there is a primal-integral locally polyhedral M-convex function $\funcdoms{\hat{g}}{V^*}{\ER}$ such that $g$ is the restriction of $\hat{g}$ to $\Mdom$.
An M-convex function $\funcdoms{g}{\Mdom}{\ER}$ is said to be \emph{integral}
if it is the restriction of an integral locally polyhedral M-convex function.

\begin{proposition}\label{prop:M-conv:unique}
    For an M-convex function $\funcdoms{g}{\Mdom}{\ER}$,
    we have $\conv{g} = \clconv{g}$,
    which forms the unique primal-integral locally polyhedral M-convex function whose restriction to $\Mdom$ is $g$.
\end{proposition}
\begin{proof}
    The proof strategy is similar to that of \Cref{prop:L-conv:unique}.
    Let $\funcdoms{g}{\Mdom}{\ER}$ be an M-convex function
    and suppose that $\funcdoms{\hat{g}}{V^*}{\ER}$ is a primal-integral locally polyhedral M-convex function satisfying $g = \hat{g}|_{\Mdom}$.
    Since $\hat{g}$ is a closed convex function that minorizes $g$,
    we have $\hat{g} \leq \clconv{g} \leq \conv{g}$.
    Thus, it suffices to show $\conv{g} \leq \hat{g}$.

    We can easily obtain $\dom(\conv{g}) = \dom{\hat{g}}$ by
    \begin{align}
        \dom(\conv{g}) = \conv(\dom{g}) = \conv(\dom{\hat{g}} \cap \Mdom) = \dom{\hat{g}},
    \end{align}
    where the first and last equalities follow from \Cref{lem:conv=clconv} and \Cref{prop:dom:LM-convex},
    respectively.

    Take any $p \in \dom(\conv{g}) = \dom{\hat{g}}$;
    our aim is to show that $(\conv{g})(p) \leq \hat{g}(p)$.
    Since $\dom{\hat{g}} = \geo{\fan(\hat{g})}$,
    there is an integral M-convex polyhedron $P \in \fan(\hat{g})$ with $p \in P$.
    Let $B \coloneqq P \cap \Mdom \subseteq \dom{g}$,
    which is an M-convex set satisfying $P = \conv{B}$.
    Thus, there is $\lambda \in \cvcoeff{B}$ such that $p = \sum_{q \in B} \lambda(q) q$.
    Since $\hat{g}$ is affine on $P$,
    we obtain
    \begin{align}
        \hat{g}(p) = \sum_{q \in B} \lambda(q) \hat{g}(q) = \sum_{q \in B} \lambda(q) g(q) \geq (\conv{g})(q),
    \end{align}
    where the inequality follows from the definition~\eqref{eq:func:conv} of $\conv{g}$.
    This completes the proof.
\end{proof}

The following immediately follows by applying the proof of \Cref{thm:integral:L-func}~(1) similarly.
\begin{theorem}
    For an M-convex function $\funcdoms{g}{\Mdom}{\ER}$,
    it is integral if and only if $\gamma_{g, p}$ is integral for all $p \in \dom{g}$.
\end{theorem}

For $B \subseteq \Mdom$ and $p \in B$,
we define
\begin{align}
    \tRS{p}{B} \coloneqq \Set{ \alpha \in \RS }{ p + \alpha \in B },
\end{align}
which plays a role as the tangent cone of $B$ at $p$ in the following sense.
\begin{proposition}\label{prop:tangent-cone}
    For an M-convex set $B \subseteq \Mdom$
    and $p \in B$,
    we have $\tcone{\conv{B}}{p} = \cone(\tRS{p}{B})$.
\end{proposition}
The above proposition (\Cref{prop:tangent-cone}) is the main bridge between the continuous and discrete sides of M-convexity,
which shows that the tangent cone of the convex hull of an M-convex set
is completely generated by feasible increments along single roots.
This identity is used repeatedly in the local optimality theorem and
in the intrinsic characterizations of M-convex sets and functions in \Cref{sec:chara}.
The proof of \Cref{prop:tangent-cone} requires further properties of spherical submodular functions, which are of independent interest,
with type-dependent arguments,
and hence we defer it to \Cref{subsec:proof}.
We further note that
this is the main point at which
the classical-type assumption enters the discrete M-convex theory.
An analogue of this identity for the exceptional root systems would
extend the subsequent M-convex results to those types.

For a function $\funcdoms{g}{\Mdom}{\ER}$ and $p \in \dom{g}$,
we define the distance function $\funcdoms{\gamma_{g,p}}{\RS}{\ER}$ by
\begin{align}
  \gamma_{g,p}(\alpha) \coloneqq g(p+\alpha) - g(p) \qquad (\alpha \in \RS).
\end{align}
Intuitively, $\gamma_{g, p}$ plays a role as the directional derivative of $g$ at $p$.

\begin{proposition}\label{prop:directional:derivative:M}
  Suppose that $\funcdoms{g}{\Mdom}{\ER}$ is an M-convex function.
  Then, we have $\cone{\gamma_{g,p}} = \hat{g}'(p; \cdot)$ for any $p \in \dom{g}$,
  where $\hat{g}$ denotes the primal-integral locally polyhedral M-convex function whose restriction to $\Mdom$ is $g$.
  In particular,
  $\gamma_{g,p}$ is a valid distance function for any $p \in \dom{g}$.
\end{proposition}
\begin{proof}
  We only show the former assertion; the latter assertion immediately follows from the former.

  Let $\funcdoms{\hat{g}}{V^*}{\ER}$ be the (unique) primal-integral locally polyhedral M-convex function such that $g$ is the restriction of $\hat{g}$ to $\Mdom$.
Take any $p \in \dom{g} \subseteq \dom{\hat{g}}$.

Since $\hat{g}'(p; \cdot)$ is M-sublinear and $\gamma_{g,p}(\alpha) = g(p+\alpha) - g(p) = \hat{g}(p+\alpha) - \hat{g}(p) \geq \hat{g}'(p; \alpha)$ for $\alpha \in \RS$,
we obtain $\cone{\gamma_{g,p}} \geq \hat{g}'(p; \cdot)$.

In the following, we prove the reverse inequality $\cone{\gamma_{g,p}} \leq \hat{g}'(p; \cdot)$.
Take any $q \in \dom(\hat{g}'(p; \cdot))$.
Then, there is an M-convex cone $K$ in $\fan(\hat{g}'(p; \cdot))$ that contains $q$.
Since $\hat{g}$ is a primal-integral locally polyhedral M-convex function,
there is an integral M-convex polyhedron $P \in \fan(\hat{g}; p)$
such that $K = \tcone{P}{p} = \cone(\RS_p(P \cap \Mdom))$,
where the former and latter equalities follow from \Cref{lem:quasi-polyhedral-convex-function}~(2) and \Cref{prop:tangent-cone},
respectively.
Thus, we obtain $q \in \cone(\RS_p(P \cap \Mdom))$.
Moreover,
it follows from \Cref{lem:quasi-polyhedral-convex-function}~(2) and $\RS_p(P \cap \Mdom) \subseteq P \in \fan(\hat{g}; p)$
that $\hat{g}'(p; \alpha) = \hat{g}(p+\alpha) - \hat{g}(p) = g(p+\alpha) - g(p)$ for $\alpha \in \RS_p(P \cap \Mdom)$,
where the last equality follows from $g = \hat{g}|_{\Mdom}$.

Assume that $q = \sum_{\alpha \in \RS_p(P \cap \Mdom)} \mu(\alpha) \alpha$
for $\mu \in \cncoeff{\RS_p(P \cap \Mdom)}$.
Since $\hat{g}'(p; \cdot)$ is linear on $K = \cone(\RS_p(P \cap \Mdom))$,
we obtain
\begin{align}
    \hat{g}'(p; q) = \sum_{\alpha \in \RS_p(P \cap \Mdom)} \mu(\alpha)\left(g(p + \alpha) - g(p)\right) = \sum_{\alpha \in \RS_p(P \cap \Mdom)} \mu(\alpha)\gamma_{g,p}(\alpha) \geq (\cone{\gamma_{g,p}})(p),\notag
\end{align}
where the inequality follows from the definition~\eqref{eq:func:cone} of the conical hull.
This implies that $\cone{\gamma_{g,p}} \leq \hat{g}'(p; \cdot)$,
which completes the proof.
\end{proof}

We can see that local optimality guarantees global optimality,
which follows from essentially the same proof as \Cref{thm:minimality:L}. 
\begin{theorem}\label{thm:local-optimality-M-convex}
    For an M-convex function $\funcdoms{g}{\Mdom}{\ER}$
    and $p \in \dom{g}$,
    we have that
    $p \in \argmin{g}$
    if and only if
    $g(p) \leq g(p + \alpha)$ for all $\alpha \in \RS$.
\end{theorem}

\subsection{Conjugacy between integral L-convex and M-convex functions}\label{subsec:conjugacy}
For a function $\funcdoms{f}{\Ldom}{\ER}$ with $\dom{f} \neq \emptyset$,
its \emph{discrete Fenchel--Legendre conjugate} $\funcdoms{f^\bullet}{\Mdom}{\ER}$
is the function defined by
\begin{align}
    f^\bullet(p) &\coloneqq \sup{\Set{ \inpr{p}{x} - f(x) }{ x \in \dom{f} }} \qquad (p \in \Mdom).
\end{align}
Similarly, for a function $\funcdoms{g}{\Mdom}{\ER}$ with $\dom{g} \neq \emptyset$,
its \emph{discrete Fenchel--Legendre conjugate} $\funcdoms{g^\bullet}{\Ldom}{\ER}$
is the function defined by
\begin{align}
    g^\bullet(x) &\coloneqq \sup{\Set{ \inpr{p}{x} - g(p) }{ p \in \dom{g} }} \qquad (x \in \Ldom).
\end{align}
We note that
$f^\bullet$ (resp. $g^\bullet$) is the restriction of $f^*$ to $\Mdom$ (resp. $g^*$ to $\Ldom$).

The following establishes the one-to-one correspondence between the integral L-convex functions and the integral M-convex functions
via the discrete Fenchel--Legendre conjugate.
\begin{theorem}\label{thm:conjugacy}
  Let $\funcdoms{f}{\Ldom}{\ER}$ and $\funcdoms{g}{\Mdom}{\ER}$ be functions.
  \begin{enumerate}[label={\textup{(\arabic*)}}]
      \item If $f$ is an integral L-convex function, then its discrete Fenchel--Legendre conjugate $f^\bullet$ is an integral M-convex function and $f = (f^\bullet)^\bullet$.
      \item If $g$ is an integral M-convex function, then its discrete Fenchel--Legendre conjugate $g^\bullet$ is an integral L-convex function and $g = (g^\bullet)^\bullet$.
  \end{enumerate}
\end{theorem}
\begin{proof}
    We only show the assertion~(1);
    the same argument yields the assertion~(2).

    For an integral L-convex function $\funcdoms{f}{\Ldom}{\ER}$,
    we have
    \begin{align}\label{eq:int-conj}
        f^\bullet = (f^*)|_{\Mdom} = (f^{***})|_{\Mdom} = (\clconv{f})^*|_{\Mdom},
    \end{align}
    where the second and last equalities follow from \Cref{lem:conjugate}.
    Since $\clconv{f}$ is an integral locally polyhedral L-convex function
    by \Cref{prop:L-conv:unique},
    its conjugate $(\clconv{f})^*$ forms an integral locally polyhedral M-convex function by \Cref{thm:conjugate-loc-poly-L-M-convex}~(3).
    Thus, $f^\bullet = (\clconv{f})^*|_{\Mdom}$ is an integral M-convex function.

    Furthermore, since $\clconv(f^\bullet) = (\clconv{f})^*$ by \Cref{prop:M-conv:unique} and the above argument,
    we obtain
    \begin{align}
        (f^\bullet)^\bullet = (\clconv(f^\bullet))^*|_{\Ldom} = (\clconv{f})^{**}|_{\Ldom} = (\clconv{f})|_{\Ldom} = f,
    \end{align}
    where the penultimate and last equalities follow from \Cref{lem:conjugate}~(2) and \Cref{prop:L-conv:unique}, respectively.
\end{proof}

Let $\Lfunc{\Z}{\R}$ and $\Lfunc{\Z}{\Z}$ be the classes of L-convex functions and of integral L-convex functions,
respectively;
the corresponding classes of M-convex functions are denoted by $\Mfunc{\Z}{\R}$ and $\Mfunc{\Z}{\Z}$.
By combining the above arguments with those in \Cref{subsec:locally-polyhedral},
we obtain the following commutative diagram.
\begin{align}\label{eq:diagram-LM-func}
\resizebox{\textwidth}{!}{%
\begin{tikzcd}[row sep=4em, column sep=1em, ampersand replacement=\&, crossing over clearance=0.8ex]
\& \& |[alias=qLRR]| \qLfunc{\R}{\R} \& \& \& \& |[alias=qMRR]| \qMfunc{\R}{\R} \& \& \\
\& \& |[alias=LRR]| \Lfunc{\R}{\R} \& \& \& \& |[alias=MRR]| \Mfunc{\R}{\R} \& \&  \\
|[alias=qLZRR]| \qLfunc{\Z|\R}{\R} \& \& \& |[alias=qLRZR]| \qLfunc{\R}{\Z|\R} \& \& |[alias=qMRZR]| \qMfunc{\R}{\Z|\R} \& \& \& |[alias=qMZRR]| \qMfunc{\Z|\R}{\R} \\[-4.5em] 
\& \& \& \rotatebox{90}{$=$} \& \& \rotatebox{90}{$=$} \& \& \& \\[-4.5em] 
\& \& \& |[alias=LRZR]| \Lfunc{\R}{\Z|\R} \& \& |[alias=MRZR]| \Mfunc{\R}{\Z|\R} \& \& \&\\
|[alias=LZRR]| \Lfunc{\Z|\R}{\R} \& \& \& \& \& \& \& \& |[alias=MZRR]| \Mfunc{\Z|\R}{\R} \\
\& \& |[alias=qLZRZR]| \qLfunc{\Z|\R}{\Z|\R} \& \& \& \& |[alias=qMZRZR]| \qMfunc{\Z|\R}{\Z|\R} \& \& \\[-4.5em] 
|[alias=LZR]| \Lfunc{\Z}{\R} \& \& \rotatebox{90}{$=$} \& \& \& \& \rotatebox{90}{$=$} \& \& |[alias=MZR]| \Mfunc{\Z}{\R}\\[-4.5em] 
\& \& |[alias=LZRZR]| \Lfunc{\Z|\R}{\Z|\R} \& \& \& \& |[alias=MZRZR]| \Mfunc{\Z|\R}{\Z|\R} \& \& \\
\& \& |[alias=LZZ]| \Lfunc{\Z}{\Z} \& \& \& \& |[alias=MZZ]| \Mfunc{\Z}{\Z} \& \&
\arrow[from=LRR, to=qLRR, hook]
\arrow[from=MRR, to=qMRR, hook]
\arrow[from=LZRR, to=qLZRR, hook]
\arrow[from=MZRR, to=qMZRR, hook]
\arrow[from=qLZRR, to=qLRR, hook]
\arrow[from=qMZRR, to=qMRR, hook]
\arrow[from=qLZRZR, to=LRZR, hook] 
\arrow[from=qLZRZR, to=LZRR, hook] 
\arrow[from=qLRZR, to=LRR, hook] 
\arrow[from=LZRR, to=LRR, hook] 
\arrow[from=qMZRZR, to=MZRR, hook] 
\arrow[from=qMZRZR, to=MRZR, hook] 
\arrow[from=MZRR, to=MRR, hook] 
\arrow[from=qMRZR, to=MRR, hook] 
\arrow[from=LZZ, to=LZRZR, shift right=0.5ex, "{\substack{\clconv(\cdot) \\[-0.2ex] \rotatebox{90}{$=$} \\[-0.2ex] \conv(\cdot) \\[-0.2ex] \rotatebox{90}{$=$} \\[-0.2ex] \ext{(\cdot)}}}"{right}]
\arrow[from=LZRZR, to=LZZ, shift right=0.5ex, "(\cdot)|_{\Ldom}"{left}]
\arrow[from=LZR, to=LZRR, shift left=0.5ex, "{\substack{\clconv(\cdot) \\[-0.2ex] \rotatebox{90}{$=$} \\[-0.2ex] \conv(\cdot) \\[-0.2ex] \rotatebox{90}{$=$} \\[-0.2ex] \ext{(\cdot)}}}"{left}]
\arrow[from=LZRR, to=LZR, shift left=0.5ex, "(\cdot)|_{\Ldom}"{right}]
\arrow[from=MZZ, to=MZRZR, shift left=0.5ex, "{\substack{\clconv(\cdot) \\[-0.2ex] \rotatebox{90}{$=$} \\[-0.2ex] \conv(\cdot)}}"{left}]
\arrow[from=MZRZR, to=MZZ, shift left=0.5ex, "(\cdot)|_{\Mdom}"{right}]
\arrow[from=MZR, to=MZRR, shift right=0.5ex, "{\substack{\clconv(\cdot) \\[-0.2ex] \rotatebox{90}{$=$} \\[-0.2ex] \conv(\cdot)}}"{right}]
\arrow[from=MZRR, to=MZR, shift right=0.5ex, "(\cdot)|_{\Mdom}"{left}]
\arrow[from=LZZ, to=LZR, hook]
\arrow[from=MZZ, to=MZR, hook]
\arrow[from=LRR, to=MRR, leftrightarrow, "(\cdot)^\ast"{above}]
\arrow[from=LZRZR, to=MZRZR, leftrightarrow, "(\cdot)^\ast"{above}]
\arrow[from=LZZ, to=MZZ, leftrightarrow, "(\cdot)^\bullet"{above}]
\arrow[from=LRZR, to=MZRR, leftrightarrow, "(\cdot)^\ast"{pos=0.5, sloped, above}, crossing over]
\arrow[from=LZRR, to=MRZR, leftrightarrow, "(\cdot)^\ast"{pos=0.5, sloped, above}, crossing over]
\end{tikzcd}%
} 
\end{align}

\subsection{Proof of \texorpdfstring{\Cref{prop:tangent-cone}}{Proposition~6.8}}\label{subsec:proof}

This subsection is devoted to showing \Cref{prop:tangent-cone}.
Before the proof,
we prepare several lemmas on spherical submodular functions,
which are of independent interest.

The following is an easy consequence of the definition of the integrality of a spherical submodular function, or immediately follows from \Cref{cor:inpr}~(2) and \Cref{prop:integral-subm}~(1).
\begin{lemma}\label{lem:rho-p}
  For an integral spherical submodular function $\rho$ and $p \in \Mdom$,
  the function $\rho - p$ is also integral spherical submodular.
\end{lemma}

Let $\CC^\circ$ be a spherical expression of $\spCC$.
For a simplex $A \in \CC^\circ$ and a function $\funcdoms{\rho}{\sk{\CC^\circ}}{\ER}$,
we define the \emph{contraction $\funcdoms{\rho / A}{\sk{\CC^\circ / A}}{\ER}$ of $\rho$ along $A$} by
\begin{align}\label{eq:def:rho/A}
    (\rho / A)(x + \spn{A}) \coloneqq \rho(x) \qquad \left(x \in \sk{\Lk{\CC^\circ}{A}}\right),
\end{align}
where we recall that $\sk{\CC^\circ / A} = \Set{ x + \spn{A} }{ x \in \sk{\Lk{\CC^\circ}{A}} }$.

\begin{lemma}\label{lem:subm:contraction}
    Let $\funcdoms{\rho}{\sk{\CC^\circ}}{\ER}$ be a spherical submodular function and $A \in \CC^\circ$.
    If $\rho(x) = 0$ for $x \in \sk{A}$,
    then
    the contraction $\rho / A$ is spherical submodular on $\sk{\CC^\circ / A}$.
    In addition, if $\rho$ is $\RS$-integral, then $\rho / A$ is $\RS|_{-A}$-integral.
\end{lemma}
\begin{proof}
    Recall that $\cone{A}$ is the simplex of $\spCC$ corresponding to $A$
    and that $\geo{\Lk{\spCC}{\cone{A}}} = \phext{\sk{\Lk{\CC^\circ}{A}}}$
    and $\geo{\St{\spCC}{\cone{A}}} = \phext{\sk{\St{\CC^\circ}{A}}}$.
    Note that any point in $V / \spn{A} = V / \spn(\cone{A})$
    is uniquely representable as $x + \spn{A}$ for some $x \in \geo{\Lk{\spCC}{\cone{A}}}$
    and that the midpoint $(x + y)/2$ of $x,y \in \geo{\Lk{\spCC}{\cone{A}}}$ belongs to $\geo{\St{\spCC}{\cone{A}}}$,
    since $\geo{\St{\spCC}{\cone{A}}}$ is convex by \Cref{lem:star_complex}.
    For $x \in \geo{\St{\spCC}{\cone{A}}}$,
    let $x_A$ denote the unique point in $\geo{\Lk{\spCC}{\cone{A}}}$ such that $x + \spn{A} = x_A + \spn{A}$.
    This $x_A$ is particularly representable as $x_A = \sum_{v \in \sk{\cell{A}{x}} \setminus \sk{A}} \mu_x(v) v$,
    where we recall that $\cell{A}{x}$ denotes the unique simplex of $\CC^\circ$ such that $\cone{\cell{A}{x}}$ contains $x$.
    By the assumption that $\rho(v) = 0$ for $v \in \sk{A}$,
    we have
    \begin{align}\label{eq:x=xA}
        \phext{\rho}(x) = \sum_{v \in \sk{A_x}} \mu_x(v) \rho(v) = \sum_{v \in \sk{A_x} \setminus \sk{A}} \mu_x(v) \rho(v) = \phext{\rho}(x_A) \qquad (x \in \geo{\St{\spCC}{\cone{A}}}).
    \end{align}

    We first show the former assertion.
    By the definition of $\rho / A$,
    we have $\phext{\rho / A}(x + \spn{A}) = \phext{\rho}(x)$ for $x \in \geo{\Lk{\spCC}{\cone{A}}}$.
    Furthermore,
    since $\rho$ is spherical submodular,
    $\phext{\rho}$ is (polyhedral) sublinear on $\geo{\spCC} = V$,
    and hence, is convex on $\geo{\St{\spCC}{\cone{A}}}$.
    Therefore,
    for any $x,y \in \geo{\Lk{\spCC}{\cone{A}}}$,
    we obtain
    \begin{align}
        \phext{\rho / A}(x + \spn{A}) + \phext{\rho / A}(y + \spn{A}) &= \phext{\rho}(x) + \phext{\rho}(y)\\
    &\geq \phext{\rho}(x + y)\\
    &= \phext{\rho}((x+y)_A)\\
    &= \phext{\rho / A}((x+y)_A + \spn{A})\\
    &= \phext{\rho / A}((x+y) + \spn{A}),\label{eq:rho/A}
    \end{align}
    where the second equality follows from~\eqref{eq:x=xA}.
    Since $\phext{\rho / A}$ is a positively homogeneous function,
    the inequality~\eqref{eq:rho/A} implies that $\phext{\rho / A}$ is convex on $V / \spn{A}$.
    Thus, $\rho / A$ is a spherical submodular function by \Cref{prop:spherical-submodular-cone}.

    For the latter assertion,
    we may assume that $\rho$ is a $\RS$-integral spherical submodular function on $\spCC^\star$.
    Let $a$ and $a_{{\mathrm{Lk}}}$ denote the mark functions of $\spCC^\star$ and $\Lk{\spCC^\star}{A}$,
    respectively.
    Recall that $a_{\mathrm{Lk}}$ can be viewed as the mark function of $\spCC^\star(\RS|_{-A})$
    that maps from $x + \spn{A}$ to $a_{\mathrm{Lk}}(x)$.
    By \Cref{lem:standard}~(1),
    for $x \in \sk{\Lk{\spCC^\star}{A}}$,
    we have $(a(x)/a_{\mathrm{Lk}}(x))x + \spn{A} \in \sk{\spCC^\star(\RS|_{-A})}$.
    Hence,
    we obtain
    \begin{align}
        \phext{\rho / A}((a(x)/a_{\mathrm{Lk}}(x))x + \spn{A}) = (a(x)/a_{\mathrm{Lk}}(x))\phext{\rho / A}(x + \spn{A})
        = (a(x)/a_{\mathrm{Lk}}(x))\rho(x)\notag
    \end{align}
    for $x \in \sk{\Lk{\spCC^\star}{A}}$.
    Since $\rho$ is a $\RS$-integral spherical submodular function on $\sk{\spCC^\star}$, we have $\rho(x) \in \Z/a(x) \cup \{ +\infty \}$ for any  $x \in \sk{\Lk{\spCC^\star}{A}} \subseteq \sk{\spCC^\star}$
    by \Cref{prop:integral-subm}~(1).
    Thus,
    \begin{align}
        \phext{\rho / A}((a(x)/a_{\mathrm{Lk}}(x))x + \spn{A}) \in \Z / a_{\mathrm{Lk}}(x) \cup \{ +\infty \}
    \end{align}
    holds
    for $x \in \sk{\Lk{\spCC^\star}{A}}$.
    This implies that $\rho / A$ is a $\RS|_{-A}$-integral spherical submodular function by \Cref{prop:integral-subm}~(1) again.
\end{proof}

For a vertex subset $S \subseteq \sk{\CC^\circ}$,
we define $\spCC[S]$ as the subcomplex of $\spCC$ induced by the vertex subset of $\spCC$ corresponding to $S$,
i.e., $\spCC[S] \coloneqq \spCC[\Set{ \cone{x} }{ x \in S }]$.
Note that $\spCC[\emptyset] = \{0\}$.
We can easily see that $\geo{\spCC[S]} = \phext{S}$.
For a function $\funcdoms{\rho}{\sk{\CC^\circ}}{\ER}$ and $p \in V^*$,
we define $\CC^\circ[\rho = p]$ as the subcomplex of $\CC^\circ$ induced by the vertex subset $\Set{ x \in \sk{\CC^\circ} }{ \rho(x) = p(x) }$
and $\spCC[\rho = p]$ as the subcomplex corresponding to $\CC^\circ[\rho = p]$,
i.e., $\spCC[\rho = p] \coloneqq \spCC[\Set{ x \in \sk{\CC^\circ} }{ \rho(x) = p(x) }]$.
We here recall that, for a polyhedron $P \subseteq V^*$ and $p \in V^*$,
let $\face{P}{p}$ denote the unique face of $P$ satisfying $p \in \relint{\face{P}{p}}$.
\begin{lemma}\label{lem:minimal-face}
    Let $\funcdoms{\rho}{\sk{\CC^\circ}}{\ER}$ be a spherical submodular function
    and $p \in \B{\rho}$.
    \begin{enumerate}[label={\textup{(\arabic*)}}]
        \item $\spCC[\rho = p]$ is a convex subcomplex of $\spCC$.
        In particular, $\spCC[\rho = p]$ and $\CC^\circ[\rho = p]$ are chamber complexes.
    \end{enumerate}
    In addition, let $A$ be a maximal simplex in $\CC^\circ[\rho = p]$.
    \begin{enumerate}[label={\textup{(\arabic*)}}, resume]
        \item $\aff{\face{\B{\rho}}{p}} = p + \spn{\RS|_{-A}}$.
        In particular, $\dim{\face{\B{\rho}}{p}} = n - \card{\sk{A}}$.
        \item $\face{\B{\rho}}{p} = p + \B{(\rho - p)/A}$ and $((\rho-p)/A)(x) > 0$ for all $x \in \sk{\CC^\circ / A}$.
    \end{enumerate}
\end{lemma}
\begin{proof}
    (1).
    The latter statement immediately follows from the former.
    Indeed, if $\spCC[\rho = p]$ is a convex subcomplex of $\spCC$,
    then it is a chamber complex by \Cref{lem:chamber-complex}.
    In addition,
    $\CC^\circ[\rho = p]$ is also a chamber complex,
    since $\CC^\circ[\rho = p]$ is isomorphic to $\spCC[\rho = p]$.

    We show the former statement.
    By \Cref{lem:convex_subcomplex},
    it suffices to see that $\geo{\spCC[\rho = p]}$ is convex.
    Since $p \in \B{\rho}$,
    we have $\inpr{p}{x} \leq \rho(x)$ for all $x \in \sk{\CC^\circ}$.
    Thus, the set $\Set{ x \in \sk{\CC^\circ} }{ \rho(x) = \inpr{p}{x} }$ coincides with $\argmin(\rho - p)$.
    Since $\rho$ is spherical submodular, $\phext{\rho}$ is convex by \Cref{prop:spherical-submodular-cone},
    implying that $\argmin(\phext{\rho} - p)$ forms a convex set.
    Thus, it follows from $\geo{\spCC[\rho = p]} = \phext{\argmin(\rho - p)} = \argmin(\phext{\rho} - p)$
    that
    $\geo{\spCC[\rho = p]}$ is convex.
    This completes the proof of the assertion~(1).

    (2).
    The latter statement immediately follows from the former,
    since $\dim{\face{\B{\rho}}{p}} = \dim{\aff{\face{\B{\rho}}{p}}} = \dim{\spn{\RS|_{-A}}} = n - \card{\sk{A}}$.
    In the following, we show the former statement.

        Since $\B{\rho}$ is the polyhedron of the form
        \begin{align}
            \B{\rho} = \Set{ q \in V^* }{ \inpr{q}{x} \leq \rho(x) \ (\forall x \in \sk{\CC^\circ}) }
        \end{align}
        and $p \in \B{\rho}$ satisfies $\inpr{p}{x} = \rho(x)$ for $x \in \sk{\CC^\circ[\rho = p]}$
        and $\inpr{p}{x} < \rho(x)$ for $x \in \sk{\CC^\circ} \setminus \sk{\CC^\circ[\rho = p]}$,
        we have
        \begin{align}\label{eq:Pp}
            \face{\B{\rho}}{p} &= \Set{ q \in \B{\rho} }{ \inpr{q}{x} = \rho(x) \ (\forall x \in \sk{\CC^\circ[\rho = p]}) }
        \end{align}
        and
        \begin{align}
            \aff{\face{\B{\rho}}{p}} &= p + \Set{ q \in V^* }{ \inpr{q}{x} = 0 \ (\forall x \in \sk{\CC^\circ[\rho = p]}) }\\
            &= p + (\spn{\sk{\CC^\circ[\rho = p]}})^\perp\\
            &= p + (\spn{\geo{\spCC[\rho = p]}})^\perp.\label{eq:affPp}
        \end{align}
        Moreover,
        since $p \in \B{\rho}$, $\spCC[\rho = p]$ and $\CC^\circ[\rho = p]$ are chamber complexes by the assertion~(1).
        Since $A$ is a maximal simplex of $\CC^\circ[\rho = p]$,
        the simplex $\cone{A}$ of $\spCC[\rho = p]$ corresponding to $A$
        is also maximal in $\spCC[\rho = p]$.
        Thus, we have $\dim{\cone{A}} = \dim{\geo{\spCC[\rho = p]}}$,
        implying that $\spn(\cone{A}) = \spn{\geo{\spCC[\rho = p]}}$.
        Since $(\spn(\cone{A}))^\perp = \spn{\RS|_{-A}}$ by \Cref{lem:span}~(1),
        we obtain
        $\aff{\face{\B{\rho}}{p}} = p + (\spn{\geo{\spCC[\rho = p]}})^\perp = p + \spn{\RS|_{-A}}$.

        (3).
        Since $\face{\B{\rho}}{p} = \B{\rho} \cap \aff{\face{\B{\rho}}{p}}$,
        we obtain
        \begin{align}
            \face{\B{\rho}}{p} &= \B{\rho} \cap (p + \spn{\RS|_{-A}})\\
            &= \Set*{ q \in V^* }{ q - p \in \spn{\RS|_{-A}} \text{ and } \inpr{q}{x} \leq \rho(x) \ (\forall x \in \sk{\CC^\circ})}\\
            &= p + \Set*{ q \in \spn{\RS|_{-A}} }{\inpr{q}{x} \leq (\rho - p)(x) \ (\forall x \in \sk{\CC^\circ}) }\\
            &\subseteq p + \Set*{ q \in \spn{\RS|_{-A}} }{ \inpr{q}{x} \leq (\rho - p)(x) \ (\forall x \in \sk{\Lk{\CC^\circ}{A}}) }. \label{eq:Pp:inc}
        \end{align}
        Here, the first equality follows from $\aff{\face{\B{\rho}}{p}} = p + \spn{\RS|_{-A}}$ by the assertion~(2).

        On the other hand,
        since $A$ is a simplex in $\CC^\circ[\rho = p]$,
        we have $(\rho - p)(x) = 0$
        for any $x \in \sk{A}$.
        Thus, $(\rho-p) / A$ is a spherical submodular function on $\CC^\circ / A$ by \Cref{lem:subm:contraction},
        which allows us to consider the base polyhedron $\B{(\rho - p) / A}$.
        Since $\CC^\circ / A$ is a spherical expression of $\spCC(\RS|_{-A})$,
        $\B{(\rho - p) / A}$ lies in $\spn{\RS|_{-A}}$.
        More precisely, $\B{(\rho - p) / A}$ is representable as
        \begin{align}\label{eq:rho-p/A}
            \B{(\rho - p) / A} = \Set*{ q \in \spn{\RS|_{-A}} }{ \inpr{q}{x} \leq (\rho-p)(x) \ (\forall x \in \sk{\Lk{\CC^\circ}{A}}) },
        \end{align}
        since $((\rho -p) / A)(x + \spn{A}) = (\rho -p)(x)$ for $x \in \sk{\Lk{\CC^\circ}{A}}$ by the definition of $(\rho-p)/A$ (see~\eqref{eq:def:rho/A})
        and $\inpr{q}{x + \spn{A}} = \inpr{q}{x}$ for $q \in \spn{\RS|_{-A}} = (\spn(\cone{A}))^\perp = (\spn{A})^\perp$ by \Cref{lem:span}~(1).

        By~\eqref{eq:Pp:inc} and~\eqref{eq:rho-p/A},
        it suffices to show that the inclusion~\eqref{eq:Pp:inc} can be replaced with the identity.
        To this end, we show the following claim.
    \begin{claim}\label{cl:proper-face}
        \begin{enumerate}[label={\textup{(\arabic*)}}]
        \item If $q$ belongs to a proper face of $\face{\B{\rho}}{p}$,
        i.e., $\face{\B{\rho}}{q} \subseteq \face{\B{\rho}}{p}$ and $\dim{\face{\B{\rho}}{q}} < \dim{\face{\B{\rho}}{p}}$,
        then there exists $x \in \sk{\Lk{\CC^\circ}{A}}$ such that $\rho(x) = q(x)$.
        \item $(\rho - p)(x) > 0$ for all $x \in \sk{\Lk{\CC^\circ}{A}}$.
    \end{enumerate}
    \end{claim}
    \begin{proof}[Proof of \Cref{cl:proper-face}]
        (1).
        Since $q \in \face{\B{\rho}}{p}$,
        we have $\sk{\CC^\circ[\rho = p]} \subseteq \sk{\CC^\circ[\rho = q]}$ (see~\eqref{eq:Pp}),
        that is,
        $\CC^\circ[\rho = p]$ is a subcomplex of $\CC^\circ[\rho = q]$.
        In addition, it follows from $\dim{\face{\B{\rho}}{q}} < \dim{\face{\B{\rho}}{p}}$ and the equation~\eqref{eq:affPp}
        that $\dim{\geo{\spCC[\rho = p]}} < \dim{\geo{\spCC[\rho = q]}}$.
        This implies that the rank of $\CC^\circ[\rho = q]$ is strictly larger than that of $\CC^\circ[\rho = p]$;
        recall that $\CC^\circ[\rho = p]$ (resp. $\CC^\circ[\rho = q]$) is isomorphic to $\spCC[\rho = p]$ (resp. $\spCC[\rho = q]$),
        and that $\CC^\circ[\rho = p]$ and $\CC^\circ[\rho = q]$ are chamber complexes by the assertion~(1).
        Hence, $A$ is a simplex of $\CC^\circ[\rho = q]$,
        but not maximal in $\CC^\circ[\rho = q]$,
        i.e., there is a simplex $A' \in \CC^\circ[\rho = q]$ with $\sk{A'} \supsetneq \sk{A}$.
        Since $A$ is a face of $A'$,
        any vertex $x \in \sk{A'} \setminus \sk{A} \neq \emptyset$ belongs to $\sk{\Lk{\CC^\circ}{A}}$;
        this $x$ satisfies $\rho(x) = q(x)$.

        (2).
        If there is $x \in \sk{\Lk{\CC^\circ}{A}}$ with $\rho(x) = p(x)$,
        then the join of $A$ and $x$ is a simplex of $\CC^\circ[\rho = p]$,
        which contradicts that $A$ is a maximal simplex in $\CC^\circ[\rho = p]$.
    \end{proof}

    Suppose to the contrary that the inclusion~\eqref{eq:Pp:inc} is strict,
    i.e., there is $q' \in \spn{\RS|_{-A}}$ such that $\inpr{q'}{x} \leq (\rho - p)(x)$ for all $x \in \sk{\Lk{\CC^\circ}{A}}$ but $p + q' \notin \face{\B{\rho}}{p}$.
    Let $\mu$ be the maximum real satisfying $p + \mu q' \in \face{\B{\rho}}{p}$;
    since $p$ belongs to a relative interior of $\face{\B{\rho}}{p}$ and $p + q' \in (\aff{\face{\B{\rho}}{p}}) \setminus \face{\B{\rho}}{p}$,
    we have $0 < \mu < 1$.
    For this $\mu$, the point $p + \mu q' \in p +\spn{\RS|_{-A}}$ belongs to a proper face of $\face{\B{\rho}}{p}$.
    Hence, by \Cref{cl:proper-face}~(1),
    there is $x \in \Lk{\CC^\circ}{A}$ such that $\rho(x) = (p + \mu q')(x)$,
    or equivalently, $\mu q'(x) = (\rho - p)(x)$.
    Since $(\rho - p)(x) > 0$ by \Cref{cl:proper-face}~(2) and $0 < \mu < 1$,
    we obtain $q'(x) > (\rho - p)(x)$.
    This contradicts the assumption of $q'$.
\end{proof}

Suppose that $\RS = \RS_1 \oplus \RS_2 \oplus \cdots \oplus \RS_k$,
in which $\RS_1, \RS_2, \dots, \RS_k$ are irreducible components of $\RS$.
Since $\spn{\RS_1}, \spn{\RS_2}, \dots, \spn{\RS_k}$ are orthogonal,
the vector spaces $V^* (= \spn{\RS})$ and $V (= (V^*)^*)$ admit orthogonal decompositions as
\begin{align}
    V^* &= \spn{\RS_1} \oplus \spn{\RS_2} \oplus \cdots \oplus \spn{\RS_k},\\
    V &= V / (\spn{\RS_1})^\perp \oplus V / (\spn{\RS_2})^\perp \oplus \cdots \oplus V / (\spn{\RS_k})^\perp,
\end{align}
where we identify $(\spn{\RS_i})^*$ with $V / (\spn{\RS_i})^\perp$ (see~\Cref{subsec:vectorspace}).
By $x = (x_1, x_2, \dots, x_k) \in V$,
we mean $x = \sum_{i = 1}^k x_i$ with $x_i \in (\spn{\RS_i})^\perp$ for $i \in [k]$;
note that such a decomposition of $x$ is unique.
Similarly, we use the notation $p = (p_1, p_2, \dots, p_k)$ for $p \in V^*$.

Along with the above decomposition, we have
$\spCC = \spCC(\RS_1) \times \spCC(\RS_2) \times \cdots \times \spCC(\RS_k)$ (see also~\cite[Exercise~3.30]{Abramenko2010}).
In particular, $\sk{\spCC} = \sk{\spCC(\RS_1)} \cup \sk{\spCC(\RS_2)} \cup \cdots \cup \sk{\spCC(\RS_k)}$ holds,
in which $\sk{\spCC(\RS_i)}$ and $\sk{\spCC(\RS_j)}$ are disjoint if $i \neq j$.
Therefore,
for the spherical expression $\CC^\circ$ of $\spCC$ generated from $\Set{ x_\ell^\circ \in \ell }{ \ell \in \sk{\spCC} }$,
we have $\CC^\circ = \CC^\circ_1 \times \CC^\circ_2 \times \cdots \times \CC^\circ_k$,
where $\CC^\circ_i$ is the spherical expression of $\spCC(\RS_i)$ generated from $\Set{ x_\ell^\circ }{ \ell \in \sk{\spCC(\RS_i)} } \subseteq \sk{\CC^\circ}$.
We refer to $\CC^\circ_i$ as the \emph{$i$th component of $\CC^\circ$}.
Since the mark of $\spCC(\RS_i)$ coincides with the restriction of that of $\spCC$ to $\sk{\spCC(\RS_i)}$ by the definition,
the $i$th component of $\spCC^\star$
is $\spCC^\star(\RS_i)$.
Let $\CC^\circ$ be a spherical expression of $\spCC = \spCC(\RS_1) \times \spCC(\RS_2) \times \cdots \times \spCC(\RS_k)$ again
and $\rho$ a function on $\CC^\circ$.
For each $i \in [k]$,
we define $\rho_i$ as the restriction of $\rho$ to the vertex set $\sk{\CC^\circ_i}$ of the $i$th component of $\CC^\circ$;
we refer to $\rho_i$ as the \emph{$i$th component of $\rho$}.

\begin{lemma}\label{lem:subm:components}
    Let $\RS = \RS_1 \oplus \RS_2 \oplus \cdots \oplus \RS_k$ be a root system in which $\RS_1, \RS_2, \dots, \RS_k$ are irreducible components of $\RS$.
    Let $\CC^\circ$ be a spherical expression of $\spCC$
    and $\CC^\circ_i$ the $i$th component of $\CC^\circ$ for $i \in [k]$.
    Then, the following hold.
    \begin{enumerate}[label={\textup{(\arabic*)}}]
        \item A function $\funcdoms{\rho}{\sk{\CC^\circ}}{\ER}$ is (integral) spherical submodular if and only if so are all of its components $\funcdoms{\rho_i}{\sk{\CC^\circ_i}}{\ER}$.
        \item For a spherical submodular function $\funcdoms{\rho}{\sk{\CC^\circ}}{\ER}$,
        we have
        \begin{align}
            \B{\rho} &= \B{\rho_1} + \cdots + \B{\rho_k},\\
            \B{\rho} \cap \Mdom &= \left(\B{\rho_1} \cap \Mdom(\RS_1)\right) + \cdots + \left(\B{\rho_k} \cap \Mdom(\RS_k)\right).
        \end{align}
        In particular, $\B{\rho}$ is $\Mdom$-integral if and only if $\B{\rho_i}$ is $\Mdom(\RS_i)$-integral for all $i \in [k]$.
    \end{enumerate}
\end{lemma}
\begin{proof}
    Since $A_x = A_{x_1} \times \cdots \times A_{x_k}$ for $x = (x_1, x_2, \dots, x_k) \in V$,
    we observe that
    \begin{align}\label{eq:rho:decomposition}
        \phext{\rho}(x) = \phext{\rho_1}(x_1) + \phext{\rho_2}(x_2) + \cdots + \phext{\rho_k}(x_k).
    \end{align}

    (1).
    By~\eqref{eq:rho:decomposition}, $\phext{\rho}$ is convex if and only if so are $\phext{\rho_1}, \phext{\rho_2}, \dots, \phext{\rho_k}$.
    Thus, by \Cref{prop:spherical-submodular-cone},
    $\rho$ is spherical submodular if and only if so are $\rho_1, \rho_2, \dots, \rho_k$.
    The integrality assertion follows from the fact that the mark of $\CC^\circ_i$ is the restriction of that of $\CC^\circ$ to $\sk{\CC^\circ_i}$ for each $i$.

    (2).
    The former assertion follows from
    \begin{align}
        p = (p_1, p_2, \dots, p_k) \in \B{\rho} &\iff \inpr{p}{x} \leq \rho(x) \text{ for all } x \in \sk{\CC^\circ}\\
        &\iff \inpr{p_i}{x} \leq \rho_i(x) \text{ for all } i \in [k] \text{ and } x \in \sk{\CC_i^\circ}\\
        &\iff p_i \in \B{\rho_i} \text{ for all } i \in [k]\\
        &\iff p \in \B{\rho_1} + \cdots + \B{\rho_k}.
    \end{align}
    The latter follows from the former and $\Mdom(\RS) = \Mdom(\RS_1) + \cdots + \Mdom(\RS_k)$ (see~\eqref{eq:Mdom-sum}).
\end{proof}

The proof of \Cref{prop:tangent-cone} requires type-dependent arguments.
The following lemma immediately follows from \Cref{lem:root_integer,lem:contraction:root-system} and the shape of the Dynkin diagram depicted in \Cref{tab:classical-Dynkin-diagrams}.
\begin{lemma}\label{lem:Dynkin-shape}
  \begin{enumerate}[label={\textup{(\arabic*)}}]
        \item A root system of type $\Dynkin{A}{3}$ is isomorphic to that of type $\Dynkin{D}{3}$.
        \item Let $\RS$ be an irreducible root system of type~B, C, or D,
        and $\simp$ a simple system of $\RS$.
        Suppose that $\{v_1, v_2, \dots, v_n\}$ is the vertex set of $D(\RS)$ as illustrated in~\eqref{tab:classical-Dynkin-diagrams}
        and that $J \subseteq D(\RS)$ such that $\RS(\simp|_{J})$ is irreducible.
        Then, the following holds:
        \begin{itemize}
          \item In the case where $\RS$ is of type~B,
            \begin{align}
              \text{$\RS(\simp|_{J})$ is of }
              \begin{cases}
                \text{type~B} & \text{if $v_n \in J$},\\
                \text{type~A} & \text{if $v_n \notin J$}.
              \end{cases}
            \end{align}
          \item In the case where $\RS$ is of type~C,
            \begin{align}
              \text{$\RS(\simp|_{J})$ is of }
              \begin{cases}
                \text{type~C} & \text{if $v_n \in J$},\\
                \text{type~A} & \text{if $v_n \notin J$}.
              \end{cases}
            \end{align}
          \item In the case where $\RS$ is of type~D,
            \begin{align}
              \text{$\RS(\simp|_{J})$ is of }
              \begin{cases}
                \text{type~D} & \text{if $\{v_{n-3}, v_{n-2}, v_{n-1}, v_n\} \subseteq J$},\\
                \text{type~A} & \text{if $\{v_{n-3}, v_{n-2}, v_{n-1}, v_n\} \not\subseteq J$}.
              \end{cases}
            \end{align}
        \end{itemize}
    \end{enumerate}
\end{lemma}

In addition, \Cref{tab:classical-positive-roots} enumerates the positive roots of types $\Dynkin{A}{n}$, $\Dynkin{B}{n}$, $\Dynkin{C}{n}$, and $\Dynkin{D}{n}$.
\begin{table}[tbp]
    \centering
    \small
    \setlength{\tabcolsep}{5pt}
    \renewcommand{\arraystretch}{1.25}

    \begin{tabularx}{\linewidth}{
        @{}
        >{\centering\arraybackslash}m{0.11\linewidth}
        >{\centering\arraybackslash}X
        @{}
    }
        \toprule
        Type
        &
        \positiveRootBlock{
            Positive root
            &
            Range of indices
        }
        \\
        \midrule

        $\Dynkin{A}{n}$
        &
        \positiveRootBlock{
            $\displaystyle
                \sum_{i=i_1}^{i_2}\alpha_i
            $
            &
            $\displaystyle
                1\leq i_1\leq i_2\leq n
            $
        }
        \\
        \typesep

        $\Dynkin{B}{n}$
        &
        \positiveRootBlock{
            $\displaystyle
                \sum_{i=i_1}^{i_2}\alpha_i
            $
            &
            $\displaystyle
                1\leq i_1\leq i_2\leq n
            $
            \\[0.6em]

            $\displaystyle
                \sum_{i=i_1}^{i_2-1}\alpha_i
                +
                \sum_{j=i_2}^{n}2\alpha_j
            $
            &
            $\displaystyle
                1\leq i_1<i_2\leq n
            $
        }
        \\
        \typesep

        $\Dynkin{C}{n}$
        &
        \positiveRootBlock{
            $\displaystyle
                \sum_{i=i_1}^{i_2}\alpha_i
            $
            &
            $\displaystyle
                1\leq i_1\leq i_2\leq n
            $
            \\[0.6em]

            $\displaystyle
                \sum_{i=i_1}^{i_2-1}\alpha_i
                +
                \sum_{j=i_2}^{n-1}2\alpha_j
                +
                \alpha_n
            $
            &
            $\displaystyle
                1\leq i_1\leq i_2\leq n-1
            $
        }
        \\
        \typesep

        \begin{tabular}[c]{@{}c@{}}
            $\Dynkin{D}{n}$
            \\
            $n\geq4$
        \end{tabular}
        &
        \positiveRootBlock{
            $\displaystyle
                \sum_{i=i_1}^{i_2}\alpha_i
            $
            &
            $\displaystyle
                1\leq i_1\leq i_2\leq n-1
            $
            \\[0.6em]

            $\displaystyle
                \sum_{i=i_1}^{n}\alpha_i
            $
            &
            $\displaystyle
                i_1\in[n-2]
            $
            \\[0.6em]

            $\displaystyle
                \sum_{i=i_1}^{n-2}\alpha_i+\alpha_n
            $
            &
            $\displaystyle
                i_1\in[n-2]
            $
            \\[0.6em]

            $\displaystyle
                \sum_{i=i_1}^{i_2-1}\alpha_i
                +
                \sum_{j=i_2}^{n-2}2\alpha_j
                +
                \alpha_{n-1}
                +
                \alpha_n
            $
            &
            $\displaystyle
                1\leq i_1<i_2\leq n-2
            $
        }
        \\

        \bottomrule
    \end{tabularx}

    \caption{
        Positive roots of root systems of the classical types with respect to a fixed simple system $\simp = \{ \alpha_1, \alpha_2, \dots, \alpha_n \}$.
    }
    \label{tab:classical-positive-roots}
\end{table}

We are ready to show \Cref{prop:tangent-cone}.
\begin{proof}[Proof of \Cref{prop:tangent-cone}]
    Letting
    \begin{align}
        \bartRS{p}{B} \coloneqq \Set{ \alpha \in \RS }{ \exists \eps > 0.\  p +\eps \alpha \in \conv{B} },
    \end{align}
    we have $\tcone{\conv{B}}{p} = \cone(\bartRS{p}{B})$,
    since $\conv{B}$ is an M-convex polyhedron.
    It is clear that $\bartRS{p}{B} \supseteq \tRS{p}{B}$, 
    implying $\tcone{\conv{B}}{p} \supseteq \cone(\tRS{p}{B})$.
    In the following, we show the reverse inclusion $\tcone{\conv{B}}{p} \subseteq \cone(\tRS{p}{B})$,
    or equivalently, $\bartRS{p}{B} \subseteq \cone(\tRS{p}{B})$.

    By \Cref{cor:Mpoly-subm},
    there is an integral spherical submodular function $\funcdoms{\rho}{\sk{\spCC^\star}}{\ER}$ satisfying $\conv{B} = \B{\rho}$.
    For notational simplicity,
    we only consider the case where $p = 0$ and $\RS$ is irreducible.
    A general case (where $p \in B$ is arbitrary and $\RS$ is not necessarily irreducible) can be reduced to the above case as follows.
    The update $\rho \leftarrow \rho - p$ reduces the case for an arbitrary $p \in B$ to the case for $p = 0$,
    since $\rho - p$ is integral spherical submodular by \Cref{lem:rho-p}
    and $B - p = \B{\rho - p} \cap \Mdom$.
    Let $\RS_0 \coloneqq \tRS{0}{B} = \RS \cap B$ and $\bar{\RS}_0 \coloneqq \bartRS{0}{B} = \Set{\alpha \in \RS}{\exists \eps > 0.\ \eps \alpha \in \B{\rho}}$.
    Suppose that $\RS = \RS_1 \oplus \RS_2 \oplus \cdots \oplus \RS_k$ with its irreducible components $\RS_1, \RS_2,  \dots, \RS_k$.
    By \Cref{lem:subm:components},
    we have $B = B_1 + B_2 + \cdots + B_k$,
    $\RS_0 = (\RS_1)_0 \cup (\RS_2)_0 \cup \cdots \cup (\RS_k)_0$,
    and $\bar{\RS}_0 = (\bar{\RS}_1)_0 \cup (\bar{\RS}_2)_0 \cup \cdots \cup (\bar{\RS}_k)_0$,
    where $B_i \coloneqq \B{\rho_i} \cap \Mdom(\RS_i)$,
    $(\RS_i)_0 \coloneqq \Set{ \alpha \in \RS_i }{ \alpha \in B_i } = \RS_i \cap B_i$,
    and $(\bar{\RS}_i)_0 \coloneqq (\bar{\RS}_i)_0(B_i) = \Set{ \alpha \in \RS_i }{ \exists \eps > 0.\  \eps \alpha \in \B{\rho_i} }$
    for each $i \in [k]$.
    Thus, if $\hat{\alpha}_i \in (\bar{\RS}_i)_0$ implies $\hat{\alpha}_i \in \cone{(\RS_i)_0}$ for each $i \in [k]$,
    then we obtain $\hat{\alpha} \in \cone{\RS_0}$ for any $\hat{\alpha} \in \bar{\RS}_0$,
    which shows the desired inclusion $\cone{\RS_0} \supseteq \bar{\RS}_0$.

    As described above,
    we assume that $p = 0$ and $\RS$ is irreducible in the following.
    Take any $\hat{\alpha} \in \bar{\RS}_0$.
    Let $\eps \coloneqq \sup{\Set{ \eps' }{ \eps'\hat{\alpha} \in \B{\rho} }} > 0$.
    If $\eps \geq 1$,
    then $\hat{\alpha} \in \B{\rho}$.
    Hence, we have $\hat{\alpha} \in \B{\rho} \cap \Mdom = B$,
    which implies $\hat{\alpha} \in \RS_0 \subseteq \cone{\RS_0}$.

    Suppose that $\eps < 1$.
    Let $A$ be a maximal simplex in $\spCC^\star[\rho = \eps \hat{\alpha}]$,
    $C = \{ x_1, x_2, \dots, x_n \}$ a chamber in $\spCC^\star$ with $C \succeq A$,
    and $\simp = \{ \alpha_1, \alpha_2, \dots, \alpha_n \}$ be the simple system corresponding to $C$,
    in which, for each $i \in [n]$, both $\tau(x_i)$ and $\tau_{\simp}(\alpha_i)$ are the vertex $v_i$ in the Dynkin diagram depicted in \Cref{tab:classical-Dynkin-diagrams}.
    By the maximality of $\eps$,
    the point $\eps\hat{\alpha}$ belongs to a proper face of $\B{\rho}$,
    i.e., $\dim{\face{\B{\rho}}{\eps\hat{\alpha}}} < \dim{\B{\rho}}$.
    Thus, we have $\tau(A) \neq \emptyset$ by \Cref{lem:minimal-face}~(2).
    For any $p = \sum_{i = 1}^n p_i \alpha_i \in V^*$,
    let $\supp_\simp{p} \coloneqq \Set{ i \in [n] }{ p_i \neq 0 }$.
    Suppose that $\hat{\alpha} = \sum_{i = 1}^n k_i \alpha_i$
    (and hence $\eps\hat{\alpha} = \sum_{i = 1}^n \eps k_i \alpha_i$).
    We define
    \begin{align}
        I_1 \coloneqq \Set{ i \in \supp_\simp{\eps\hat{\alpha}} }{ \eps k_i \in \Z }, \qquad
        I_{1/2} \coloneqq \Set{ i \in \supp_\simp{\eps\hat{\alpha}} }{ \eps k_i \notin \Z }.
    \end{align}
    Then, the following hold:
    \begin{claim}\label{cl:positive-F}
        \begin{enumerate}[label={\textup{(\arabic*)}}]
        \item $\hat{\alpha}$ is a positive root with respect to $\Delta$.
        \item $\Set{v_i \in D(\RS)}{i \in I_{1/2}} \subseteq D(\RS) \setminus \tau(A)$, or equivalently, $\Set{v_i \in D(\RS)}{i \in I_{1/2}} \cap \tau(A) = \emptyset$.
        \item $\Set{v_i \in D(\RS)}{i \in \supp_\simp{\hat{\alpha}}} \cap \tau(A) \neq \emptyset$. In addition, $\Set{v_i \in D(\RS)}{i \in I_1} \cap \tau(A) \neq \emptyset$.
        \item There exists $v_i \in D(\RS)$ such that $k_i \geq 2$.
    \end{enumerate}
    \end{claim}
    \begin{proof}
        (1).
        We first see that there is $x \in \sk{A}$ such that $\rho(x) > 0$.
        Suppose to the contrary that $\inpr{\eps \hat{\alpha}}{x} = \rho(x) = 0$ for all $x \in \sk{A}$.
        Then, $\hat{\alpha}$ is orthogonal to $\spn{\sk{A}} = \spn{A} = (\spn{\RS|_{-A}})^\perp$,
        where the last equality follows from \Cref{lem:span}~(1).
        Hence, $\hat{\alpha} \in \spn{\RS|_{-A}}$.
        Since $\aff{\face{\B{\rho}}{\eps \hat{\alpha}}} = \eps \hat{\alpha} + \spn{\RS|_{-A}}$ by \Cref{lem:minimal-face}~(2),
        for sufficiently small $\eps' > 0$
        we obtain $(\eps + \eps')\hat{\alpha} \in \face{\B{\rho}}{\eps\hat{\alpha}} \subseteq \B{\rho}$.
        This contradicts the maximality of $\eps$.

        Suppose that $x_i \in \sk{A}$ satisfies $\rho(x_i) > 0$.
        In the representation $\hat{\alpha} = \sum_{i = 1}^n k_i \alpha_i$ of $\hat{\alpha}$ with respect to $\Delta$,
        the coefficient $k_i$ of $\alpha_i$ is equal to $k_i = a(x_i)\inpr{\hat{\alpha}}{x_i} = a(x_i)\rho(x_i)/\eps > 0$ by \Cref{cor:inpr}~(1).
        This implies that $\hat{\alpha}$ is a positive root with respect to $\Delta$.

        (2).
        Suppose to the contrary that $\Set{v_i \in D(\RS)}{i \in I_{1/2}} \cap \tau(A) \neq \emptyset$.
        Then, take any $i \in I_{1/2}$ with $v_i \in \tau(A)$.
        Since $\aff{\face{\B{\rho}}{\eps \hat{\alpha}}} = \eps \hat{\alpha} + \spn{\RS|_{-A}} = \eps \hat{\alpha} + \spn{\simp|_{-A}}$ by \Cref{lem:minimal-face}~(2) (and~\eqref{eq:spanA}),
        for any $p = \sum_{i = 1}^n p_i \alpha_i \in \aff{\face{\B{\rho}}{\eps \hat{\alpha}}}$,
        the $i$th coefficient $p_i$ of $p$ with respect to $\simp$ coincides with that $\eps k_i$ of $\eps \hat{\alpha}$ by $v_i \in \tau(A)$.
        This implies that $p_i$ is not an integer by $i \in I_{1/2}$.
        Thus, we obtain $\aff{\face{\B{\rho}}{\eps\hat{\alpha}}} \cap \Mdom = \emptyset$, particularly, $\face{\B{\rho}}{\eps\hat{\alpha}} \cap \Mdom = \emptyset$.
        Therefore, $\face{\B{\rho}}{\eps\hat{\alpha}}$ is not $\Mdom$-integral, which contradicts the $\Mdom$-integrality of $\B{\rho}$ by \Cref{lem:integral-polyhedron}.

        (3).
        If $\Set{v_i \in D(\RS)}{i \in \supp_\simp{\hat{\alpha}}} \subseteq D(\RS) \setminus \tau(A)$,
        then $\hat{\alpha} \in \spn{\simp|_{-A}} = \spn{\RS|_{-A}}$,
        which contradicts the maximality of $\eps$ by the same argument as in the proof of the assertion~(1).
        The latter $\Set{v_i \in D(\RS)}{i \in I_1} \cap \tau(A) \neq \emptyset$ follows from
        $\supp_\simp{\hat{\alpha}} = I_{1/2} \cup I_1$, the assertion~(2), and the former statement $\Set{v_i \in D(\RS)}{i \in \supp_\simp{\hat{\alpha}}} \cap \tau(A) \neq \emptyset$.

        (4).
        Since $I_1 \neq \emptyset$ by the assertion~(3),
        at least one of $\eps k_i$ must be a positive integer in the representation $\hat{\alpha} = \sum_{i = 1}^n k_i \alpha_i$.
        Thus, by $\eps < 1$, we obtain $k_i \geq 2$ for some $i \in [n]$.
    \end{proof}

The remainder of the proof is type-dependent,
and requires the additional specific arguments on the root systems.

(Type A).
Since $\hat{\alpha}$ is a positive root with respect to $\Delta$,
we have $\hat{\alpha} = \sum_{i \in I} \alpha_i$ for some $I \subseteq [n]$ (see \Cref{tab:classical-positive-roots}),
i.e.,
the coefficient of $\alpha_i$ is at most one;
this contradicts \Cref{cl:positive-F}~(4).
This means that, in the case of type A, the strict inequality $\eps < 1$ cannot occur,
that is,
$\hat{\alpha} \in \bar{\RS}_0$ implies $\hat{\alpha} \in \RS_0$.

For the remaining types B, C, and D,
we further review several facts as follows.
Let $\hat{\rho} \coloneqq (\rho - \eps \hat{\alpha})/A$,
which is a spherical submodular function on the spherical expression $\spCC^\star / A$ of $\spCC(\RS|_{-A})$ by \Cref{lem:subm:contraction}.
By \Cref{lem:minimal-face}~(3),
we have $\face{\B{\rho}}{\eps\hat{\alpha}} = \eps \hat{\alpha} + \B{\hat{\rho}}$.
Moreover, $\RS|_{-A}$ is not necessarily irreducible,
i.e.,
$\RS|_{-A}$ may admit the decomposition $(\RS|_{-A})_1 \oplus (\RS|_{-A})_2 \oplus \cdots \oplus (\RS|_{-A})_k$
for the irreducible components $(\RS|_{-A})_1, (\RS|_{-A})_2, \dots, (\RS|_{-A})_k$.
By \Cref{lem:subm:components}~(2),
$\B{\hat{\rho}}$ is representable as
$\B{\hat{\rho}} = \B{\hat{\rho}_1} + \B{\hat{\rho}_2} + \cdots + \B{\hat{\rho}_k}$,
where $\hat{\rho}_\ell$ denotes the $\ell$th component of $\hat{\rho}$.
Since $\hat{\rho}(x) > 0$ for all $x \in \sk{\spCC^\star / A}$ (\Cref{lem:minimal-face}~(2)),
all of $\B{\hat{\rho}_1}, \dots, \B{\hat{\rho}_k}$ contain the origin $0$.
Therefore, for any subset $S \subseteq [k]$,
we have
\begin{align}\label{eq:F}
    \eps \hat{\alpha} + \sum_{\ell \in S} \B{\hat{\rho}_\ell} \subseteq \face{\B{\rho}}{\eps\hat{\alpha}} \subseteq \B{\rho}.
\end{align}
Our aim is to show the existence of $\alpha^+,\alpha^-\in \B{\rho} \cap \RS$ such that $\hat{\alpha} = \alpha^+ + \alpha^-$, which implies that $\hat{\alpha} \in \cone{\RS_0}$.
Here, a spherical submodular function $\rho$ is said to be \emph{half-integral} if $2\rho$ is integral.

We are ready to consider the remaining types.

(Type B).
By \Cref{cl:positive-F}~(4),
$\hat{\alpha}$ is a positive root of the form
\begin{align}
    \hat{\alpha} = \sum_{i = i_1}^{i_2-1} \alpha_i + \sum_{j = i_2}^n 2 \alpha_j
\end{align}
for some $i_1, i_2 \in [n]$ with $1 \leq i_1 < i_2 \leq n$ (see \Cref{tab:classical-positive-roots}).
Since $I_1 \neq \emptyset$ (\Cref{cl:positive-F}~(3)),
we obtain $\eps = 1/2$.
Therefore,
\begin{align}
    \eps \hat{\alpha} = \sum_{i = i_1}^{i_2-1} \frac{1}{2}\alpha_i + \sum_{j = i_2}^n \alpha_j  \in \B{\rho},
\end{align}
implying that $I_{1/2} = [i_1, i_2-1]$ and $I_1 = [i_2, n]$.

\begin{claim}\label{cl:half-integral-subm}
    \begin{enumerate}[label={\textup{(\arabic*)}}]
        \item All of $\hat{\rho}, \hat{\rho}_1, \dots, \hat{\rho}_k$ are half-integral spherical submodular functions.
        \item There is a connected component $J$ of the Dynkin diagram $D(\RS|_{-A})$ containing all $v_i$ with $i \in [i_1, i_2 - 1]$.
        In addition, the induced subgraph of $D(\RS)$ by $J$ forms a Dynkin diagram of type~A.
        \item Let $\alpha' \coloneqq \sum_{i = i_1}^{i_2-1} \alpha_i$.
    There is $\ell \in [k]$ such that $\B{\hat{\rho}_\ell}$ contains both of $\pm \alpha'/2$.
    \end{enumerate}
\end{claim}
\begin{proof}
    (1).
    Since $\rho$ is integral, $\hat{\alpha} \in \Mdom(\RS)$, and $\eps = 1/2$,
    the function $\hat{\rho}$ is a half-integral spherical submodular function by \Cref{lem:rho-p}.
    Since $\hat{\rho}$ is half-integral,
    so are $\hat{\rho}_1, \hat{\rho}_2, \dots, \hat{\rho}_k$ by \Cref{lem:subm:components}~(1).

    (2).
    By \Cref{cl:positive-F}~(2), $\tau(A)$ contains no $v_i$ with $i \in I_{1/2} = [i_1, i_2 - 1]$.
    Hence, there is a connected component $J$ of $D(\RS|_{-A})$ containing all $v_i$ with $i \in [i_1, i_2 - 1]$.
    In addition, by \Cref{cl:positive-F}~(3),
    $\tau(A)$ contains some $v_j$ with $j \in I_1 =[i_2, n]$.
    Hence, this connected component $J$ does not contain $v_n$,
    which implies that 
    the induced subgraph of $D(\RS)$ by $J$
    forms a Dynkin diagram of type~A by \Cref{lem:Dynkin-shape}~(2).

    (3).
    Let $\ell \in [k]$ be the index with $D((\RS|_{-A})_\ell) = J$
    and $\RS_J \coloneqq (\RS|_{-A})_\ell$ for notational simplicity.
    By the assertion~(2) and \Cref{lem:contraction:root-system}~(1),
    $\RS_J$ is a root system of type A.
    Since $\hat{\rho}_\ell > 0$ and $\pm\alpha' \in \spn{\RS_J}$,
    there is a sufficiently small $\eps' > 0$ such that $\pm \eps'\alpha' \in \B{\hat{\rho}_\ell}$.
    Moreover, since $\hat{\rho}_\ell$ is half-integral,
    we can take $\eps' = 1/2$ by the same argument as in case of the type A.
    Thus, we obtain $\pm \alpha'/2 \in \B{\hat{\rho}_\ell}$.
\end{proof}

Let $\alpha_+ \coloneqq (\hat{\alpha} + \alpha')/2 = \sum_{i = i_1}^{n} \alpha_i$
and $\alpha_- \coloneqq (\hat{\alpha}- \alpha')/2 = \sum_{j = i_2}^{n} \alpha_j$.
Then, we have $\alpha_+, \alpha_- \in \RS$ (see \Cref{tab:classical-positive-roots})
and $\alpha_+, \alpha_- \in \eps \hat{\alpha} + \B{\hat{\rho}_\ell} \subseteq \face{\B{\rho}}{\eps\hat{\alpha}} \subseteq \B{\rho}$ by \Cref{cl:half-integral-subm}~(3) and the relation~\eqref{eq:F}.
In particular, $\alpha_+, \alpha_- \in \RS_0$.
Since $\hat{\alpha} = \alpha_+ + \alpha_-$,
we obtain $\hat{\alpha} \in \cone{\RS_0}$.

(Type C).
By \Cref{cl:positive-F}~(4),
$\hat{\alpha}$ is a positive root of the form
\begin{align}
    \hat{\alpha} = \sum_{i = i_1}^{i_2-1} \alpha_i + \sum_{j = i_2}^{n-1} 2 \alpha_j + \alpha_n
\end{align}
for some $i_1, i_2 \in [n]$ with $1 \leq i_1 \leq i_2 \leq n-1$ (see \Cref{tab:classical-positive-roots}).
Since $I_1 \neq \emptyset$ (\Cref{cl:positive-F}~(3)),
we obtain $\eps = 1/2$.
Therefore,
\begin{align}
    \eps \hat{\alpha} = \sum_{i = i_1}^{i_2-1} \frac{1}{2}\alpha_i + \sum_{j = i_2}^{n-1}  \alpha_j + \frac{1}{2}\alpha_n
\end{align}
implying that $I_{1/2} = [i_1, i_2-1] \cup \{n\}$ and $I_1 = [i_2, n-1]$.

By the same proof as \Cref{cl:half-integral-subm},
all of $\hat{\rho}, \hat{\rho}_1, \dots, \hat{\rho}_k$ are half-integral spherical submodular functions,
there is a connected component $J$ of $D(\RS|_{-A})$ containing all $v_i$ with $i \in [i_1, i_2 - 1]$ that is of type A,
and there is $\ell \in [k]$ such that $\B{\hat{\rho}_\ell}$ contains both of $\pm \alpha'/2$.
In addition, by $n \in I_{1/2}$ and \Cref{cl:positive-F}~(2),
there is a different connected component $J'$ of $D(\RS|_{-A})$ from $J$ containing $v_n$,
which is of the form $\{ v_s, v_{s+1}, \dots,v_n \}$ for some $s \in [i_2+1, n-1]$.
Let $\ell' \in [k]$ be the index with $D((\RS|_{-A})_{\ell'}) = J'$
and $\RS_{J'} \coloneqq (\RS|_{-A})_{\ell'}$.
We note that $\ell' \neq \ell$.

\begin{claim}\label{cl:>=1/2}
    For any $x \in \sk{\Lk{\spCC^\star}{A}}$ with $\tau(x) \in J'$,
    we have $\hat{\rho}_{\ell'}(x + \spn{A}) \geq 1/2$.
\end{claim}
\begin{proof}
    Let $p \in [s,n]$.
    The direct calculation of $\inpr{\hat{\alpha}}{\alpha_p^\vee} = \sum_{i = i_1}^{i_2 -1}\inpr{\alpha_i}{\alpha_p^\vee} + 2\sum_{j = i_2}^{n -1}\inpr{\alpha_j}{\alpha_p^\vee} + \inpr{\alpha_n}{\alpha_p^\vee}$,
    based on the definition~\eqref{eq:coroot} of the coroot $\alpha_p^{\vee}$,
    the conditions \eqref{cond:D0}--\eqref{cond:D2}, and the shape of Dynkin diagram of type C (see \Cref{tab:classical-Dynkin-diagrams}),
    produces $\inpr{\hat{\alpha}}{\alpha_p^\vee} = 0$.
    Thus, by the definition~\eqref{eq:reflection-dual} of $s_\alpha$, we obtain $s_{\alpha_p}(\hat{\alpha}) = \hat{\alpha} - \inpr{\hat{\alpha}}{\alpha_p^\vee} \alpha_p = \hat{\alpha}$.
    Since $W_{\RS_{J'}} = \gen{s_{\alpha_p}}{ p \in [s,n]}$ by \Cref{lem:Weyl-simple-system} and \Cref{lem:contraction:root-system}~(1),
    we have $w \hat{\alpha} = \hat{\alpha}$ for any $w \in W_{\RS|_{J'}}$.
    Furthermore, since the coefficient of $\alpha_p$ in $\hat{\alpha}$ is $a(x_p)$,
    we have $\inpr{\hat{\alpha}}{x_p} = a(x_p)\inpr{\alpha_p}{x_p} = 1$ for any $p \in [s,n]$.

    Take any $x \in \sk{\Lk{\spCC^\star}{A}}$ with $\tau(x) \in J'$, say, $\tau(x) = v_p$.
    Since there is $w \in W_{\RS_{J'}}$ such that $x = wx_p$,
    we obtain $\inpr{\hat{\alpha}}{x} = \inpr{\hat{\alpha}}{wx_p} = \inpr{w^{-1}\hat{\alpha}}{x_p} = \inpr{\hat{\alpha}}{x_p} = 1$.
    Thus, $\hat{\rho}_{\ell'}(x + \spn{A}) = \hat{\rho}(x + \spn{A}) = (\rho - \hat{\alpha}/2)(x) = \rho(x) - 1/2 \in \Z/2 \cup \{+\infty\}$,
    where $\rho(x) \in \Z/a(x) \cup \{+\infty\}$ and $a(x) \in \{1,2\}$.
    Since $\hat{\rho}_{\ell'}$ only takes a positive value,
    we obtain $\hat{\rho}_{\ell'}(x + \spn{A}) \geq 1/2$.
\end{proof}
Since $|\inpr{\alpha_n}{x}| \leq 1$ for any $x \in \sk{\spCC^\star}$,
we obtain $|\inpr{\alpha_n/2}{x}| \leq 1/2$.
Thus, by the above claim, $|\inpr{\alpha_n/2}{x}| \leq \hat{\rho}_{\ell'}(x + \spn{A})$ holds for any $x \in \sk{\Lk{\spCC^\star}{A}}$ with $\tau(x) \in J'$,
that is, $\pm\alpha_n/2 \in \B{\hat{\rho}_{\ell'}}$.

Let $\alpha_+ \coloneqq (\hat{\alpha} + \alpha' + \alpha_n)/2 = \sum_{i = i_1}^n \alpha_i$ and $\alpha_- \coloneqq (\hat{\alpha} - \alpha' - \alpha_n)/2 = \sum_{i = i_2}^{n - 1} \alpha_i$.
Then, we have $\alpha_+, \alpha_- \in \RS$ (see \Cref{tab:classical-positive-roots}) and $\alpha_+, \alpha_- \in \eps \hat{\alpha} + \B{\hat{\rho}_\ell} + \B{\hat{\rho}_{\ell'}} \subseteq \face{\B{\rho}}{\eps\hat{\alpha}} \subseteq \B{\rho}$ by the above arguments and the relation~\eqref{eq:F}.
In particular, $\alpha_+, \alpha_- \in \RS_0$.
Since $\hat{\alpha} = \alpha_+ + \alpha_-$,
we obtain $\hat{\alpha} \in \cone{\RS_0}$.

(Type D).
Since $\RS$ is irreducible and the root system of type $\Dynkin{D}{3}$ is isomorphic to that of type $\Dynkin{A}{3}$ by \Cref{lem:Dynkin-shape}~(1),
we may assume that $n \geq 4$.
By \Cref{cl:positive-F}~(4),
$\hat{\alpha}$ is a positive root of the form
\begin{align}
    \hat{\alpha} = \sum_{i = i_1}^{i_2-1} \alpha_i + \sum_{j = i_2}^{n-2} 2 \alpha_j + \alpha_{n-1} + \alpha_n
\end{align}
for some $i_1, i_2 \in [n]$ with $1 \leq i_1 < i_2 \leq n-2$ (see \Cref{tab:classical-positive-roots}).
Since $I_1 \neq \emptyset$ (\Cref{cl:positive-F}~(3)),
we obtain $\eps = 1/2$.
Therefore,
\begin{align}
    \eps \hat{\alpha} = \sum_{i = i_1}^{i_2-1} \frac{1}{2}\alpha_i + \sum_{j = i_2}^{n-2}  \alpha_j + \frac{1}{2}\alpha_{n-1} + \frac{1}{2}\alpha_n,
\end{align}
implying that $I_{1/2} = [i_1, i_2-1] \cup \{n-1, n\}$ and $I_1 = [i_2, n-2]$.

By the same proof as \Cref{cl:half-integral-subm},
all of $\hat{\rho}, \hat{\rho}_1, \dots, \hat{\rho}_k$ are half-integral spherical submodular functions,
there is a connected component $J$ of $D(\RS|_{-A})$ containing all $v_i$ with $i \in [i_1, i_2 - 1]$ that is of type A,
and there is $\ell \in [k]$ such that $\B{\hat{\rho}_\ell}$ contains both of $\pm \alpha'/2$.
In addition, by $n-1, n \in I_{1/2}$ and \Cref{cl:positive-F}~(2),
there is a connected component $J'$ (resp. $J''$) of $D(\RS|_{-A})$ containing $v_{n-1}$ (resp. $v_n$) that is different from $J$.
Let $\ell', \ell'' \in [k]$ be the indices with $D((\RS|_{-A})_{\ell'}) = J'$ and $D((\RS|_{-A})_{\ell''}) = J''$, respectively,
and define $\RS_{J'} \coloneqq (\RS|_{-A})_{\ell'}$ and $\RS_{J''} \coloneqq (\RS|_{-A})_{\ell''}$.

We further consider the following two cases:
(i) $\tau(A)$ contains $v_{n-2}$, or equivalently, $J'$ and $J''$ are distinct;
and (ii) $\tau(A)$ does not contain $v_{n-2}$, or equivalently, $J' = J''$.

(i). In this case, we have $J' = \{ v_{n-1} \}$ and $J'' = \{ v_n \}$.
Then, both $\RS_{J'}$ and $\RS_{J''}$ are root systems of type $\Dynkin{A}{1}$.
Since $\hat{\rho}_{\ell'}, \hat{\rho}_{\ell''}$ are half-integral,
we obtain $\pm \alpha_{n-1}/2 \in \B{\hat{\rho}_{\ell'}}$ and $\pm \alpha_n/2 \in \B{\hat{\rho}_{\ell''}}$.
Let $\alpha_+ \coloneqq (\hat{\alpha} + \alpha' + \alpha_{n-1} + \alpha_n)/2 = \sum_{i = i_1}^n \alpha_i$
and $\alpha_- \coloneqq (\hat{\alpha} - \alpha' - \alpha_{n-1} - \alpha_n)/2 = \sum_{i = i_2}^{n-2} \alpha_i$.
Then, we have $\alpha_+, \alpha_- \in \RS$ (see \Cref{tab:classical-positive-roots}) and $\alpha_+, \alpha_- \in \eps \hat{\alpha} + \B{\hat{\rho}_\ell} + \B{\hat{\rho}_{\ell'}} + \B{\hat{\rho}_{\ell''}} \subseteq \face{\B{\rho}}{\eps\hat{\alpha}} \subseteq \B{\rho}$ by the above arguments and the relation~\eqref{eq:F}.
In particular, $\alpha_+, \alpha_- \in \RS_0$.
Since $\hat{\alpha} = \alpha_+ + \alpha_- \in \cone{\RS_0}$,
we obtain $\hat{\alpha} \in \cone{\RS_0}$.

(ii).
In this case, we have $J' = J'' = \{ v_s, v_{s+1}, \dots, v_{n-1}, v_n \}$ for some $s \in [i_2 + 1, n-2]$.
By the same proof as \Cref{cl:>=1/2} (with the shape of the Dynkin diagram of type D; see \Cref{tab:classical-Dynkin-diagrams}),
$\hat{\rho}_{\ell'}(x + \spn{A}) \geq 1/2$ holds for $x \in \sk{\Lk{\spCC^\star}{A}}$ with $\tau(x) \in J'$.

By the direct calculation based on the definition~\eqref{eq:reflection-dual} of $s_\alpha$, the conditions \eqref{cond:D0}--\eqref{cond:D2}, and the shape of the Dynkin diagram of type D (see \Cref{tab:classical-Dynkin-diagrams}),
for $t \in [s, n]$ and $t' \in [s, n-2]$,
we can obtain
\begin{align}
    s_{\alpha_t}\left(\sum_{i = t'}^{n-2} 2\alpha_i + \alpha_{n-1} + \alpha_n\right)
    &=
    \begin{dcases}
        \sum_{i = t'-1}^{n-2} 2\alpha_i + \alpha_{n-1} + \alpha_n & \text{if $t = t'-1$},\\
        \sum_{i = t'+1}^{n-2} 2\alpha_i + \alpha_{n-1} + \alpha_n & \text{if $t = t'$},\\
        \sum_{i = t'}^{n-2} 2\alpha_i + \alpha_{n-1} + \alpha_n & \text{otherwise},
    \end{dcases}\\
    s_{\alpha_t}(\alpha_{n-1} - \alpha_n) &=
    \begin{cases}
        -(\alpha_{n-1} + \alpha_n) & \text{if $t = n-1$},\\
        \alpha_{n-1} + \alpha_n & \text{if $t = n$},\\
        \alpha_{n-1} - \alpha_n & \text{otherwise}.
    \end{cases}
\end{align}
Hence, letting
\begin{align}
    \Mdom' \coloneqq \Set*{ \pm\left( \sum_{i = t'}^{n-2} 2\alpha_i + \alpha_{n-1} + \alpha_n \right) }{ t' \in [s, n-1] } \cup \{ \pm(\alpha_{n-1} - \alpha_n) \} \subseteq \Mdom,\notag
\end{align}
we have $W_{\RS_{J'}} \Mdom' = \Mdom'$.
Moreover, for any $p \in \Mdom'$ and $t \in [s, n]$,
we obtain $|\inpr{p}{x_t}| \leq 1$, since the coefficient of each $\alpha_t$ in $p \in \Mdom'$ is $\pm a(x_t)$ or $0$.
Thus, for any $x \in \sk{\Lk{\spCC^\star}{A}}$ with $\tau(x) \in J'$, or equivalently, $t \in [s,n]$ and $w \in W_{\RS_{J'}}$ with $x = w x_t$,
we obtain $|\inpr{\alpha_{n-1} + \alpha_{n}}{x}| = |\inpr{\alpha_{n-1} + \alpha_{n}}{wx_t}| = |\inpr{w^{-1}(\alpha_{n-1} + \alpha_{n})}{x_t}| \leq 1$.
This implies that $\pm(\alpha_{n-1}+\alpha_n)/2 \in \B{\hat{\rho}_{\ell'}}$.

Let $\alpha_+ \coloneqq (\hat{\alpha} + \alpha' + \alpha_{n-1}+\alpha_n)/2 = \sum_{i = i_1}^n \alpha_i$
and $\alpha_- \coloneqq (\hat{\alpha} - \alpha' - \alpha_{n-1}-\alpha_n)/2 = \sum_{i = i_2}^{n-2} \alpha_i$.
Then,
we have $\alpha_+, \alpha_- \in \RS$ (see \Cref{tab:classical-positive-roots}) and $\alpha_+, \alpha_- \in \eps \hat{\alpha} + \B{\hat{\rho}_\ell} + \B{\hat{\rho}_{\ell'}} \subseteq \face{\B{\rho}}{\eps\hat{\alpha}} \subseteq \B{\rho}$ by the above arguments and the relation~\eqref{eq:F}.
In particular, $\alpha_+, \alpha_- \in \RS_0$.
Since $\hat{\alpha} = \alpha_+ + \alpha_-$,
we obtain $\hat{\alpha} \in \cone{\RS_0}$.
\end{proof}

The following is a direct consequence of the proof of \Cref{prop:tangent-cone},
which is of independent interest.
\begin{corollary}
    Let $\RS$ be an irreducible root system,
    $B$ an M-convex set of type $\RS$,
    and $p \in B$.
    Suppose that $\alpha \in \RS$ belongs to $\tcone{\conv{B}}{p}$,
    and
    let $\eps \coloneqq \sup{\Set{ \eps' }{ p + \eps' \alpha \in \conv{B} }}$.
    Then, the following hold.
    \begin{enumerate}[label={\textup{(\arabic*)}}]
        \item If $\RS$ is of type A, then $\eps \geq 1$, i.e., $p + \alpha \in B$.
        \item If $\RS$ is of type B, C, or D,
        then either $\eps \geq 1$, i.e., $p + \alpha \in B$,
        or $\eps = 1/2$ and there are two roots $\alpha_+, \alpha_- \in \RS$ such that $p + \alpha_+, p + \alpha_- \in B$ and $\alpha = \alpha_+ + \alpha_-$.
    \end{enumerate}
\end{corollary}

\section{Characterizations of discrete convexity}\label{sec:chara}
In this section, we provide combinatorial characterizations for the discrete convex sets and functions defined in \Cref{sec:DCA:discrete}. 
\Cref{subsec:chara:L-convex-sets,subsec:chara:L-convex-functions} characterize L-convex sets and functions in terms of the discrete midpoint convexity. 
\Cref{subsec:chara:M-convex-sets} characterizes M-convex sets and functions through hole-freeness and the maximality property.
Finally, \Cref{subsec:chara:M-convex-functions} characterizes M-convex functions based on the M-convexity of their effective domains, the M-sublinearity of their (discrete version of) directional derivatives, and a local-to-global lower bound condition.

\subsection{L-convex sets}\label{subsec:chara:L-convex-sets}
The following theorem summarizes characterizations of L-convex sets,
including the \emph{discrete midpoint convexity} corresponding to the condition~(d) below.
\begin{theorem}\label{thm:chara:L-convex-set}
    For a nonempty set $D \subseteq \Ldom$,
    the following are equivalent:
    \begin{enumerate}[label={\textup{(\alph*)}}]
        \item $D$ is an L-convex set.
        \item $\isc{\euCC}{D}$ is a convex subcomplex of $\euCC$
        \item $\ext{D}$ is convex in $V$.
        \item $\sk{\cell{A}{(x+y)/2}} \subseteq D$ for $x, y \in D$.
    \end{enumerate}
\end{theorem}
\begin{proof}
    (a) $\Leftrightarrow$ (b) $\Leftrightarrow$ (c) immediately follows from the definition of L-convexity and \Cref{lem:convex_subcomplex,lem:chara-integral-L-convex-poly}.

    (c) $\Rightarrow$ (d).
    Recall that $\ext{D}$ is the geometric realization $\geo{\isc{\euCC}{D}}$ of the subcomplex $\isc{\euCC}{D}$ induced by $D$.
    Take any $x, y \in D$.
    Since $\ext{D}$ is convex in $V$, we have $(x+y)/2 \in \ext{D} = \geo{\isc{\euCC}{D}}$,
    which implies $\cell{A}{(x+y)/2} \in \isc{\euCC}{D}$, i.e., $\sk{\cell{A}{(x+y)/2}} \subseteq D$.

    (d) $\Rightarrow$ (a).
    Let $D \subseteq \Ldom$ be a nonempty set satisfying the condition in the assertion~(d).
    It suffices to show that $\isc{\euCC}{D}$ is convex,
    i.e., for any $A, B \in \isc{\euCC}{D}$, we have $AB \in \isc{\euCC}{D}$.
    Indeed, if $\isc{\euCC}{D}$ is convex,
    then $\geo{\isc{\euCC}{D}}$ is an integral L-convex polyhedron by \Cref{lem:chara-integral-L-convex-poly},
    and hence $D = \geo{\isc{\euCC}{D}} \cap \Ldom$ is an L-convex set.

    We prove this by induction on $d(A,B)$.
    If $d(A,B) = 0$, then $A$ and $B$ are joinable and $\sk{AB} = \sk{A} \cup \sk{B}$.
    Hence we obtain $AB \in \isc{\CC}{D}$, since $\sk{A} \cup \sk{B} \subseteq D$.

    Suppose that $d(A,B) > 0$.
    We consider the following three cases: (i) $A$ and $B$ are vertices, (ii) $A$ is a vertex, and (iii) $A$ and $B$ are general simplices.

    (i). Suppose that $\sk{A} = \{x\}$ and $\sk{B} = \{y\}$.
    We first observe that, for any $w \in (x, y)$,
    we have $xy = x\cell{A}{w}$.
    Indeed, the convex combination $(1-\eps)x + \eps y$ for sufficiently small $\eps > 0$ can be simulated by some convex combination of $x$ and $w$.
    Set $z \coloneqq (x+y)/2 \in (x, y)$.
    The condition (d) says that $\cell{A}{z} \in \isc{\CC}{D}$.
    By \Cref{lem:sign-sequence}~(7),
    we have $d(x, y) \geq d(x, \cell{A}{z})$.
    If $d(x, y) > d(x, \cell{A}{z})$,
    then by induction hypothesis we have $\isc{\CC}{D} \ni x\cell{A}{z} = xy$.
    If $d(x, y) = d(x, \cell{A}{z})$,
    then \Cref{lem:sign-sequence}~(7) says that $x$ is a vertex of $\cell{A}{z}$,
    which implies that $xy = x \cell{A}{z} = \cell{A}{z}$.
    Hence $xy \in \isc{\CC}{D}$ holds.

    (ii). Suppose that $A = x$ and $\sk{B} = \{ y_0, y_1, \dots, y_k \}$.
    Since $d(x, B) \geq d(x, y_i)$,
    by induction hypothesis and the argument in (i)
    we have $xy_i \in \isc{\CC}{D}$ for each $i$.
    Furthermore, $xy_0, xy_1, \dots, xy_k$ are joinable and its join is
    $(xy_0)(xy_1) \cdots (xy_k) = xy_0\cdots y_k = AB$,
    in which the first equality follows from \Cref{lem:sign-sequence}~(2) and the second follows from $A = x$ and $B = y_0 y_1 \dots y_k$.
    Hence we obtain $AB \in \isc{\CC}{D}$.

    (iii).
    Suppose that $\sk{A} = \{ x_0, \dots, x_k \}$ and $\sk{B} = \{ y_0, \dots, y_\ell \}$.
    Then, we have $d(A, B) = d(x_0 x_1 \cdots x_k, B) \geq d(x_0 x_1 \dots x_{k-1}, x_kB) \geq \cdots \geq d(x_0, x_{1} \cdots x_k B)$,
    where the inequalities follow from \Cref{lem:sign-sequence}~(4).
    Hence, by induction hypothesis and the argument in (ii),
    we obtain $x_0x_{1} \cdots x_k B = AB \in \isc{\CC}{D}$.
\end{proof}

Using the explicit descriptions of $\sk{\cell{A}{(x+y)/2}}$ for $x, y \in \sk{\euCC(\RS)}$ given in \Cref{subsec:DMC},
we can rephrase the condition (d) of \Cref{thm:chara:L-convex-set} with more explicit formulae for $\RSA$, $\RSB$, $\RSC$, and $\RSD$,
which are given in \Cref{cor:DMC:set}.

\subsection{L-convex functions}\label{subsec:chara:L-convex-functions}

As in the case of L-convex sets, we can characterize L-convex functions by the discrete midpoint convexity.
Here, the \emph{discrete midpoint convexity} is referred to as the condition for a function $\funcdoms{f}{\Ldom}{\ER}$ with $\dom{f} \neq \emptyset$ that,
for each $x, y \in \Ldom$, it satisfies
\begin{align}\label{eq:midpointconvexity}
    f(x) + f(y) \geq 2\sum_{v \in \Ldom} \cell{\lambda}{z}(v) f(v),
\end{align}
where $z$ is the midpoint $(x+y)/2$ of $x$ and $y$.
\begin{theorem}\label{thm:chara:L-conv-func}
    For a function $\funcdoms{f}{\Ldom}{\ER}$ with $\dom{f} \neq \emptyset$,
    the following are equivalent:
    \begin{enumerate}[label={\textup{(\alph*)}}]
        \item $f$ is L-convex.
        \item $\ext{f}$ is convex in $V$.
        \item $f$ satisfies the discrete midpoint convexity~\eqref{eq:midpointconvexity}.
        \item $\dom{f}$ is L-convex and $f$ satisfies the discrete midpoint convexity~\eqref{eq:midpointconvexity} for any two adjacent maximal simplices $A_1, A_2$ of $\euCC[\dom{f}]$,
        $x \in \sk{A_1} \setminus \sk{A_2}$, and $y \in \sk{A_2} \setminus \sk{A_1}$.
    \end{enumerate}
\end{theorem}
\begin{proof}
  The implication (a) $\Rightarrow$ (b) immediately follows from \Cref{prop:L-conv:unique}.

  (b) $\Rightarrow$ (a).
  Since the affine hyperplane arrangement $\affcalH(\RS)$ is locally finite, i.e., every point of $V$ has a neighborhood that meets only finitely many affine hyperplanes in $\affcalH(\RS)$ (see e.g.,~\cite[Definition~10.7]{Abramenko2010}) and
  $\ext{f}$ is affine on the closure of each simplex,
  we can conclude that
  $\ext{f}$ is locally polyhedral by its convexity.
  
  Let $P = \argmin(\ext{f} - p) \in \fan(\ext{f})$ and take $x \in P$.
Since $\ext{f}-p$ is affine on $\cl{\cell{A}{x}}$ and attains its minimum
at $x \in\cell{A}{x} = \relint(\cl{\cell{A}{x}})$, every vertex of $\cell{A}{x}$ also belongs to $P$.
Thus, $P$ is the geometric realization of a subcomplex of
$\euCC$.
In addition,
since $P$ is convex, this subcomplex is convex by \Cref{lem:convex_subcomplex}, and hence $P$ is an integral
L-convex polyhedron by \Cref{lem:chara-integral-L-convex-poly}.
  Therefore,
  $\ext{f}$ is a primal-integral locally polyhedral L-convex function.
  Since $f = \ext{f}|_{\Ldom}$,
  we conclude that $f$ is L-convex.

  (b) $\Rightarrow$ (c).
  By the definition of $\ext{f}$ and its convexity, for $x,y \in \Ldom$, we have
  $f(x) + f(y) = \ext{f}(x) + \ext{f}(y) \geq 2\ext{f}((x+y)/2) = 2\sum_{v \in \Ldom} \cell{\lambda}{(x+y)/2}(v) f(v)$.

  (c) $\Rightarrow$ (d).
  The discrete midpoint convexity~\eqref{eq:midpointconvexity} of $f$
  implies that $\sk{\cell{A}{(x+y)/2}} \subseteq \dom{f}$ for any $x, y \in \dom{f}$.
  Thus, by \Cref{thm:chara:L-convex-set},
  $\dom{f}$ is an L-convex set.
  The latter statement is immediate.

  (d) $\Rightarrow$ (b).
  Suppose that $f$ satisfies the condition~(d).
  Recall that $\dom{\ext{f}} = \geo{\euCC[\dom{f}]}$ by the definition of the Lov\'{a}sz extension.
  Since $\dom{f}$ is an L-convex set,
  $\euCC[\dom{f}]$ is a nonempty convex subcomplex of $\euCC$ by \Cref{thm:chara:L-convex-set},
  implying that $\euCC[\dom{f}]$ is a chamber complex by \Cref{lem:chamber-complex}.

  We see the following claim:
  \begin{claim}\label{cl:A1-U-A2}
    For any two adjacent maximal simplices (chambers) $A_1, A_2$ of $\euCC[\dom{f}]$,
    the set $\cl{A_1} \cup \cl{A_2}$ is convex.
    In addition,
    the restriction of $\ext{f}$ to $\cl{A_1} \cup \cl{A_2}$
    is convex.
  \end{claim}
  \begin{proof}[Proof of \Cref{cl:A1-U-A2}]  
    After restricting the affine hyperplane arrangement $\affcalH(\RS)$ to $\aff{\geo{\euCC[\dom{f}]}}$, the simplices $A_1$ and $A_2$ are adjacent chambers separated by $H \coloneqq \aff(\cl{A_1}\cap\cl{A_2})$,
    implying that
    their sign sequences differ only at this hyperplane, or equivalently,
    that they are strictly separated by exactly one hyperplane
    (see, e.g.,~\cite[Proposition~3.78]{Abramenko2010}).
    Hence, $\cl{A_1}\cup\cl{A_2}$ is
    the intersection of the closed half-spaces determined by their
    common signs at all hyperplanes other than $H$, and is therefore
    convex.

    Since $\ext{f}$ is affine on $\cl{A_1}$ and on $\cl{A_2}$,
    the restriction of $\ext{f}$ to $\cl{A_1} \cup \cl{A_2}$ is the piecewise affine function consisting of two pieces $\cl{A_1}$ and $\cl{A_2}$.
    Let $\ell_1$ (resp. $\ell_2$) denote the affine extension of $\ext{f}|_{\cl{A_1}}$ (resp. $\ext{f}|_{\cl{A_2}}$) to $\aff{\geo{\euCC[\dom{f}]}}$.
    Then, we have $\ell_1(z) = \ell_2(z)$ for $z \in H$.
    Let $x \in \sk{A_1} \setminus \sk{A_2}$ and $y \in \sk{A_2} \setminus \sk{A_1}$;
    we may assume that $(x+y)/2 \in \cl{A_1}$.
    The condition~(d), or more precisely,
    the discrete midpoint convexity~\eqref{eq:midpointconvexity} for $x$ and $y$,
    implies that
    \begin{align}
      \ell_1(x) + \ell_2(y) = \ext{f}(x) + \ext{f}(y) \geq 2 \ext{f}((x+y)/2) = 2 \ell_1((x+y)/2) = \ell_1(x) + \ell_1(y).
    \end{align}
    Thus, $\ell_2(y) \geq \ell_1(y)$, or equivalently, $\ell_2(z) \geq \ell_1(z)$ holds for $z \in \cl{A_2}$.
    In addition,
    since the affine function $\ell_2-\ell_1$ vanishes on the hyperplane $H$ and
$A_1$ and $A_2$ lie on opposite sides of $H$,
the inequality $\ell_1(z) \geq \ell_2(z)$ holds for $z \in \cl{A_1}$.
Thus, we obtain $\ext{f}(z) = \max\{ \ell_1(z), \ell_2(z)\}$ for $z \in \cl{A_1} \cup \cl{A_2}$,
which proves that $\ext{f}$ is convex on $\cl{A_1} \cup \cl{A_2}$.
  \end{proof}

  Take any $a, b \in \dom{\ext{f}} = \geo{\euCC[\dom{f}]}$;
  our aim is to show that $\ext{f}$ is convex on the closed segment $[a, b]$,
  which implies that $\ext{f}$ is convex.
  Note that $[a, b] \subseteq \dom{\ext{f}}$ by the convexity of $\dom{\ext{f}}$.
  In the following argument, codimension is taken relative to $\geo{\euCC[\dom{f}]}$.

  We first consider the case where the segment
$[a,b]$ is contained in the closure of some simplex,
or
the relative interior $(a,b)$ meets no simplex of codimension at
least two;
we say that the pair $(a,b)$ is \emph{in simple position} in this case. 

By the local finiteness of the affine hyperplane arrangement
$\affcalH(\RS)$,
the segment $[a,b]$ meets only finitely many affine hyperplanes.
Thus, there are distinct points $a \eqqcolon a_0,a_1,\dots,a_k \coloneqq b$
appearing in this order on $[a,b]$ and maximal simplices $A_1,A_2,\ldots,A_k$
of $\euCC[\dom{f}]$ such that $[a_{i-1},a_i] \subseteq \cl{A_i}$.
If $k = 1$, i.e., $[a,b]$ is included in the closure of some simplex, then $\ext{f}$ is affine, and hence, convex on $[a,b]$.
Suppose that $k \geq 2$.
By the assumption, $A_i$ and $A_{i+1}$ are
adjacent for every $i\in[k-1]$.
It follows from \Cref{cl:A1-U-A2} and the inclusion $[a_{i-1}, a_{i+1}] \subseteq \cl{A_{i}} \cup \cl{A_{i+1}}$
that
$\ext{f}$ is convex on $[a_0, a_2], [a_1, a_3], [a_2, a_4], \dots, [a_{k-2}, a_{k}]$.
Since consecutive intervals overlap in a nondegenerate interval,
these convexity properties can be successively glued together.
Thus, $\ext{f}$ is convex on $[a_0, a_{k}] = [a, b]$.

We then consider the case where the pair $(a,b)$ is not in simple position,
i.e., $[a,b]$ is not contained in the closure of a single simplex
and its relative interior meets some simplex of codimension at
least two.
By the local finiteness of $\affcalH(\RS)$,
there are 
only finitely many simplices of $\euCC[\dom{f}]$ that meet $\conv(\{a\} \cup B(b; \eps))$,
where $B(b; \eps)$ denotes the open ball centered at $b$ with radius $\eps$.
Let $B_1, B_2, \dots, B_k$ be all of such simplices of $\euCC[\dom{f}]$ having codimension at least two.
Then, the set $a + \cone(B_i - a)$ is of codimension at least one for each $i$.
Let $B' \coloneqq (B(b; \eps) \cap \geo{\euCC[\dom{f}]}) \setminus \bigcup_{i = 1}^k(a + \cone(B_i - a))$.
Since $B'$ is a nonempty relatively open subset of $\aff{\geo{\euCC[\dom{f}]}}$,
it cannot be covered
by a finite union of subsets of codimension at least one.
Therefore, we obtain $B'\neq\emptyset$.
By the definition of $B'$,
the pair $(a, b')$ is in simple position for any $b' \in B'$.

By the above argument,
there is a sequence $(b_m)_{m = 1,2,\dots}$ such that $(a, b_m)$ is in simple position for all $m$
and $\lim_{m \to \infty} b_m = b$.
Since $[a, b_m]$ is in simple position,
for every $\lambda \in [0,1]$, we have $(1-\lambda) \ext{f}(a) + \lambda \ext{f}(b_m) \geq \ext{f}((1-\lambda)a + \lambda b_m)$.
Since $\ext{f}$ is continuous on $\geo{\euCC[\dom{f}]}$,
letting $m \to \infty$, we obtain $(1-\lambda) \ext{f}(a) + \lambda \ext{f}(b) \geq \ext{f}((1-\lambda)a + \lambda b)$,
which implies that $\ext{f}$ is convex on $[a,b]$.
\end{proof}

\begin{remark}\label{rmk:proof:subm}
    We can show the implication (c) $\Rightarrow$ (b) of \Cref{thm:submodular-chara} as follows.
    Fix $x_* \in \Ldom$
    and let $\funcdoms{\rho}{\nbor{x_*}}{\ER}$ be a function with $\rho(x_*) < +\infty$
    satisfying the submodular inequality~\eqref{eq:subm-ineq}.
    Let $\rho_{\Ldom}$ be the function obtained from $\rho$ by extending the domain to $\Ldom$,
    i.e.,
    \begin{align}
        \rho_{\Ldom}(x) \coloneqq
        \begin{cases}
            \rho(x) & \text{if $x \in \nbor{x_*}$},\\
            +\infty & \text{otherwise}.
        \end{cases}
    \end{align}
    Then, $\rho_{\Ldom}$ clearly satisfies the discrete midpoint inequality~\eqref{eq:midpointconvexity},
    which implies that $\ext{\rho_{\Ldom}}$ is convex in $V$ by (c) $\Rightarrow$ (b) of \Cref{thm:chara:L-conv-func}.
    Thus, $\ext{\rho}$ is convex,
    since $\ext{\rho}$ is the restriction of $\ext{\rho_{\Ldom}}$ to $\geo{\St{\euCC}{x_*}}$,
    where we recall that $\geo{\St{\euCC}{x_*}}$ is a convex set by \Cref{lem:star_complex}.
\end{remark}

As in \Cref{subsec:chara:L-convex-sets},
using the explicit descriptions of $\cell{\lambda}{z} \in \cvcoeff{\sk{\cell{A}{z}}}$ for $x, y \in \sk{\euCC(\RS)}$ with $z \coloneqq (x+y)/2$ given in \Cref{subsec:DMC},
we can rephrase the condition (d) of \Cref{thm:chara:L-conv-func} with more explicit formulae for $\RSA$, $\RSB$, $\RSC$, and $\RSD$,
which are given in \Cref{cor:DMC:func}.

\subsection{M-convex sets}\label{subsec:chara:M-convex-sets}
We summarize characterizations of M-convex sets, including the maximality property~\eqref{cond:MP} (corresponding to the condition~(c)) and the discrete tangent-cone property (corresponding to the condition~(d)), as follows.
Here,
a subset $B \subseteq \Mdom$ is said to be \emph{hole-free} if $B = (\conv{B}) \cap \Mdom$.
\begin{theorem}\label{thm:chara:M-convex-set}
    For a nonempty set $B \subseteq \Mdom$,
    the following are equivalent:
    \begin{enumerate}[label={\textup{(\alph*)}}]
        \item $B$ is an M-convex set.
        \item $B = \B{\rho} \cap \Mdom$ for some integral spherical submodular function $\rho$.
        \item $B$ is a hole-free set having the maximality property~\eqref{cond:MP}.
        \item $B$ is a hole-free set and $q-p \in \cone(\tRS{p}{B})$ for any $p,q \in B$,
        i.e.,
        there exist $\beta_1, \beta_2, \dots, \beta_k \in \tRS{p}{B}$ and nonnegative reals $\mu_1, \mu_2, \dots, \mu_k \in \R_+$ such that
        $p + \sum_{\ell = 1}^k \mu_\ell \beta_\ell = q$.
    \end{enumerate}
\end{theorem}
\begin{proof}
    (a) $\Leftrightarrow$ (b) follows from \Cref{cor:Mpoly-subm}.

    (a) $\Rightarrow$ (d).
    Let $B$ be an M-convex set.
    The hole-freeness of $B$ is clear.
    Since $\conv{B}$ is an integral M-convex polyhedron,
    for any $p, q \in B$,
    we have $q - p \in \tcone{\conv{B}}{p} = \cone(\tRS{p}{B})$,
    where the identity follows from \Cref{prop:tangent-cone}.
    Thus, the assertion~(d) holds.

    (d) $\Rightarrow$ (c).
    It suffices to show that $B$ has the maximality property~\eqref{cond:MP}.
    Fix any simple system $\Delta = \{ \alpha_1, \alpha_2, \dots, \alpha_n \}$
    and
    take any $p = \sum_{i = 1}^n p_i \alpha_i, q = \sum_{i = 1}^n q_i \alpha_i \in B$.
    Let $r = \sum_{i = 1}^n r_i \alpha_i \coloneqq p$.
    By the assumption,
    there exist $\beta_1, \beta_2, \dots, \beta_k \in \RS$ and nonnegative reals $\mu_1, \mu_2, \dots, \mu_k \in \R_+$ satisfying
    \begin{align}
        r + \sum_{\ell = 1}^k \mu_\ell \beta_\ell = q \qquad \text{and} \qquad r+\beta_\ell \in B \quad (\ell \in [k]).
    \end{align}
    Suppose that $r_i < q_i$.
    Since $r + \sum_{\ell = 1}^k \mu_\ell \beta_\ell = q$,
    there is at least one of $\beta_1, \beta_2, \dots, \beta_k$, say $\beta_1 = \sum_{i = 1}^n b_i \alpha_i$, such that $b_i > 0$.
    Hence we obtain that $r_i < r_i + b_i$ and $\beta_1$ is a positive root with respect to $\simp$, i.e., $r <_\simp r+\beta_1$.
    We update $r \leftarrow r + \beta_1 \in B$.
    Since each positive root has nonnegative integral coefficients with respect to $\simp$, the value $\sum_{i=1}^n \max\{q_i-r_i,0\}$
decreases by at least one at each update.
Hence, the procedure terminates after finitely many steps;
    we finally obtain $r \in B$ such that $r \geq_\Delta p$ and $r \geq_\Delta q$.
    Thus, $B$ has \eqref{cond:MP}.

    (c) $\Rightarrow$ (a).
    By the if part of \Cref{thm:chara:M-poly}~(2),
    which has been already shown,
    we have
    $\conv{B}$ is an integral M-convex polyhedron.
    Since $B$ is hole-free by the assumption,
    we obtain $B = (\conv{B}) \cap \Mdom$,
    which implies that $B$ is an M-convex set, i.e., the assertion~(a).
\end{proof}

\begin{remark}\label{rmk:Mpoly:proof}
    We are ready to complete the proof of \Cref{thm:chara:M-poly}~(2).
    If $P$ is an integral M-convex polyhedron,
    then $B \coloneqq P \cap \Mdom$ forms an M-convex set
    and it has~\eqref{cond:MP} by (a) $\Rightarrow$ (c) of \Cref{thm:chara:M-convex-set}.
    Since $P = \conv{B}$,
    we obtain the only-if part of \Cref{thm:chara:M-poly}~(2).
\end{remark}

\subsection{M-convex functions}\label{subsec:chara:M-convex-functions}
The following theorem provides an intrinsic characterization of M-convex functions,
in terms of the discrete tangent-cone property related to the condition~(d) of \Cref{thm:chara:M-convex-set}. 
\begin{theorem}\label{thm:chara-M-conv-func}
    A function $\funcdoms{g}{\Mdom}{\ER}$ with $\dom{g} \neq \emptyset$ is M-convex if and only if it satisfies the following conditions:
    \begin{itemize}
        \item $\dom{g}$ is an M-convex set;
        \item for any $p \in \dom{g}$, the function $\gamma_{g,p}$ is valid,
        or equivalently, $\cone{\gamma_{g,p}}$ is M-sublinear; and
        \item for any $p, q \in \dom{g}$,
        we have
        $g(q) \geq g(p) + (\cone{\gamma_{g,p}})(q-p)$,
        or equivalently,
        there exist $\beta_1, \beta_2, \dots, \beta_k \in \RS_p(\dom{g})$ and nonnegative reals $\mu_1, \mu_2, \dots, \mu_k \in \R_+$ satisfying
        \begin{align}\label{eq:ineq:M-convex-func}
            p + \sum_{\ell = 1}^k \mu_\ell \beta_\ell = q \quad \text{and} \quad g(q) \geq g(p) + \sum_{\ell = 1}^k \mu_\ell \gamma_{g,p}(\beta_\ell).
        \end{align}
    \end{itemize}
\end{theorem}
Before the proof, we review a general fact on the tangent cone of a polyhedron.
Recall again that,
for a polyhedron $P \subseteq V^*$ and $p \in P$,
let $\face{P}{p}$ denote the unique face of $P$ satisfying $p \in \relint{\face{P}{p}}$.
\begin{lemma}\label{lem:tcone-fan}
    Let $P$ be a polyhedron and $p, q \in P$
    such that $q \in \face{P}{p}$.
    Then, we have $\tcone{\tcone{P}{q}}{p-q} = \tcone{P}{p}$.
\end{lemma}
\begin{proof}
    Suppose that $\tcone{P}{q} = \cone\{d_1, d_2, \dots, d_m\}$.
    Since $P$ is a polyhedron,
    for each $i \in [m]$,
    there is $\eps_i > 0$ such that $q + \eps_i d_i \in P$.
    Let $P' \coloneqq q + \conv\{0,\eps_1 d_1, \eps_2 d_2, \dots, \eps_m d_m\} \subseteq P$.
    Then, there is a closed neighborhood $N_q$ of $q$ such that $N \coloneqq P' \cap N_q = P \cap N_q = (\tcone{P}{q} + q) \cap N_q$.
    Thus, for $p' \coloneqq q + \eps(p-q) \in N \cap \relint{\face{P}{p}}$ with a sufficiently small $\eps > 0$,
    we have $\tcone{N}{p'} = \tcone{P}{p'} = \tcone{P'}{p'} = \tcone{\tcone{P}{q}}{p' - q}$.
    Furthermore, since $\face{P}{p'} = \face{P}{p}$ and $\face{(\tcone{P}{q})}{p' - q} = \face{(\tcone{P}{q})}{p - q}$,
    we obtain $\tcone{P}{p'} = \tcone{P}{p}$ and $\tcone{\tcone{P}{q}}{p'-q} = \tcone{\tcone{P}{q}}{p-q}$ by \Cref{lem:poly:tcone}.
    Thus, $\tcone{P}{p} = \tcone{\tcone{P}{q}}{p-q}$ holds.
\end{proof}

\begin{proof}[Proof of \Cref{thm:chara-M-conv-func}]
  The only-if part immediately follows from \Cref{prop:dom:LM-convex} and \Cref{prop:directional:derivative:M}.

We show the if part.
Suppose that $\funcdoms{g}{\Mdom}{\ER}$ satisfies the three conditions in the statement.
Since $\dom{g}$ is an M-convex set, $\conv(\dom{g})$ is an integral M-convex polyhedron,
particularly, a closed convex set.
Thus, we have $\dom(\clconv{g}) = \conv(\dom{g})$ by \Cref{lem:conv=clconv}.

Our aim in the following is to show that
\begin{enumerate}[label={\textup{(\Roman*)}}, ref={\textup{\Roman*}}]
    \item \label{assertion:I}
    there is a subset $\fan' \subseteq \fan(\clconv{g})$ such that every member of $\fan'$ is an integral M-convex polyhedron and $\conv(\dom{g}) = \bigcup_{P \in \fan'} P$; and
    \item \label{assertion:II}
    $g = \clconv{g}|_{\Mdom}$.
\end{enumerate}
If \eqref{assertion:I} holds, then $\clconv{g}$ is a primal-integral locally polyhedral M-convex function by \Cref{prop:loc-poly-LM-convex-sufficient-condition}~(2); and if \eqref{assertion:II} holds, then $g$ is the restriction of $\clconv{g}$ to $\Mdom$.
Thus, $g$ forms an M-convex function.

To this end, we first prove the following.
Here, we define
\begin{align}
    \barfan(g) &\coloneqq \Set{ \conv(B) }{ B \in \fan(g) },\\
    \barfan(g; q) &\coloneqq \Set{ P \in \barfan(g) }{ q \in P } \qquad (q \in V^*).
\end{align}
Also, let $\geo{\barfan(g)} \coloneqq \bigcup_{P \in \barfan(g)} P$.
\begin{claim}\label{cl:barfan}
    \begin{enumerate}[label={\textup{(\arabic*)}}]
        \item Every $B \in \fan(g)$ is an M-convex set.
        \item Let $p \in \dom{g}$ and $x \in V$.
        Then, $p \in \argmin(g-x)$ if and only if $(\gamma_{g, p} - x)(\alpha) \geq 0$ for all $\alpha \in \RS$.
        \item Let $p \in \dom{g}$ and $x \in V$.
        Then, $\argmin(\cone{\gamma_{g,p}} - x) \in \fan(\cone{\gamma_{g,p}})$ if and only if
        $\argmin(g-x) \in \fan(g;p)$.
        In addition, we have $\argmin(\cone{\gamma_{g,p}} - x) = \tcone{\conv(\argmin(g-x))}{p}$ if it is nonempty.
        \item For $x \in V$ with $\argmin(g-x) \neq \emptyset$,
        we have $\argmin(\clconv{g} - x) = \conv(\argmin(g-x))$.
        In particular, we have $\barfan(g) \subseteq \fan(\clconv{g})$.
        \item $\barfan(g)$ forms a polyhedral complex consisting of integral M-convex polyhedra.
        \item For $p \in \geo{\barfan(g)}$ and $q \in \conv(\dom{g})$, there is $P \in \barfan(g; p)$ such that $q-p \in \tcone{P}{p}$.
    \end{enumerate}
\end{claim}
\begin{proof}[Proof of \Cref{cl:barfan}]
    (1).
    Take any $B = \argmin(g-x) \in \fan(g)$.
    We only see the case of $x = 0$, i.e., $B = \argmin{g} \neq \emptyset$;
    the proof for the general case is essentially the same.
    Our aim is to show that $\argmin{g}$ satisfies the condition~(d) of \Cref{thm:chara:M-convex-set},
    implying that $\argmin{g}$ is an M-convex set.

    Suppose to the contrary that $\argmin{g}$ is not hole-free.
    Take $p \in (\conv(\argmin{g}) \cap \Mdom) \setminus \argmin{g}$.
    Since $p$ belongs to $\conv(\argmin{g})$,
    there is $\lambda \in \cvcoeff{\argmin{g}}$ with $p = \sum_{q \in \argmin{g}} \lambda(q)q$.
    Suppose that $\supp{\lambda} = \{p_1, p_2, \dots, p_m\}$.
    Since $\dom{g}$ is hole-free and $p \in \conv(\dom{g}) \cap \Mdom$,
    we have $p \in \dom{g}$.
    Thus, we obtain $\min{g} < g(p) < +\infty$.

    By the assumption,
    for each $i \in [m]$,
    there are $\beta_1^i, \beta_2^i, \dots, \beta_{k_i}^i \in \RS_p(\dom{g})$ and positive reals $\mu_1^i, \mu_2^i, \dots, \mu_{k_i}^i$ such that
    \begin{align}
        \sum_{\ell = 1}^{k_i} \mu_\ell^i \beta_\ell^i = p_i - p \quad \text{and} \quad \sum_{\ell = 1}^{k_i} \mu_\ell^i \gamma_{g,p}(\beta_\ell^i) \leq g(p_i) - g(p) < 0.
    \end{align}
    Thus, we obtain
    \begin{align}
        \sum_{i = 1}^m \lambda_i \left(\sum_{\ell = 1}^{k_i} \mu_\ell^i \beta_\ell^i\right) = \sum_{i = 1}^m \lambda_i (p_i - p) = 0 \quad \text{and} \quad \sum_{i = 1}^m \lambda_i \left(\sum_{\ell = 1}^{k_i} \mu_\ell^i \gamma_{g,p}(\beta_\ell^i)\right) < 0,
    \end{align}
    implying that $\phext{\gamma_{g,p}}(0) = -\infty$, i.e., $\gamma_{g,p}$ is not valid.
    This contradicts the assumption that $\gamma_{g, p}$ is valid.
    Therefore, $\argmin{g}$ is hole-free.

    Take any $p, q \in \argmin{g}$.
    By the assumption,
    there are $\beta_1, \beta_2, \dots, \beta_k \in \RS_p(\dom{g})$ and nonnegative reals $\mu_1, \mu_2, \dots, \mu_k \in \R_+$
    such that $p + \sum_{\ell = 1}^k \mu_\ell \beta_\ell = q$
    and $g(q) \geq g(p) + \sum_{\ell = 1}^k \mu_\ell \gamma_{g,p}(\beta_\ell)$;
    we may assume that $\mu_\ell>0$ for all $\ell\in[k]$.
    Since $p \in \argmin{g}$,
    we have $\gamma_{g,p}(\beta_\ell) = g(p + \beta_\ell) - g(p) \geq 0$ for all $\ell$.
    Thus, by $g(q) \geq g(p) + \sum_{\ell = 1}^k \mu_\ell \gamma_{g,p}(\beta_\ell)$
    and $q \in \argmin{g}$,
    we obtain $g(p + \beta_\ell) = g(p) = \min{g}$ for all $\ell$,
    which implies that
    $p + \beta_\ell \in \argmin{g}$,
    or equivalently, $\beta_\ell \in \tRS{p}{\argmin{g}}$.
    Therefore,
    $\argmin{g}$ is an M-convex set by \Cref{thm:chara:M-convex-set}.

    (2).
    We only consider the case of $x = 0$;
    the proof for the general case is essentially the same.

    (Only-if part).
    Suppose that $p \in \argmin{g}$.
    Then, for all $\alpha \in \RS$, we have $g(p+\alpha) \geq g(p)$, implying that $\gamma_{g,p}(\alpha) \geq 0$.

    (If part).
    Suppose that there is $\hat{p} \in \dom{g}$ with $g(\hat{p}) < g(p)$.
    Since $(\cone{\gamma_{g,p}})(\hat{p} - p) \leq g(\hat{p}) - g(p) < 0$,
    there is $\alpha \in \dom{\gamma_{g,p}} \subseteq \RS$
    such that $\gamma_{g,p}(\alpha) < 0$,
    which implies that $g(p+\alpha) = g(p) + \gamma_{g,p}(\alpha)  < g(p)$ for some $\alpha \in \RS$.

    (3).
    Fix an arbitrary $p \in \dom{g}$.
    For $x \in V$, we have
    \begin{align}
      \argmin(\cone{\gamma_{g,p}} - x) \in \fan(\cone{\gamma_{g,p}}) 
      &\iff \argmin(\cone(\gamma_{g,p} - x)) \in \fan(\cone{\gamma_{g,p}})\\
      &\iff (\gamma_{g,p}-x)(\alpha) \geq 0 \quad (\alpha \in \RS)\\
      &\iff p \in \argmin(g-x)\\
      &\iff \argmin(g-x) \in \fan(g; p),
    \end{align}
    where the first equivalence follows from $\cone{\gamma_{g,p}} - x = \cone(\gamma_{g,p} - x)$
    and the third follows from the assertion~(2).
    This implies the first statement.

    In addition, for $x \in V$ with $K \coloneqq \argmin(\cone{\gamma_{g,p}} - x) \in \fan(\cone{\gamma_{g,p}})$, or equivalently, $B \coloneqq \argmin(g-x) \in \fan(g; p)$,
    we have
    \begin{align}
        K = \cone{\Set{ \alpha \in \RS }{ (\gamma_{g,p} - x)(\alpha) = 0 }}
        = \cone(\tRS{p}{B})
        = \tcone{\conv{B}}{p},
    \end{align}
    where the last equality follows from \Cref{prop:tangent-cone}.
    Thus, the second statement follows.

    (4).
    We first see that
    \begin{align}\label{eq:clconv-g-lower-bound}
        \clconv{g}(q) \geq (\cone{\gamma_{g,p}})(q-p) + g(p)
    \end{align}
    holds for $p \in \dom{g}$ and $q \in V^*$.
    Fix an arbitrary $p \in \dom{g}$.
    If $q \notin \dom(\clconv{g})$, then the above inequality clearly holds.
    Take any $q \in \dom(\clconv{g}) = \conv(\dom{g})$.
    Then, we have
    \begin{align}
        q - p \in \conv(\dom{g}) - p \subseteq \tcone{\conv(\dom{g})}{p} = \cone(\tRS{p}{\dom{g}}) = \dom(\cone{\gamma_{g,p}})\notag,
    \end{align}
    where the first equality follows from \Cref{prop:tangent-cone}.
    Since $\cone{\gamma_{g,p}}$ is a polyhedral sublinear function,
    we have $\dom(\cone{\gamma_{g,p}}) = \geo{\fan(\cone{\gamma_{g,p}})}$ by \Cref{lem:quasi-polyhedral-convex-function}~(1),
    i.e.,
    there is $x \in V$ such that $q - p \in \argmin(\cone{\gamma_{g,p}} - x) \in \fan(\cone{\gamma_{g,p}})$.
    Fix such $x$.
    Since $\cone{\gamma_{g,p}}$ is linear on $\argmin(\cone{\gamma_{g,p}} - x)$,
    we obtain $(\cone{\gamma_{g,p}})(q-p) = \inpr{q-p}{x}$.
    Moreover, it follows from $\argmin(\cone{\gamma_{g,p}} - x) \in \fan(\cone{\gamma_{g,p}})$ and the former statement of the assertion~(3) that $\argmin(g-x) \in \fan(g; p)$, i.e., $p \in \argmin(g-x)$.
    Thus, $g(p') \geq g(p) + \inpr{p'-p}{x}$ holds for any $p' \in \Mdom$,
    which implies that $\clconv{g}(q) \geq g(p) + \inpr{q-p}{x}$ by the definition~\eqref{eq:func:clconv} of $\clconv{g}$.
    By combining them,
    we obtain $\clconv{g}(q) \geq g(p) + (\cone{\gamma_{g,p}})(q-p)$.

    We are ready to show the assertion~(4).
    We only consider the case where $x = 0$ and $B \coloneqq \argmin{g} \in \fan(g)$;
    the proof for the general case is essentially the same.
    Since $\argmin{g} \neq \emptyset$,
    we have $\min{g} = \min{\clconv{g}}$,
    which immediately implies
    the inclusion $\argmin(\clconv{g}) \supseteq \conv{B}$.
    Thus,
    it suffices to show the reverse inclusion $\argmin(\clconv{g}) \subseteq \conv{B}$.

    Take any $p \notin \conv{B}$; our aim is to prove that $p \notin \argmin(\clconv{g})$.
    Since $B$ is an M-convex set by the assertion~(1),
    there is an integral spherical submodular function $\rho$ such that $\conv{B} = \B{\rho}$.
    Hence, there is $\hat{x} \in \dom{\rho}$ such that $\inpr{p}{\hat{x}} > \rho(\hat{x})$.
    Since $\B{\rho} = \B{\cone{\rho}}$ and $\rho(\hat{x}) = (\cone{\rho})(\hat{x})$,
    the set $P' \coloneqq \Set{ q \in \B{\rho} }{ \inpr{q}{\hat{x}} = \rho(\hat{x}) }$
    is a nonempty face of $\B{\rho}$ by \Cref{lem:conjugate-poly-sublinear}.
    In addition, $P'$ is an integral M-convex polyhedron by \Cref{lem:face:M-poly}.
    Thus, $B' \coloneqq P' \cap \Mdom = \Set{ q \in B }{ \inpr{q}{\hat{x}} = \rho(\hat{x}) }$ is an M-convex set.

    Fix an arbitrary $\hat{p} \in B'$.
    Since $\hat{p} \in B = \argmin{g}$,
    we have $\argmin(\cone{\gamma_{g, \hat{p}}}) \in \fan(\cone{\gamma_{g,\hat{p}}})$
    by the former statement of assertion~(3).
    Thus,
    the latter statement of the assertion~(3) implies that
    \begin{align}
        (\cone{\gamma_{g, \hat{p}}})(q)
        \begin{cases}
            = 0 & \text{if $q \in \tcone{\conv{B}}{\hat{p}}$},\\
            > 0 & \text{otherwise}.
        \end{cases}
    \end{align}
    Since $\tcone{\conv{B}}{\hat{p}} \subseteq \Set{ q - \hat{p} \in V^* }{ \inpr{q - \hat{p}}{\hat{x}} \leq 0 } = \Set{ q \in V^* }{ \inpr{q}{\hat{x}} \leq \rho(\hat{x}) } - \hat{p}$ and $\inpr{p}{\hat{x}} > \rho(\hat{x})$,
    we have $p - \hat{p} \notin \tcone{\conv{B}}{\hat{p}}$, i.e., $(\cone{\gamma_{g, \hat{p}}})(p - \hat{p}) > 0$.
    Thus,
    \begin{align}
        \clconv{g}(p) \geq (\cone{\gamma_{g, \hat{p}}})(p - \hat{p}) + g(\hat{p}) > g(\hat{p}) = \min(\clconv{g})
    \end{align}
    holds,
    where the first inequality follows from~\eqref{eq:clconv-g-lower-bound}.
    This implies $p \notin \argmin(\clconv{g})$.

    (5).
    By the assertion~(1),
    every $P \in \barfan(g)$ is an integral M-convex polyhedron,
    which implies that any nonempty face of $P \in \barfan(g)$ is an integral M-convex polyhedron by \Cref{lem:face:M-poly}.
    Hence, it suffices to show that $\barfan(g)$ forms a polyhedral complex, i.e.,
    (i) if $P \in \barfan(g)$ and $Q$ is a nonempty face of $P$, then $Q \in \barfan(g)$;
    and
    (ii) if $P, Q \in \barfan(g)$ with $P \cap Q \neq \emptyset$, then $P \cap Q \in \barfan(g)$.

    (i).
    Take any $P \in \barfan(g)$ and any nonempty face $Q$ of $P$.
    Since $P$ is an integral M-convex polyhedron, so is $Q$ by \Cref{lem:face:M-poly}.
    Let $B \coloneqq P \cap \Mdom$ and $B' \coloneqq Q \cap \Mdom$ be the M-convex sets.
    We note that $B \in \fan(g)$.

    Choose an arbitrary $p \in B' \subseteq B$.
    By the assertion~(3),
    we have
    $\tcone{P}{p} \in \fan(\cone{\gamma_{g,p}})$.
    In addition,
    since
    $\tcone{Q}{p}$ is a nonempty face of $\tcone{P}{p}$
    and $\fan(\cone{\gamma_{g,p}})$ forms a polyhedral fan,
    we obtain $\tcone{Q}{p} \in \fan(\cone{\gamma_{g,p}})$.
    Thus, by the assertion~(3), there is $x' \in V$ such that
    \begin{align}
        \argmin(\cone{\gamma_{g,p}} - x') = \tcone{Q}{p} = \tcone{\conv(\argmin(g-x'))}{p}.
    \end{align}

    Let $P' \coloneqq \conv(\argmin(g-x'))$;
    our aim is to show that $Q = P'$,
    which implies that $Q \in \barfan(g)$.
    By the assertion~(4),
    we have $P, P' \in \fan(\clconv{g})$.
    It follows from the identity $\tcone{P'}{p} = \tcone{Q}{p}$
    that $\relint{P'} \cap \relint{Q} \neq \emptyset$,
    implying further that $\relint{P'} \cap P \neq \emptyset$.
    By \Cref{lem:argmin-support}~(2), $P'$ is a face of $P$.
    Since $Q$ is also a face of $P$ and $\relint{P'} \cap \relint{Q} \neq \emptyset$, we obtain $P' = Q$.

(ii).
Take any $P, Q \in \barfan(g)$ with $P \cap Q \neq \emptyset$.
Suppose that $P = \conv(\argmin(g-x))$ and $Q = \conv(\argmin(g-x'))$.
By the assertion~(4),
we have $P, Q \in \fan(\clconv{g})$,
particularly,
$P = \argmin(\clconv{g} - x)$ and $Q = \argmin(\clconv{g} - x')$.
Furthermore,
since $P \cap Q \neq \emptyset$,
we obtain
$P \cap Q = \argmin(\clconv{g} - x) \cap \argmin(\clconv{g} - x') = \argmin(\clconv{g} - (x+x')/2) \in \fan(\clconv{g})$,
implying that
$P \cap Q$ forms a nonempty face of both $P$ and $Q$.
Thus, it follows from (i) that $P \cap Q \in \barfan(g)$.

(6).
Take any $p \in \geo{\barfan(g)}$ and $q \in \conv(\dom{g})$.
    Since $p \in \geo{\barfan(g)}$ and $\barfan(g)$ is a polyhedral complex by the assertion~(5),
    there is a (unique) integral M-convex polyhedron $\hat{P} \in \barfan(g)$ such that $p \in \relint(\hat{P})$.

    Choose an arbitrary $\hat{p} \in \hat{P} \cap \Mdom$.
    Since $\hat{P} \in \barfan(g; \hat{p})$,
    we have $\tcone{\hat{P}}{\hat{p}} \in \fan(\cone{\gamma_{g,\hat{p}}})$
    by the assertion~(3).
    Moreover,
    we obtain
    \begin{align}
        [p,q] - \hat{p} \in \tcone{\conv(\dom{g})}{\hat{p}} = \cone(\tRS{\hat{p}}{\dom{g}}) = \dom(\cone{\gamma_{g,\hat{p}}}) = \geo{\fan(\cone{\gamma_{g,\hat{p}}})},\notag
    \end{align}
    where the first equality follows from \Cref{prop:tangent-cone}
    and the last follows from \Cref{lem:quasi-polyhedral-convex-function}
    and the fact that $\cone{\gamma_{g, \hat{p}}}$ is a polyhedral sublinear function.
    Since $\fan(\cone{\gamma_{g,\hat{p}}})$ is a polyhedral fan consisting of finitely many polyhedral cones,
    there is a sufficiently small $\delta > 0$ such that $[p, p + \delta(q-p)] - \hat{p}$ belongs to some $K \in \fan(\cone{\gamma_{g,\hat{p}}})$,
    which implies $q - p \in \tcone{K}{p-\hat{p}}$.
    In addition,
    by the assertion~(3),
    there is $P \in \barfan(g; \hat{p})$ with $K = \tcone{P}{\hat{p}}$.
    Since $p - \hat{p} \in K = \tcone{P}{\hat{p}}$ and $P$ is a polyhedron,
    we have $\hat{p} + \eps(p - \hat{p}) \in P$
    for a sufficiently small $\eps > 0$.
    On the other hand, by $p \in \relint{\hat{P}}$ and $\hat{p} \in \hat{P}$,
    we obtain $P \cap \relint{\hat{P}} \neq \emptyset$.
    Thus, by \Cref{lem:argmin-support}~(2) and the assertion~(4), $\hat{P}$ is a face of $P$
    and $\face{P}{p} = \hat{P}$.
    By \Cref{lem:tcone-fan},
    we obtain $\tcone{P}{p} = \tcone{\tcone{P}{\hat{p}}}{p - \hat{p}} = \tcone{K}{p - \hat{p}} \ni q - p$.
\end{proof}

We are ready to show the assertions~\eqref{assertion:I} and \eqref{assertion:II},
which complete the proof of the if part (and hence \Cref{thm:chara-M-conv-func}).

We first consider the assertion~\eqref{assertion:II}.
Fix any $p \in \Mdom$.
Since $\dom{g}$ is hole-free,
if $p \notin \dom{g}$,
then $p \notin \conv(\dom{g}) = \dom(\clconv{g})$.
Thus, $g(p) = \clconv{g}(p) = +\infty$.
Suppose that $p \in \dom{g}$.
Since $\fan(\cone{\gamma_{g,p}})$ is a nonempty polyhedral fan,
there is $x \in V$ such that $\argmin(\cone{\gamma_{g,p}} - x) \in \fan(\cone{\gamma_{g,p}})$,
which implies that $p \in \argmin(g-x)$ by \Cref{cl:barfan}~(3).
For such $x$,
by \Cref{cl:barfan}~(4),
we have $\argmin(\clconv{g} - x) = \conv(\argmin(g-x)) \ni p$.
Hence, we obtain $\clconv{g}(p) - \inpr{p}{x} = \min(\clconv{g}-x) = \min(g-x) = g(p) - \inpr{p}{x}$,
which implies that $\clconv{g}(p) = g(p)$.
Thus, the assertion~\eqref{assertion:II} holds.

We finally see the assertion~\eqref{assertion:I}.
By
\Cref{cl:barfan}~(4),
we have $\barfan(g) \subseteq \fan(\clconv{g})$.
Furthermore, by \Cref{cl:barfan}~(1),
every member of $\barfan(g)$ is an integral M-convex polyhedron.
Thus, it suffices to prove that $\conv(\dom{g}) = \geo{\barfan(g)}$.

Suppose to the contrary that $\conv(\dom{g}) \supsetneq \geo{\barfan(g)}$.
Take any $q \in \conv(\dom{g}) \setminus \geo{\barfan(g)}$.
Since $\geo{\barfan(g)}$ is a closed set by \Cref{lem:union-closed:M},
the minimum of $\|q - p'\|_2$ over $p' \in \geo{\barfan(g)}$ exists,
which is strictly positive.
Let $p \in \argmin_{p' \in \geo{\barfan(g)}} \|q-p'\|_2$.
Then, by \Cref{cl:barfan}~(6),
there is $P \in \barfan(g; p)$ such that $q - p \in \tcone{P}{p}$,
i.e., $\hat{p} \coloneqq p + \eps(q-p) \in P \subseteq \geo{\barfan(g)}$ for a sufficiently small $\eps > 0$.
This contradicts the choice of $p$ as the closest point in $\geo{\barfan(g)}$ to $q$,
since $\|q - \hat{p}\|_2 < \|q - p\|_2$.
\end{proof}

\begin{remark}
  The condition that $\dom{g}$ is an M-convex set in \Cref{thm:chara-M-conv-func} can be replaced by the condition that $\dom{g}$ is a nonempty hole-free subset of $\Mdom$.
  Indeed, the nonempty hole-freeness of $\dom{g}$ and the third condition~\eqref{eq:ineq:M-convex-func} in \Cref{thm:chara-M-conv-func} imply that $\dom{g}$ forms an M-convex set by the implication (d) $\Rightarrow$ (a) in \Cref{thm:chara:M-convex-set}.
\end{remark}

\section{Connections with existing concepts}\label{sec:examples}
In this section, we clarify how the discrete convexity notions
introduced in this paper are related to existing classes.
Our purpose is to determine when the present notions coincide with
known ones and, when they do not, to identify the precise differences
in their discrete domains, integrality conditions, and conditions on
function values.

In type~A, the present framework recovers the original DCA,
including L${}^\natural$- and M${}^\natural$-convex functions and
ordinary submodular functions.
In types~B, C, and~D, we compare the present notions with
(valuated) $\Delta$-matroids, (valuated) even $\Delta$-matroids,
bisubmodular functions, BS-convex sets and functions, and UJ-convex
functions.
For general root systems, we compare M-convex polytopes and
spherical submodular functions with Coxeter matroids and Coxeter
submodular functions, respectively.
These comparisons show which existing notions are recovered as
special cases of the present framework and where the present
definitions differ from them.

Several of these correspondences require a change of domain or a
normalization, such as taking the quotient by the all-one direction,
passing to a zero-sum lattice, or rescaling the representative points on the vertices of a spherical Coxeter complex.
We make these identifications explicit and distinguish the
real-valued and integral settings whenever the corresponding notions
of integrality differ.
Throughout this section, the symbol $\simeq$ denotes the natural
correspondence induced by the identification or normalization under
consideration.
The integrality comparisons and explicit midpoint descriptions
used below rely on \Cref{subsec:proof:prop:integral-subm,subsec:DMC}.

\subsection{Original DCA}\label{subsec:original-DCA}

We compare the L- and M-convex functions of type $\RSA$
introduced in this paper with the corresponding function classes in
the original DCA~\cite{Murota1998-yj,Murota2003-yb}.

\subsubsection{L- and L${}^\natural$-convex functions}\label{subsubsec:example:L}
We first present formal definitions of L- and L${}^\natural$-convex functions.
A function $\funcdoms{\tilde{f}}{\Z^{n+1}}{\ER}$ with
$\dom{\tilde{f}} \neq \emptyset$ is said to be \emph{L-convex} if it
satisfies the following two conditions:
\begin{itemize}
  \item It is linear in the direction $\onevec_{n+1}$, i.e., there
  is $r \in \R$ such that
  \begin{align}\label{eq:linear}
    \tilde{f}(x+\onevec_{n+1})
    =
    \tilde{f}(x)+r
    \qquad
    (x\in\Z^{n+1}).
  \end{align}

  \item It satisfies the discrete midpoint convexity, i.e.,
  \begin{align}\label{eq:DMC}
    \tilde{f}(x)+\tilde{f}(y)
    \geq
    \tilde{f}\rbra*{\floor*{\frac{x+y}{2}}}
    +
    \tilde{f}\rbra*{\ceil*{\frac{x+y}{2}}}
    \qquad
    (x,y\in\Z^{n+1}).
  \end{align}
\end{itemize}
Here,
for a real number $r \in \mathbb{R}$, let $\floor{r}$ and $\ceil{r}$ denote the floor function (rounding down to the nearest integer) and the ceiling function (rounding up to the nearest integer), respectively.
A function $f$ on $\Z^{n+1}$ is said to be \emph{translation-invariant} if
$f(x + \onevec_{n+1}) = f(x)$ for all $x \in \Z^{n+1}$.
A translation-invariant L-convex function is exactly an L-convex function with $r = 0$ in the condition~\eqref{eq:linear}.
A function $\funcdoms{f}{\Z^n}{\ER}$ with
$\dom{f} \neq \emptyset$ is said to be
\emph{L${}^\natural$-convex} if there is a translation-invariant L-convex function
$\funcdoms{\tilde{f}}{\Z^{n+1}}{\ER}$ such that
$f(x) = \tilde{f}(x,0)$ for all $x \in \Z^n$.
In other words, an L${}^\natural$-convex function on $\Z^n$ is the
restriction of a translation-invariant L-convex function on $\Z^{n+1}$.
It is well known~\cite[Theorem~7.7]{Murota2003-yb} that the discrete midpoint convexity~\eqref{eq:DMC} characterizes the L${}^\natural$-convexity,
i.e.,
a function $\funcdoms{f}{\Z^n}{\ER}$ with
$\dom{f} \neq \emptyset$ is L${}^\natural$-convex if and only if it satisfies~\eqref{eq:DMC}.
An L-convex or L${}^\natural$-convex function is said to be
\emph{integral} if it is integer-valued.
The classes of (integral) translation-invariant L-convex functions and (integral) L${}^\natural$-convex functions are essentially equivalent.

A translation-invariant function on $\Z^{n+1}$ can be regarded
as a function on the primal discrete domain $\sk{\euCC(\RSA)}$ for type
$\RSA$,
where $\sk{\euCC(\RSA)} = \Z^{n+1}/\spn{\onevec_{n+1}}$ (see~\Cref{tab:classical-coxeter-vertices}).
Indeed, for a translation-invariant function $\tilde{f}$ on $\Z^{n+1}$,
we define the function on $\Z^{n+1} / \spn{\onevec_{n+1}}$ by $[x] \mapsto \tilde{f}(x)$ for $x \in \Z^{n+1}$ and $[x] \coloneqq x + \spn{\onevec_{n+1}} \in \Z^{n+1} / \spn{\onevec_{n+1}}$,
which is well defined.
Conversely, every function on this quotient has a unique
translation-invariant lift to $\Z^{n+1}$ through the quotient map.

The above argument with the characterization of the L-convexity in terms of the discrete midpoint convexity (given in \Cref{cor:DMC:func}) immediately implies that
the class of translation-invariant L-convex functions naturally corresponds to that of L-convex functions of type $\RSA$.
Specifically,
for a translation-invariant L-convex function $\funcdoms{\tilde{f}}{\Z^{n+1}}{\ER}$,
the function
$[x] \mapsto \tilde{f}(x)$ for $x \in \Z^{n+1}$
is L-convex of type $\RSA$.
Conversely, for an L-convex function $\funcdoms{f}{\sk{\euCC(\RSA)}}{\ER}$ of type $\RSA$,
the function $x \mapsto f([x])$ for $x \in \Z^{n+1}$ forms a translation-invariant L-convex function on $\Z^{n+1}$.
Therefore, we obtain the following:
\begin{align}
    &\{ \text{translation-invariant L-convex functions on $\Z^{n+1}$} \}\\
    &\simeq \{ \text{L${}^\natural$-convex functions on $\Z^n$} \}\\
    &\simeq \{ \text{L-convex functions of type $\RSA$} \}.
\end{align}

We next make the correspondence for integral functions precise.
Let $\funcdoms{\tilde{f}}{\Z^{n+1}}{\ER}$ be a
translation-invariant L-convex function, and
$\funcdoms{f}{\sk{\euCC(\RSA)}}{\ER}$ the corresponding
L-convex function of type $\RSA$, i.e.,
$f([x]) = \tilde{f}(x)$ for $x\in\Z^{n+1}$.
By \Cref{thm:integral:L-func}~(2), 
the function $f$ is
integral if and only if there is $c \in \R$ such that $f([x])-c \in \Z$ for all $[x]\in\dom{f}$.
Under the above correspondence, this condition is equivalent to $\tilde{f}(x)-c \in \Z$ for all $x\in\dom{\tilde{f}}$.
Thus, if $\tilde{f}$ is integral in the standard sense,
i.e., if it is integer-valued, then $f$ is integral in the sense
of this paper.
Conversely, if $f$ is integral, then the translation-invariant
lift of $f-c$ is an integer-valued L-convex function for some
$c\in\R$.
Since the correspondence between translation-invariant L-convex
functions on $\Z^{n+1}$ and L${}^\natural$-convex functions on
$\Z^n$ preserves the function values, the same statement holds for
L${}^\natural$-convex functions.
Consequently, we
have the equivalences
\begin{align}
  &\{\text{integral translation-invariant L-convex functions on
  $\Z^{n+1}$ with an additive constant}\}\notag
  \\
  &\simeq
  \{\text{integral L${}^\natural$-convex functions on $\Z^n$ with an additive constant}\}\notag
  \\
  &\simeq
  \{\text{integral L-convex functions of type $\RSA$}\}.\label{eq:example:integral:L}
\end{align}

As a special case, we obtain the relations among (integral) submodular functions on $\{0,1\}^{n+1}$, (integral) spherical submodular functions of type $\RSA$,
and (integral) submodular functions on $\sk{\St{\euCC(\RSA)}{0}}$.
Here, a function $\funcdoms{\rho}{\{0,1\}^{n+1}}{\ER}$ with $\dom{\rho} \neq \emptyset$ is said to be \emph{submodular}
if it satisfies the \emph{submodular inequality}:
\begin{align}
    \rho(x) + \rho(y) \geq \rho(x \wedge y) + \rho(x \vee y) \qquad (x, y \in \{0,1\}^{n+1}),
\end{align}
where $(x \wedge y)(i) \coloneqq \min\{ x(i), y(i) \}$ and $(x \vee y)(i) \coloneqq \max\{ x(i), y(i) \}$ for $i \in [n+1]$.
A submodular function is said to be \emph{integral} if it is integer-valued.
Since $x \wedge y = \floor{(x+y)/2}$ and $x \vee y = \ceil{(x+y)/2}$,
we can represent the submodular inequality as:
$\rho(x) + \rho(y) \geq \rho(\floor{(x+y)/2}) + \rho(\ceil{(x+y)/2})$ for $x, y \in \{0,1\}^{n+1}$.
That is, submodular functions on $\{0,1\}^n$ can be viewed as the restriction of L${}^\natural$-convex functions to $\{0,1\}^n$.

By an argument similar to that for L- and L${}^\natural$-convex functions,
we obtain
\begin{align}
    &\{ \text{(integral) submodular functions $\rho$ on $\{0,1\}^{n+1}$ such that $\rho(0) = \rho(\onevec_{n+1}) < +\infty$} \}\\
    &\simeq \{ \text{(integral) submodular functions $\rho$ of type $\RSA$ (with $\rho([0]) \in \Z$)} \}
\end{align}
and
\begin{align}
    &\{ \text{(integral) submodular functions $\rho$ on $\{0,1\}^{n+1}$ such that $\rho(0) = \rho(\onevec_{n+1}) = 0$} \}\notag\\
    &\simeq \{ \text{(integral) spherical submodular functions of type $\RSA$} \}.\label{eq:subm=typeA}
\end{align}
Here, the condition that $\rho(0) = \rho(\onevec_{n+1})$ corresponds to the translation-invariant condition for L-convex functions.

\subsubsection{M- and M${}^\natural$-convex functions}
The formal definitions of M- and
M${}^\natural$-convex functions are as follows.
A function $\funcdoms{\tilde{g}}{\Z^{n+1}}{\ER}$ with
$\dom{\tilde{g}}\neq\emptyset$ is said to be \emph{M-convex} if,
for any $p,q\in\dom{\tilde{g}}$ and $i\in[n+1]$ with
$p(i)>q(i)$, there is $j\in[n+1]$ with $p(j)<q(j)$ such that
\begin{align}\label{eq:M-exchange-original}
  \tilde{g}(p)+\tilde{g}(q)
  \geq
  \tilde{g}(p-e_i+e_j)
  +
  \tilde{g}(q+e_i-e_j).
\end{align}
We can easily see that the effective domain of an M-convex function $\funcdoms{\tilde{g}}{\Z^{n+1}}{\ER}$ lies on the set $\Set{ p \in \Z^{n+1} }{ \sum_{i = 1}^{n+1} p(i) = r }$ for some $r \in \Z$ (see e.g.,~\cite[Proposition~6.1]{Murota2003-yb}).
A function on $\Z^{n+1}$ is said to be \emph{zero-sum} if $\dom{\tilde{g}} \subseteq \Set{ p \in \Z^{n+1} }{ \sum_{i = 1}^{n+1} p(i) = 0 }$.
Since $\Mdom(\RSA) = \Set{ p \in \Z^{n+1} }{ \sum_{i = 1}^{n+1} p(i) = 0 }$,
a zero-sum function on $\Z^{n+1}$ can be naturally regarded as a function on $\Mdom(\RSA)$.
A function $\funcdoms{g}{\Z^n}{\ER}$ with
$\dom{g}\neq\emptyset$ is said to be
\emph{M${}^\natural$-convex} if there is a zero-sum M-convex function $\funcdoms{\tilde{g}}{\Z^{n+1}}{\ER}$ such that $g(p) = \tilde{g}(p, -\sum_{i = 1}^n p(i))$ for $p \in \Z^n$.
In other words, an M${}^\natural$-convex function on $\Z^n$ is the projection of a zero-sum M-convex function on $\Z^{n+1}$ to $\Z^n$.
An M-convex or M${}^\natural$-convex function is said to be
\emph{integral} if it is integer-valued.

The classes of zero-sum M-convex functions on $\Z^{n+1}$ and
M-convex functions of type $\RSA$ are in fact naturally equivalent,
after identifying the zero-sum lattice with $\Mdom(\RSA)$, i.e.,
we have
\begin{align}
    &\{ \text{zero-sum M-convex functions on $\Z^{n+1}$} \}\\
    &\simeq \{ \text{M${}^\natural$-convex functions on $\Z^n$} \}\\
    &\simeq \{ \text{M-convex functions of type $\RSA$} \}.
\end{align}
Indeed,
it is known~\cite[Theorem~6.43~(1)]{Murota2003-yb} that 
a function
$\funcdoms{\tilde{g}}{\Z^{n+1}}{\ER}$ with $\dom{\tilde{g}} \neq \emptyset$ is M-convex if and only if it satisfies
\begin{itemize}
  \item $\tilde{g}$ is the restriction of $\clconv{\tilde{g}}$ to $\Z^{n+1}$; and
  \item $\fan(\clconv{\tilde{g}})$ consists of integral M-convex polyhedra.
\end{itemize}
Here, it follows from the equivalence~\eqref{eq:subm=typeA}, \Cref{cor:Mpoly-subm}, and \cite[(4.40) of Section~4.8]{Murota2003-yb}, which states the equivalence between the integral M-convex polyhedra and integral submodular functions, that
the integral M-convex polyhedra lying in $\Set{p \in \R^{n+1}}{\sum_{i = 1}^{n+1} p(i) = 0}$ are exactly the same as the integral M-convex polyhedra of type $\RSA$.
Thus, by \Cref{lem:argmin-support}~(1), \Cref{lem:union-closed:M}, and \Cref{prop:loc-poly-LM-convex-sufficient-condition}~(2),
the second bullet is equivalent to the condition that $\clconv{\tilde{g}}$ is a primal-integral locally polyhedral M-convex function.
Hence, $\tilde{g}$ satisfies both bullets if and only if the restriction of $\tilde{g}$ to $\Mdom(\RSA)$ forms an M-convex function of type $\RSA$.

In the integral case,
by combining the equivalence~\eqref{eq:example:integral:L} with the conjugacy between the integral L- and M-convex functions (\Cref{thm:conjugacy} and \cite[Theorem~8.12~(1)]{Murota2003-yb}),
we immediately obtain
\begin{align}
  &\{\text{zero-sum integral M-convex functions on
  $\Z^{n+1}$ with an additive constant}\}\notag
  \\
  &\simeq
  \{\text{integral M${}^\natural$-convex functions on $\Z^n$ with an additive constant}\}\notag
  \\
  &\simeq
  \{\text{integral M-convex functions of type $\RSA$}\}.
\end{align}

\subsection{(Valuated) \texorpdfstring{$\Delta$}{delta}-matroids and even \texorpdfstring{$\Delta$}{delta}-matroids}
\label{subsec:example:delta-matroids}

We compare $\Delta$-matroids and their valuated versions with
M-convex sets and functions of types~$\RSB$ (or $\RSC$) and~$\RSD$.
For notational simplicity,
we identify a subset $X \subseteq [n]$ with its characteristic vector $e_X \coloneqq \sum_{i\in X}e_i \in \{0,1\}^n$,
which induces the identification of $2^{[n]}$ with $\{0,1\}^n$.
In particular, we use the notation $e_X \symdiff Y \coloneqq e_{X \symdiff Y}$ for $X, Y \subseteq [n]$,
where $\symdiff$ denotes the symmetric difference.
For $p \in \{0,1\}^n$,
we refer to $\card{\supp{p}}$ as the \emph{size} of $p$.

A nonempty subset $B \subseteq \{0,1\}^n$
is called a \emph{$\Delta$-matroid}~\cite{Bouchet1987-gl,Dress1986-oh,Chandrasekaran1988-wc} if it satisfies the symmetric
exchange axiom: for any $p,q\in B$ and
$i \in \supp{(p - q)}$, there is
$j \in \supp{(p - q)}$ such that
$p \symdiff \{i,j\} \in B$.
A $\Delta$-matroid $B$ is said to be \emph{even} if the sizes of all
its members have the same parity, i.e.,
$\card{\supp{p}} \equiv \card{\supp{q}} \pmod 2$ for all $p,q\in B$.
That is, either every member has even size or every
member has odd size.
For an (even) $\Delta$-matroid $B$,
its convex hull $\conv{B}$ is called the
\emph{(even) $\Delta$-matroid polytope} of $B$.

The following polyhedral characterizations are well known (\cite{Gelfand1987-mn} and \cite[Section~6.3 and Theorem~6.3.1]{Borovik2011-io}; see also \Cref{subsec:Coxeter-matroids} below).
For a nonempty subset $B \subseteq \{0,1\}^n$,
we have
\begin{align}
    \text{$B$ is a $\Delta$-matroid} &\iff \text{every edge of $\conv{B}$ is parallel to a root in $\RSB$},\label{eq:Delta-typeB}\\
    \text{$B$ is an even $\Delta$-matroid} &\iff \text{every edge of $\conv{B}$ is parallel to a root in $\RSD$}\label{eq:evenDelta-typeD},
\end{align}
Since a tangent cone of a polytope is generated by its incident
edge directions, the above characterizations identify
$\Delta$-matroid polytopes and even $\Delta$-matroid polytopes with
particular M-convex polytopes of types~$\RSB$ and~$\RSD$,
respectively.

Let $B \subseteq \{0,1\}^n$ be a $\Delta$-matroid.
Since $\Mdom(\RSB) = \Z^n$ (see~\Cref{tab:root-coroot-lattices}),
we have $B \subseteq \Mdom(\RSB)$.
Then, the equivalence~\eqref{eq:Delta-typeB} immediately implies that
\begin{align}\label{eq:delta-matroid-M-set-B}
  \{
    \text{$\Delta$-matroids on $[n]$}
  \}
  =
  \{
      \text{M-convex sets of type $\RSB$ contained in $\{0,1\}^n$}
  \}.
\end{align}
Thus, $\Delta$-matroids can be viewed as M-convex sets of type $\RSB$.

Let $B$ be an even $\Delta$-matroid.
Since $\Mdom(\RSD) = \Set*{p\in\Z^n}{\sum_{i=1}^n p(i)\in2\Z}$ (see~\Cref{tab:root-coroot-lattices}),
the set of $p \in \{0,1\}^n$ having even size equals $\{0,1\}^n \cap \Mdom(\RSD)$.
Thus, if every member of $B$ has even size,
then we have $B \subseteq \Mdom(\RSD)$.
Then, the equivalence~\eqref{eq:evenDelta-typeD} immediately implies that
\begin{align}
  &\{
    \text{even $\Delta$-matroids consisting of even-size elements of $\{0,1\}^n$}
  \}\notag\\
  &=
  \{
      \text{M-convex sets of type $\RSD$ contained in $\Mdom(\RSD) \cap \{0,1\}^n$}
  \}.\label{eq:even-delta-matroid-M-set-D}
\end{align}
In the case where every member of $B$ has odd
size,
we consider the twist of $B$ by an arbitrary singleton $\{i\} \subseteq [n]$.
Here,
for $S\subseteq[n]$, the \emph{twist} of $B$ by
$S$ is defined by $B \ast S \coloneqq \Set*{ p\symdiff S}{ p\in B}$,
which is again a $\Delta$-matroid
(resp., an even $\Delta$-matroid)
if $B$ is a $\Delta$-matroid
(resp., an even $\Delta$-matroid).
Since the one-element twist $B \ast \{i\}$
consists entirely of even-size feasible sets,
it is an M-convex set of type~$\RSD$ by
\eqref{eq:even-delta-matroid-M-set-D}.
Thus, even $\Delta$-matroids can be regarded as
M-convex sets of type~$\RSD$, either directly or after a
one-element twist.

We then consider the valuated case.
A function $\funcdoms{\nu}{\{0,1\}^n}{\ER}$ is called a \emph{valuated $\Delta$-matroid} (in the sense of~\cite{cheung2025valuateddeltamatroidsprincipal}) if $\fan(\nu)$ consists of $\Delta$-matroids.
Since $\dom{\nu} \subseteq \{0,1\}^n$, which is finite,
the above definition of valuated $\Delta$-matroids
immediately implies that $\nu = (\clconv{\nu})|_{\{0,1\}^n}$ and $\fan(\clconv{\nu}) = \{ \conv{B} \mid B \in \fan(\nu) \}$.
Combining this with the equivalence~\eqref{eq:delta-matroid-M-set-B},
we obtain
\begin{align}
  &\{
    \text{valuated $\Delta$-matroids on $\{0,1\}^n$}
  \}
  \notag\\
  &\simeq
  \{
      \text{M-convex functions $g$ of type $\RSB$ such that $\dom{g} \subseteq \{0,1\}^n$}
  \}.
\end{align}
Similarly, a \emph{valuated even $\Delta$-matroid} is a function $\funcdoms{\nu}{\{0,1\}^n}{\ER}$ such that $\fan(\nu)$ consists of even $\Delta$-matroids;
either every member of the effective domain has even size or every member of the effective domain has odd size.
Then, by the same argument as above, we have
\begin{align}
  &\{
    \text{valuated even $\Delta$-matroids $\nu$ such that $\dom{\nu} \subseteq \{0,1\}^n \cap \Mdom(\RSD)$}
  \}\\
  &\simeq
  \{
      \text{M-convex functions $g$ of type $\RSD$ such that $\dom{g} \subseteq \{0,1\}^n \cap \Mdom(\RSD)$}
  \}.
\end{align}

We finally comment on the definition of valuated $\Delta$-matroids via the simultaneous exchange property,
which has been used in the field of combinatorial optimization.
A $\Delta$-matroid $B$ is said to be \emph{symmetric} (or \emph{strong})
if, for any $p, q \in B$ and $i \in \supp{(p-q)}$, there is $j \in \supp{(p-q)}$ such that both $p \symdiff \{i,j\}$ and $q \symdiff \{i,j\}$ belong to $B$. 
It is known (see e.g.,~\cite{calvert2026characterisationsstrongdeltamatroids}) that $\Delta$-matroids are not symmetric in general, whereas all even $\Delta$-matroids are symmetric.

In the earlier literature
\cite{Dress1991-ty,Wenzel1993-ee,Takazawa2014-rl},
valuated
$\Delta$-matroids are only defined over symmetric $\Delta$-matroids as follows:
A function $\funcdoms{\nu}{\{0,1\}^n}{\ER}$ is called a \emph{valuated $\Delta$-matroid} if,
for any $p, q \in \dom{\nu}$ and $i \in \supp{(p-q)}$, there is $j \in \supp{(p-q)}$ such that
\begin{align}\label{eq:valuated-simultaneous-exchange}
  \nu(p)+\nu(q)
  \geq
  \nu\rbra*{
    p\symdiff \{i,j\}
  }
  +
  \nu\rbra*{
    q\symdiff \{i,j\}
  },
\end{align}
which is a direct generalization of the simultaneous exchange axiom of (valuated) matroids and M-convex functions (see~\eqref{eq:M-exchange-original}).
The condition~\eqref{eq:valuated-simultaneous-exchange} immediately implies
that $\dom{\nu}$ and all members of $\fan(\nu)$ form symmetric $\Delta$-matroids.
Thus, the exchange-based valuated $\Delta$-matroids form a proper subclass
of the valuated $\Delta$-matroids in the sense of~\cite{cheung2025valuateddeltamatroidsprincipal}.

\subsection{Bisubmodular functions}\label{subsec:example:bisubmodular}

We next compare bisubmodular functions with submodular functions of
types~$\RSB$ and~$\RSC$.
A function $\funcdoms{\rho}{\{-1,0,1\}^n}{\ER}$ with
$\dom{\rho} \neq \emptyset$ is said to be \emph{bisubmodular} if it satisfies the \emph{bisubmodular inequality}:
\begin{align}\label{eq:bisubmodular}
  \rho(x)+\rho(y)
  \geq
  \rho(x\sqcap y)+\rho(x\sqcup y)
  \qquad
  (x,y\in\{-1,0,1\}^n),
\end{align}
where 
\begin{align}
  (x\sqcap y)(i)
  \coloneqq
  \begin{cases}
    x(i) & \text{if $x(i)=y(i)$},\\
    0 & \text{if $x(i)\neq y(i)$},
  \end{cases}
  \qquad
  (x\sqcup y)(i)
  \coloneqq
  \begin{cases}
    x(i) & \text{if $x(i)=y(i)$ or $y(i) = 0$},\\
    y(i) & \text{if $x(i) = 0$},\\
    0 & \text{if $\{x(i), y(i)\} = \{ -1, 1 \}$},
  \end{cases}
\end{align}
for $i \in [n]$.
A bisubmodular function $\rho$ is said to be \emph{normalized} if $\rho(0) = 0$,
and
\emph{integral} if it is
integer-valued.

We may easily check that $\{-1,0,1\}^n = \nbor{0}(\RSC)$ and
\begin{align}\label{eq:bisubmodular:rounding}
    x\sqcap y = \even*{\frac{x+y}{2}}, \qquad x \sqcup y =
  \odd[\bigg]{\frac{x+y}{2}}
\end{align}
for $x, y \in \{-1,0,1\}^n$, where
for $c \in \Z/2$, we define
\begin{align}
    \even{c} &\coloneqq
    \begin{cases}
        c & \text{if $c \in \Z$},\\
        \text{the nearest even integer to $c$} & \text{if $c \in \Z + 1/2$},
    \end{cases}\\
    \odd{c} &\coloneqq
    \begin{cases}
        c & \text{if $c \in \Z$},\\
        \text{the nearest odd integer to $c$} & \text{if $c \in \Z + 1/2$}.
    \end{cases}
\end{align}
Thus, by the midpoint formula in terms of $\RSC$ (given in~\eqref{eq:DMC:coeff:C}) and
\Cref{thm:submodular-chara}, the bisubmodular inequality
\eqref{eq:bisubmodular} is precisely the submodular inequality on
$\nbor{0}(\RSC)$.
Therefore, a function $\rho$ on $\{-1,0,1\}^n$ with $\rho(0) < +\infty$ is bisubmodular if
and only if it is a submodular function of type~$\RSC$.

The equivalence $\RSB \sim \RSC$ immediately implies that a bisubmodular function can also be regarded as a submodular function of type $\RSB$.
We here present the precise formulation; recall \Cref{tab:classical-coxeter-chambers,tab:classical-coxeter-vertices}.
One can easily check that
\begin{align}
    \nbor{0}(\RSB) = \Set{x/\kappa(x)}{ x \in \{ -1,0,1 \}^n},
\end{align}
where
\begin{align}
    \kappa(x) \coloneqq
    \begin{cases}
        1 & \text{if $\card{\supp{x}}\leq1$},\\
        2 & \text{if $\card{\supp{x}}\geq2$},
    \end{cases}
    \qquad (x \in \{ -1,0,1 \}^n).
\end{align}
Equivalently, letting $a_{\mathrm{B}}$ be the mark function of $\euCC(\RSB)$,
we have
\begin{align}
    \{ -1,0,1 \}^n = \Set{ a_{\mathrm{B}}(y) y }{y \in \nbor{0}(\RSB)}.
\end{align}
In particular, the set $\{-1,0,1\}^n \setminus \{0\} = \opnbor{0}(\RSC)$ can be viewed as the vertex set of a spherical expression of $\spCC(\RSB)$.
For a function $\funcdoms{\rho}{\{-1,0,1\}^n}{\ER}$ with $\rho(0) < +\infty$,
we define $\funcdoms{\rho_{\mathrm{B}}}{\nbor{0}(\RSB)}{\ER}$ by
\begin{align}\label{eq:def:rho_B}
    \rho_{\mathrm{B}}(y) \coloneqq \rho(0) + \frac{\rho(a_{\mathrm{B}}(y)y) - \rho(0)}{a_{\mathrm{B}}(y)} \qquad (y \in \nbor{0}(\RSB)).
\end{align}
Then, $\rho$ is bisubmodular if and only if $\rho_{\mathrm{B}}$ is submodular of type $\RSB$,
since
\begin{align}
    &\text{$\rho$ is bisubmodular}\\
    &\iff \text{$\rho$ is submodular of type $\RSC$}\\
    &\iff \text{$x \mapsto \rho(x) - \rho(0)$ for $x \in \opnbor{0}(\RSC)$ is spherical submodular of type $\RSC$}\\
    &\iff \text{$a_{\mathrm{B}}(y)y \mapsto \rho(a_{\mathrm{B}}(y)y) - \rho(0)$ for $y \in \opnbor{0}(\RSB)$ is spherical submodular of type $\RSB$}\\
    &\iff \text{$y \mapsto \rho_{\mathrm{B}}(y) - \rho_{\mathrm{B}}(0)$ for $y \in \opnbor{0}(\RSB)$ is spherical submodular of type $\RSB$}\\
    &\iff \text{$\rho_{\mathrm{B}}$ is submodular of type $\RSB$}.
\end{align}

Combining the above arguments, we obtain
\begin{align}
  &\{
    \text{bisubmodular functions on $\{-1,0,1\}^n$ with $\rho(0) < +\infty$}
  \}\\
  &=
  \{
    \text{submodular functions of type $\RSC$}
  \}\\
  &\simeq
  \{
    \text{submodular functions of type $\RSB$}
  \}
\end{align}
and
\begin{align}
  &\{
    \text{normalized bisubmodular functions on $\{-1,0,1\}^n$}
  \}\notag\\
  &\simeq
  \{
    \text{spherical submodular functions of type $\RSC$}
  \}\notag\\
  &\simeq
  \{
    \text{spherical submodular functions of type $\RSB$}
  \}.\label{eq:normalized:bisubmodular}
\end{align}

The distinction between types~$\RSB$ and
$\RSC$ appears when integrality is considered.
By the definition~\eqref{eq:def:rho_B},
a bisubmodular function $\funcdoms{\rho}{\{-1,0,1\}^n}{\ER}$ with $\rho(0) \in \Z$ is integral if and only if
$\rho_{\mathrm{B}}(0) \in \Z$
and
$\rho_{\mathrm{B}}(y) \in \Z/a_{\mathrm{B}}(y) \cup \{+\infty\}$ for $y \in \nbor{0}(\RSB)$.
By $a_{\mathrm{B}}(0) = 1$ and \Cref{cor:integral-subm},
these are equivalent to the conditions that $\rho_{\mathrm{B}}(0) \in \Z$ and $\rho_{\mathrm{B}}$ is integral submodular of type $\RSB$.
On the other hand,
letting $a_{\mathrm{C}}$ denote the mark function of $\euCC(\RSC)$, we have
\begin{align}
    a_{\mathrm{C}}(x) =
    \begin{cases}
        1 & \text{if $\card{\supp{x}} \in \{0,n\}$},\\
        2 & \text{if $1 \leq \card{\supp{x}} \leq n-1$},
    \end{cases}
    \qquad (x \in \{ -1,0,1 \}^n = \nbor{0}(\RSC));
\end{align}
see \Cref{tab:classical-coxeter-vertices}.
Then, by $a_{\mathrm{C}}(0) = 1$ and \Cref{cor:integral-subm} again,
a submodular function $\funcdoms{\rho}{\{-1,0,1\}^n}{\ER}$ of type $\RSC$ with $\rho(0) \in \Z$
is integral if and only if
\begin{align}
    \rho(x) \in
    \begin{cases}
        \EZ & \text{if $\card{\supp{x}} \in \{0,n\}$},\\
        \Z/2 \cup \{+\infty\} & \text{if $1 \leq \card{\supp{x}} \leq n-1$},
    \end{cases}
    \qquad (x \in \{ -1,0,1 \}^n = \nbor{0}(\RSC)).
\end{align}

Consequently, for $n\geq2$, we have
\begin{align}\label{eq:integral:bisubmodular:typeBC}
  &\{
    \text{integral bisubmodular functions $\rho$ on $\{-1,0,1\}^n$ with $\rho(0) < +\infty$}
  \}
  \\
  &\simeq
  \{
    \text{integral submodular functions $\rho$ of type $\RSB$ with $\rho(0) \in \Z$}
  \}
  \\
  &\subsetneq
  \{
    \text{integral submodular functions $\rho$ of type $\RSC$ with $\rho(0) \in \Z$}
  \}
\end{align}
and
\begin{align}
  &\{
    \text{normalized integral bisubmodular functions}
  \}\notag
  \\
  &\simeq
  \{
    \text{integral spherical submodular functions of type $\RSB$}
  \}\notag
  \\
  &\subsetneq
  \{
    \text{integral spherical submodular functions of type $\RSC$}
  \}.\label{eq:normalized:integral:bisubmodular}
\end{align}

For a normalized bisubmodular function $\rho$ on $\{-1,0,1\}^n$,
the submodular polyhedron $\B{\rho}$ (in our sense) is called the \emph{bisubmodular polyhedron} in the context of bisubmodular optimization,
in which we regard $\rho$ as a submodular function of type $\RSC$.
In addition, a bisubmodular polyhedron is said to be \emph{integral}
if it is generated from an integral normalized bisubmodular function.
By using the notion of bisubmodular and M-convex polyhedra,
the relations~\eqref{eq:normalized:bisubmodular} and~\eqref{eq:normalized:integral:bisubmodular}
can be rephrased as
\begin{align}
  \{
    \text{bisubmodular polyhedra}
  \}
  &=
  \{
    \text{M-convex polyhedra of type $\RSB$}
  \}\\
  &=
  \{
    \text{M-convex polyhedra of type $\RSC$}
  \},\\
  \{
    \text{integral bisubmodular polyhedra}
  \}
  &=
  \{
    \text{integral M-convex polyhedra of type $\RSB$}
  \}\\
  &\subsetneq
  \{
    \text{integral M-convex polyhedra of type $\RSC$}
  \}\label{eq:integral:bisubmodular:polyhedra}.
\end{align}

\subsection{BS-convex sets and functions}\label{subsec:example:BS-convex}

We next compare BS-convex sets and functions
\cite{Fujishige2014-bj,Iwamasa2024-jw} with the
M-convexity of type $\RSB$ introduced in this paper.
A set $B \subseteq\Z^n$ is said to be \emph{BS-convex} if $B = \B{\rho} \cap \Z^n$
for some normalized integral bisubmodular function $\rho$ on $\{-1,0,1\}^n$.
Thus, it immediately follows from the identity $\Mdom(\RSB) = \Z^n$ (see~\Cref{tab:root-coroot-lattices}) and the relation~\eqref{eq:integral:bisubmodular:polyhedra}
that
\begin{align}\label{eq:BS-set:typeB-M-set}
  \{
    \text{BS-convex sets}
  \}
  =
  \{
    \text{M-convex sets of type $\RSB$}
  \}.
\end{align}

A function $\funcdoms{g}{\Z^n}{\EZ}$ (resp. $\funcdoms{g}{\Z^n}{\ER}$) with $\dom{g} \neq \emptyset$ is said to be \emph{BS-convex} (resp. \emph{real-valued BS-convex})\footnote{
In~\cite{Fujishige2014-bj}, the definition
requires only that every member of $\fan(g)$ be a BS-convex set.
This condition alone does not imply integral convexity.
The strengthened definition adopted here, which additionally
requires
$\dom{g}=\dom(\clconv{g})\cap\Z^n$
and
$\dom{g}=\geo{\fan(g)}$,
was formulated in consultation with Satoru Fujishige, who is the author of~\cite{Fujishige2014-bj}, through private
communication.
} if
$\dom{g} = \dom{(\clconv{g})} \cap \Z^n$,
$\dom{g} = \geo{\fan(g)}$,
and
$\fan(g)$ consists of BS-convex sets.
Since
a real-valued BS-convex function $g$ is \emph{integrally convex}, i.e.,
$(\clconv{g})(q) = \clconv(g|_{[p, p + \onevec_{n}]})(q)$
for any $p \in \Z^n$ and $q \in [p, p + \onevec_{n}]$ (see~\cite{Moriguchi2019-dl,Murota2022-ff,Murota2023-xc}),
we have $g = (\clconv{g})|_{\Z^n}$ and $\dom(\clconv{g}) = \bigcup_{B \in \fan(g)} \conv{B}$.
Thus, by the identity~\eqref{eq:BS-set:typeB-M-set}, any real-valued BS-convex function is M-convex of type $\RSB$.
Since the converse clearly holds, we obtain
\begin{align}
  \{
    \text{BS-convex functions}
  \}
  &=
  \{
    \text{integer-valued M-convex functions of type $\RSB$}
  \}
  \\
  &\subsetneq
  \{
    \text{integral M-convex functions of type $\RSB$}
  \}
  \\
  &\subsetneq
  \{
    \text{M-convex functions of type $\RSB$}
  \}\\
  &=
  \{
    \text{real-valued BS-convex functions}
  \}.
\end{align}

\subsection{UJ-convex functions}\label{subsec:example:UJ-convex}

A function $\funcdoms{f}{\Z^n}{\ER}$ with $\dom{f} \neq \emptyset$ is said to be \emph{UJ-convex}~\cite{Fujishige2014-bj} (see also \cite[Example~12]{Hirai2018-lw})
if its Lov\'{a}sz extension $\ext{f}$ along with the Euclidean Coxeter complex $\euCC(\RSC)$ forms a convex function.
By the above definition, the identity $\sk{\euCC(\RSC)} = \Z^n$ (see~\Cref{tab:classical-coxeter-vertices}), and \Cref{thm:chara:L-conv-func},
UJ-convex functions and L-convex functions of type $\RSC$ are exactly the same objects, i.e.,
\begin{align}\label{eq:UJ=typeC-L}
  \{
    \text{UJ-convex functions on $\Z^n$}
  \}
  =
  \{
    \text{L-convex functions of type $\RSC$}
  \}.
\end{align}

A UJ-convex function is said to be \emph{integral} if it is integer-valued.
On the other hand, by \Cref{thm:integral:L-func},
an L-convex function $\funcdoms{f}{\Z^n}{\ER}$ of type $\RSC$ is integral (in our sense) if
\begin{align}
  f(x)
  \in
  \begin{cases}
    \EZ & \text{if $x\in(2\Z)^n\cup(2\Z+1)^n$},\\
    \Z/2 \cup \{ +\infty \} & \text{otherwise};
  \end{cases}
\end{align}
see \Cref{tab:classical-coxeter-vertices} for the mark function of $\euCC(\RSC)$.
Hence, every integral UJ-convex function is an integral
L-convex function of type~$\RSC$, but the converse need not hold.
Consequently, for $n\geq2$, we obtain
\begin{align}\label{eq:integral-UJ:typeC}
  \{
    \text{integral UJ-convex functions on $\Z^n$}
  \}
  \subsetneq
  \{
    \text{integral L-convex functions of type $\RSC$}
  \}.\notag
\end{align}

\subsection{Coxeter matroids}\label{subsec:Coxeter-matroids}
Let $\RS$ be an irreducible root system of rank $n$, and let
$W_{\RS}$ be its Weyl group (recall the definition~\eqref{eq:def:Weyl-group} of $W_{\RS}$).
Fix a point $p\in V^*$ and consider its Weyl-group orbit
$W_{\RS}p$.
A nonempty subset $B\subseteq W_{\RS}p$ is a
\emph{Coxeter matroid}~\cite{Borovik2011-io} if $B$ satisfies the maximality property~\eqref{cond:MP}, i.e., for every simple
system $\simp$ of $\RS$, the set $B$ has a unique maximal member
with respect to $\leq_{\simp}$.

Since all points in $W_{\RS}p$ have the same norm by $W_{\RS} \leq \OG(V^*)$, every
member of $B$ is a vertex of $\conv{B}$.
The Coxeter-matroid maximality property is therefore equivalent to
the condition that, for every simple system $\simp$, the polytope
$\conv{B}$ has a unique maximal member with respect to
$\leq_{\simp}$.
By \Cref{cor:M-convex-set:MP}, this is equivalent
to the M-convexity of $\conv{B}$.
Since the tangent cone of a polytope at a vertex is generated by
the directions of the incident edges, a polytope is M-convex of
type $\RS$ if and only if every edge is parallel to a root in
$\RS$.
Consequently, for $B \subseteq W_{\RS}p$, we have
\begin{align}\label{eq:Coxeter-matroid-M-polytope}
  \text{$B$ is a Coxeter matroid}
  &\quad\Longleftrightarrow\quad
  \text{$\conv{B}$ is an M-convex polytope of type $\RS$}\\
  &\quad\Longleftrightarrow\quad
  \text{every edge of $\conv{B}$ is parallel to a root in $\RS$}.
\end{align}
This recovers the polyhedral characterization of Coxeter matroids
due to Gelfand and Serganova~\cite{Gel-fand1987-da},
usually referred to as the \emph{Gelfand--Serganova theorem}; see~\cite[Section~6.3 and Theorem~6.3.1]{Borovik2011-io}.
In type~$\RSA$, after translating an ordinary matroid base polytope
to the zero-sum hyperplane, this characterization specializes to
the classical edge characterization of matroid base polytopes due
to Gelfand, Goresky, MacPherson, and
Serganova~\cite[Theorem~4.1]{Gelfand1987-mn} (see also~\cite[Theorem~17.1]{Fujishige2005-sp}).

As we can see above~\eqref{eq:Coxeter-matroid-M-polytope}, \emph{Coxeter matroid polytopes}, which are the convex hulls of Coxeter matroids, are precisely the M-convex polytopes
whose vertex sets are contained in a single Weyl-group orbit.
In addition, $\conv{B}$ is integral in the sense of this paper if
and only if $p \in\Mdom(\RS)$.
Indeed, if $p\in\Mdom(\RS)$, then
$B\subseteq\Mdom(\RS)$ and
$\conv{B}$ is integral.
Conversely, if $\conv{B}$ is integral, then each of its
vertices belongs to $\Mdom(\RS)$; hence one, and therefore every,
point in $W_{\RS}p$ belongs to $\Mdom(\RS)$.

\subsection{Coxeter submodular functions}\label{subsec:Coxeter-submodular}
Again, let $\RS$ be an irreducible root system of rank $n$ and $W_{\RS}$ its Weyl group.
Fix an arbitrary simple system $\simp = \{\alpha_1,\dots,\alpha_n\}$
of $\RS$.
Let $\omega_1^\vee,\dots,\omega_n^\vee \in V$ be the corresponding
fundamental coweights,
and define the set
\begin{align}
    \mathcal{R}_{\RS^\vee} \coloneqq W_{\RS} \{ \omega_1^\vee,\dots,\omega_n^\vee \} = \Set*{ w \omega_i^\vee }{ w \in W_\RS, \ i \in [n] } \subseteq V,
\end{align}
which is the union of the $W_{\RS}$-orbits of the fundamental coweights of $\RS$.
We note that $\mathcal{R}_{\RS^\vee}$ can be viewed as the vertex set of a certain spherical expression of $\spCC(\RS)$.
Specifically, \Cref{lem:inpr} implies
\begin{align}
    \sk{\spCC^\star(\RS)} = \Set*{ w\omega_i^\vee / a_i  }{ w \in W_\RS, \ i \in [n] },
\end{align}
where $a_1, a_2, \dots,a_n$ denote the marks of $\alpha_1, \alpha_2, \dots, \alpha_n$, respectively.

A finite-valued function
$\funcdoms{\rho}{\mathcal{R}_{\RS^\vee}}{\R}$ is called a
\emph{Coxeter submodular function of type $\RS^\vee$} or \emph{$\RS^\vee$-submodular function}\footnote{
In~\cite{Ardila2020-ll}, a \emph{$\RS$-submodular function} is defined
on the union of the $W_{\RS}$-orbits of the \emph{fundamental weights} $\omega_1, \omega_2, \dots, \omega_n \in V^*$ corresponding to $\simp$,
i.e., $\inpr{\omega_i}{\alpha_j^\vee} = \delta_{ij}$,
where $\delta_{ij}$ denotes the Kronecker delta.
The fundamental weights of $\RS^\vee$ are the fundamental
coweights of $\RS$, and hence the domain
$\mathcal{R}_{\RS^\vee}$ is exactly the domain used for a
$\RS^\vee$-submodular function in their convention.
Under the inner-product identification of $V$ with $V^*$, the
spherical Coxeter complexes of $\RS$ and $\RS^\vee$ are naturally identified.
}~\cite{Ardila2020-ll}
if its positively homogeneous Lov\'{a}sz extension $\phext{\rho}$ of $\rho$ along with $\spCC(\RS)$ is convex.
Therefore, \Cref{prop:spherical-submodular-cone} immediately implies that
\begin{align}
  &\{
    \text{$\RS^\vee$-submodular functions}
  \}\\
  &\simeq
  \{
    \text{finite-valued spherical submodular functions of type $\RS$}\label{eq:Coxeter-submodular}
  \}.
\end{align}
In particular, a $\RS^\vee$-submodular function $\funcdoms{\rho}{\mathcal{R}_{\RS^\vee}}{\R}$ corresponds to the finite-valued spherical submodular function $\funcdoms{\rho^\star}{\sk{\spCC(\RS)^\star}}{\R}$ of type $\RS$ defined by
\begin{align}\label{eq:Coxeter-spherical-submodular}
    \rho^\star(w\omega_i^\vee / a_i) \coloneqq \rho(w\omega_i^\vee) / a_i.
\end{align}
It is worth mentioning that $\B{\rho} = \B{\rho^\star}$ holds.

The correspondence also preserves integrality in the following
sense. 
A $\RS^\vee$-submodular function $\funcdoms{\rho}{\mathcal{R}_{\RS^\vee}}{\R}$ is said to be \emph{integral} (or \emph{discrete} in~\cite{Ardila2020-ll}) if it is integer-valued, i.e.,
\begin{align}
  \rho(w\omega_i^\vee)
  \in
  \Z
  \qquad
  (w\in W_{\RS},\ i\in[n]).
\end{align}
Thus, the corresponding finite-valued spherical submodular
function $\funcdoms{\rho^\star}{\sk{\spCC^\star(\RS)}}{\R}$
satisfies
\begin{align}
  \rho^\star(w\omega_i^\vee/a_i)
  \in
  \Z/a_i
  \qquad
  (w\in W_{\RS},\ i\in[n])
\end{align}
by the relation~\eqref{eq:Coxeter-spherical-submodular},
which is equivalent to the condition that $\rho^\star$ is integral by \Cref{prop:integral-subm}~(1).
Thus, we obtain
\begin{align}
  &\{
    \text{integral $\RS^\vee$-submodular functions}
  \}\\
  &\simeq
  \{
    \text{integral finite-valued spherical submodular functions of type $\RS$}
  \}.
\end{align}

Ardila et al.~\cite{Ardila2020-ll} originally introduced
Coxeter submodular functions to describe the deformation of the $\RS$-permutahedra, which are called \emph{generalized Coxeter permutahedra (of type $\RS$)} or \emph{Coxeter polymatroids (of type $\RS$)}.
We note that, under the inner-product identification of $V$ with $V^*$,
generalized Coxeter permutahedra of type $\RS$ and of type $\RS^\vee$ are the same objects.
Specifically, a polytope $P \subseteq V^*$ is a generalized Coxeter permutahedron of type $\RS$
if and only if $P = \B{\rho}$ for some $\RS^\vee$-submodular function $\rho$.
Thus, the relation~\eqref{eq:Coxeter-submodular} immediately implies
\begin{align}
\{\text{generalized Coxeter permutahedra of type $\RS$}\} = \{ \text{M-convex polytopes of type $\RS$} \}.\notag
\end{align}
In addition, our notion can capture the \emph{extended} deformation of the $\RS$-permutahedra in the following sense:
\begin{align}
\{\text{extended generalized Coxeter permutahedra of type $\RS$}\} = \{ \text{M-convex polyhedra of type $\RS$} \}.\notag
\end{align}

The class of
generalized Coxeter permutahedra contains
\emph{weight polytopes}~\cite{Fulton1991-dr},
\emph{Coxeter graphic zonotopes}~\cite{Zaslavsky1982-sy},
Coxeter matroid polytopes~\cite{Borovik2011-io}
(see also \Cref{subsec:Coxeter-matroids}), and
\emph{Coxeter associahedra}~\cite{Hohlweg2011-ab},
whereas its extended version also contains
\emph{parsets}~\cite{Reiner1992-ev}; see~\cite[Proposition~4.4]{Ardila2020-ll}.
Thus, the notion of M-convex polyhedra places these classes in a
common framework based on root systems and convex analysis.

\section{Conclusion}\label{sec:conclusion}

In this paper, we have developed a framework of discrete convex
analysis over classical root systems by adopting the vertex set of
the Euclidean Coxeter complex and the root lattice as the primal and
dual discrete domains, respectively.
We introduced the corresponding L- and M-convex sets and functions,
established their local-to-global optimality properties and
combinatorial characterizations, and proved the one-to-one
correspondence between integral L-convex functions and integral
M-convex functions through the discrete Fenchel--Legendre conjugate.
The examples in \Cref{sec:examples} relate this framework to the
original DCA and to several existing notions associated with
classical root systems.

We summarize several directions for future research as follows.

\paragraph{Exceptional root systems.}
Most of the L-convex theory and the discrete conjugacy theorem
remain valid for arbitrary irreducible root systems, including the
exceptional types.
The principal obstacle to extending the corresponding local and
intrinsic results on the M-convex side is the discrete tangent-cone property
described in \Cref{prop:tangent-cone}.
Establishing this property for the exceptional root systems is a natural direction for future research.

\paragraph{Duality beyond type A.}
The discrete conjugacy theorem proved in this paper gives a
  one-to-one correspondence between integral L- and M-convex
  functions.
  In continuous convex analysis and in the original DCA, the
  corresponding conjugacy theory leads further to a Fenchel duality
  theorem, which provides a general min--max relation; see~\cite[Section~31]{Rockafellar1996-wm} and~\cite[Chapter~8]{Murota2003-yb}.

The direct analogue of the discrete Fenchel duality in our setting might be the min--max relation of the following form:
For an integral L-convex function $\funcdoms{f}{\sk{\euCC(\RS)}}{\ER}$ 
and an integral L-concave function $\funcdoms{h}{\sk{\euCC(\RS)}}{\R \cup \{-\infty\}}$ such that $f \geq h$, and $\dom{f} \cap \dom{h} \neq \emptyset$ or $\dom{f^\bullet} \cap \dom{h^\circ} \neq \emptyset$, we have
\begin{align}\label{eq:discrete-Fenchel-duality:L}
    \inf{\Set{f(x) - h(x)}{ x \in \sk{\euCC(\RS)} }} = \sup{\Set{h^\circ(p) - f^\bullet(p)}{ p \in \Mdom(\RS) }},
\end{align}
or equivalently,
for an integral M-convex function $\funcdoms{g}{\Mdom(\RS)}{\ER}$ 
and an integral M-concave function $\funcdoms{k}{\Mdom(\RS)}{\R \cup \{-\infty\}}$ such that $g \geq k$, and $\dom{g} \cap \dom{k} \neq \emptyset$ or $\dom{g^\bullet} \cap \dom{k^\circ} \neq \emptyset$, we have
\begin{align}\label{eq:discrete-Fenchel-duality:M}
    \inf{\Set{g(p) - k(p)}{ p \in \Mdom(\RS) }} = \sup{\Set{k^\circ(x) - g^\bullet(x)}{ x \in \sk{\euCC(\RS)} }},
\end{align}
where $h^\circ(p) \coloneqq -(-h)^\bullet(-p)$ and $k^\circ(x) \coloneqq -(-k)^\bullet(-x)$ are the Fenchel--Legendre concave conjugates of $h$ and $k$,
respectively.

The usual Fenchel duality does not extend uniformly beyond
type A; it already fails in type $\Phi_{\Dynkin{C}{3}}$.
For example, the two sets $B_1, B_2 \subseteq \Mdom(\RS_{\Dynkin{C}{3}}) = \Set{ p \in (\Z/2)^3 }{ p(1) + p(2) + p(3) \in \Z }$ defined by
\begin{align}
    B_1 &\coloneqq \{ (0,0,0), (0,0,1), (-1/2,1/2,1), (-1/2,1/2,0) \},\\
    B_2 &\coloneqq \{ (-1/2,0,-1/2), (-1/2, 0, 1/2), (0, 1/2, 1/2), (0, 1/2, -1/2) \},
\end{align}
which are M-convex of type $\RS_{\Dynkin{C}{3}}$,
satisfy $B_1 \cap B_2 = \emptyset$ but $\conv{B_1} \cap \conv{B_2} = \{ (-1/4, 1/4, p_3) \mid 0 \leq p_3 \leq 1/2 \}$.
These two sets provide a counterexample to the duality.
Indeed, let $g \coloneqq \delta_{B_1}$ and $k \coloneqq -\delta_{B_2}$, where
$\delta_B$ denotes the indicator function of $B$.
Then, $g$ is integral M-convex, $k$ is integral M-concave,
$g\geq k$, and the conjugates $g^\bullet$ and $k^\circ$ are
finite-valued.
Since $B_1\cap B_2=\emptyset$, the primal value $\inf{\Set{g(p) - k(p)}{ p \in \Mdom(\RS) }}$ is $+\infty$.
On the other hand,
we have $g^\bullet(x) = \max\Set{\inpr{p}{x}}{p\in B_1}$ and 
$k^\circ(x) = \min\Set{\inpr{q}{x}}{q\in B_2}$.
Since $k^\circ(x) \leq \inpr{\bar{p}}{x} \leq g^\bullet(x)$
for any
$\bar{p}\in\conv{B_1}\cap\conv{B_2}$,
the dual value $\sup{\Set{k^\circ(x) - g^\bullet(x)}{ x \in \sk{\euCC(\RS_{\Dynkin{C}{3}})} }}$ is at most zero; actually, it is exactly zero since $x=0$ is feasible.
Thus, the duality does not hold.

  An important research direction is to establish a general duality
  theorem that remains valid for all classical types.
  Such a theorem may take a form different from the usual Fenchel duality
  theorem.
  Identifying an appropriate formulation, together with verifiable
  conditions for strong duality and optimality certificates, is an
  important direction from the viewpoint of discrete optimization.

  \paragraph{Algorithmic and computational aspects.}
The results of this paper are primarily structural and
concern convex analysis.
Developing an algorithmic theory for the L- and M-convex functions
introduced here is another important direction for future research.
For example, the local-to-global optimality properties established here suggest
steepest descent algorithms based on the neighborhood structure of the
Euclidean Coxeter complex and on root directions in the root
lattice;
it remains to
derive bounds on the number of iterations to reach an optimal solution.
It is of interest to determine which algorithmic techniques from
the original DCA extend to the other classical types and which
require type-specific approaches.

  \paragraph{Conjugacy on Euclidean buildings.}
  The L-convexity studied in this paper is defined on the vertex set
  of a Euclidean Coxeter complex, which may be regarded as a single
  apartment of a \emph{Euclidean building}~\cite[Chapter~11]{Abramenko2010}.
  L-convexity on Euclidean buildings has been developed
  extensively in recent years
  \cite{Hirai2017-gx,Hirai2018-lw,Hamada2021-gw,Hirai2025-go}.
  A natural next problem is to identify a conjugate class for these
  functions and to establish an analogue of the conjugacy
  developed in this paper.

  A Euclidean building is not naturally embedded into a Euclidean space and does
  not admit a canonical global dual space or root lattice in the
  same manner as a single apartment.
  It is therefore necessary to determine suitable replacements for
  linear perturbations, dual discrete points, and the
  Fenchel--Legendre conjugate.
  Addressing these questions is likely to require a 
  theory on metric spaces beyond Euclidean spaces from the viewpoint of convex analysis,
  which has recently attracted significant attention (see e.g.,~\cite{Bacak2014-jw,Burgisser2019-zq,Hirai2024-ao} and references therein).
  Such a theory would provide a further connection between discrete
  convex analysis and building theory.

  \paragraph{Applications.}
  Another direction is to identify concrete applications of the
  L- and M-convexity notions introduced in this paper.
  The possible applications are not limited to combinatorial
  optimization.
  The original DCA has been used in operations research, economics,
  game theory, and several areas of pure mathematics; see the surveys
  \cite{Murota2008-ej,Murota2024}.
  It is natural to investigate whether the present extensions can
  model problems in these areas whose underlying discrete structures
  exhibit symmetries of classical types other than type~A.

  The framework may also be relevant to areas in which Coxeter
  matroids, Coxeter submodular functions, and generalized Coxeter
  permutahedra arise.
  As discussed in \Cref{subsec:Coxeter-matroids,subsec:Coxeter-submodular}, these objects can be
  interpreted through M-convex polyhedra of the appropriate root
  system type.
  This interpretation places several constructions associated with Coxeter groups
  within a common framework based on convex analysis.
  It is therefore of interest to determine whether the structural,
  conjugacy, and optimality results developed here can yield new
  tools for their combinatorial, geometric, or algorithmic study.

\section*{Acknowledgments}
The author thanks Soichiro Fujii for valuable discussions,
and Hiroshi Hirai for helpful comments,
including his suggestion that led to a
simpler proof of \Cref{thm:chara:L-conv-func}.
The author also thanks Satoru Fujishige for clarifying the
definition of BS-convex functions through private communication.
This work was supported by JSPS KAKENHI Grant Numbers JP22K17854, JP24K21315, JP24K02901, Japan.

\newpage
\appendix
\crefalias{section}{appendix}

\section{Deferred proofs}\label{sec:deferred}
\subsection{Proof of \texorpdfstring{\Cref{prop:distance-function}}{Proposition~5.7}}\label{sec:proof:prop:distance-function}
    The only-if part is obvious from the definition of tight distance functions.
    We prove the if part,
    i.e., $\gamma(\alpha) = (\cone{\gamma})(\alpha)$ for all $\alpha \in \cone(\dom{\gamma}) \cap \RS$.
    Note that $\gamma \geq \cone{\gamma}$ generally holds by the definition~\eqref{eq:func:cone} of $\cone{\gamma}$.
    The equivalence between the two formulations in each type is straightforward to verify from the explicit descriptions of the corresponding root systems 
    (see \Cref{tab:classical-coordinate-realizations}).
    Hence, in the remainder of the proof, we assume the latter coordinate-wise conditions.

    Fix an arbitrary $\hat{\alpha} \in \cone(\dom{\gamma}) \cap \RS$.
    Suppose to the contrary that $\gamma(\hat{\alpha}) > (\cone{\gamma})(\hat{\alpha})$.
    Then, there is a nonnegative coefficient $\mu' \in \cncoeff{\dom{\gamma}}$
    such that $\hat{\alpha} = \sum_{\alpha \in \dom{\gamma}} \mu'(\alpha) \alpha$ and $\sum_{\alpha \in \dom{\gamma}} \mu'(\alpha) \gamma(\alpha) < \gamma(\hat{\alpha})$.
    Letting $r \coloneqq \sum_{\alpha \in \dom{\gamma}} \mu'(\alpha) \gamma(\alpha)$ for such $\mu'$,
    we define the following LP:
    \begin{align}
    \begin{array}{ll}
	\text{Minimize} & \displaystyle \sum_{\alpha \in \dom{\gamma}} \mu(\alpha) \left(\| \alpha - \hat{\alpha} \|_1 + \| \alpha \|_1 \right)\\
	\text{subject to} & \displaystyle \hat{\alpha} = \sum_{\alpha \in \dom{\gamma}} \mu(\alpha) \alpha, \\
        & \mu(\alpha) \geq 0 \quad (\alpha \in \dom{\gamma}),\\
        & \displaystyle \sum_{\alpha \in \dom{\gamma}} \mu(\alpha) \gamma(\alpha) \leq r.
	\end{array}    
\end{align}
Since the set of feasible solutions of the above LP is nonempty
and the objective function is lower bounded by $0$ on the feasible set,
the above LP admits an optimal solution.
Let $\hat{\mu}$ denote such an optimal solution.
Our aim is to show that we can construct another feasible solution that attains a smaller objective value than $\hat{\mu}$,
which leads to a contradiction.

First, we see the following conditions that guarantee the existence of a feasible solution attaining a smaller objective value.
Here, for $p \in \R^n$,
let $\supp^+{p} \coloneqq \Set{ i \in [n]}{p(i) > 0}$ and $\supp^-{p} \coloneqq \Set{ i \in [n]}{p(i) < 0}$.
\begin{claim}\label{cl:optimality-condition}
    Let $\mu$ be a feasible solution of the above LP.
    \begin{enumerate}[label={\textup{(\arabic*)}}]
        \item If there are $\alpha, -\alpha \in \supp{\mu}$, then we can construct another feasible solution attaining a smaller objective value than $\mu$.
        \item If there are
        $\alpha_1, \alpha_2 \in \supp{\mu}$ such that $\alpha_1 + \alpha_2 \in \RS$
        and
        $i \in (\supp^\sigma(\alpha_1) \setminus \supp^\sigma(\hat{\alpha})) \cap \supp^{-\sigma}(\alpha_2)$ for some $i$ and $\sigma \in \{\pm1\}$,
        then we can construct another feasible solution attaining a smaller objective value than $\mu$.
    \end{enumerate}
\end{claim}
\begin{proof}[Proof of \Cref{cl:optimality-condition}]
    (1).
Suppose that there are $\alpha, -\alpha \in \supp{\mu}$.
Then, update the coefficient as $\mu(\alpha) \leftarrow \mu(\alpha) - \delta$ and $\mu(-\alpha) \leftarrow \mu(-\alpha) - \delta$ for a sufficiently small $\delta > 0$.
Since $\alpha + (-\alpha) = 0$ and $\gamma(\alpha) + \gamma(-\alpha) \geq 0$,
the resulting coefficient is also a feasible solution.
Furthermore,
the objective value strictly decreases.

(2).
Suppose that there are
$\alpha_1, \alpha_2 \in \supp{\mu}$ such that $\alpha_1 + \alpha_2 \in \RS$
and
$i \in (\supp^\sigma(\alpha_1) \setminus \supp^\sigma(\hat{\alpha})) \cap \supp^{-\sigma}(\alpha_2)$ for some $i$ and $\sigma \in \{\pm1\}$.
Since $\alpha_1 + \alpha_2 \in \RS$,
we can update the coefficient as $\mu(\alpha_1) \leftarrow \mu(\alpha_1) - \delta$, $\mu(\alpha_2) \leftarrow \mu(\alpha_2) - \delta$, and $\mu(\alpha_1 + \alpha_2) \leftarrow \mu(\alpha_1 + \alpha_2) + \delta$ for a sufficiently small $\delta > 0$.
Since $\gamma(\alpha_1) + \gamma(\alpha_2) \geq \gamma(\alpha_1 + \alpha_2)$,
the resulting coefficient is also a feasible solution.
Furthermore,
since $\| \alpha_1 - \hat{\alpha} \|_1 \geq \| \alpha_1 + \alpha_2 - \hat{\alpha} \|_1$, $\| \alpha_2 - \hat{\alpha} \|_1 \geq 0$,
and $\| \alpha_1 \|_1 + \| \alpha_2 \|_1 > \| \alpha_1 + \alpha_2 \|_1$,
we obtain 
$\| \alpha_1 - \alpha \|_1 + \| \alpha_2 - \alpha \|_1 + \|\alpha_1\|_1 + \|\alpha_2 \|_1  > \| \alpha_1 + \alpha_2 - \alpha \|_1 + \| \alpha_1 + \alpha_2 \|_1$.
Here, the inequality $\| \alpha_1 - \hat{\alpha} \|_1 \geq \| \alpha_1 + \alpha_2 - \hat{\alpha} \|_1$
follows from the facts that $\| \alpha_1 - \sigma e_i - \hat{\alpha} \|_1 = \| \alpha_1 - \hat{\alpha} \|_1 - 1$
and $\| \alpha_1 + \alpha_2 - \hat{\alpha} \|_1 \leq \| \alpha_1 - \sigma e_i - \hat{\alpha} \| + 1$
by the assumption $i \in (\supp^\sigma(\alpha_1) \setminus \supp^\sigma(\hat{\alpha})) \cap \supp^{-\sigma}(\alpha_2)$.
This implies that the objective value strictly decreases.
\end{proof}
We note that $\supp{\hat{\mu}} \setminus \{ \hat{\alpha} \} \neq \emptyset$, since the optimal solution $\hat{\mu} \in M_{\dom{\gamma}}$ satisfies $\hat{\alpha} = \sum_{\alpha \in \dom{\gamma}} \mu(\alpha) \alpha$
and $\sum_{\alpha \in \dom{\gamma}} \hat{\mu}(\alpha) \gamma(\alpha) \leq r < \gamma(\hat{\alpha})$.
Suppose that there is $\alpha_1 \in \supp{\hat{\mu}} \setminus \{ \hat{\alpha} \}$ with $(\supp^+(\alpha_1) \setminus \supp^+(\hat{\alpha})) \cup (\supp^-(\alpha_1) \setminus \supp^-(\hat{\alpha})) \neq \emptyset$;
we may assume that $i \in \supp^+(\alpha_1) \setminus \supp^+(\hat{\alpha})$.
Then,
since $\hat{\alpha} = \sum_{\alpha \in \dom{\gamma}} \hat{\mu}(\alpha) \alpha$,
there is $\alpha_2 \in \supp{\hat{\mu}}$ with $i \in\supp^-(\alpha_2)$;
we refer to such $\alpha_2$ as a \emph{counterpart} of $\alpha_1$ (with respect to $i$).
By \Cref{cl:optimality-condition},
if $\alpha_1 + \alpha_2 = 0$ or $\alpha_1 + \alpha_2 \in \RS$,
then we can construct a feasible solution whose objective value is smaller than that of $\hat{\mu}$,
which contradicts the optimality of $\hat{\mu}$.
Therefore,
it suffices to consider the case where every $\alpha_1 \in \supp{\hat{\mu}} \setminus \{ \hat{\alpha} \}$ satisfies either
\begin{itemize}
    \item $(\supp^+(\alpha_1) \setminus \supp^+(\hat{\alpha})) \cup (\supp^-(\alpha_1) \setminus \supp^-(\hat{\alpha})) = \emptyset$; or
    \item $(\supp^+(\alpha_1) \setminus \supp^+(\hat{\alpha})) \cup (\supp^-(\alpha_1) \setminus \supp^-(\hat{\alpha})) \neq \emptyset$
    and
    $\alpha_1 + \alpha_2 \notin \RS \cup \{0\}$
    for any counterpart $\alpha_2 \in \supp{\hat{\mu}}$ of $\alpha_1$.
\end{itemize}

In the following, we derive a contradiction for each type of root system by proving the existence of a feasible solution whose objective value is smaller than that of $\hat{\mu}$.

(Type~A).
In the case of $\RS = \RSA$,
we can easily see that any $\alpha_1 \in \supp{\hat{\mu}} \setminus \{ \hat{\alpha} \}$ admits its counterpart $\alpha_2$ such that $\alpha_1 + \alpha_2 \in \RSA \cup \{0\}$.
That is, we can always find a feasible solution whose objective value is smaller than that of $\hat{\mu}$,
which contradicts the optimality of $\hat{\mu}$.

(Type~B).
We consider the following two cases:
(i) there is $\alpha \in \supp{\hat{\mu}} \setminus \{ \hat{\alpha} \}$ such that $(\supp^+(\alpha) \setminus \supp^+(\hat{\alpha})) \cup (\supp^-(\alpha) \setminus \supp^-(\hat{\alpha})) \neq \emptyset$;
and (ii) there is no such $\alpha \in \supp{\hat{\mu}} \setminus \{ \hat{\alpha} \}$.

(i). Let $\alpha_1 \in \supp{\hat{\mu}} \setminus \{ \hat{\alpha} \}$ with $(\supp^+(\alpha_1) \setminus \supp^+(\hat{\alpha})) \cup (\supp^-(\alpha_1) \setminus \supp^-(\hat{\alpha})) \neq \emptyset$;
we may assume that $1 \in \supp^+(\alpha_1) \setminus \supp^+(\hat{\alpha})$.
Then, there is a counterpart $\alpha_2 \in \supp{\hat{\mu}}$ of $\alpha_1$ such that $1 \in \supp^-(\alpha_2)$.
Since $\alpha_1 + \alpha_2 \notin \RSB \cup \{0\}$,
we obtain $\alpha_1 = e_1 + \sigma e_j$ and $\alpha_2 = -e_1 + \sigma e_j$ for some $j \neq 1$ and $\sigma \in \{\pm1\}$;
we may assume that $\alpha_1 = e_1 + e_2$ and $\alpha_2 = -e_1 + e_2$.
Then,
$\| \alpha_1 - \hat{\alpha} \|_1 + \| \alpha_2 - \hat{\alpha} \|_1 \geq 2\| (\alpha_1 + \alpha_2)/2 - \hat{\alpha} \|_1 = 2\| \sigma e_j - \hat{\alpha} \|$ holds by the convexity of the function $p \mapsto \| p - \hat{\alpha} \|_1$,
and $\| \alpha_1 \|_1 + \| \alpha_2 \|_1 > 2 \| (\alpha_1 + \alpha_2)/2 \|_1 = 2 \|\sigma e_j\|_1$ holds.
We update the coefficient as
$\hat{\mu}(\alpha_1) \leftarrow \hat{\mu}(\alpha_1) - \delta$,
$\hat{\mu}(\alpha_2) \leftarrow \hat{\mu}(\alpha_2) - \delta$,
and
$\hat{\mu}(\sigma e_j) \leftarrow \hat{\mu}(\sigma e_j) + 2\delta$
for a sufficiently small $\delta > 0$.
Since $\gamma(\alpha_1) + \gamma(\alpha_2) \geq 2 \gamma(\sigma e_j)$ and $\| \alpha_1 - \hat{\alpha} \|_1 + \| \alpha_2 - \hat{\alpha} \|_1 + \|\alpha_1\|_1 + \| \alpha_2 \|_1 > 2\| \sigma e_j - \hat{\alpha} \|_1 + 2\| \sigma e_j \|_1$,
the resulting coefficient is also a feasible solution
and
the objective value strictly decreases.
This contradicts the optimality of $\hat{\mu}$.

(ii). In this case,
we obtain
$\supp{\hat{\mu}} \subseteq \{ \hat{\alpha}, \alpha_1, \alpha_2 \}$,
in which $\hat{\alpha} = \sigma_i e_i + \sigma_j e_j$,
$\alpha_1 = \sigma_i e_i$, and $\alpha_2 = \sigma_j e_j$ for distinct $i,j$ and $\sigma_i, \sigma_j \in \{ \pm1 \}$.
Since $\hat{\alpha} = \sum_{\alpha \in \dom{\gamma}} \hat{\mu}(\alpha) \alpha$ and $\supp{\hat{\mu}} \setminus \{ \hat{\alpha}\} \neq \emptyset$,
we have $\alpha_1, \alpha_2 \in \supp{\hat{\mu}}$.
We here update the coefficient as
$\hat{\mu}(\alpha_1) \leftarrow \hat{\mu}(\alpha_1) - \delta$,
$\hat{\mu}(\alpha_2) \leftarrow \hat{\mu}(\alpha_2) - \delta$,
and
$\hat{\mu}(\hat{\alpha}) \leftarrow \hat{\mu}(\hat{\alpha}) + \delta$
for a sufficiently small $\delta > 0$.
Since $\alpha_1 + \alpha_2 = \hat{\alpha} \in \RSB$,
$\gamma(\alpha_1) + \gamma(\alpha_2) \geq \gamma(\hat{\alpha})$,
and
$\| \alpha_1 - \hat{\alpha} \|_1 + \| \alpha_2 - \hat{\alpha} \|_1 + \| \alpha_1 \|_1 + \| \alpha_2 \|_1 > \| \alpha_1 + \alpha_2 - \hat{\alpha} \|_1 + \| \alpha_1 + \alpha_2 \|_1$,
the resulting coefficient is also a feasible solution
and
the objective value strictly decreases.
This contradicts the optimality of $\hat{\mu}$.

(Type~C).
Since,
for any $\alpha, \alpha' \in 2\RSC$ with $\supp^\sigma(\alpha) \cap \supp^{-\sigma}(\alpha') \neq \emptyset$ for some $\sigma \in \{\pm1\}$,
we have $\alpha + \alpha' = 0$ or $\alpha + \alpha' \in 2\RSC$,
it suffices to consider the case where there are no $\alpha \in \supp{\hat{\mu}} \setminus \{ \hat{\alpha} \}$ such that $(\supp^+(\alpha) \setminus \supp^+(\hat{\alpha})) \cup (\supp^-(\alpha) \setminus \supp^-(\hat{\alpha})) \neq \emptyset$,
which is the same as the case (ii) of Type~B.
Then, by the same argument as in the case (ii) of Type~B,
we may assume that $\hat{\alpha} = \sigma_i e_i + \sigma_j e_j$ for some distinct $i,j$ and $\sigma_i, \sigma_j \in \{ \pm1 \}$
and
$\alpha_1 \coloneqq 2\sigma_i e_i, \alpha_2 \coloneqq 2\sigma_j e_j \in \supp{\hat{\mu}}$.
We can update the coefficient as
$\hat{\mu}(\alpha_1) \leftarrow \hat{\mu}(\alpha_1) - \delta$,
$\hat{\mu}(\alpha_2) \leftarrow \hat{\mu}(\alpha_2) - \delta$,
and
$\hat{\mu}(\hat{\alpha}) \leftarrow \hat{\mu}(\hat{\alpha}) + 2\delta$
for a sufficiently small $\delta > 0$.
Since
$\gamma(\alpha_1) + \gamma(\alpha_2) \geq 2\gamma(\hat{\alpha})$
and
$\| \alpha_1 - \hat{\alpha} \|_1 + \| \alpha_2 - \hat{\alpha} \|_1 + \|\alpha_1\|_1 + \|\alpha_2\|_1 > 2\| (\alpha_1 + \alpha_2)/2 - \hat{\alpha} \|_1 + 2\| (\alpha_1 + \alpha_2)/2 \|_1$,
the resulting coefficient is also a feasible solution
and
the objective value strictly decreases.
This contradicts the optimality of $\hat{\mu}$.

(Type~D).
For any $\alpha \in \RSD \setminus \{ \hat{\alpha} \}$,
we have $(\supp^+(\alpha) \setminus \supp^+(\hat{\alpha})) \cup (\supp^-(\alpha) \setminus \supp^-(\hat{\alpha})) \neq \emptyset$.
Hence,
we only need to consider the case where,
for any $\alpha_1 \in \supp{\hat{\mu}} \setminus \{ \hat{\alpha} \}$
and any counterpart $\alpha_2$ of $\alpha_1$,
we have $\alpha_1 + \alpha_2 \notin \RSD \cup \{0\}$.
That is, if $\alpha_1 = \sigma_i e_i + \sigma_j e_j$ with $i \in \supp^{\sigma_i}(\alpha_1) \setminus \supp^{\sigma_i}(\hat{\alpha})$,
then a counterpart $\alpha_2 \in \supp{\hat{\mu}}$ of $\alpha_1$ with respect to $i$ must be $\alpha_2 = -\sigma_i e_i + \sigma_j e_j$.

Take any $\alpha_1 \in \supp{\hat{\mu}} \setminus \{ \hat{\alpha} \}$;
we may assume that $\alpha_1 = e_1 + e_2$.
If $\supp(\alpha_1) (= \{1,2\}) \subseteq (\supp^+(\alpha_1) \setminus \supp^+(\hat{\alpha})) \cup (\supp^-(\alpha_1) \setminus \supp^-(\hat{\alpha})) (= \supp^+(\alpha_1) \setminus \supp^+(\hat{\alpha}))$,
then the unique counterpart $\alpha_2 = -e_1 + e_2$ of $\alpha_1$ with respect to $1$
and $\alpha_2' = e_1 - e_2$ of $\alpha_1$ with respect to $2$
satisfy $\alpha_2 + \alpha_2' = 0$,
which contradicts the optimality of $\hat{\mu}$ by \Cref{cl:optimality-condition}~(1).
Hence, we may assume that $1 \in \supp^+(\alpha_1) \setminus \supp^+(\hat{\alpha}) \not\ni 2$,
i.e.,
$\hat{\alpha} = e_2 + \sigma e_i$ for some $i \neq 2$ and $(i, \sigma) \neq (1,1)$.
Note that the unique counterpart $\alpha_2$ of $\alpha_1$ with respect to $1$ is $\alpha_2 = - e_1 + e_2$.

Suppose that $(i, \sigma) = (1,-1)$, i.e.,
$\hat{\alpha} = -e_1 + e_2 = \alpha_2$.
Since $\hat{\alpha} = \sum_{\alpha \in \dom{\gamma}} \hat{\mu}(\alpha) \alpha$ and there are no $\alpha \in \supp{\hat{\mu}}$ different from $\alpha_2$ satisfying $1 \in \supp^-(\alpha)$,
there must exist $\alpha_1' \in \supp{\hat{\mu}}$ such that $2 \in \supp^-(\alpha_1')$,
i.e., $\alpha_1' = \sigma' e_j - e_2$ for some $j \neq 2$ and $\sigma' \in \{\pm1\}$.
If $j = 1$, then $\alpha_1' = e_1 - e_2 = -\alpha_2$ or $\alpha_1' = -e_1 - e_2 = -\alpha_1$.
Hence, in both cases,
we obtain a contradiction to the optimality of $\hat{\mu}$ by \Cref{cl:optimality-condition}~(1).
Otherwise, i.e., if $j \neq 1$,
then $\supp(\alpha_1') \subseteq (\supp^+(\alpha_1') \setminus \supp^+(\hat{\alpha})) \cup (\supp^-(\alpha_1') \setminus \supp^-(\hat{\alpha}))$.
By the same argument as above,
we obtain the counterparts $\alpha$ and $\alpha'$ of $\alpha_1'$ with respect to $2$ and $j$, respectively,
that satisfy $\alpha + \alpha' = 0$.
This contradicts the optimality of $\hat{\mu}$ by \Cref{cl:optimality-condition}~(1).

Suppose that $i \neq 1$;
we may assume that
$\hat{\alpha} = e_2 + e_3$.
Since $\hat{\alpha} = \sum_{\alpha \in \dom{\gamma}} \hat{\mu}(\alpha) \alpha$,
there is $\alpha_3 \in \supp{\hat{\mu}} \setminus \{ \hat{\alpha} \}$ such that $3 \in \supp^+(\alpha_3)$,
i.e., $\alpha_3 = e_3 + \sigma'' e_j$ for some $j$ and $\sigma'' \in \{\pm1\}$ with $j \neq 3$ and $(j,\sigma'') \neq (2,1)$.
By the same argument as above,
we can derive a contradiction if $\supp(\alpha_3) = \{2,3\}$.

If $j = 2$, i.e., $\alpha_3 = -e_2 + e_3$,
then by the same argument as the case of $(i, \sigma) = (1,-1)$ above,
we can derive a contradiction.
If $j \neq 2$, then we may assume that $\alpha_3 = e_3 + e_4$.
The unique counterpart $\alpha_4$ of $\alpha_3$ with respect to $4$ is $\alpha_4 = e_3 - e_4$.
Thus, we have $\alpha_1 = e_1 + e_2, \alpha_2 = -e_1 + e_2, \alpha_3 = e_3 + e_4, \alpha_4 = e_3 - e_4\in \supp{\hat{\mu}}$
and $\hat{\alpha} = e_2 + e_3 = (\alpha_1 + \alpha_2 + \alpha_3 + \alpha_4)/2$.
We here update the coefficient as
$\hat{\mu}(\alpha_1) \leftarrow \hat{\mu}(\alpha_1) - \delta$,
$\hat{\mu}(\alpha_2) \leftarrow \hat{\mu}(\alpha_2) - \delta$,
$\hat{\mu}(\alpha_3) \leftarrow \hat{\mu}(\alpha_3) - \delta$,
and
$\hat{\mu}(\alpha_4) \leftarrow \hat{\mu}(\alpha_4) - \delta$,
and
$\hat{\mu}(\hat{\alpha}) \leftarrow \hat{\mu}(\hat{\alpha}) + 2\delta$
for a sufficiently small $\delta > 0$.
Since
$\gamma(\alpha_1) + \gamma(\alpha_2) + \gamma(\alpha_3) + \gamma(\alpha_4) \geq 2\gamma(\hat{\alpha})$
and
$\| \alpha_1 - \hat{\alpha} \|_1 + \| \alpha_2 - \hat{\alpha} \|_1 + \| \alpha_3 - \hat{\alpha} \|_1 + \| \alpha_4 - \hat{\alpha} \|_1 + \|\alpha_1\|_1 + \|\alpha_2\|_1 + \|\alpha_3\|_1 + \|\alpha_4\|_1 > 2\| \hat{\alpha} - \hat{\alpha} \|_1 + 2\| \hat{\alpha} \|_1$,
the resulting coefficient is also a feasible solution
and
the objective value strictly decreases.
This contradicts the optimality of $\hat{\mu}$.

\subsection{Proof of \texorpdfstring{\Cref{prop:integral-subm}}{Proposition~5.17}~(2)}\label{subsec:proof:prop:integral-subm}
    Suppose that $\tau(x_*) = v_{i} \in \tilde{D}(\RS)$.
    Recall that $\Lk{\euCC}{x_*}$ is a spherical expression of $\spCC(\tilde{\RS}|_{- x_*})$
    and $\opnbor{x_*} - x_* = \sk{\Lk{\euCC}{x_*}}$.
    Let $\rho^\circ$ be the spherical submodular function on $\opnbor{x_*} - x_*$
    corresponding to $\rho$, i.e., $\rho^\circ(x - x_*) \coloneqq \rho(x) - \rho(x_*)$ for $x \in \opnbor{x_*}$,
    and $\sigma$ the L-sublinear function of type $\tilde{\RS}|_{-x_*}$ whose restriction to $\opnbor{x_*} - x_*$ is $\rho^\circ$.
    We note that $\sigma$ can be viewed as an L-sublinear function of type $\RS$, since $\tilde{\RS}|_{-x_*} \precsim \RS$, or more precisely, $\tilde{\RS}|_{-x_*} \subseteq \RS$.

    We first show the former statement.
    Assume that there is $r \in \R$ such that
    $\rho(x) - r \in \Z/a(x)$ for all $x \in \dom{\rho}$,
    and
    let $K$ be an arbitrary maximal L-convex cone in $\fan(\sigma)$.
    Our aim is to prove that there is $p_K \in \Mdom$ such that $K = \argmin(\sigma - p_K)$ as in the proof of the if part of the assertion~(1).

    Let $\CC_K$ be the convex subcomplex of $\spCC(\tilde{\RS}|_{-x_*})$ whose geometric realization is $K$.
    Fix an arbitrary maximal simplex $A_K$ of $\CC_K$
    and let $A_*$ be the simplex of $\Lk{\euCC}{x_*}$ corresponding to $A_K$.

    Suppose that $\esimp \coloneqq \{ \alpha_0, \alpha_1, \dots, \alpha_n \}$ is the extended simple system corresponding to some chamber $C$ of $\St{\euCC}{A_*}$,
    in which $\sk{C} = \{ x_0 ,x_1, \dots, x_n \}$
    and $\tau(x_k) = \tau_{\esimp}(\alpha_k) = v_k \in \tilde{D}(\RS)$.
    Note that $x_* = x_i$, since $\tau(x_*) = v_i$.
    Then, we define the vector $p_K$ by
    \begin{align}\label{eq:pK:simp}
        p_K \coloneqq\sum_{x_k \in \sk{A_*} \cup \{x_i\}} a(x_k)\left( \rho(x_k) - r \right) \alpha_k.
    \end{align}
    Since $\rho(x) - r \in \Z/a(x)$ for $x \in \sk{A_*} \cup \{x_i\}$ by the assumption,
    we have $a(x_k)(\rho(x_k)-r) \in \Z$.
    Thus, we obtain $p_K \in \Mdom$.

    In the following, we show that $K = \argmin(\sigma - p_K)$.
    By the same argument as in the proof of the if part of the assertion~(1),
    it suffices to prove that
    \begin{align}\label{eq:aim:inpr}
        \inpr{p_K}{x_k - x_i}\ (= \rho^\circ(x_k - x_i)) = \rho(x_k) - \rho(x_i) \qquad (x_k \in \sk{A_*}).
    \end{align}

    Let $C_*$ be the simplex of $\euCC$ with $\sk{C_*} = \sk{C} \setminus \{x_i\}$,
    which can be viewed as a chamber of $\Lk{\euCC}{x_i}$ with $A_* \preceq C_*$.
    Then, $\esimp|_{-x_i} = \esimp \setminus \{ \alpha_i \}$ is the simple system of $\tilde{\RS}|_{-x_i}$ corresponding to $C_*$.
    Let $a_{\rm Lk}$ denote the mark function of $\Lk{\euCC}{x_i}$.
    Suppose that $p_K$ is representable as
    $p_K = \sum_{\alpha_k \in \esimp|_{-x_i}} p_k \alpha_k$;
    we refer to this representation as the \emph{$\esimp|_{-x_i}$-representation} of $p_K$.
    Then, by \Cref{cor:inpr}~(1) and \Cref{lem:standard}~(1),
    we obtain
    \begin{align}\label{eq:pk:inpr}
        \inpr{p_K}{x_k - x_i} = \left(p_k / a_{\rm Lk}(x_k)\right) \cdot \left(a_{\rm Lk}(x_k)/a(x_k)\right) = p_k / a(x_k).
    \end{align}

    Since $\alpha_i = -\sum_{j \in [0,n] : j \neq i} (a(x_j)/a(x_i)) \alpha_j$,
    the $\esimp|_{-x_i}$-representation of $p_K$ is obtained as follows:
    \begin{align}
      p_K &= a(x_i) (\rho(x_i) - r) \alpha_i + \sum_{x_k \in \sk{A_*}} a(x_k)\left( \rho(x_k) - r \right) \alpha_k\\
      &= -\sum_{j \in [0,n] : j \neq i} a(x_j) (\rho(x_i) - r) \alpha_j + \sum_{x_k \in \sk{A_*}} a(x_k)\left( \rho(x_k) - r \right) \alpha_k\\
      &= \sum_{x_k \in \sk{A_*}} a(x_k)\left( \rho(x_k) - \rho(x_i) \right) \alpha_k - \sum_{x_j \notin \sk{A_*} \cup \{x_i\}} a(x_j) (\rho(x_i) - r) \alpha_j.
    \end{align}
    Thus, it immediately follows from~\eqref{eq:pk:inpr} that $\inpr{p_K}{x_k - x_i} = \rho(x_k)-\rho(x_i)$ for $x_k \in \sk{A_*}$,
    implying the equations~\eqref{eq:aim:inpr}.

    We then show the latter statement;
    assume that $\rho$ is $\RS$-integral and that, for each maximal simplex $A$ of $\euCC[\dom{\rho}]$ and $x \in \sk{A} \setminus \{x_*\}$, there is a vertex $x' \in \sk{A} \setminus \{x\}$ such that $a(x') = 1$.

    Since $\rho$ is $\RS$-integral, so is $\sigma$.
    Thus, there is $p \in \Mdom$ such that $\argmin(\sigma - p)$ forms a maximal L-convex cone in $\fan(\sigma)$.
    For such $p$,
    we define
    \begin{align}\label{eq:r}
        r \coloneqq (\rho(x_*) - \inpr{p}{x_*}) - \lfloor \rho(x_*) - \inpr{p}{x_*} \rfloor.
    \end{align}
    Then, the following claim holds.
    \begin{claim}\label{cl:r}
      The value of $r$ is independent of the choice of $p \in \Mdom$ such that $\argmin(\sigma - p)$ is a maximal L-convex cone in $\fan(\sigma)$.
    \end{claim}
    \begin{proof}[Proof of \Cref{cl:r}]
      For $p \in \Mdom$,
      we define $r(p) \coloneqq (\rho(x_*) - \inpr{p}{x_*}) - \lfloor \rho(x_*) - \inpr{p}{x_*} \rfloor$.
      Since $\fan(\sigma)$ is a polyhedral fan arising from a proper polyhedral sublinear function,
      for any two maximal L-convex cones $K, K' \in \fan(\sigma)$,
      there is a sequence $K \eqqcolon K_0, K_1, \dots, K_\ell \coloneqq K'$ of maximal L-convex cones in $\fan(\sigma)$ such that $K_{j-1} \cap K_{j}$ is a codimension-one L-convex cone in $\fan(\sigma)$, i.e., $\dim(K_{j-1} \cap K_{j}) = \dim{K_j} - 1$ (such as a gallery of a chamber complex; see \Cref{subsec:simplicial_complex}).
      Thus,
      it suffices to show that
      $\inpr{p - p'}{x_*} \in \Z$ for any $p, p' \in \Mdom$ such that $K \coloneqq \argmin(\sigma - p)$ and $K' \coloneqq \argmin(\sigma - p')$ form maximal L-convex cones in $\fan(\sigma)$ and
      \begin{enumerate}[label={\textup{(\roman*)}}]
        \item $K = K'$; or
        \item $\dim(K \cap K') = \dim{K} - 1$.
      \end{enumerate}
      Indeed, this implies that $\inpr{p - p'}{x_*} \in \Z$ for any $p, p' \in \Mdom$ such that $\argmin(\sigma - p)$ and $\argmin(\sigma - p')$ form maximal L-convex cones in $\fan(\sigma)$,
      and hence, that $r(p) = r(p')$.

      In the following, we consider the above two cases (i) $K = K'$ and (ii) $\dim(K \cap K') = \dim{K} - 1$.
      Let $\CC_K$ be the convex subcomplex of $\spCC(\tilde{\RS}|_{-x_*})$ whose geometric realization is $K$.

      (i) $K = K'$.
      Let
      $A_K$ be an arbitrary maximal simplex of $\CC_K$,
      $A_*$ the corresponding simplex of $\Lk{\euCC}{x_*}$ to $A_K$,
      and $A$ the join of $A_*$ and $x_*$, i.e., the simplex such that $\sk{A} = \sk{A_*} \cup \{x_*\}$.
      We note that $A$ is a maximal simplex of $\euCC[\dom{\rho}]$.
      Suppose that $\esimp \coloneqq \{ \alpha_0, \alpha_1, \dots, \alpha_n \}$ is the extended simple system corresponding to some chamber $C$ of $\euCC$ (or $\St{\euCC}{x_*}$) with $A_* \preceq A \preceq C$,
    in which $\sk{C} = \{ x_0 ,x_1, \dots, x_n \}$
    and $\tau(x_k) = \tau_{\esimp}(\alpha_k) = v_k \in \tilde{D}(\RS)$.
    We note that $x_0 \in \Mdom^{\vee}(\RS)$ and $\Lk{\euCC}{x_0} - x_0 = \Lk{\euCC}{0} = \spCC^\star$,
    where the second equality follows from \Cref{lem:link-standard}.
    Since $\tau(x_*) = v_i$, we have $x_* = x_i$.
    Thus, we obtain that
    \begin{align}
        \inpr{p - p'}{x_0} &\in \Z,\label{eq:x0}\\
        \inpr{\alpha}{x_* - x_0} &=
        \begin{cases}
            1/a(x_*) & \text{if $x_* \neq x_0$ and $\alpha = \alpha_i$},\\
            -1 & \text{if $x_* \neq x_0$ and $\alpha = \alpha_0$},\\
            0 & \text{if $x_* = x_0$ or $\alpha \in \simp \setminus \{\alpha_i\}$}
        \end{cases}\label{eq:inpr:xi-x0}.
    \end{align}
    Here, \eqref{eq:x0} follows from $p-p' \in \Mdom$ and $x_0 \in \Mdom^{\vee}(\RS)$,
    and \eqref{eq:inpr:xi-x0} follows from \Cref{lem:inpr} and $\alpha_0 = -\sum_{k = 1}^n a(x_k) \alpha_k$.

    Since $\inpr{p - p'}{x} = 0$ for any $x \in K$,
    we have $p - p' \in (\spn{K})^{\perp} = (\spn{A})^{\perp} = \spn{(\esimp|_{-x_*})|_{-A_*}} = \spn{\esimp|_{-A}}$,
    where the second equality follows from \Cref{lem:span}~(1).
    Thus, $p - p' \in (\spn{\esimp|_{-A}}) \cap \Mdom$.
    We further divide the analysis into the following two cases: (i-1) $x_0 \in \sk{A}$ and (i-2) $x_0 \notin \sk{A}$.
    
    (i-1) $x_0 \in \sk{A}$.
    In this case, we have $\esimp|_{-A} \subseteq \simp \setminus \{ \alpha_i \}$,
    which implies that $\inpr{p-p'}{x_* - x_0} = 0$ by~\eqref{eq:inpr:xi-x0}.
    Thus, we obtain from~\eqref{eq:x0} that
    \begin{align}
        \inpr{p - p'}{x_*} = \inpr{p - p'}{x_* - x_0} + \inpr{p - p'}{x_0} = \inpr{p - p'}{x_0} \in \Z.
    \end{align}

    (i-2) $x_0 \notin \sk{A}$.
    In this case, we have $x_* \neq x_0$.
    Let $I \coloneqq \Set{ k \in [0,n] }{ x_k \in \sk{A} }$.
    We note that $i \in I \not\ni 0$ and $\spn{\esimp|_{-A}} = \spn{\Set{ \alpha_k \in \esimp }{ k \in [0,n] \setminus I }}$.
    Suppose that $p - p' = \sum_{k \in [0,n] \setminus I} c_k \alpha_k$.
    Then, we obtain
    \begin{align}
        \inpr{p - p'}{x_*} &= \inpr{p - p'}{x_* - x_0} + \inpr{p - p'}{x_0}\\
        &= \sum_{k \in [0,n] \setminus I} c_k\inpr{\alpha_k}{x_* - x_0} + \inpr{p - p'}{x_0}\\
        &= -c_0 + \inpr{p - p'}{x_0}.\label{eq:inpr:xi-x0:ii}
    \end{align}
    Here, since $p - p'$ is representable as
    \begin{align}
        p - p' &= c_0 \alpha_0 + \sum_{k \in [n] \setminus I} c_k \alpha_k\\
        &= \sum_{k \in [n] \setminus I} (c_k - a(x_k)c_0) \alpha_k - \sum_{k \in I} a(x_k) c_0 \alpha_k.
    \end{align}
    and $p - p' \in \Mdom$,
    we have $c_k - a(x_k)c_0 \in \Z$ for $k \in [n] \setminus I$, and $a(x_k) c_0 \in \Z$ for $k \in I$.
    By the assumption, there is $k \in I$ such that $a(x_k) = 1$.
    Hence, $c_0$ must be an integer by $a(x_k) c_0 \in \Z$.
    Thus, by~\eqref{eq:x0} and~\eqref{eq:inpr:xi-x0:ii},
    we obtain $\inpr{p - p'}{x_*} \in \Z$.

    (ii) $\dim(K \cap K') = \dim{K} - 1$.
    Let $A_K$ be a maximal simplex of $\CC_K$ such that there is a codimension-one face $A_K'$ of $A_K$ satisfying $A_K' \subseteq K \cap K'$,
    or equivalently, $\spn{A_K'} = K \cap K'$.
    Suppose that $\tau(A_K') = \tau(A_K) \setminus \{v_{j}\}$.
    We similarly define $A_*$, $A$, $C$, and $\esimp$ as in the case for (i).

    Since $\inpr{p - p'}{x} = 0$ for any $x \in K \cap K'$,
    we have $p - p' \in (\spn{K \cap K'})^{\perp} = (\spn{A_K'})^{\perp} = \spn{\esimp|_{-(\tau(A) \setminus \{v_{j}\})}}$.
    Thus, $p - p' \in \spn{\esimp|_{-(\tau(A) \setminus \{v_{j}\})}} \cap \Mdom$.
    By the same argument as in (i) above,
    we can see that $\inpr{p - p'}{x_*} \in \Z$,
    where we remark that, by the assumption, there is $x_k \in \sk{A} \setminus \{x_j\}$ such that $a(x_k) = 1$.
    \end{proof}

    Recall the definition~\eqref{eq:r} of $r$.
    Since $\rho(x_*) - r = \inpr{p}{x_*} + \lfloor \rho(x_*) - \inpr{p}{x_*}\rfloor$
    for some $p \in \Mdom$,
    we have $\rho(x_*) - r \in \Z/a(x_*)$ by \Cref{cor:inpr}~(1).

    Take any $x \in \dom{\rho}$.
    Let $K \in \fan(\sigma)$ be a maximal L-convex cone containing $x - x_*$
    and $p \in \Mdom$ with $K = \argmin(\sigma - p)$.
    Then, we have
    \begin{align}
        \rho(x) - r &= \rho(x) - (\rho(x_*) - \inpr{p}{x_*}) + \lfloor \rho(x_*) - \inpr{p}{x_*} \rfloor\\
        &= \inpr{p}{x} + \lfloor \rho(x_*) - \inpr{p}{x_*} \rfloor,
    \end{align}
    where the first and second equalities follow from \Cref{cl:r} and $\inpr{p}{x - x_*} = \rho(x) - \rho(x_*)$.
    Thus, by \Cref{cor:inpr}~(1),
    we obtain $\rho(x) - r \in \Z/a(x)$.

\section{Explicit descriptions of the discrete midpoint convexity}\label{subsec:DMC}
We provide the explicit discrete midpoint formulae on the Euclidean Coxeter complexes of type $\RSA$, $\RSB$, $\RSC$, and $\RSD$.

\begin{example}[Discrete midpoint on type A]\label{ex:complex_A}
    Recall that $\sk{\euCC(\RSA)} = \Z^{n+1} / \spn{\onevec_{n+1}}$; see~\Cref{tab:classical-coxeter-vertices}.
    Here,
    we can easily check that, for $x, x', y, y' \in \Z^{n+1}$ with $[x] = [x']$ and $[y] = [y']$,
    while $\sbra*{\floor{(x+y)/2}} \neq \sbra*{\floor{(x'+y')/2}}$ and $\sbra*{\ceil{(x+y)/2}} \neq \sbra*{\ceil{(x'+y')/2}}$ in general,
  the identity
  \begin{align}
    \set*{\sbra*{\floor*{\frac{x+y}{2}}}, \sbra*{\ceil*{\frac{x+y}{2}}}} = \set*{ \sbra*{\floor*{\frac{x'+y'}{2}}}, \sbra*{\ceil*{\frac{x'+y'}{2}}}}
    \end{align}
    always holds.
    Based on this fact,
    whenever $\sbra*{\floor{(x+y)/2}}$ and $\sbra*{\ceil{(x+y)/2}}$ appear symmetrically,
    we denote the two elements of the set $\set*{\sbra*{\floor{(x+y)/2}}, \sbra*{\ceil{(x+y)/2}} }$ by $\floor*{([x] + [y])/2}$ and $\ceil*{([x] + [y])/2}$,
    namely,
    \begin{align}
        \set*{\floor*{\frac{[x] + [y]}{2}}, \ceil*{\frac{[x] + [y]}{2}}} = \set*{\sbra*{\floor*{\frac{x+y}{2}}}, \sbra*{\ceil*{\frac{x+y}{2}}}}
    \end{align}
    for $[x], [y] \in \Z^{n+1}/\spn{\onevec_{n+1}} = \sk{\euCC(\RSA)}$.

  It follows from a direct calculation that
  \begin{align}
    \sk{\face{A}{([x]+[y])/2}} &= \set*{\floor*{\frac{[x]+[y]}{2}}, \ceil*{\frac{[x]+[y]}{2}}},\label{eq:DMC:vertex:A}\\
    \frac{[x]+[y]}{2} &= \frac{1}{2} \floor*{\frac{[x]+[y]}{2}} + \frac{1}{2} \ceil*{\frac{[x]+[y]}{2}}\label{eq:DMC:coeff:A}
\end{align}
for $[x], [y] \in \sk{\euCC(\RSA)}$.
Here, we note that $\floor{([x]+[y])/2}, \ceil{([x]+[y])/2}$ are well defined by the above argument.
\end{example}

\begin{example}[Discrete midpoint on type C]\label{ex:complex_C}
    Recall that $\sk{\euCC(\RSC)} = \Z^n$; see~\Cref{tab:classical-coxeter-vertices}.
    It follows from a direct calculation that
  \begin{align}
    \sk{\face{A}{(x+y)/2}} &= \set*{\even*{\frac{x+y}{2}}, \odd[\bigg]{\frac{x+y}{2}}},\label{eq:DMC:vertex:C}\\
    \frac{x+y}{2} &= \frac{1}{2} \even*{\frac{x+y}{2}} + \frac{1}{2} \odd[\bigg]{\frac{x+y}{2}}\label{eq:DMC:coeff:C}
    \end{align}
    for $x, y \in \sk{\euCC(\RSC)}$,
    where $(x+y)/2 \in (\Z/2)^n$.
\end{example}
Before considering the cases of $\RSB$ and $\RSD$,
we prepare several rounding operations as follows:
For $c \in \Z/4$,
we define
\begin{align}
    \qeven{c} &\coloneqq
    \even{2c}/2 =
    \begin{cases}
        c & \text{if $c \in \Z/2$},\\
        \text{the nearest integer to $c$} & \text{if $c \in (\Z/4) \setminus (\Z/2)$},
    \end{cases}\\
    \qodd{c} &\coloneqq
    \odd{2c}/2 =
    \begin{cases}
        c & \text{if $c \in \Z/2$},\\
        \text{the nearest half-integer to $c$} & \text{if $c \in (\Z/4) \setminus (\Z/2)$},
    \end{cases}\\
    \qqeven{c} &\coloneqq
    \begin{cases}
        c & \text{if $c \in \Z/2$},\\
        \qeven{c} - 1 & \text{if $c \in (\Z/4) \setminus (\Z/2)$ and $c < \qeven{c}$},\\
        \qeven{c} + 1 & \text{if $c \in (\Z/4) \setminus (\Z/2)$ and $c > \qeven{c}$},
    \end{cases}\\
    \qqodd{c} &\coloneqq
    \begin{cases}
        c & \text{if $c \in \Z/2$},\\
        \qodd{c} - 1 & \text{if $c \in (\Z/4) \setminus (\Z/2)$ and $c < \qodd{c}$},\\
        \qodd{c} + 1 & \text{if $c \in (\Z/4) \setminus (\Z/2)$ and $c > \qodd{c}$},
    \end{cases}\\
    \qfloor{c} &\coloneqq
    \begin{cases}
        \floor{c} & \text{if $c \in \Z + 1/2$},\\
        \qeven{c} & \text{otherwise, i.e., $c \in \Z \cup \rbra*{(\Z/4) \setminus (\Z/2)}$},
    \end{cases}\\
    \qceil{c} &\coloneqq
    \begin{cases}
        \ceil{c} & \text{if $c \in \Z + 1/2$},\\
        \qeven{c} & \text{otherwise, i.e., $c \in \Z \cup \rbra*{(\Z/4) \setminus (\Z/2)}$},
    \end{cases}\\
    \qqfloor{c} &\coloneqq
    \begin{cases}
        c & \text{if $c \in \Z + 1/2$},\\
        c - 1/2 & \text{if $c \in \Z$},\\
        \qodd{c} & \text{otherwise, i.e., $c \in (\Z/4) \setminus (\Z/2)$},
    \end{cases}\\
    \qqceil{c} &\coloneqq
    \begin{cases}
        c & \text{if $c \in \Z + 1/2$},\\
        c + 1/2 & \text{if $c \in \Z$},\\
        \qodd{c} & \text{otherwise, i.e., $c \in (\Z/4) \setminus (\Z/2)$}.
    \end{cases}
\end{align}
Note that, for $c \in \Z/4$, we have $\qfloor{c}, \qceil{c} \in \Z$ and $\qqfloor{c}, \qqceil{c} \in \Z + 1/2$.
These rounding operations are extended to vectors component-wise.
\begin{example}[Discrete midpoint on type B]\label{ex:complex_B}
    Recall that
    \begin{align}
        \sk{\euCC(\RSB)} &= \Set{ x \in (\Z/2)^n }{ \card{\Set{ i \in [n] }{ x(i) \in \Z + 1/2 }} \neq 1 }\\
        &= \Set{ x \in (\Z/2)^n }{ \card{\Set{ i \in [n] }{ x(i) \in \Z}} \neq n-1 };
    \end{align}
    see~\Cref{tab:classical-coxeter-vertices}.
    Let $x, y \in \sk{\euCC(\RSB)}$ and $z \coloneqq (x+y)/2 \in (\Z/4)^n$.
  Then, we can easily observe that
  \begin{align}
    \qeven{z} \notin \sk{\euCC(\RSB)} &\iff \card{\Set{i \in [n]}{z(i) \in \Z + 1/2}} = 1,\\
    \qodd{z} \notin \sk{\euCC(\RSB)} &\iff \card{\Set{i \in [n]}{z(i) \in \Z}} = n-1.
\end{align}
Thus, $\qeven{z}, \qodd{z} \notin \sk{\euCC(\RSB)}$ if and only if
$z \in (\Z/2)^n$ and $z \notin \sk{\euCC(\RSB)}$;
in this case, $\even{z}$ and $\odd{z}$ are well defined.

Based on the above observation and a direct calculation, we obtain
  \begin{align}
            \sk{\face{A}{z}} &=
            \begin{cases}
                \{\qeven{z}, \qodd{z}\} & \text{if $\qeven{z}, \qodd{z} \in \sk{\euCC(\RSB)}$},\\
                \{\qeven{z}, \qqeven{z}\} & \text{if $\qeven{z} \in \sk{\euCC(\RSB)} \not\ni \qodd{z}$},\\
                \{\qodd{z}, 
                \qfloor{z}, \qceil{z}\} & \text{if $\qeven{z} \notin \sk{\euCC(\RSB)} \ni \qodd{z}$},\\
                \{\even{z}, \odd{z}\} & \text{if $\qeven{z}, \qodd{z} \notin \sk{\euCC(\RSB)}$},
            \end{cases}\label{eq:DMC:vertex:B}\\
            z &=
            \begin{dcases}
                \frac{1}{2}\qeven{z} + \frac{1}{2}\qodd{z} & \text{if $\qeven{z}, \qodd{z} \in \sk{\euCC(\RSB)}$},\\
                \frac{3}{4}\qeven{z} + \frac{1}{4}\qqeven{z} & \text{if $\qeven{z} \in \sk{\euCC(\RSB)} \not\ni \qodd{z}$},\\
                \frac{1}{2}\qodd{z} + \frac{1}{4}\qfloor{z} + \frac{1}{4}\qceil{z} & \text{if $\qeven{z} \notin \sk{\euCC(\RSB)} \ni \qodd{z}$},\\
                \frac{1}{2}\even{z} + \frac{1}{2}\odd{z} & \text{if $\qeven{z}, \qodd{z} \notin \sk{\euCC(\RSB)}$}.
            \end{dcases}\label{eq:DMC:coeff:B}
        \end{align}
        Here, in the case of $\qeven{z}, \qodd{z} \notin \sk{\euCC(\RSB)}$, we can replace $\even{z}$ and $\odd{z}$ with $\floor{z}$ and $\ceil{z}$.
\end{example}

\begin{example}[Discrete midpoint on type D]\label{ex:complex_D}
    Recall that
  \begin{align}
    \sk{\euCC(\RSD)} &= \Set*{ x \in (\Z/2)^n }{ \card{\Set{ i \in [n] }{ x(i) \in \Z+1/2}} \notin \{1,n-1\} }\\
    &= \Set*{ x \in (\Z/2)^n }{ \card{\Set{ i \in [n] }{ x(i) \in \Z}} \notin \{1,n-1\} };
  \end{align}
  see~\Cref{tab:classical-coxeter-vertices}.
  Let $x, y \in \sk{\euCC(\RSD)}$ and $z \coloneqq (x+y)/2 \in (\Z/4)^n$.
  Then, we can easily observe that
  \begin{align}
    \qeven{z} \notin \sk{\euCC(\RSD)} &\iff \card{\Set{i \in [n]}{z(i) \in \Z + 1/2}} \in \{1,n-1\},\\
    \qodd{z} \notin \sk{\euCC(\RSD)} &\iff \card{\Set{i \in [n]}{z(i) \in \Z}} \in \{ 1,n-1 \}.
\end{align}
Thus, $\qeven{z}, \qodd{z} \notin \sk{\euCC(\RSD)}$ if and only if
one of the following conditions~\eqref{cond:N1}--\eqref{cond:N3} holds:
\begin{enumerate}[label={\textup{(N\arabic*)}}, ref={\textup{N\arabic*}}]
    \item $\card{\Set{i \in [n]}{z(i) \in \Z + 1/2}} = 1$ and $\card{\Set{i \in [n]}{z(i) \in \Z}} = n-1$.\label{cond:N1}
    \item $\card{\Set{i \in [n]}{z(i) \in \Z + 1/2}} = n-1$ and $\card{\Set{i \in [n]}{z(i) \in \Z}} = 1$.\label{cond:N2}
    \item $\card{\Set{i \in [n]}{z(i) \in \Z + 1/2}} = \card{\Set{i \in [n]}{z(i) \in \Z}} = 1$.\label{cond:N3}
\end{enumerate}
We note that $z \in (\Z/2)^n$ if $z$ satisfies~\eqref{cond:N1} or~\eqref{cond:N2}.

Based on the above observation and a direct calculation,
we obtain
  \begin{align}
            \sk{\face{A}{z}} &=
            \begin{cases}
                \{\qeven{z}, \qodd{z}\} & \text{if $\qeven{z}, \qodd{z} \in \sk{\euCC(\RSD)}$},\\
                \{\qeven{z}, \qfloor{z}, \qceil{z}\} & \text{if $\qeven{z} \in \sk{\euCC(\RSD)} \not\ni \qodd{z}$},\\
                \{\qodd{z}, \qqodd{z}\} & \text{if $\qeven{z} \notin \sk{\euCC(\RSD)} \ni \qodd{z}$},\\
                \{\even{z}, \odd{z}\} & \text{if $z$ satisfies \eqref{cond:N1}},\\
                \{\qqfloor{z}, \qqceil{z}\} & \text{if $z$ satisfies \eqref{cond:N2}},\\
                \{\qfloor{z}, \qceil{z}, \qqfloor{z}, \qqceil{z}\} & \text{if $z$ satisfies \eqref{cond:N3}},
            \end{cases}\label{eq:DMC:vertex:D}\\
            z &=
            \begin{dcases}
                \frac{1}{2}\qeven{z} + \frac{1}{2}\qodd{z} & \text{if $\qeven{z}, \qodd{z} \in \sk{\euCC(\RSD)}$},\\
                \frac{1}{2}\qeven{z} + \frac{1}{4} \qfloor{z} + \frac{1}{4} \qceil{z} & \text{if $\qeven{z} \in \sk{\euCC(\RSD)} \not\ni \qodd{z}$},\\
                \frac{3}{4}\qodd{z} + \frac{1}{4}\qqodd{z} & \text{if $\qeven{z} \notin \sk{\euCC(\RSD)} \ni \qodd{z}$},\\
                \frac{1}{2}\even{z} + \frac{1}{2}\odd{z} & \text{if $z$ satisfies \eqref{cond:N1}},\\
                \frac{1}{2}\qqfloor{z} + \frac{1}{2}\qqceil{z} & \text{if $z$ satisfies \eqref{cond:N2}},\\
                \frac{1}{4}\qfloor{z} + \frac{1}{4}\qceil{z} + \frac{1}{4}\qqfloor{z} + \frac{1}{4}\qqceil{z} & \text{if $z$ satisfies \eqref{cond:N3}}.
            \end{dcases}\label{eq:DMC:coeff:D}
        \end{align}
        Here, in the case where $z$ satisfies \eqref{cond:N1}, we can replace $\even{z}$ and $\odd{z}$ with $\floor{z}$ and $\ceil{z}$.
\end{example}

By the descriptions~\eqref{eq:DMC:vertex:A},~\eqref{eq:DMC:vertex:B},~\eqref{eq:DMC:vertex:C}, and~\eqref{eq:DMC:vertex:D} on the vertex sets of the simplex containing the midpoint of two vertices of $\Ldom$,
we can rephrase the condition (d) of \Cref{thm:chara:L-convex-set} for each case of types~A, B, C, and D as follows.
\begin{corollary}\label{cor:DMC:set}
A nonempty set $D \subseteq \Ldom$ is L-convex if and only if the following holds for each case of $\RS = \RSA$ (Type~A), $\RS =\RSB$ (Type~B), $\RS =\RSC$ (Type~C), and $\RS =\RSD$ (Type~D):
    \begin{description}[font=\normalfont]
        \item[(Type~A)] For any $[x], [y] \in D$, we have $\floor{([x]+[y])/2}, \ceil{([x]+[y])/2} \in D$.
        \item[(Type~B)] For any $x, y \in D$, we have
        \begin{align}
            \begin{cases}
                \qeven{z}, \qodd{z} \in D & \text{if $\qeven{z}, \qodd{z} \in \sk{\euCC(\RSB)}$},\\
                \qeven{z}, \qqeven{z} \in D & \text{if $\qeven{z} \in \sk{\euCC(\RSB)} \not\ni \qodd{z}$},\\
                \qodd{z}, 
                \qfloor{z}, \qceil{z} \in D & \text{if $\qeven{z} \notin \sk{\euCC(\RSB)} \ni \qodd{z}$},\\
                \even{z}, \odd{z} \in D & \text{if $\qeven{z}, \qodd{z} \notin \sk{\euCC(\RSB)}$},
            \end{cases}
        \end{align}
        where $z \coloneqq (x+y)/2$.
        \item[(Type~C)] For any $x,y \in D$, we have $\even{(x+y)/2}, \odd{(x+y)/2} \in D$.
        \item[(Type~D)] For any $x,y \in D$, we have
        \begin{align}
            \begin{cases}
                \qeven{z}, \qodd{z} \in D & \text{if $\qeven{z}, \qodd{z} \in \sk{\euCC(\RSD)}$},\\
                \qeven{z}, \qfloor{z}, \qceil{z} \in D & \text{if $\qeven{z} \in \sk{\euCC(\RSD)} \not\ni \qodd{z}$},\\
                \qodd{z}, \qqodd{z} \in D & \text{if $\qeven{z} \notin \sk{\euCC(\RSD)} \ni \qodd{z}$},\\
                \even{z}, \odd{z} \in D & \text{if $z$ satisfies \eqref{cond:N1}},\\
                \qqfloor{z}, \qqceil{z} \in D & \text{if $z$ satisfies \eqref{cond:N2}},\\
                \qfloor{z}, \qceil{z}, \qqfloor{z}, \qqceil{z} \in D & \text{if $z$ satisfies \eqref{cond:N3}},
            \end{cases}
        \end{align}
        where $z \coloneqq (x+y)/2$.
    \end{description}
\end{corollary}

Similarly, by the descriptions~\eqref{eq:DMC:coeff:A},~\eqref{eq:DMC:coeff:B},~\eqref{eq:DMC:coeff:C}, and~\eqref{eq:DMC:coeff:D} on the convex combinations representing the midpoint of two vertices of $\Ldom$,
we can rephrase the condition (d) of \Cref{thm:chara:L-conv-func} for each case of types~A, B, C, and D as follows.
\begin{corollary}\label{cor:DMC:func}
    A function $\funcdoms{f}{\Ldom}{\ER}$ with $\dom{f} \neq \emptyset$ is L-convex if and only if the following holds for each case of $\RS = \RSA$ (Type~A), $\RS =\RSB$ (Type~B), $\RS =\RSC$ (Type~C), and $\RS =\RSD$ (Type~D):
    \begin{description}[font=\normalfont]
        \item[(Type~A)] For any $[x], [y] \in \dom{f}$, we have
        \begin{align}
            f([x]) + f([y]) \geq f\rbra*{\floor*{\frac{[x]+[y]}{2}}} + f\rbra*{\ceil*{\frac{[x]+[y]}{2}}}
        \end{align}
        \item[(Type~B)] For any $x, y \in \dom{f}$, we have
        \begin{align}
            f(x) + f(y)
            \geq
            \begin{dcases}
                f(\qeven{z}) + f(\qodd{z}) & \text{if $\qeven{z}, \qodd{z} \in \sk{\euCC(\RSB)}$},\\
                \frac{3}{2}f(\qeven{z}) + \frac{1}{2}f(\qqeven{z}) & \text{if $\qeven{z} \in \sk{\euCC(\RSB)} \not\ni \qodd{z}$},\\
                f(\qodd{z}) + \frac{1}{2}f(\qfloor{z}) + \frac{1}{2}f(\qceil{z}) & \text{if $\qeven{z} \notin \sk{\euCC(\RSB)} \ni \qodd{z}$},\\
                f(\even{z}) + f(\odd{z}) & \text{if $\qeven{z}, \qodd{z} \notin \sk{\euCC(\RSB)}$}.
            \end{dcases}
        \end{align}
        where $z \coloneqq (x+y)/2$.
        \item[(Type~C)] For any $x,y \in \dom{f}$, we have
        \begin{align}
            f(x) + f(y) \geq f\rbra*{\even*{\frac{x+y}{2}}} + f\rbra*{\odd[\bigg]{\frac{x+y}{2}}}.
        \end{align}
        \item[(Type~D)] For any $x,y \in \dom{f}$, we have
        \begin{align}
            f(x) + f(y) \geq
            \begin{dcases}
                f(\qeven{z}) + f(\qodd{z}) & \text{if $\qeven{z}, \qodd{z} \in \sk{\euCC(\RSD)}$},\\
                f(\qeven{z}) + \frac{1}{2} f(\qfloor{z}) + \frac{1}{2} f(\qceil{z}) & \text{if $\qeven{z} \in \sk{\euCC(\RSD)} \not\ni \qodd{z}$},\\
                \frac{3}{2}f(\qodd{z}) + \frac{1}{2}f(\qqodd{z}) & \text{if $\qeven{z} \notin \sk{\euCC(\RSD)} \ni \qodd{z}$},\\
                f(\even{z}) + f(\odd{z}) & \text{if $z$ satisfies \eqref{cond:N1}},\\
                f(\qqfloor{z}) + f(\qqceil{z}) & \text{if $z$ satisfies \eqref{cond:N2}},\\
                \frac{1}{2}\rbra*{f(\qfloor{z}) + f(\qceil{z}) + f(\qqfloor{z}) + f(\qqceil{z})} & \text{if $z$ satisfies \eqref{cond:N3}},
            \end{dcases}
        \end{align}
        where $z \coloneqq (x+y)/2$.
    \end{description}
\end{corollary}

\end{document}